\documentclass{amsart}
\usepackage[a4paper, margin=3cm]{geometry}
\usepackage{amsthm}
\usepackage{amsmath,amssymb}
\usepackage{mathrsfs}
\usepackage{url}
\usepackage[all]{xy}
\usepackage{comment}
\makeindex
\newtheorem{theorem}{Theorem}[section]
\newtheorem*{theorem*}{Theorem}
\newtheorem{lemma}[theorem]{Lemma}
\newtheorem*{lemma*}{Lemma}
\newtheorem{proposition}[theorem]{Proposition}
\newtheorem*{proposition*}{Proposition}
\newtheorem{corollary}[theorem]{Corollary}
\newtheorem*{corollary*}{Corollary}
\newtheorem{claim}[theorem]{Claim}
\newtheorem*{claim*}{Claim}

\newtheorem*{fact*}{Fact}

\newtheorem*{conjecture*}{Conjecture}
\theoremstyle{definition}
\newtheorem{definition}[theorem]{Definition}
\newtheorem*{definition*}{Definition}

\newtheorem*{example*}{Example}
\newtheorem{remark}[theorem]{Remark}
\newtheorem*{remark*}{Remark}

\newtheorem*{question*}{Question}
\newtheorem{assumption}[theorem]{Assumption}
\newtheorem*{assumption*}{Assumption}
\numberwithin{equation}{section}
\allowdisplaybreaks[1]

\newcommand{\abs}[1]{\left\lvert#1\right\rvert}
\newcommand{\norm}[1]{\left\lVert#1\right\rVert}

\DeclareMathOperator{\End}{End}

\DeclareMathOperator{\Ad}{Ad}

\DeclareMathOperator{\ind}{ind}
\DeclareMathOperator{\Ind}{Ind}
\DeclareMathOperator{\vol}{vol}
\DeclareMathOperator{\Hom}{Hom}

\DeclareMathOperator{\red}{red}

\DeclareMathOperator{\Inf}{Inf}
\DeclareMathOperator{\cpt}{c}
\DeclareMathOperator{\id}{id}
\DeclareMathOperator{\im}{im}
\DeclareMathOperator{\unr}{unr}
\DeclareMathOperator{\ord}{ord}
\DeclareMathOperator{\supp}{supp}
\DeclareMathOperator{\ext}{ext}
\DeclareMathOperator{\aff}{aff}
\DeclareMathOperator{\Mor}{Mor}
\DeclareMathOperator{\Sol}{Sol}

\DeclareMathOperator{\otw}{otherwise}
\DeclareMathOperator{\reg}{reg}
\DeclareMathOperator{\terms}{terms}
\address{graduate school of mathematical science, the university of tokyo, 3-8-1 komaba, meguroku, tokyo 153-8914, japan.}
\email{kohara@ms.u-tokyo.ac.jp}
\begin{document}
\title{A comparison of endomorphism algebras}
\author{Kazuma Ohara}
\date{}
\maketitle
\begin{abstract}
Let $F$ be a non-archimedean local field and $G$ be a connected reductive group over $F$. 
For a Bernstein block in the category of smooth complex representations of $G(F)$, we have two kinds of progenerators:
the compactly induced representation $\ind_{K}^{G(F)} (\rho)$ of a type $(K, \rho)$, and the parabolically induced representation $I_{P}^{G}(\Pi^{M})$ of a progenerator $\Pi^{M}$ of a Bernstein block for a Levi subgroup $M$ of $G$.
In this paper, we construct an explicit isomorphism of these two progenerators.
Moreover, we compare the description of the endomorphism algebra $\End_{G(F)}\left(\ind_{K}^{G(F)} (\rho)\right)$ for a depth-zero type $(K, \rho)$ in \cite{MR1235019} with the description of the endomorphism algebra $\End_{G(F)}\left(I_{P}^{G}(\Pi^{M})\right)$ in \cite{MR4432237}, that are described in terms of affine Hecke algebras. 
\end{abstract}
\tableofcontents
\section{Introduction}
Let $F$ be a non-archimedean local field and $G$ be a connected reductive group over $F$. As explained in \cite{MR771671}, the category $\mathcal{R}(G(F))$ of smooth complex representations of $G(F)$ is decomposed into a product $\prod_{[M, \sigma]_{G}} \mathcal{R}^{[M, \sigma]_G}(G(F))$ of full subcategories $\mathcal{R}^{[M, \sigma]_G}(G(F))$, called Bernstein blocks. Bernstein blocks are parametrized by inertial equivalence classes $[M, \sigma]_{G}$ of cuspidal pairs, that consist of Levi subgroups $M$ of $G$ and irreducible supercuspidal representations $\sigma$ of $M(F)$.

One of the ways to study the block $\mathcal{R}^{\mathfrak{s}}(G(F))$ associated with an inertial equivalence class $\mathfrak{s}$ of cuspidal pairs is using the theory of types.
Let $\mathfrak{S}$ be a finite set of inertial equivalence classes of cuspidal pairs.
A pair $(K, \rho)$ of a compact open subgroup $K$ of $G(F)$ and an irreducible smooth representation $(\rho, V_{\rho})$ of $K$ is called an $\mathfrak{S}$-type if the full subcategory
\[
\mathcal{R}^{\mathfrak{S}}(G(F)) = \prod_{\mathfrak{s} \in \mathfrak{S}} \mathcal{R}^{\mathfrak{s}}(G(F))
\]
is precisely the full subcategory of $\mathcal{R}(G(F))$ consisting of smooth representations that are generated by their $\rho$-isotypic components.
In this case, $\mathcal{R}^{\mathfrak{S}}(G(F))$ is equivalent to the category of right modules over the endomorphism algebra $\End_{G(F)} \left(
\ind_{K}^{G(F)} (\rho)
\right)$ of the compactly induced representation $\ind_{K}^{G(F)} (\rho)$ of $\rho$ \cite[Theorem~4.3]{MR1643417}. 
In particular, if we obtain an $\{\mathfrak{s}\}$-type $(K, \rho)$ for an inertial equivalence class $\mathfrak{s}$ of cuspidal pairs, we can study the block $\mathcal{R}^{\mathfrak{s}}(G(F))$ by using the endomorphism algebra $\End_{G(F)} \left(
\ind_{K}^{G(F)} (\rho)
\right)$.
If a pair $(K, \rho)$ is an $\mathfrak{S}$-type for some finite set $\mathfrak{S}$ of inertial equivalence classes of cuspidal pairs, we say that $(K, \rho)$ is a type in $G$.
%In particular, if $\mathfrak{S}$ is a singleton $\{\mathfrak{s}\}$, and a pair $(K, \rho)$ is an $\mathfrak{S}$-type, we say that $(K, \rho)$ is an $\mathfrak{s}$-type.
Many kinds of types have been constructed, for instance, \cite{MR1204652} for $\text{GL}_{n}$, \cite{MR2390287, MR3157998} for classical groups, and \cite{MR1621409} for the principal series of split groups.

In \cite{MR1253198} and \cite{MR1371680}, Moy and Prasad defined the notion of depth of types and constructed types called \emph{depth-zero types}.
For a depth-zero type $(K, \rho)$, Morris proved that the endomorphism algebra $\End_{G(F)} \left(
\ind_{K}^{G(F)} (\rho)
\right)$ is isomorphic to an extension of an affine Hecke algebra $\mathcal{H}^{\Mor}$ by a twisted group algebra \cite[Theorem~7.12]{MR1235019}.
We note that a type $(K, \rho)$ considered in \cite{MR1235019} is not necessarily an $\mathfrak{S}$-type for a singleton $\mathfrak{S} = \{\mathfrak{s}\}$.

%gave the generators and relations for the endomorphism algebra $\End_{G(F)} \left(
%\ind_{K}^{G(F)} (\rho)
%\right)$ \cite[Theorem~7.12]{MR1235019}.
%More precisely, Morris proved that the endomorphism algebra $\End_{G(F)} \left(
%\ind_{K}^{G(F)} (\rho)
%\right)$ is isomorphic to an extension of an affine Hecke algebra $\mathcal{H}^{\Mor}$ by a twisted group algebra.
% $\mathbb{C}[C(J, \rho), \chi]$.

On the other hand, there is another approach to study a block.
Let $M$ be a Levi subgroup of $G$ and $\sigma$ be an irreducible supercuspidal representation of $M(F)$.
We fix a parabolic subgroup $P$ of $G$ with Levi factor $M$ and unipotent radical $U$.
Then, the Bernstein block $\mathcal{R}^{\mathfrak{s}}(G(F))$ associated with the inertial equivalence class $\mathfrak{s}$ of the pair $(M, \sigma)$ in $G$ can be studied as follows.
Let $\sigma_{1}$ be an irreducible subrepresentation of $\sigma\restriction_{M^1}$, where $M^1$ denotes the intersection of the kernels of unramified characters of $M(F)$.
Then, according to \cite[Theorem~1.8.1.1]{MR2508719}, the parabolically induced representation $I_{P}^{G}\left(\ind_{M^1}^{M(F)} (\sigma_{1})\right)$ of the compactly induced representation $\ind_{M^1}^{M(F)} (\sigma_{1})$ of $\sigma_{1}$ is a progenerator of $\mathcal{R}^{\mathfrak{s}}(G(F))$.
Hence, according to \cite[Theorem~1.8.2.1]{MR2508719}, $\mathcal{R}^{\mathfrak{s}}(G(F))$ is equivalent to the category of right modules over the endomorphism algebra $\End_{G(F)} \left(
I_{P}^{G}\left(\ind_{M^1}^{M(F)} (\sigma_{1})\right)
\right)$.
The structure of the endomorphism algebra $\End_{G(F)} \left(
I_{P}^{G}\left(\ind_{M^1}^{M(F)} (\sigma_{1})\right)
\right)$ is studied in \cite{MR4432237}.
Under some assumptions, Solleveld proved that the endomorphism algebra $\End_{G(F)} \left(
I_{P}^{G}\left(\ind_{M^1}^{M(F)} (\sigma_{1})\right)
\right)$ is isomorphic to an extension of an affine Hecke algebra $\mathcal{H}^{\Sol}$ by a twisted group algebra \cite[Theorem~10.9]{MR4432237}.
% $\mathbb{C}[R(\mathfrak{s}_{M}), \chi']$ \cite[Theorem~10.9]{MR4432237}.

In this paper, we compare the endomorphism algebra $\End_{G(F)} \left(
\ind_{K}^{G(F)} (\rho)
\right)$ associated with a type $(K, \rho)$ with the endomorphism algebra $\End_{G(F)} \left(
I_{P}^{G}\left(\ind_{M^1}^{M(F)} (\sigma_{1})\right)
\right)$ obtained from the progenerator $I_{P}^{G}\left(\ind_{M^1}^{M(F)} (\sigma_{1})\right)$.
In particular, we compare the description of $\End_{G(F)} \left(
\ind_{K}^{G(F)} (\rho)
\right)$ in \cite[Theorem~7.12]{MR1235019} with the description of $\End_{G(F)} \left(
I_{P}^{G}\left(\ind_{M^1}^{M(F)} (\sigma_{1})\right)
\right)$ in \cite[Theorem~10.9]{MR4432237}.

We explain the main results of this paper briefly.
Let $(K, \rho)$ be a depth-zero type considered in \cite{MR1235019}.
We suppose Assumption~\ref{parahoric=stabilizer} and Assumption~\ref{parahoricvermultiplicity1} on $(K, \rho)$.
Assumption~\ref{parahoric=stabilizer} is necessary for the type $(K, \rho)$ to be an $\mathfrak{S}$-type for a singleton $\mathfrak{S} = \{\mathfrak{s}\}$.
Assumption~\ref{parahoricvermultiplicity1} is essentially the same as \cite[Working hypothesis~10.2]{MR4432237}, that is supposed in \cite[Theorem~10.9]{MR4432237}.
In many cases, these assumptions are satisfied (see Remark~\ref{remarkaboutassumptionmultiplicity1} and the paragraph following Assumption~\ref{parahoric=stabilizer}).
From the type $(K, \rho)$, we can define a Levi subgroup $M$.
We write $K_{M} = K \cap M(F)$ and $\rho_{M} = \rho\restriction_{K_{M}}$.
Then, we can prove that $(K, \rho)$ is a $G$-cover of $(K_{M}, \rho_{M})$ in the sense of \cite[Definition~8.1]{MR1643417}.
We construct an explicit isomorphism
\[
I_{U} \colon \ind_{K}^{G(F)} (\rho) \rightarrow I^{G}_{P} \left( \ind_{K_{M}}^{M(F)} (\rho_{M}) \right).
\]
Thus, we also have an isomorphism of endomorphism algebras
\[
I_{U} \colon \End_{G(F)}\left(\ind_{K}^{G(F)} (\rho)\right) \rightarrow \End_{G(F)}\left(I^{G}_{P} \left( \ind_{K_{M}}^{M(F)} (\rho_{M}) \right)\right).
\]
In this part, we need not to suppose Assumption~\ref{parahoric=stabilizer}, Assumption~\ref{parahoricvermultiplicity1}, or even that $(K, \rho)$ is a depth-zero type.
Hence, we can apply the result to any $G$-cover $(K, \rho)$ of a pair $(K_{M}, \rho_{M})$.
We also prove that the isomorphism $I_{U}$ is compatible with the injections
\[
I_{P}^{G} \colon \End_{M(F)}\left( \ind_{K_{M}}^{M(F)} (\rho_{M}) \right) \rightarrow \End_{G(F)}\left(I^{G}_{P} \left( \ind_{K_{M}}^{M(F)} (\rho_{M}) \right)\right)
\]
and
\[
t_{P} \colon \End_{M(F)}\left( \ind_{K_{M}}^{M(F)} (\rho_{M}) \right) \rightarrow \End_{G(F)}\left(\ind_{K}^{G(F)} (\rho)\right)
\]
defined in \cite[Corollary~7.12]{MR1643417}.

We define an irreducible supercuspidal representation $\sigma$ of $M(F)$ such that the compactly induced representation
$
\ind_{K_{M}}^{M^1} (\rho_{M})
$
is an irreducible subrepresentation of $\sigma\restriction_{M^1}$.
We write $\sigma_{1} =~ \ind_{K_{M}}^{M^1} (\rho_{M})$.
Then, the transitivity of the compact induction implies
\[
\ind_{K_{M}}^{M(F)} (\rho_{M}) \simeq \ind_{M^1}^{M(F)} (\sigma_{1}).
\]
Hence, we have an isomorphism
\[
T_{\rho_{M}} \colon \End_{G(F)}\left(I^{G}_{P} \left( \ind_{K_{M}}^{M(F)} (\rho_{M}) \right)\right) \rightarrow \End_{G(F)}\left(I^{G}_{P} \left( \ind_{M^1}^{M(F)} (\sigma_{1}) \right)\right).
\]
Composing it with $I_{U}$, we have an isomorphism
\[
T_{\rho_{M}} \circ I_{U} \colon \End_{G(F)}\left(\ind_{K}^{G(F)} (\rho)\right) \rightarrow \End_{G(F)}\left(I^{G}_{P} \left( \ind_{M^1}^{M(F)} (\sigma_{1}) \right)\right).
\]

The main purpose of this paper is to compare the description of the left hand side of $T_{\rho_{M}} \circ I_{U}$ in \cite{MR1235019} with the description  of the right hand side of $T_{\rho_{M}} \circ I_{U}$ in \cite{MR4432237}.
The left hand side of $T_{\rho_{M}} \circ I_{U}$ is described in terms of an affine Hecke algebra $\mathcal{H}^{\Mor}$ associated with a based root datum 
\[
\mathcal{R}^{\Mor} = \left(
\Hom_{\mathbb{Z}} \left(
\mathbb{Z} (R^{\Mor})^{\vee}, \mathbb{Z}
\right), R^{\Mor},
\mathbb{Z} (R^{\Mor})^{\vee},  (R^{\Mor})^{\vee}, \Delta^{\Mor}
\right).
\]
More precisely, there is a subalgebra $\mathcal{H}(R(J, \rho))$ of $\End_{G(F)}\left(\ind_{K}^{G(F)} (\rho)\right)$ and an isomorphism
\[
I^{\Mor} \colon \mathcal{H}(R(J, \rho)) \rightarrow \mathcal{H}^{\Mor}.
\]
On the other hand, the right hand side of $T_{\rho_{M}} \circ I_{U}$ is described in terms of an affine Hecke algebra $\mathcal{H}^{\Sol}$ associated with a based root datum
\[
\mathcal{R}^{\Sol}
=
\left(
\left( M_{\sigma}/M^{1} \right)^{\vee}, R^{\Sol}, M_{\sigma}/M^{1}, (R^{\Sol})^{\vee}, \Delta^{\Sol}
\right).
\]
We have a subalgebra $\mathcal{H}\left(W(\Sigma_{\mathfrak{s}_{M}, \mu})\right)$ of $\End_{G(F)}\left(I^{G}_{P} \left( \ind_{M^1}^{M(F)} (\sigma_{1}) \right)\right)$ and an isomorphism
\[
I^{\Sol} \colon \mathcal{H}\left(W(\Sigma_{\mathfrak{s}_{M}, \mu})\right) \rightarrow \mathcal{H}^{\Sol}.
\]

The first main result of this paper is as follows:
\begin{theorem}[Theorem~\ref{maintheoremrootsystem}]
There is a canonical identification 
\[
\begin{cases}
R^{\Mor} &= R^{\Sol}, \\
\Delta^{\Mor} &= - \Delta^{\Sol}.
\end{cases}
\]
\end{theorem}
Hence, we can identify the Weyl group $W_{0}(R^{\Mor})$ of $R^{\Mor}$ with the Weyl group $W_{0}(R^{\Sol})$ of $R^{\Sol}$.
We can also identify the set of simple reflections in $W_{0}(R^{\Mor})$ corresponding to the basis $\Delta^{\Mor}$ of $R^{\Mor}$ with the set of simple reflections in $W_{0}(R^{\Sol})$ corresponding to the basis $\Delta^{\Sol}$ of $R^{\Sol}$. 
For a simple reflection
\[
s \in W_{0}(R^{\Mor}) = W_{0}(R^{\Sol}),
\]
let $T_{s}^{\Mor}$ denote the element of the standard basis of $\mathcal{H}^{\Mor}$ corresponding to $s$, and let $T_{s}^{\Sol}$ denote the element of the standard basis of $\mathcal{H}^{\Sol}$ corresponding to $s$.
The second main result of this paper is as follows:
\begin{theorem}[Theorem~\ref{maintheoremisomofaffinehecke}]
There is an explicitly defined involution
\[
\iota \colon \mathcal{H}^{\Sol} \rightarrow \mathcal{H}^{\Sol}
\]
such that
\[
\left(
\iota \circ I^{\Sol} \circ T_{\rho_{M}} \circ I_{U} \circ (I^{\Mor})^{-1}
\right)
(T_{s}^{\Mor}) = T_{s}^{\Sol}
\]
for any simple reflection $s$ associated with 
\[
\alpha \in \Delta^{\Sol} = - \Delta^{\Mor}
\]
that is not the unique simple root in a type $A_{1}$ irreducible component of $R^{\Sol}$ or a long root in a type $C_{n} \ (n \ge 2)$ irreducible component of $R^{\Sol}$.
\end{theorem}
We also describe the image of $T^{\Mor}_{s}$ for a simple reflection $s$ associated with
\[
\alpha \in \Delta^{\Sol} = - \Delta^{\Mor}
\]
that is the unique simple root in a type $A_{1}$ irreducible component of $R^{\Sol}$ or a long root in a type $C_{n} \ (n \ge 2)$ irreducible component of $R^{\Sol}$ (see Theorem~\ref{maintheoremisomofaffinehecke}).

For
\[
\alpha \in R^{\Mor} = R^{\Sol},
\]
let $\theta_{- \alpha^{\vee}}$ denote the element of the group algebra $\mathbb{C}[\mathbb{Z} (R^{\Mor})^{\vee}]$ that corresponds to $- \alpha^{\vee} \in (R^{\Mor})^{\vee}$.
We note that $\mathbb{C}[\mathbb{Z} (R^{\Mor})^{\vee}]$ is a subalgebra of the affine Hecke algebra $\mathcal{H}^{\Mor}$. 
We also write $\theta_{\alpha^{\vee}}$ for the element of the group algebra
\[
\mathbb{C}[\mathbb{Z} (R^{\Sol})^{\vee}] \subset \mathcal{H}^{\Sol}
\]
corresponding to $\alpha^{\vee} \in (R^{\Sol})^{\vee}$.
Then, we also prove:
\begin{theorem}[Corollary~\ref{maintheoremc=1}]
For
\[
\alpha \in R^{\Mor} = R^{\Sol},
\]
we have
\[
\left(
\iota \circ I^{\Sol} \circ T_{\rho_{M}} \circ I_{U} \circ (I^{\Mor})^{-1}
\right)(\theta_{- \alpha^{\vee}}) = \theta_{\alpha^{\vee}}.
\]
\end{theorem}

We sketch the outline of this paper.
In Section~\ref{An explicit isomorphism}, we construct an isomorphism
\begin{align}
\label{isomorphismIUinintroduction}
I_{U} \colon \End_{G(F)}\left(\ind_{K}^{G(F)} (\rho)\right) \rightarrow \End_{G(F)}\left(I^{G}_{P} \left( \ind_{K_{M}}^{M(F)} (\rho_{M}) \right)\right)
\end{align}
for a $G$-cover $(K, \rho)$ of $(K_{M}, \rho_{M})$.
In Section~\ref{Hecke algebra injections}, we prove that isomorphism~\eqref{isomorphismIUinintroduction} is compatible with the injections
\[
I_{P}^{G} \colon \End_{M(F)}\left( \ind_{K_{M}}^{M(F)} (\rho_{M}) \right) \rightarrow \End_{G(F)}\left(I^{G}_{P} \left( \ind_{K_{M}}^{M(F)} (\rho_{M}) \right)\right)
\]
and
\[
t_{P} \colon \End_{M(F)}\left( \ind_{K_{M}}^{M(F)} (\rho_{M}) \right) \rightarrow \End_{G(F)}\left(\ind_{K}^{G(F)} (\rho)\right).
\]
In Section~\ref{The case of depth-zero types}, we review the description of the endomorphism algebra $\End_{G(F)}\left(\ind_{K}^{G(F)} (\rho)\right)$ for a depth-zero type $(K, \rho)$ in \cite{MR1235019}.
We also rewrite the description in terms of an affine Hecke algebra.
In Section~\ref{A review of Solleveld's results}, we review the description of the endomorphism algebra $\End_{G(F)}\left(I_{P}^{G}\left(\ind_{M^1}^{M(F)}(\sigma_{1})\right)\right)$ in \cite{MR4432237}.
In Section~\ref{Statements of main results}, we explain how to connect the right hand side of isomorphism~\eqref{isomorphismIUinintroduction} with an object of Section~\ref{A review of Solleveld's results} and state the main results of this paper.
In Section~\ref{Some lemmas about elements of}, we prepare some lemmas to prove the main results.
In Section~\ref{Comparison of Morris and Solleveld's endomorphism algebras : maximal case}, we prove the main results in case that $M$ is a maximal proper Levi subgroup of $G$.
Finally, in Section~\ref{Comparison of Morris and Solleveld's endomorphism algebras: general case}, we prove the main results for general cases.
\subsection*{Acknowledgment}
I am deeply grateful to my supervisor Noriyuki Abe for his enormous support and helpful advice. He checked the draft and gave me useful comments. I am supported by the FMSP program at Graduate School of Mathematical Sciences, the University of Tokyo and JSPS KAKENHI Grant number JP22J22712.
\section{Notation and assumptions}
Let $F$ be a non-archimedean local field of residue characteristic $p$, and let $k_F$ denote its residue field.
We write $q_{F} = \abs{k_{F}}$.
Let $\ord_{F}$ denote the unique discrete valuation on $F^{\times}$ such that the image of $\ord_{F}$ is $\mathbb{Z}$.
%We fix an algebraic closure $\overline{F}$ of $F$ and write $F^{\text{sep}}$ for the separable closure of $F$ in $\overline{F}$. We denote by $\overline{k_F}$ the residue field of $F^{\text{sep}}$.
%We write $F^{u}$ for the maximal unramified extension of $F$. For an algebraic field extension $E$ over $F$, we write $\mathcal{O}_{E}$ for its ring of integers.
%We denote by $\Gamma_F$ the absolute Galois group $\text{Gal}(F^{\text{sep}}/F)$ of $F$, and by $W_F$ the Weil group of $F$.
%We also denote by $I_F$ the inertia subgroup of $\Gamma_F$, by $P_F$ the wild inertia subgroup of $\Gamma_F$, and by $\text{Frob}$ the geometric Frobenius element in $\Gamma_F/I_F$.
%For $r\ge 0$, let $W_{F}^{r}$ be the $r$-th subgroup of $I_F$ in its upper numbering filtration computed with respect to the unique discrete valuation $\text{ord}_{F}$ on $F^{\times}$ with the value group $\mathbb{Z}$.
%For an $F$-split torus $A$, let $X^{*}(A)$ denote the root lattice and $X_{*}(A)$ denote the cocharacter lattice.

Let $G$ be a connected reductive group defined over $F$.
For a connected reductive group $H$, especially for a Levi subgroup of $G$, let $X_{\unr}(H)$ denote the set of unramified characters of $H(F)$, and let 
\[
H^1 = \bigcap_{\chi \in X_{\unr}(H)} \ker(\chi).
\]
For a parabolic subgroup $P$ of $G$ with Levi factor $M$ and unipotent radical $U$, let $\overline{P}$ denote the opposite parabolic subgroup of $P$ and $\overline{U}$ denote the unipotent radical of $\overline{P}$.
We define the modular function
\[
\delta_{P} \colon M(F) \rightarrow \mathbb{R}_{>0}
\]
as \cite[II.3.7]{MR2567785}.
Hence, for any compactly supported smooth function $f$ on $U(F)$, $m \in M(F)$, and Haar measure $du$ on $U(F)$, we have
\[
\int_{U(F)} f(mum^{-1}) \, du = \delta_{P}(m) \int_{U(F)} f(u) du.
\]
For a smooth representation $(\pi, V)$ of $G(F)$, let $(\pi_{U}, V_{U})$ denote the (un-normalized) Jacquet module of $(\pi, V)$ with respect to $P$, and let
\[
j_{U}(\pi) \colon V \to V_{U}
\] 
denote the canonical quotient map.
For a smooth representation $(\tau, W)$ of $M(F)$, let 
\[
\left(\Ind_{P}^{G}(\tau), \Ind_{P}^{G}(W)\right)
\]
denote the (un-normalized) parabolically induced representation of $(\tau, W)$ with respect to $P$.
Here, we realize $\Ind_{P}^{G}(\tau)$ as the right regular representation on
\[
\Ind_{P}^{G}(W) = \left\{
f \colon G(F) \rightarrow W \colon \text{smooth} \mid 
f(umg) = \tau(m) \cdot f(g) \, (u \in U(F), m \in M(F), g \in G(F))
\right\}.
\]
We write the normalized Jacquet functor and the normalized parabolic induction functor as $r_{P}^{G}$ and $I_{P}^{G}$, respectively.
Hence, for a smooth representation $(\pi, V)$ of $G(F)$,
\[
r_{P}^{G}(\pi) = \pi_{U} \otimes \delta_{P}^{1/2},
\]
and for a smooth representation $(\tau, W)$ of $M(F)$, 
\[
I_{P}^{G}(\tau) = \Ind_{P}^{G}(\tau \otimes \delta_{P}^{-1/2}).
\]

Let $K$ be an open subgroup of a locally profinite group $H$.
For a smooth representation $(\rho, V_{\rho})$ of $K$, let 
\[
\left(
\ind_{K}^{H} (\rho), \ind_{K}^{H} (V_{\rho})
\right)
\]
denote the compactly induced representation of $(\rho, V_{\rho})$.
Here, we realize $\ind_{K}^{H}(\rho)$ as the right regular representation on
\[
\ind_{K}^{H}(V_{\rho}) = \left\{
f \colon H \rightarrow V_{\rho} \colon \text{compactly supported modulo $K$} \mid 
f(kg) = \rho(k) \cdot f(g) \, (k \in K, g \in G(F))
\right\}.
\]

Let $K$ be a compact open subgroup of a locally profinite group $H$.
For a smooth representation $(\pi, V)$ of $H$ and an irreducible smooth representation $(\rho, V_{\rho})$ of $K$, let $V^{(K, \rho)}$ denote the $(K, \rho)$-isotypic subspace of $V$.
If $\rho$ is the trivial representation of $K$, we simply write $V^{K}$ for $V^{(K, \rho)}$. 

For any smooth representation $(\rho, V_{\rho})$ of a locally profinite group, let $(\rho^{\vee}, V_{\rho}^{\vee})$ denote the contragredient representation of $(\rho, V_{\rho})$.

Suppose that $K$ is a subgroup of a group $H$ and $h \in H$. 
Let $^hK$ denote the subgroup $hKh^{-1}$ of $H$.
If $\rho$ is a representation of $K$, $^h\!\rho$ denotes the representation $x\mapsto \rho(h^{-1}xh)$ of $^hK$. 
We sometimes write $h \rho$ for $^h\!\rho$.
If $\text{Hom}_{K\cap ^h\!K}(^h\!\rho, \rho)$ is non-zero, we say $h$ \emph{intertwines} $\rho$.
We write
\[
I_{H}(\rho) = \{
h \in H \mid \text{$h$ intertwines $\rho$}
\}.
\]

For a group $H$, let $\mathbb{C}[H]$ denote the group algebra of $H$ over $\mathbb{C}$ and
\[
\{
\theta_{h} \mid h \in H
\}
\]
denote the standard basis of $\mathbb{C}[H]$.

For a vector space $V$ over a field $\mathbb{K}$, let $V^{*}$ denote the dual vector space
\[
V^{*} = \Hom_{\mathbb{K}}(V, \mathbb{K}).
\]
\section{An explicit isomorphism}
\label{An explicit isomorphism}
We recall the definition of $G$-covers.
Let $M$ be a Levi subgroup of $G$ and $K$ be a compact open subgroup of $G(F)$.
We write $K_{M} = K \cap M(F)$.
For a parabolic subgroup $P$ with Levi factor $M$ and unipotent radical $U$, we also write $K_{U} = K \cap U(F)$ and $K_{\overline{U}} = K \cap \overline{U}(F)$.
We say that $K$ decomposes with respect to $U, M, \overline{U}$ if
\[
K= K_{U} \cdot K_{M} \cdot K_{\overline{U}}.
\]
Let $(\rho, V_{\rho})$ be an irreducible smooth representation of $K$ and $(\rho_{M}, V_{\rho_{M}})$ be an irreducible smooth representation of $K_{M}$.
The pair $(K, \rho)$ is called a $G$-cover of $(K_{M}, \rho_{M})$ if for any parabolic subgroup $P=MU$ with Levi factor $M$, we have
\begin{enumerate}
\item $K$ decomposes with respect to $U, M, \overline{U}$.
\item $K_{U}$ and $K_{\overline{U}}$ are contained in the kernel of $\rho$, and $\rho\restriction_{K_{M}} = \rho_{M}$.
\item For any irreducible smooth representation $(\pi, V)$ of $G(F)$, the restriction of $j_{U}(\pi)$ to $V^{(K, \rho)}$ is an injection.
\end{enumerate}
The notion of $G$-covers is originally introduced in \cite[Definition~8.1]{MR1643417}.
Here, we use a reformation given in \cite[Th{\'e}or{\`e}me 1]{MR1490115} (see also \cite[Section~4.1]{MR1901371} and \cite[Section~4.2]{MR3753917}).
In the presence of (1) and (2), the third condition is equivalent to the condition below (see \cite[Proposition~7.14]{MR1643417}):
\begin{description}
\item[(3')] For any smooth representation $(\pi, V)$ of $G(F)$, $j_{U}(\pi)$ induces an isomorphism
\[
V^{(K, \rho)} \rightarrow V_{U}^{(K_{M}, \rho_{M})}.
\]
\end{description}
We note that if $(K, \rho)$ is a $G$-cover of $(K_{M}, \rho_{M})$, the representation space $V_{\rho}$ of $\rho$ is equal to the representation space $V_{\rho_{M}}$ of $\rho_{M}$.

The following Lemma will be used below:
\begin{lemma}
\label{contraofcover}
Let $(K, \rho)$ be a $G$-cover of $(K_{M}, \rho_{M})$. Then, $(K, \rho^{\vee})$ is a $G$-cover of $(K_{M}, (\rho_{M})^{\vee})$.
\end{lemma}
\begin{proof}
It is obvious that $(K, \rho^{\vee})$ satisfies the first two conditions of $G$-covers.
We will prove that $(K, \rho^{\vee})$ satisfies the third condition.
Let $(\pi, V)$ be an irreducible (hence admissible) smooth representation of $G(F)$. 
We write $\langle , \rangle$ for the canonical $G(F)$-invariant pairing on $V \times V^{\vee}$.
Then, $\langle , \rangle$ restricts to a perfect pairing on 
\[
V^{(K, \rho)} \times (V^{\vee})^{(K, \rho^{\vee})}.
\]
%Since $(\pi, V)$ is irreducible, $(\pi, V)$ is admissible, and the map
%\[
%v \mapsto [\check{v} \mapsto \langle v, \check{v} \rangle]
%\]
%gives an isomorphism
%\[
%V \rightarrow \check{\check{V}}.
%\]
%Moreover, this isomorphism restricts to an isomorphism
%\[
%V^{(K, \rho)} \rightarrow (\check{V}^{(K, \check{\rho})})^{*} := \Hom_{\mathbb{C}}(\check{V}^{(K, \check{\rho})}, \mathbb{C}).
%\]
On the other hand, for any parabolic subgroup $P$ with Levi factor $M$ and unipotent radical $U$, we can define a canonical perfect pairing $\langle, \rangle_{U}$ on 
\[V_{U}^{(K_{M}, \rho_{M})} \times \left(V^{\vee}\right)_{\overline{U}}^{(K_{M}, (\rho_{M})^{\vee})}\]
as follows.

Let $K^{+}$ denote the kernel of $\rho$ and $K_{M}^{+}$ denote the kernel of $\rho_{M}$. Since $(K, \rho)$ is a $G$-cover of $(K_{M}, \rho_{M})$, we obtain
\[
K^{+} = K_{U} \cdot K_{M}^{+} \cdot K_{\overline{U}}.
\]
According to \cite[Th{\'e}or{\`e}me~VI.6.1]{MR2567785}, $j_{U}(\pi)$ induces a surjection
\[
V^{K^{+}} \rightarrow V_{U}^{K_{M}^{+}}.
\]
Moreover, according to \cite[Proposition~~VI.6.1]{MR2567785}, this surjection has a canonical section
\[
s_{P}^{K^{+}} \colon V_{U}^{K_{M}^{+}} \rightarrow V^{K^{+}},
\]
hence we obtain a decomposition
\[
V^{K^{+}} = \im(s_{P}^{K^{+}}) \oplus \ker(j_{U}(\pi)).
\]
Similarly, there exists a canonical section
\[
s_{\overline{P}}^{K^{+}} \colon \left(V^{\vee}\right)_{\overline{U}}^{K_{M}^{+}} \rightarrow (V^{\vee})^{K^{+}}
\]
of the surjection
\[
j_{\overline{U}}(\pi^{\vee}) \colon (V^{\vee})^{K^{+}} \rightarrow \left(V^{\vee}\right)_{\overline{U}}^{K_{M}^{+}},
\]
and we obtain a decomposition
\[
(V^{\vee})^{K^{+}} = \im(s_{\overline{P}}^{K^{+}}) \oplus \ker(j_{\overline{U}}(\pi^{\vee})).
\]
Moreover, $\im(s_{P}^{K^{+}})$ is orthogonal to $\ker(j_{\overline{U}}(\pi^{\vee}))$, and $\im(s_{\overline{P}}^{K^{+}})$ is orthogonal to $\ker(j_{U}(\pi))$ with respect to $\langle , \rangle$ (see the proof of \cite[Proposition~VI.9.6]{MR2567785}).
Hence, the pairing $\langle , \rangle$ restricts to a perfect pairing on 
\[
\im(s_{P}^{K^{+}}) \times \im(s_{\overline{P}}^{K^{+}}),
\]
and
\begin{align}
\label{pairing}
([v], [\check{v}]) \mapsto \langle 
s_{P}^{K^{+}}([v]) , s_{\overline{P}}^{K^{+}}([\check{v}])
\rangle
\end{align}
defines a perfect pairing on
\[
\left(V_{U}\right)^{K_{M}^{+}} \times \left(V^{\vee}\right)_{\overline{U}}^{K_{M}^{+}}.
\]
We define the perfect pairing $\langle, \rangle_{U}$ on 
\[\left(V_{U}\right)^{(K_{M}, \rho_{M})} \times \left(V^{\vee}\right)_{\overline{U}}^{(K_{M}, (\rho_{M})^{\vee})}\]
as the restriction of \eqref{pairing}.

%Since $(K, \rho)$ is a $G$-cover of $(K_{M}, \rho_{M})$, condition (3') implies that the natural map from $V$ to its Jacquet module $V_{U}$ induces an isomorphism
%\[
%V^{(K, \rho)} \rightarrow \left(V_{U}\right)^{(K_{M}, \rho_{M})},
%\]
%and the inverse map is given by $s_{P}^{K^{+}}$.
Now, we obtain isomorphisms
\begin{align*}
(V^{\vee})^{(K, \rho^{\vee})} & \simeq \left( V^{(K, \rho)} \right)^{*}\\
& \simeq \left(V_{U}^{(K_{M}, \rho_{M})} \right)^{*}\\
& \simeq \left(V^{\vee}\right)_{\overline{U}}^{(K_{M}, (\rho_{M})^{\vee})}.
\end{align*}
Here, the first isomorphism is given by the perfect pairing $\langle , \rangle$ on
\[
V^{(K, \rho)} \times (V^{\vee})^{(K, \rho^{\vee})},
\] 
the second isomorphism is given by the isomorphism
\[
j_{U}(\pi) \colon V^{(K, \rho)} \rightarrow V_{U}^{(K_{M}, \rho_{M})}
\]
of condition (3') of $G$-covers, and the third isomorphism is given by the perfect pairing $\langle, \rangle_{U}$ on
\[V_{U}^{(K_{M}, \rho_{M})} \times \left(V^{\vee}\right)_{\overline{U}}^{(K_{M}, (\rho_{M})^{\vee})}.\]
The construction of $\langle, \rangle_{U}$ implies that the composition of these isomorphisms coincides with the map
\[
j_{\overline{U}}(\pi^{\vee}) \colon (V^{\vee})^{(K, \rho^{\vee})} \rightarrow \left(V^{\vee}\right)_{\overline{U}}^{(K_{M}, (\rho_{M})^{\vee})}.
\]
Thus, we have proved that for any irreducible smooth representation $(\pi, V)$ of $G(F)$ and parabolic subgroup $P$ with Levi factor $M$ and unipotent radical $U$, $j_{\overline{U}}(\pi^{\vee})$ induces an isomorphism
\[
(V^{\vee})^{(K, \rho^{\vee})} \rightarrow \left(V^{\vee}\right)_{\overline{U}}^{(K_{M}, (\rho_{M})^{\vee})}.
\]
In particular, the restriction of $j_{\overline{U}}(\pi^{\vee})$ to $(V^{\vee})^{(K, \rho^{\vee})}$ is an injection.
Since 
\[
(\pi, V) \leftrightarrow (\pi^{\vee}, V^{\vee})
\]
gives a bijection of the set of irreducible smooth representations of $G(F)$, and 
\[
P \leftrightarrow \overline{P}
\]
gives a bijection of the set of parabolic subgroups with Levi factor $M$, we conclude that $(K, \rho^{\vee})$ satisfies condition (3), hence it is a $G$-cover of $(K_{M}, (\rho_{M})^{\vee})$.
\end{proof}

We fix a parabolic subgroup $P$ with Levi factor $M$ and unipotent radical $U$.
From a $G$-cover $(K, \rho)$ of $(K_{M}, \rho_{M})$, we obtain two kinds of representations:
\begin{enumerate}
\item The compactly induced representation $\ind_{K}^{G(F)} (\rho)$.
\item The parabolically induced representation $I^{G}_{P} \left( \ind_{K_{M}}^{M(F)} (\rho_{M}) \right)$.
\end{enumerate}
According to \cite[Lemma~B.3]{2020arXiv201102456B}, these two representations are isomorphic.
However, \cite[Lemma~B.3]{2020arXiv201102456B} is proved by using the uniqueness of adjoints, and the isomorphism is not described explicitly.
We will give an explicit isomorphism between these representations following the arguments in the proof of \cite[Corollary~3.6]{MR2485794}.
\begin{lemma}
\label{isomU1}
The map
\[
I_{U, 1} \colon f \mapsto [g \mapsto
[m \mapsto \delta_{P}(m)^{1/2} \cdot f(mg)]
]\index{$I_{U, 1}$}
\]
gives an isomorphism
\[
\ind_{U(F)K_{M}}^{G(F)} \left(\Inf(\rho_{M})\right) \rightarrow I^{G}_{P} \left( \ind_{K_{M}}^{M(F)} (\rho_{M}) \right).
\]
Here, 
\[
\left(\Inf(\rho_{M}), V_{\rho_{M}} \right)
\]
denotes the inflation of $\rho_{M}$ to $U(F)K_{M}$ via the canonical map
\[
U(F)K_{M} \to U(F)K_{M}/U(F) \simeq K_{M}.
\]
\end{lemma}
\begin{proof}
A straightforward calculation shows that the map
\[
F \mapsto [g \mapsto \left(F(g)\right)(1)]
\]
gives the inverse map.
\end{proof}
Next, we consider the map
\[
I_{U, 2} \colon \ind_{K}^{G(F)} (\rho) \rightarrow \ind_{U(F)K_{M}}^{G(F)} \left(\Inf(\rho_{M})\right)
\]
defined as
\[
I_{U, 2} \colon f \mapsto [
g \mapsto \int_{U(F)} f(ug) du
].\index{$I_{U, 2}$}
\]
Here, we use the Haar measure $du$ on $U(F)$ such that the volume of $K_{U}$ is equal to $1$.
\begin{proposition}
\label{isomintegral}
The map $I_{U,2}$ gives an isomorphism
\[
\ind_{K}^{G(F)} (\rho) \rightarrow \ind_{U(F)K_{M}}^{G(F)} \left(\Inf(\rho_{M})\right).
\]
\end{proposition}
\begin{proof}
We prepare some spaces of functions on $G(F)$:
\begin{itemize}
\item
Let $C_{\cpt}^{\infty} (G(F), \rho)$ denote the space of  compactly supported smooth functions 
\[
f \colon G(F) \to V_{\rho}.
\]
We define a representation $l_{\rho}$ of $K$ on $C_{\cpt}^{\infty} (G(F), \rho)$ as
\[
\left(
l_{\rho}(k) \cdot f
\right)(g) = \rho(k) \cdot f(k^{-1}g)
\]
for $k \in K$, $g \in G(F)$ and $f \in C_{\cpt}^{\infty} (G(F), \rho)$.
\item
Let $C_{\cpt}^{\infty} (U(F) \backslash G(F), \rho_{M})$ denote the space of smooth functions 
\[
f \colon G(F) \to V_{\rho_{M}}
\]
that are left $U(F)$-invariant and compactly supported modulo $U(F)$.
We define a representation $l_{\rho_{M}}$ of $K_{M}$ on $C_{\cpt}^{\infty} (U(F) \backslash G(F), \rho_{M})$ as 
\[
\left(
l_{\rho_{M}}(k) \cdot f
\right)(g)
= \rho_{M}(k) \cdot f(k^{-1}g)
\]
for $k \in K_{M}$, $g \in G(F)$ and $f \in C_{\cpt}^{\infty} (U(F) \backslash G(F), \rho_{M})$.
\item
Let $C_{\cpt}^{\infty}(G(F))$ denote the space of  compactly supported smooth functions 
\[
f \colon G(F) \to \mathbb{C}.
\]
We define a representation $l_{\text{reg}}$ of $G(F)$ on $C_{\cpt}^{\infty}(G(F))$ as
\[
(l_{\text{reg}}(g) \cdot f)(h) = f(g^{-1}h)
\]
for $g, h \in G(F)$ and $f \in C_{\cpt}^{\infty}(G(F))$.
\item
Let $C_{\cpt}^{\infty}(U(F) \backslash G(F))$ denote the space of smooth functions 
\[
f \colon G(F) \to \mathbb{C}
\]
that are left $U(F)$-invariant and compactly supported modulo $U(F)$.
We define a representation $l_{\reg, M}$ of $M(F)$ on $C_{\cpt}^{\infty}(U(F) \backslash G(F))$ as
\[
(l_{\reg, M}(m) \cdot f)(g) = f(m^{-1}g)
\]
for $m \in M(F)$, $g \in G(F)$, and $f \in C_{\cpt}^{\infty} (U(F) \backslash G(F))$.
\end{itemize}
We also define a representation of $K$ on $C_{\cpt}^{\infty}(G(F)) \otimes V_{\rho}$ as 
$
l_{\reg}\restriction_{K} \otimes \rho
$,
and a representation of $K_{M}$ on $C_{\cpt}^{\infty}(U(F) \backslash G(F)) \otimes V_{\rho_{M}}$ as 
$
l_{\reg, M}\restriction_{K_{M}} \otimes \rho_{M}
$.
The definition of representations $l_{\rho}$ and $l_{\reg}$ implies that the map
\[
f \otimes v \mapsto [g \mapsto f(g)\cdot v]
\]
gives a $K$-equivariant isomorphism
\[
\left(
l_{\reg}\restriction_{K} \otimes \rho, \, C_{\cpt}^{\infty}(G(F)) \otimes V_{\rho}
\right)
 \rightarrow 
 \left(
 l_{\rho}, \,
 C_{\cpt}^{\infty} (G(F), \rho)
 \right).
 \]
 On the other hand, the definition of the compact induction implies that as vector spaces, we have
\[
\ind_{K}^{G(F)} (V_{\rho}) = C_{\cpt}^{\infty}(G(F), \rho)^{K}.
\]
Thus, we obtain that 
\[
\ind_{K}^{G(F)} (V_{\rho}) \simeq \left(C_{\cpt}^{\infty}(G(F)) \otimes V_{\rho}\right)^{K}
\]
as vector spaces.
%We define an action of $K$ on $C_{\cpt}^{\infty}(G(F)) \otimes V_{\rho}$ and an action of $K_{M}$ on $C_{\cpt}^{\infty}(U(F)\backslash(G(F)) \otimes V_{\rho_{M}}$ as the tensor product actions.
%Similarly, we obtain a $K_{M}$-equivariant isomorphism
%\[C_{\cpt}^{\infty}(U(F) \backslash G(F)) \otimes V_{\rho_{M}} \rightarrow C_{\cpt}^{\infty} (U(F) \backslash G(F), \rho_{M}).\]
%The definition of the compact induction implies that as vector spaces, we have
%\[
%\ind_{K}^{G(F)} (V_{\rho}) = C_{\cpt}^{\infty}(G(F), \rho)^{K}
%\]
%and
%\[
%\ind_{U(F)K_{M}}^{G(F)} (V_{\rho_{M}}) = C_{\cpt}^{\infty} (U(F) \backslash G(F), \rho_{M})^{K_M}.
%\]
%Hence, we obtain isomorphisms (as vector spaces)
%\[
%\ind_{K}^{G(F)} (V_{\rho}) \simeq \left(C_{\cpt}^{\infty}(G(F)) \otimes V_{\rho}\right)^{K}
%\]
%and
Similarly, the definition of representations $l_{\rho_{M}}$ and $l_{\reg, M}$ implies that the map
\[
f \otimes v \mapsto [g \mapsto f(g)\cdot v]
\]
gives a $K_{M}$-equivariant isomorphism
\[
\left(
l_{\reg, M}\restriction_{K_{M}} \otimes \rho_{M}, \, C_{\cpt}^{\infty}(U(F) \backslash G(F)) \otimes V_{\rho_{M}}
\right)
 \rightarrow 
 \left(
 l_{\rho_{M}}, \,
C_{\cpt}^{\infty} (U(F) \backslash G(F), \rho_{M})
 \right).
 \]
Hence, we obtain an isomorphism of vector spaces
\begin{align*}
\ind_{U(F)K_{M}}^{G(F)} (V_{\rho_{M}}) &=
C_{\cpt}^{\infty} (U(F) \backslash G(F), \rho_{M})^{K_{M}} \\
&\simeq 
\left(C_{\cpt}^{\infty}(U(F) \backslash G(F)) \otimes V_{\rho_{M}}\right)^{K_{M}}.
\end{align*}
Under these isomorphisms, the map
\[
I_{U, 2} \colon \ind_{K}^{G(F)} (\rho) \rightarrow \ind_{U(F)K_{M}}^{G(F)} \left(\Inf(\rho_{M})\right)
\]
is translated into the map
\[
I_{U, 2}' \colon \left(C_{\cpt}^{\infty}(G(F)) \otimes V_{\rho}\right)^{K} \rightarrow \left(C_{\cpt}^{\infty}(U(F) \backslash G(F)) \otimes V_{\rho_{M}}\right)^{K_{M}} 
\]
defined as
\[
I_{U, 2}'(f \otimes v) = [g \mapsto \int_{U(F)} f(ug) du] \otimes v.
\]
We will prove that $I_{U, 2}'$ is an isomorphism.
Since $K$ and $K_{M}$ are compact subgroups, the representation $l_{\reg}\restriction_{K}$ of $K$ on $C_{\cpt}^{\infty}(G(F))$ and the representation $l_{\reg, M}\restriction_{K_{M}}$ of $K_{M}$ on $C_{\cpt}^{\infty}(U(F) \backslash G(F))$ are semisimple.
We write
\[
C_{\cpt}^{\infty}(G(F)) = \bigoplus_{\rho'} C_{\cpt}^{\infty}(G(F))^{(K, \rho')}
\]
and
\[
C_{\cpt}^{\infty}(U(F) \backslash G(F)) = \bigoplus_{\rho'_{M}} C_{\cpt}^{\infty}(U(F) \backslash G(F))^{(K_{M}, \rho'_{M})},
\]
where $\rho'$ and $\rho'_{M}$ run thorough irreducible smooth representations of $K$ and $K_{M}$, respectively.
For $\rho' \not\simeq \rho^{\vee}$ and $\rho'_{M} \not \simeq (\rho_{M})^{\vee}$, we have
\[
\left(C_{\cpt}^{\infty}(G(F))^{(K, \rho')} \otimes V_{\rho} \right)^{K} = \{0\},
\] 
and
\[
\left(C_{\cpt}^{\infty}(U(F) \backslash G(F))^{(K_{M}, \rho'_{M})} \otimes V_{\rho_{M}} \right)^{K_{M}} = \{0\},
\]
respectively.
Thus, we obtain that
\[
\left(C_{\cpt}^{\infty}(G(F)) \otimes V_{\rho}\right)^{K} = \left(C_{\cpt}^{\infty}(G(F))^{(K, \rho^{\vee})} \otimes V_{\rho}\right)^{K}
\]
and
\[
\left(C_{\cpt}^{\infty}(U(F) \backslash G(F)) \otimes V_{\rho_{M}}\right)^{K_{M}}  = \left(C_{\cpt}^{\infty}(U(F) \backslash G(F))^{(K_{M}, (\rho_{M})^{\vee})} \otimes V_{\rho_{M}}\right)^{K_{M}}.
\]
Moreover, since $K$ decomposes with respect to $U, M, \overline{U}$, the groups $K_{U}$ and $K_{\overline{U}}$ are contained in the kernel of $\rho$ and $\rho^{\vee}$, and $\rho\restriction_{K_{M}} = \rho_{M}$, we have
\[
\left(C_{\cpt}^{\infty}(G(F))^{(K, \rho^{\vee})} \otimes V_{\rho}\right)^{K} = \left(C_{\cpt}^{\infty}(G(F))^{(K, \rho^{\vee})} \otimes V_{\rho_{M}}\right)^{K_{M},}.
\]
Then, the claim follows from Lemma~\ref{integralisomlemma} below.
\end{proof}
\begin{lemma}[cf.\,{\cite[Remarque~3.4]{MR2485794}}]
\label{integralisomlemma}
The map
\[
f \mapsto [g \mapsto \int_{U(F)} f(ug) du]
\]
induces an isomorphism
\[
C_{\cpt}^{\infty}(G(F))^{(K, \rho^{\vee})} \to C_{\cpt}^{\infty}(U(F) \backslash G(F))^{(K_{M}, (\rho_{M})^{\vee})}.
\]
\end{lemma}
\begin{proof}
Recall that we defined a representation $l_{\text{reg}}$ of $G(F)$ on $C_{\cpt}^{\infty}(G(F))$ as
\[
(l_{\text{reg}}(g) \cdot f)(h) = f(g^{-1}h)
\]
for $g, h \in G(F)$ and $f \in C_{\cpt}^{\infty}(G(F))$.
The map
\[
C_{\cpt}^{\infty}(G(F)) \rightarrow C_{\cpt}^{\infty}(U(F) \backslash G(F))
\]
defined as
\[
f \mapsto [g \mapsto \int_{U(F)} f(ug) du]
\]
factors through the map 
\[
j_{U}(l_{\text{reg}}) \colon C_{\cpt}^{\infty}(G(F)) \rightarrow C_{\cpt}^{\infty}(G(F))_{U}
\]
and induces a $K_{M}$-equivariant isomorphism
\[
C_{\cpt}^{\infty}(G(F))_{U} \rightarrow C_{\cpt}^{\infty}(U(F) \backslash G(F)).
\]
Thus, the claim follows from Lemma~\ref{contraofcover} and condition~(3') of $G$-covers.
\end{proof}
We write
\[
I_{U}\index{$I_{U}$} := I_{U, 1} \circ I_{U, 2} \colon \ind_{K}^{G(F)} (\rho) \rightarrow I^{G}_{P} \left( \ind_{K_{M}}^{M(F)} (\rho_{M}) \right).
\]
According to Lemma~\ref{isomU1} and Proposition~\ref{isomintegral}, $I_{U}$ is an isomorphism.
We use the same symbols $I_{U, 1}$, $I_{U, 2}$, and $I_{U}$ for the isomorphisms of endomorphism algebras
\begin{align*}
I_{U, 1} \colon \End_{G(F)}\left(\ind_{U(F)K_{M}}^{G(F)} \left(\Inf(\rho_{M})\right)\right) &\rightarrow \End_{G(F)}\left(I^{G}_{P} \left( \ind_{K_{M}}^{M(F)} (\rho_{M}) \right)\right), \\
I_{U, 2} \colon \End_{G(F)}\left(\ind_{K}^{G(F)} (\rho)\right) &\rightarrow \End_{G(F)}\left(\ind_{U(F)K_{M}}^{G(F)} \left(\Inf(\rho_{M})\right)\right), \\
I_{U} \colon \End_{G(F)}\left(\ind_{K}^{G(F)} (\rho)\right) &\rightarrow \End_{G(F)}\left(I^{G}_{P} \left( \ind_{K_{M}}^{M(F)} (\rho_{M}) \right)\right)
\end{align*}
induced by $I_{U, 1}$, $I_{U, 2}$, and $I_{U}$, respectively.
\section{Hecke algebra injections}
\label{Hecke algebra injections}
We use the same notation as Section~\ref{An explicit isomorphism}.
In particular, let $(K, \rho)$ be a $G$-cover of $(K_{M}, \rho_{M})$.
In Section~\ref{An explicit isomorphism}, we constructed an isomorphism
\[
I_{U} \colon \End_{G(F)}\left(\ind_{K}^{G(F)} (\rho)\right) \rightarrow \End_{G(F)}\left(I^{G}_{P} \left( \ind_{K_{M}}^{M(F)} (\rho_{M}) \right)\right).
\]
Since $I_{P}^{G}$ is a faithful functor,  it provides a natural injection
\[
I_{P}^{G} \colon \End_{M(F)}\left( \ind_{K_{M}}^{M(F)} (\rho_{M}) \right) \rightarrow \End_{G(F)}\left(I^{G}_{P} \left( \ind_{K_{M}}^{M(F)} (\rho_{M}) \right)\right).
\]
On the other hand, there exists a natural injection
\[
t_{P} \colon \End_{M(F)}\left( \ind_{K_{M}}^{M(F)} (\rho_{M}) \right) \rightarrow \End_{G(F)}\left(\ind_{K}^{G(F)} (\rho)\right)
\]
defined in \cite[Corollary~7.12]{MR1643417}.

We will explain the definition of $t_{P}$.
First, we recall the definition of the Hecke algebra associated with $(K, \rho)$.
Let $\mathcal{H}(G(F), \rho)\index{$\mathcal{H}(G(F), \rho)$}$ denote the space of compactly supported functions 
\[
\phi \colon  G(F)\to \text{End}_{\mathbb{C}}(V_{\rho})
\]
satisfying
\[
\phi(k_1 g k_2)=\rho(k_1)\circ \phi(g) \circ \rho(k_2),
\]
for all $k_1, k_2 \in K$ and $g \in G(F)$. 
The standard convolution product
\[
\left(\phi_1*\phi_2\right)(x)=\int_{G(F)} \phi_1(y) \circ \phi_2(y^{-1}x) dy
\]
with $\phi_1, \phi_2 \in \mathcal{H}(G(F), \rho)$ and $x \in G(F)$ gives $\mathcal{H}(G(F), \rho)$ a structure of a $\mathbb{C}$-algebra.
We call $\mathcal{H}(G(F), \rho)$ the Hecke algebra associated with the pair $(K, \rho)$.
Here, we normalize the Haar measure $dy$ on $G(F)$ such that the volume of $K$ is equal to $1$.
We note that the isomorphism class of $\mathcal{H}(G(F), \rho)$ does not depend on the choice of the Haar measure on $G(F)$ used to define the convolution product.
For $g \in G(F)$ and $\phi \in \mathcal{H}(G(F), \rho)$, we have
\[
\phi(g) \in \Hom_{K\cap ^g\!K}(^g\!\rho, \rho).
\]
Thus, the support of $\phi$ is contained in $I_{G(F)}(\rho)$.
\begin{remark}
\label{rmkrl}
The definition of $\mathcal{H}(G(F), \rho)$ above is different from the definition of $\mathcal{H}(G(F), \rho)$ in \cite[Section~2]{MR1643417}.
More precisely, our $\mathcal{H}(G(F), \rho)$ denotes $\mathcal{H}(G(F), \rho^{\vee})$ in \cite[Section~2]{MR1643417}. 
According to \cite[(2.3)]{MR1643417}, there exists a canonical anti-isomorphism
\[\mathcal{H}(G(F), \rho) \simeq \mathcal{H}(G(F), \rho^{\vee})\]
that inverts the supports of functions.
Thus, we may apply the results of \cite{MR1643417} to our cases with suitable modifications.
%In particular, there exists a canonical equivalence of categories
%\[\Mod-\mathcal{H}(G(F), \rho) \simeq \mathcal{H}(G(F), \check{\rho})-\Mod.\]
\end{remark}
According to \cite[(2.6)]{MR1643417} and Remark~\ref{rmkrl}, there exists an isomorphism
\begin{align}
\label{heckevsend}
\mathcal{H}(G(F), \rho) \simeq \End_{G(F)}\left(\ind_{K}^{G(F)} (\rho)\right).
\end{align}
We write the isomorphism above explicitly.
For $v \in V_{\rho}$, we define $f_{v} \in \ind_{K}^{G(F)} (V_{\rho})$ as
\[
f_v(g) =
\begin{cases}
\rho(g) \cdot v & (g \in K ), \\
0 & (\text{otherwise}).
\end{cases} 
\]
Then, for $\Phi \in \End_{G(F)}\left(\ind_{K}^{G(F)} (\rho)\right)$, the corresponding element $\phi \in \mathcal{H}(G(F), \rho)$ is defined as
\[\phi(g) \cdot v = \left(\Phi(f_{v})\right)(g)\]
for $g \in G(F)$ and $v \in V_{\rho}$.
Conversely, for $\phi \in \mathcal{H}(G(F), \rho)$, the corresponding element $\Phi \in~ \End_{G(F)}\left(\ind_{K}^{G(F)} (\rho)\right)$ is defined as
\[
\left(\Phi(f)\right)(x) = \int_{G(F)} \phi(y) \cdot f(y^{-1}x) dy
\]
for $f \in \ind_{K}^{G(F)} (\rho)$ and $x \in G(F)$.

Similarly, we define the Hecke algebra $\mathcal{H}(M(F), \rho_{M})\index{$\mathcal{H}(M(F), \rho_{M})$}$ associated with $(K_{M}, \rho_{M})$.
We also have an isomorphism
\begin{align}
\label{Mverofheckevsend}
\mathcal{H}(M(F), \rho_{M}) \simeq \End_{M(F)}\left( \ind_{K_{M}}^{M(F)} (\rho_{M}) \right)
\end{align}
corresponding to \eqref{heckevsend}.

Next, we recall the definition of positive elements \cite[Definition~6.5]{MR1643417}.
An element $z \in M(F)$ is called positive relative to $K$ and $U$, if it satisfies the conditions
\[
zK_{U}z^{-1} \subset K_{U}, \ z^{-1}K_{\overline{U}}z \subset K_{\overline{U}}.
\]
\begin{lemma}
\label{lemmacalculationofdelta}
Suppose that $z \in M(F)$ is positive relative to $K$ and $U$.
Then, we have
\[
\delta_{P}(z) = \abs{K_{U}/zK_{U}z^{-1}}.
\]
\end{lemma}
\begin{proof}
Since $z$ is positive relative to $K$ and $U$, we have $z K_{U} z^{-1} \subset K_{U}$.
Then, substituting the characteristic function of $zK_{U}z^{-1}$ to $f$ in the equation
\[
\int_{U(F)} f(zuz^{-1}) \, du = \delta_{P}(z) \int_{U(F)} f(u) du,
\]
we obtain that
\[
\abs{K_{U}/zK_{U}z^{-1}} = \delta_{P}(z).
\]
\end{proof}
Let $I^{+}\index{$I^{+}$}$ denote the set of positive elements $z \in I_{M(F)}(\rho_{M})$, and let $\mathcal{H}^{+}(M(F), \rho_{M})$ denote the space of functions in $\mathcal{H}(M(F), \rho_{M})$ whose support is contained in $I^{+}$.
According to \cite[Proposition~6.3 (iii)]{MR1643417}, for $\phi \in \mathcal{H}(M(F), \rho_{M})$ with support contained in $K_{M} z K_{M}$ for some $z \in~ I_{M(F)}(\rho_{M})$, there exists a unique function $T(\phi) \in \mathcal{H}(G(F), \rho)$ with support contained in $K z K$, and such that $\left(T(\phi)\right)(z) = \phi(z)$.
According to \cite[Corollary~6.12]{MR1643417},  $\mathcal{H}^{+}(M(F), \rho_{M})$ is a subalgebra of $\mathcal{H}(M(F), \rho_{M})$, and $T$ induces an injective homomorphism
\[
T \colon \mathcal{H}^{+}(M(F), \rho_{M}) \rightarrow \mathcal{H}(G(F), \rho).
\]
Moreover, according to \cite[Theorem~7.2 (i)]{MR1643417} (see also \cite[Proposition~7.14]{MR1643417}), $T$ extends uniquely to an injective homomorphism
\[
t \colon \mathcal{H}(M(F), \rho_{M}) \rightarrow \mathcal{H}(G(F), \rho).
\]
We define
\[
t_{P}\index{$t_{P}$} \colon \mathcal{H}(M(F), \rho_{M}) \rightarrow \mathcal{H}(G(F), \rho)
\]
as
\[
t_{P}(\phi) = t(\phi \cdot \delta_{P}^{-1/2}),
\]
where $\phi \cdot \delta_{P}^{-1/2}$ denotes the function
\[
m \mapsto \phi(m)\delta_{P}(m)^{-1/2}
\]
in $\mathcal{H}(M(F), \rho_{M})$.
\begin{remark}
Our definition of $t_{P}$ is different from that of \cite{MR1643417}.
The difference is due to the fact that we use the normalized parabolic induction, while the un-normalized parabolic induction is used in \cite{MR1643417}.
\end{remark}
The following characterization of $t_{P}$ is a trivial consequence of \cite[Theorem~7.2 (i)]{MR1643417}.
\begin{lemma}
\label{characterizationoftp}
Let 
\[
t' \colon \mathcal{H}(M(F), \rho_{M}) \rightarrow \mathcal{H}(G(F), \rho)
\]
be a homomorphism such that
\[
t'(\phi) = t_{P}(\phi) \ \left(= T(\phi \cdot \delta_{P}^{-1/2})\right)
\]
for any $\phi \in \mathcal{H}^{+}(M(F), \rho_{M})$.
Then, we obtain $t' = t_{P}$.
\end{lemma}

The following Lemma will be used later:
\begin{lemma}
\label{tPwhenmissubalgebra}
Suppose that the subspace
\[
\mathcal{H}(G(F), \rho)_{M} = \{
\phi \in \mathcal{H}(G(F), \rho) \mid \supp(\phi) \subset K \cdot M(F) \cdot K
\}
\]
is a subalgebra of $\mathcal{H}(G(F), \rho)$.
Let $\phi$ be an element of $\mathcal{H}(M(F), \rho_{M})$ whose support is contained in $K_{M} z K_{M}$ for some $z \in I_{M(F)} (\rho_{M})$.
Then, we obtain
\[
t_{P}(\phi) = \frac{
\abs{K_{M} / \left(
K_{M} \cap zK_{M}z^{-1}
\right)}^{1/2}
}{
\abs{K / \left(
K \cap zKz^{-1}
\right)}^{1/2}
}T(\phi).
\]
In particular, the injection $t_{P}$ does not depend on the choice of $P$. 
\end{lemma}
\begin{proof}
The lemma follows from \cite[Proposition~5.1]{MR1621409} and \cite[Remark~5.2]{MR1621409}.
We note that some stronger conditions are supposed in \cite{MR1621409}. However, the results of \cite{MR1621409} may also be extended without difficulty to our case.
We explain the proof briefly.

Let $\phi$ be an element of $\mathcal{H}(M(F), \rho_{M})$ whose support is contained in $K_{M} z K_{M}$ for some $z \in ~ I_{M(F)} (\rho_{M})$.
According to the proof of \cite[Theorem~7.2 (ii)]{MR1643417}, there exists $c \in \mathbb{R}_{>0}$ such that
\begin{align}
\label{uptoconst}
t_{P}(\phi) = c \cdot T(\phi).
\end{align}
We will prove that
\begin{align}
\label{calculationofc}
c= \frac{
\abs{K_{M} / \left(
K_{M} \cap zK_{M}z^{-1}
\right)}^{1/2}
}{
\abs{K / \left(
K \cap zKz^{-1}
\right)}^{1/2}
}.
\end{align}
First, we consider the positive case, that is, we suppose that $z \in I^{+}$.
Then, the definition of $t_{P}$ implies that
\[
t_{P}(\phi) = T\left(
\phi \cdot \delta_{P}^{-1/2}
\right) = \delta_{P}(z)^{-1/2} T(\phi).
\] 
According to Lemma~\ref{lemmacalculationofdelta}, we obtain
\[
c = \delta_{P}(z)^{-1/2} = \abs{K_{U}/zK_{U}z^{-1}}^{-1/2}.
\]
On the other hand, since $K$ decomposes with respect to $U, M, \overline{U}$, and $z$ normalizes $U, M, \overline{U}$, we have
\[
K= K_{U} \cdot K_{M} \cdot K_{\overline{U}},
\]
\[
zKz^{-1} = (zK_{U}z^{-1}) \cdot (zK_{M}z^{-1}) \cdot (zK_{\overline{U}}z^{-1}),
\]
and
\[
K \cap zKz^{-1} = (K_{U} \cap zK_{U}z^{-1}) \cdot (K_{M} \cap zK_{M}z^{-1}) \cdot (K_{\overline{U}} \cap zK_{\overline{U}}z^{-1}).
\]
Moreover, since $z \in I^{+}$, we obatin
\[
K_{U} \cap zK_{U}z^{-1} = zK_{U}z^{-1}
\]
and
\[
K_{\overline{U}} \cap zK_{\overline{U}}z^{-1} = K_{\overline{U}}.
\]
Hence,
\begin{align*}
\frac{
\abs{K_{M} / \left(
K_{M} \cap zK_{M}z^{-1}
\right)}^{1/2}
}{
\abs{K / \left(
K \cap zKz^{-1}
\right)}^{1/2}
} & =
\frac{
\abs{K_{M} / \left(
K_{M} \cap zK_{M}z^{-1}
\right)}^{1/2}
}{
\abs{K_{U} / \left(
K_{U} \cap zK_{U}z^{-1}
\right)}^{1/2}\abs{K_{M} / \left(
K_{M} \cap zK_{M}z^{-1}
\right)}^{1/2}\abs{K_{\overline{U}} / \left(
K_{\overline{U}} \cap zK_{\overline{U}}z^{-1}
\right)}^{1/2}
}\\
&=
\frac{
\abs{K_{M} / \left(
K_{M} \cap zK_{M}z^{-1}
\right)}^{1/2}
}{
\abs{K_{U} / zK_{U}z^{-1}
}^{1/2}\abs{K_{M} / \left(
K_{M} \cap zK_{M}z^{-1}
\right)}^{1/2}\abs{K_{\overline{U}} /K_{\overline{U}}}^{1/2}
}\\
&= \abs{K_{U}/zK_{U}z^{-1}}^{-1/2} \\
&= c.
\end{align*} 
Thus, we obtain equation~\eqref{calculationofc}.

To prove equation~\eqref{calculationofc} for general $z \in I_{M(F)}(\rho_{M})$, we define norms on $\mathcal{H}(G(F), \rho)$ and $\mathcal{H}(M(F), \rho_{M})$.
Fix a $K$-invariant norm on $V_{\rho} = V_{\rho_{M}}$, and let $\norm{\cdot}$ denote the operator norm on $\End_{\mathbb{C}}(V_{\rho})$. 
For $\phi \in ~\mathcal{H}(G(F), \rho)$, we define $\norm{\phi}_{G}$ as
\[
\norm{\phi}_{G} = \left(
\int_{G(F)} \norm{\phi(g)}^{2} dg
\right)^{1/2}.
\]
Here, we normalize the Haar measure $dg$ on $G(F)$ such that the volume of $K$ is equal to $1$.
A straightforward calculation shows
\[
\norm{\phi_{1} * \phi_{2}}_{G} \le \norm{\phi_{1}}_{G} \norm{\phi_{2}}_{G}
\]
for $\phi_{1}, \phi_{2} \in \mathcal{H}(G(F), \rho)$.
We also define $\norm{\phi}_{M}$ for $\phi \in \mathcal{H}(M(F), \rho_{M})$, similarly.
Then, equation~\eqref{calculationofc} can be rephrased in terms of the norms: 
\begin{claim}
Equation~\eqref{calculationofc} is equivalent to 
\begin{align}
\label{normpreserving}
\norm{t_{P}(\phi)}_{G} = \norm{\phi}_{M}.
\end{align}
\end{claim}
\begin{proof}
The definition of $\norm{\cdot}_{M}$ implies that
\begin{align}
\label{normphiM}
\norm{\phi}_{M} = 
\abs{K_{M} z K_{M} / K_{M}}^{1/2} \cdot \norm{\phi(z)}.
\end{align}
On the other hand, comparing the norms of both sides of equation~\eqref{uptoconst}, we obtain
\begin{align}
\label{normcomparison}
\norm{t_{P}(\phi)}_{G} = c \cdot \norm{T(\phi)}_{G}.
\end{align}
Moreover, the definition of $T$ implies that
\begin{align}
\label{normofTphiG}
\norm{T(\phi)}_{G} &= 
\abs{KzK/K}^{1/2} \cdot \norm{\left(T(\phi)\right)(z)}\\
&= \abs{KzK/K}^{1/2} \cdot \norm{\phi(z)}.\notag
%&= \frac{
%\abs{KzK/K}^{1/2}
%}{
%\abs{K_{M} z K_{M} / K_{M}}^{1/2}
%}\norm{\phi}_{M}\\
%&= \frac{
%\abs{K / \left(
%K \cap zKz^{-1}
%\right)}^{1/2}
%}
%{
%\abs{K_{M} / \left(
%K_{M} \cap zK_{M}z^{-1}
%\right)}^{1/2}
%}
%\norm{\phi}_{M}.
\end{align}
Comparing \eqref{normphiM}, \eqref{normcomparison} with \eqref{normofTphiG}, we obtain the claim.
\end{proof}
We prove equation~\eqref{normpreserving} for general case.
Let $\zeta$ be a strongly $(U, K)$-positive element in the center of $M$ (see \cite[Definition~6.16]{MR1643417}).
Hence, $\zeta$ is a positive element, and there exists a positive integer $n$ such that $\zeta^{n} z \in I^{+}$. 
Replacing $\zeta$ with $\zeta^{n}$, we may assume that $n=1$.
Let $\phi_{\zeta}$ denote the unique element of $\mathcal{H}(M(F), \rho_{M})$ with support $\zeta K_{M}$ such that
\[
\phi_{\zeta}(\zeta) = \id_{V_{\rho_{M}}}.
\]
A straightforward calculation shows that
\begin{align}
\label{lefttranslationphim}
\left(\phi_{\zeta} *\phi\right)(m) = \phi(\zeta^{-1} m)
\end{align}
for $m \in M(F)$.
In particular, $\phi_{\zeta} * \phi$ is supported on $K_{M} \zeta z K_{M} \subset I^{+}$.
Hence, the result for positive case implies that
\[
\norm{t_{P}\left(\phi_{\zeta} *\phi\right)}_{G} = \norm{\phi_{\zeta} *\phi}_{M}.
\]
On the other hand, the definition of $\norm{\cdot}_{M}$ and equation~\eqref{lefttranslationphim} imply that
\[
\norm{\phi_{\zeta} *\phi}_{M} = \norm{\phi}_{M}.
\]
Thus, to prove \eqref{normpreserving}, it suffices to show
\[
\norm{t_{P}\left(\phi_{\zeta} *\phi\right)}_{G} = \norm{t_{P}(\phi)}_{G}.
\]
Since $\zeta \in I^{+}$, the result for positive case implies that
\[
\norm{t_{P}(\phi_{\zeta})}_{G} = \norm{\phi_{\zeta}}_{M} = 1.
\]
Hence,
\begin{align*}
\norm{t_{P}\left(\phi_{\zeta} *\phi\right)}_{G}  &= \norm{t_{P}(\phi_{\zeta}) * t_{P}(\phi)}_{G} \\
&\le \norm{t_{P}(\phi_{\zeta})}_{G} \norm{t_{P}(\phi)}_{G}\\
&= \norm{t_{P}(\phi)}_{G}.
\end{align*}
On the other hand, we can prove 
\[
\norm{t_{P}(\phi)}_{G} \le \norm{t_{P}\left(\phi_{\zeta} *\phi\right)}_{G}
\]
as follows.
Let $\phi_{\zeta^{-1}}$ denote the unique element of $\mathcal{H}(M(F), \rho_{M})$ with support $\zeta^{-1} K_{M}$, and such that
\[
\phi_{\zeta^{-1}}(\zeta^{-1}) = \id_{V_{\rho_{M}}}.
\]
Then, $\phi_{\zeta^{-1}}$ is the inverse of $\phi_{\zeta}$ in $\mathcal{H}(M(F), \rho_{M})$, hence $t_{P}(\phi_{\zeta^{-1}})$ is the inverse of $t_{P}(\phi_{\zeta})$ in $\mathcal{H}(G(F), \rho)$.
According to the proof of \cite[Theorem~7.2 (ii)]{MR1643417}, there exists $c' \in \mathbb{R}_{>0}$ such that
\begin{align*}
t_{P}(\phi_{\zeta^{-1}}) = c' \cdot T(\phi_{\zeta^{-1}}).
\end{align*}
Since $\zeta$ is a positive element in the center of $M$, the result for positive case implies that
\begin{align*}
t_{P}(\phi_{\zeta}) &= \frac{
\abs{K_{M} / \left(
K_{M} \cap \zeta K_{M} \zeta^{-1}
\right)}^{1/2}
}{
\abs{K / \left(
K \cap \zeta K \zeta^{-1}
\right)}^{1/2}
}T(\phi_{\zeta})\\
&= \abs{K / \left(
K \cap \zeta K \zeta^{-1}
\right)}^{-1/2} T(\phi_{\zeta}).
\end{align*}
Hence, we obtain
\begin{align*}
1 &= t_{P}(\phi_{\zeta^{-1}}) * t_{P} (\phi_{\zeta}) \\
&=c' \cdot \abs{K / \left(
K \cap \zeta K \zeta^{-1}
\right)}^{-1/2}  T(\phi_{\zeta^{-1}}) * T(\phi_{\zeta}).
\end{align*}
Comparing the values at $1$, we obtain
\begin{align*}
\id_{V_{\rho}} &= c' \cdot \abs{K / \left(
K \cap \zeta K \zeta^{-1}
\right)}^{-1/2}  \int_{G(F)} \left(
T(\phi_{\zeta^{-1}})
\right)(g) \circ\left(
T(\phi_{\zeta})
\right)(g^{-1}) dg\\
&= c' \cdot \abs{K / \left(
K \cap \zeta K \zeta^{-1}
\right)}^{-1/2}  \abs{K \zeta^{-1} K/K} \id_{V_{\rho}}\\
&= c' \cdot \abs{K \zeta K/K}^{-1/2} \abs{K \zeta^{-1} K/K} \id_{V_{\rho}}.
%&= c' \cdot \abs{K \zeta^{-1} K/K}^{-1/2} \abs{K \zeta^{-1} K/K} \id_{V_{\rho}}\\
%&= c' \cdot \abs{K \zeta^{-1} K/K}^{1/2} \id_{V_{\rho}}.
\end{align*}
Hence, we obtain
\begin{align}
\label{valueofc'notcalculated}
c' = \abs{K \zeta K/K}^{1/2} \abs{K \zeta^{-1} K/K}^{-1} 
\end{align}
Recall that we normalize the Haar measure on $G(F)$ such that the volume of $K$ is equal to $1$.
For an open subset $U$ of $G(F)$, let $\vol(U)$ denote the volume of $U$ with respect to this Haar measure.
Then, we have
\begin{align*}
\abs{K \zeta K/K} &= \vol\left(
K \zeta K
\right) \\
&= \vol\left(
\left(
K \zeta K
\right)^{-1}
\right) \\
&= \vol\left(
K \zeta^{-1} K
\right) \\
&= \abs{K \zeta^{-1} K/K}.
\end{align*}
Substituting it to equation~\eqref{valueofc'notcalculated}, we have
\[
c' = \abs{K \zeta^{-1} K/K}^{-1/2}.
\]
Thus, we have
\begin{align*}
\norm{t_{P}(\phi_{\zeta^{-1}})}_{G} &= \norm{c' \cdot T(\phi_{\zeta^{-1}})}_{G}\\
&= c' \norm{T(\phi_{\zeta^{-1}})}_{G}\\
&= c' \abs{K \zeta^{-1} K/K}^{1/2} \norm{
\left(
T(\phi_{\zeta^{-1}})(\zeta^{-1})
\right)
}\\
&= c' \abs{K \zeta^{-1} K/K}^{1/2} \norm{
\phi_{\zeta^{-1}}(\zeta^{-1})
}\\
&= c' \abs{K \zeta^{-1} K/K}^{1/2}\\
&= 1.
\end{align*}
Now, we obtain
\begin{align*}
\norm{t_{P}(\phi)}_{G} &= \norm{t_{P}(\phi_{\zeta^{-1}}) * t_{P}\left(\phi_{\zeta} *\phi\right)}_{G}\\
&\le \norm{t_{P}(\phi_{\zeta^{-1}})}_{G} \norm{t_{P}\left(\phi_{\zeta} *\phi\right)}_{G} \\
&= \norm{t_{P}\left(\phi_{\zeta} *\phi\right)}_{G} .
\end{align*}
%According to the proof of \cite[Theorem~7.2 (ii)]{MR1643417}, we obtain
%\[
%t_{P}(\phi) = \delta_{P}(z)^{-1/2} \abs{K_{U} / \left( K_{U} \cap z^{-1}K_{U}z \right)}^{-1} T(\phi).
%\]
%Substituting the characteristic function of $zK_{U}z^{-1}$ to $f$ in the equation
%\[
%\int_{U(F)} f(zuz^{-1}) \, du = \delta_{P}(z) \int_{U(F)} f(u) du,
%\]
%we obtain
%\begin{align*}
%\delta_{P}(z) &= \frac{\vol{K_{U}}}{\vol{zK_{U}z^{-1}}}\\ 
%&= \frac{\abs{K_{U}/\left( K_{U} \cap zK_{U}z^{-1}\right)}}{\abs{z K_{U} z^{-1}/\left( K_{U} \cap zK_{U}z^{-1}\right)}}\\
%&= \frac{\abs{K_{U}/\left( K_{U} \cap zK_{U}z^{-1}\right)}}{\abs{K_{U} /\left( K_{U} \cap z^{-1}K_{U}z\right)}}.
%\end{align*}
%Here, $\vol{K_{U}}$ and $\vol{zK_{U}z^{-1}}$ denote the volume of $K_{U}$ and $zK_{U}z^{-1}$ with respect to a Haar measure on $U(F)$ respectively.
%Now, we obtain
%\begin{align*}
%t_{P}(\phi) &= \delta_{P}(z)^{-1/2} \abs{K_{U} / \left( K_{U} \cap z^{-1}K_{U}z \right)}^{-1} T(\phi)\\
%&= \abs{K_{U}/\left( K_{U} \cap zK_{U}z^{-1}\right)}^{-1/2} \abs{K_{U} /\left( K_{U} \cap z^{-1}K_{U}z\right)}^{1/2} \abs{K_{U} / \left( K_{U} \cap z^{-1}K_{U}z \right)}^{-1} T(\phi)\\
%&= \abs{K_{U}/\left( K_{U} \cap zK_{U}z^{-1}\right)}^{-1/2} \abs{K_{U} /\left( K_{U} \cap z^{-1}K_{U}z\right)}^{-1/2} T(\phi)
%\end{align*}
\end{proof}

We transport $t_{P}\index{$t_{P}$}$ to an injective homomorphism
\[
\End_{M(F)}\left( \ind_{K_{M}}^{M(F)} (\rho_{M}) \right) \rightarrow \End_{G(F)}\left(\ind_{K}^{G(F)} (\rho)\right)
\]
via isomorphism~\eqref{heckevsend} and isomorphism~\eqref{Mverofheckevsend}, and use the same symbol $t_{P}$ for it.

Now, we have two injections
\[
I_{P}^{G} \colon \End_{M(F)}\left( \ind_{K_{M}}^{M(F)} (\rho_{M}) \right) \rightarrow \End_{G(F)}\left(I^{G}_{P} \left( \ind_{K_{M}}^{M(F)} (\rho_{M}) \right)\right)
\]
and
\[
t_{P} \colon \End_{M(F)}\left( \ind_{K_{M}}^{M(F)} (\rho_{M}) \right) \rightarrow \End_{G(F)}\left(\ind_{K}^{G(F)} (\rho)\right).
\]
The following proposition claims that the isomorphism
\[
I_{U} \colon \End_{G(F)}\left(\ind_{K}^{G(F)} (\rho)\right) \rightarrow \End_{G(F)}\left(I^{G}_{P} \left( \ind_{K_{M}}^{M(F)} (\rho_{M}) \right)\right)
\]
is compatible with these injections.
\begin{proposition}
\label{compatibility}
The following diagram commutes:
\[
\xymatrix{
\End_{M(F)}\left( \ind_{K_{M}}^{M(F)} (\rho_{M}) \right) \ar[d]_-{t_{P}} \ar[r]^-{\id} \ar@{}[dr]|\circlearrowleft & \End_{M(F)}\left( \ind_{K_{M}}^{M(F)} (\rho_{M}) \right) \ar[d]^-{I_{P}^{G}} \\
\End_{G(F)}\left(\ind_{K}^{G(F)} (\rho)\right) \ar[r]^-{I_{U}} & \End_{G(F)}\left(I^{G}_{P} \left( \ind_{K_{M}}^{M(F)} (\rho_{M}) \right)\right). 
}
\]
\end{proposition}
\begin{proof}
According to Lemma~\ref{characterizationoftp} for $t' = I_{U}^{-1} \circ I_{P}^{G}$, it suffices to show that
\[
I_{U}(t_{P}(\Phi)) = I_{P}^{G}(\Phi)
\]
for any $\Phi \in \End_{M(F)}\left( \ind_{K_{M}}^{M(F)} (\rho_{M}) \right)$ that corresponds to an element $\phi \in \mathcal{H}^{+}(M(F), \rho_{M})$ via isomorphism~\eqref{Mverofheckevsend}.
Moreover, we may suppose that the support of $\phi$ is contained in $K_{M} z K_{M}$ for some $z \in I^{+}$.
Let $\phi \in \mathcal{H}^{+}(M(F), \rho_{M})$ be such an element.
Then, the element $\Phi \in \End_{M(F)}\left( \ind_{K_{M}}^{M(F)} (\rho_{M}) \right)$ corresponding to $\phi$ via isomorphism~\eqref{Mverofheckevsend} is defined as 
\[
\left(\Phi(f)\right)(x) = \int_{M(F)} \phi(y) \cdot f(y^{-1}x) dy
\]
for $f \in \ind_{K_{M}}^{M(F)} (\rho_{M})$ and $x \in M(F)$.

Let $F_{U}$ be an element of $I^{G}_{P} \left( \ind_{K_{M}}^{M(F)} (\rho_{M}) \right)$.
We write 
\[
f_{U} = I_{U, 1}^{-1}(F_{U})
\]
and
\[
f = I_{U}^{-1}(F_{U}).
\]
To prove the proposition, it suffices to show that
\[
\left(I_{U}(t_{P}(\Phi))\right)(F_{U})
= \left(I_{P}^{G}(\Phi)\right)(F_{U}),
\]
equivalently, 
\begin{align}
\label{compatibilityenoughtoshow}
I_{U,2}\left(\left(t_{P}(\Phi)\right)(f)\right) = I_{U,1}^{-1}\left(\left(I_{P}^{G}(\Phi)\right)(F_{U})\right).
\end{align}
for any $F_{U} \in I^{G}_{P} \left( \ind_{K_{M}}^{M(F)} (\rho_{M}) \right)$.

First, we calculate the left hand side of \eqref{compatibilityenoughtoshow}.
For $x \in G(F)$, we have
\begin{align*}
\left(\left(t_{P}(\Phi)\right)(f)\right)(x) &= \int_{G(F)} \left(t_{P}(\phi)\right)(y) \cdot f(y^{-1}x) dy\\
&= \int_{G(F)} \left(T(\phi \cdot \delta_{P}^{-1/2})\right)(y) \cdot f(y^{-1}x) dy\\
&= \int_{K z K} \left(T(\phi \cdot \delta_{P}^{-1/2})\right)(y) \cdot f(y^{-1}x) dy\\
&= \sum_{y \in KzK/K} \left(T(\phi \cdot \delta_{P}^{-1/2})\right)(y) \cdot f(y^{-1}x) \\
&= \sum_{k \in K/(K \cap zKz^{-1})} \left(T(\phi \cdot \delta_{P}^{-1/2})\right)(kz) \cdot f(z^{-1}k^{-1}x).
\end{align*}
Since $K$ decomposes with respect to $U, M, \overline{U}$, and $z$ normalizes $U, M, \overline{U}$, we have
\[
K= K_{U} \cdot K_{M} \cdot K_{\overline{U}},
\]
\[
zKz^{-1} = (zK_{U}z^{-1}) \cdot (zK_{M}z^{-1}) \cdot (zK_{\overline{U}}z^{-1}),
\]
and
\[
K \cap zKz^{-1} = (K_{U} \cap zK_{U}z^{-1}) \cdot (K_{M} \cap zK_{M}z^{-1}) \cdot (K_{\overline{U}} \cap zK_{\overline{U}}z^{-1}).
\]
Moreover, since $z \in I^{+}$, we obatin
\[
K_{U} \cap zK_{U}z^{-1} = zK_{U}z^{-1}
\]
and
\[
K_{\overline{U}} \cap zK_{\overline{U}}z^{-1} = K_{\overline{U}}.
\]
Hence, we may rewrite the index of the summation as 
\[
\sum_{k \in K/(K \cap zKz^{-1})} (*)(k) = \sum_{k_{M} \in K_{M}/ (K_{M} \cap zK_{M}z^{-1})} \sum_{k_{U} \in K_{U}/zK_{U}z^{-1}} (*)(k_{U}k_{M}).
\]
Thus, for $g \in G(F)$, we obtain
\begin{align*}
& \quad \left(
I_{U,2}\left(\left(t_{P}(\Phi)\right)(f)\right)
\right)(g)\\ &= \int_{U(F)} \left(\left(t_{P}(\Phi)\right) (f)\right)(ug) du\\
&= \int_{U(F)} \sum_{k_{M} \in K_{M}/ (K_{M} \cap zK_{M}z^{-1})} \sum_{k_{U} \in K_{U}/zK_{U}z^{-1}} \left(T(\phi \cdot \delta_{P}^{-1/2})\right)(k_{U}k_{M}z) \cdot f(z^{-1}k_{M}^{-1}k_{U}^{-1}ug) du\\
&= \delta_{P}(z)^{-1/2} \int_{U(F)} \sum_{k_{M} \in K_{M}/ (K_{M} \cap zK_{M}z^{-1})} \sum_{k_{U} \in K_{U}/zK_{U}z^{-1}} \left(T(\phi)\right)(k_{U}k_{M}z) \cdot f(z^{-1}k_{M}^{-1}k_{U}^{-1}ug) du\\
&= \delta_{P}(z)^{-1/2} \int_{U(F)} \sum_{k_{M} \in K_{M}/ (K_{M} \cap zK_{M}z^{-1})} \sum_{k_{U} \in K_{U}/zK_{U}z^{-1}} \left(\rho(k_{U}k_{M}) \circ \left(T(\phi)\right)(z)\right) \cdot f(z^{-1}k_{M}^{-1}k_{U}^{-1}ug) du\\
&= \delta_{P}(z)^{-1/2} \int_{U(F)} \sum_{k_{M} \in K_{M}/ (K_{M} \cap zK_{M}z^{-1})} \sum_{k_{U} \in K_{U}/zK_{U}z^{-1}} \left(\rho_{M}(k_{M}) \circ \phi(z)\right) \cdot f(z^{-1}k_{M}^{-1}k_{U}^{-1}ug) du\\
&= \delta_{P}(z)^{-1/2} \sum_{k_{M} \in K_{M}/ (K_{M} \cap zK_{M}z^{-1})} \sum_{k_{U} \in K_{U}/zK_{U}z^{-1}} \left(\rho_{M}(k_{M}) \circ \phi(z)\right) \cdot \left(\int_{U(F)} f(z^{-1}k_{M}^{-1}k_{U}^{-1}ug) du\right)\\
&= \delta_{P}(z)^{-1/2} \sum_{k_{M} \in K_{M}/ (K_{M} \cap zK_{M}z^{-1})} \sum_{k_{U} \in K_{U}/zK_{U}z^{-1}} \left(\rho_{M}(k_{M}) \circ \phi(z)\right) \cdot \left(\int_{U(F)} f(z^{-1}k_{M}^{-1}ug) du\right)\\
&= \delta_{P}(z)^{-1/2} \abs{K_{U}/zK_{U}z^{-1}} \sum_{k_{M} \in K_{M}/ (K_{M} \cap zK_{M}z^{-1})} \left(\rho_{M}(k_{M}) \circ \phi(z)\right) \cdot \left(\int_{U(F)} f(z^{-1}k_{M}^{-1}ug) du\right)\\
&= \delta_{P}(z)^{-1/2} \cdot \delta_{P}(z) \sum_{k_{M} \in K_{M}/ (K_{M} \cap zK_{M}z^{-1})} \left(\rho_{M}(k_{M}) \circ \phi(z)\right) \cdot \left(\int_{U(F)} f(z^{-1}k_{M}^{-1}ug) du\right)\\
&= \delta_{P}(z)^{1/2} \sum_{k_{M} \in K_{M}/ (K_{M} \cap zK_{M}z^{-1})} \left(\rho_{M}(k_{M}) \circ \phi(z)\right) \cdot \left(\int_{U(F)} f(z^{-1}k_{M}^{-1}ug) du\right).
%&= \sum_{k_{M} \in K_{M}/ (K_{M} \cap zK_{M}z^{-1})} \sum_{k_{U} \in K_{U}/zK_{U}z^{-1}} T(\phi \cdot \delta_{P}^{-1/2})(k_{M}z) \cdot \left(\int_{U(F)} f(z^{-1}k_{M}^{-1}ug) du\right)\\
%&= \abs{K_{U}/zK_{U}z^{-1}} \sum_{k_{M} \in K_{M}/ (K_{M} \cap zK_{M}z^{-1})} T(\phi \cdot \delta_{P}^{-1/2})(k_{M}z) \cdot \left(\int_{U(F)} f(z^{-1}k_{M}^{-1}ug) du\right)\\
%&= \delta_{P}(z) \sum_{k_{M} \in K_{M}/ (K_{M} \cap zK_{M}z^{-1})} T(\phi \cdot \delta_{P}^{-1/2})(k_{M}z) \cdot \left(\int_{U(F)} f(z^{-1}k_{M}^{-1}ug) du\right)\\
%&= \delta_{P}(z)^{1/2} \sum_{k_{M} \in K_{M}/ (K_{M} \cap zK_{M}z^{-1})} T(\phi)(k_{M}z) \cdot \left(\int_{U(F)} f(z^{-1}k_{M}^{-1}ug) du\right)\\
%&= \delta_{P}(z)^{1/2} \sum_{k_{M} \in K_{M}/ (K_{M} \cap zK_{M}z^{-1})} \rho(k_{M})
\end{align*}
We used Lemma~\ref{lemmacalculationofdelta} for the second equality from the last.

Next, we calculate the right hand side of \eqref{compatibilityenoughtoshow}.
For $g \in G(F)$, we have
\begin{align*}
& \quad \left(
I_{U,1}^{-1}\left(\left(I_{P}^{G}(\Phi)\right)(F_{U})\right)
\right)(g)\\
&= 
\left(
\left(
\left(I_{P}^{G}(\Phi)\right)(F_{U})
\right)(g)
\right)(1)\\
&=\left(\Phi \left( F_{U}(g) \right)\right) (1)\\
&= \int_{M(F)} \phi(y) \cdot \left(F_{U}(g)\right)(y^{-1}) dy\\
&= \int_{K_M z K_{M} } \phi(y) \cdot \left(F_{U}(g)\right)(y^{-1}) dy\\
&= \sum_{y \in K_M z K_M/K_{M}} \phi(y) \cdot \left(F_{U}(g)\right)(y^{-1})\\
&=  \sum_{k_{M} \in K_{M}/ (K_{M} \cap zK_{M}z^{-1})} \phi(k_{M}z) \cdot \left(F_{U}(g)\right)(z^{-1}k_{M}^{-1})\\
&= \sum_{k_{M} \in K_{M}/ (K_{M} \cap zK_{M}z^{-1})} \left(\rho_{M}(k_{M}) \circ \phi(z) \right) \cdot \left(F_{U}(g)\right)(z^{-1}k_{M}^{-1})\\
&= \sum_{k_{M} \in K_{M}/ (K_{M} \cap zK_{M}z^{-1})} \left(\rho_{M}(k_{M}) \circ \phi(z) \right) \cdot \left( \delta_{P}(z^{-1}k_{M}^{-1})^{1/2} \cdot f_{U}(z^{-1}k_{M}^{-1}g)\right)\\
&= \delta_{P}(z)^{-1/2} \sum_{k_{M} \in K_{M}/ (K_{M} \cap zK_{M}z^{-1})} \left(\rho_{M}(k_{M}) \circ \phi(z) \right) \cdot \left(f_{U}(z^{-1}k_{M}^{-1}g)\right)\\
&= \delta_{P}(z)^{-1/2} \sum_{k_{M} \in K_{M}/ (K_{M} \cap zK_{M}z^{-1})} \left(\rho_{M}(k_{M}) \circ \phi(z) \right) \cdot \left(\int_{U(F)} f(uz^{-1}k_{M}^{-1}g) du\right)\\
&= \delta_{P}(z)^{-1/2} \sum_{k_{M} \in K_{M}/ (K_{M} \cap zK_{M}z^{-1})} \left(\rho_{M}(k_{M}) \circ \phi(z) \right) \cdot \left(\int_{U(F)} f(z^{-1}k_{M}^{-1}(k_{M}zuz^{-1}k_{M}^{-1})g) du\right)\\
&= \delta_{P}(z)^{-1/2} \cdot \delta_{P}(k_{M}z) \sum_{k_{M} \in K_{M}/ (K_{M} \cap zK_{M}z^{-1})} \left(\rho_{M}(k_{M}) \circ \phi(z) \right) \cdot \left(\int_{U(F)} f(z^{-1}k_{M}^{-1}ug) du\right)\\
&= \delta_{P}(z)^{-1/2} \cdot \delta_{P}(z) \sum_{k_{M} \in K_{M}/ (K_{M} \cap zK_{M}z^{-1})} \left(\rho_{M}(k_{M}) \circ \phi(z) \right) \cdot \left(\int_{U(F)} f(z^{-1}k_{M}^{-1}ug) du\right)\\
&= \delta_{P}(z)^{1/2} \sum_{k_{M} \in K_{M}/ (K_{M} \cap zK_{M}z^{-1})} \left(\rho_{M}(k_{M}) \circ \phi(z) \right) \cdot \left(\int_{U(F)} f(z^{-1}k_{M}^{-1}ug) du\right),
\end{align*}
that is equal to $\left(
I_{U,2}\left(\left(t_{P}(\Phi)\right)(f)\right)
\right)(g)$. 
\end{proof}
\section{The case of depth-zero types}
\label{The case of depth-zero types}
In this section, we recall the description of the endomorphism algebra $\End_{G(F)}\left(\ind_{K}^{G(F)} (\rho)\right)$ for a depth-zero type $(K, \rho)$ in \cite{MR1235019}.
Let $S\index{$S$}$ be a maximal split torus of $G$, and let $\Phi = \Phi(G, S)\index{$\Phi$}$ and $\Phi^{\vee} = \Phi^{\vee}(G, S)\index{$\Phi^{\vee}$}$ denote the set of relative roots and the set of relative coroots with respect to $S$, respectively.
%For $\alpha \in \Phi$, let $\alpha^{\vee}$ denote the corresponding coroot in $\Phi^{\vee}$.
Let $V\index{$V$}$ denote the $\mathbb{R}$-span of $\Phi^{\vee}(G, S)$.
Let $\mathcal{A} = \mathcal{A}(G, S)\index{$\mathcal{A}$}$ denote the reduced apartment of $S$.
Hence, $\mathcal{A}$ is an affine space whose vector space of translations is $V$.
The work of Bruhat and Tits \cite{MR327923} associates to $(G, S)$ an affine root system $\Phi_{\aff} = \Phi_{\aff}(G, S)\index{$\Phi_{\aff}$}$ on $\mathcal{A}$ (see \cite[Section~1]{MR546588}).
For $a \in \Phi_{\aff}$, let $Da$ denote the gradient of $a$.
For a subset $\Psi \subset \Phi_{\aff}$, we write
\[
D\Psi = \{
Da \mid a \in \Psi
\}.
\]
We note that
\[
D\Phi_{\aff} = \Phi.
\]
We write $A'\index{$A'$}$ for the space of affine-linear functions on $\mathcal{A}$, that is spanned by $\Phi_{\aff}$.
For $\alpha \in \Phi$, let $s_{\alpha}$ denote the corresponding reflection on $V$, and for $a \in \Phi_{\aff}$, let $s_{a}$ denote the corresponding reflection on $\mathcal{A}$. 
Let $W_{0}\index{$W_{0}$}$ denote the Weyl group of the root system $\Phi$ and $W_{\aff}\index{$W_{\aff}$}$ denote the affine Weyl group of the affine root system $\Phi_{\aff}$.
Hence, $W_{0}$ is generated by $s_{\alpha} \ (\alpha \in \Phi)$, and $W_{\aff}$ is generated by $s_{a} \ (a \in \Phi_{\aff})$.
We define the derivative $Dw \in W_{0}$ of an element $w \in W_{\aff}$ as
\[
w(x+v) = w(x) + (Dw)(v)
\]
for all $x \in \mathcal{A}$ and $v \in V$.
We also write 
\[
W\index{$W$} = N_{G}(S)(F)/Z_{G}(S)(F)_{0},
\]
where $N_{G}(S)(F)\index{$N_{G}(S)(F)$}$ denotes the normalizer of $S$ in $G(F)$, and $Z_{G}(S)(F)_{0}\index{$Z_{G}(S)(F)_{0}$}$ denotes the unique parahoric subgroup of the minimal semi-standard Levi subgroup $Z_{G}(S)\index{$Z_{G}(S)$}$ of $G$ with respect to $S$.
We fix lifts of elements of $W$ in $N_{G}(S)(F)$ as \cite[Proposition~5.2]{MR1235019} and write $\dot{w} \in N_{G}(S)(F)$ for the lift of $w \in W$.
For a subset $H$ of $G(F)$ containing $Z_{G}(S)(F)_{0}$, let $W_{H}$ denote the subset
\[
\left(
N_{G}(S)(F) \cap H
\right) / Z_{G}(S)(F)_{0}
\]
of $W$.
According to \cite[1.2]{MR546588}, $W$ acts on the affine space $\mathcal{A}$.
Let $\nu$ denote this action.
Let $G'$ denote the open subgroup of $G(F)$ generated by all parahoric subgroups of $G(F)$.
According to \cite[5.2.12]{MR756316}, the restriction of $\nu$ to $W_{G'}$ induces an isomorphism
\[
W_{G'} \simeq W_{\aff}
\]
 (see also \cite[3.2, 3.12]{MR1235019}).
 We identify $W_{G'}$ with $W_{\aff}$ and regard $W_{\aff}$ as a subgroup of $W$.
 According to \cite[1.7]{MR546588}, $W_{\aff}$ is a normal subgroup of $W$.
%According to \cite[1.2]{MR546588} and \cite[1.7]{MR546588}, $W$ acts on the affine space $\mathcal{A}$, and we can consider $W_{\aff}$ as a normal subgroup of $W$ via the action. 
%\begin{remark}
%\label{identificationofWwithitslift}
%We sometimes identify an element of $W$ with its lift.
%If a property $p$ about $N_{G}(S)(F)$ does not depend on the choice of lifts, we simply say $p(w)$ for $w \in W$ instead of ``$p(\dot{w})$ for any lift $\dot{w}$ of $w$''.
%Moreover, for a subset $H$ of $G(F)$, we write $W \cap H$ for
%\[
%\{
%w \in W \mid \text{any lift of $w$ is contained in $H$}
%\}.
%\]
%For a subset $W_{1}$ of $W$ and a subset $H$ of $G(F)$, we write $W_{1} \subset H$ if any lift of the elements of $W_{1}$ is contained in $H$.
%\end{remark}
We fix a chamber $C\index{$C$}$ of the affine root system $\Phi_{\aff}$, and let $B \subset \Phi_{\aff}\index{$B$}$ denote the corresponding basis of $\Phi_{\aff}$. 
The chamber $C$ determines a set of positive affine roots $\Phi_{\aff}^{+}\index{$\Phi_{\aff}^{+}$}$ as
\[
\Phi_{\aff}^{+} = \{
a \in \Phi_{\aff} \mid a(x) >0 \ (x \in C)
\}.
\]
We assume that the affine root system $\Phi_{\aff}$ is irreducible.
\begin{remark}
We assume that $\Phi_{\aff}$ is irreducible since it is supposed in \cite{MR1235019} (see \cite[3.14 (a)]{MR1235019}).
However, the modifications of \cite{MR1235019} and our results in case that $\Phi_{\aff}$ is not irreducible can be obtained by arguing component by component.
\end{remark}

We fix a proper subset $J \subset B\index{$J$}$. Then, we can associate $J$ with a parahoric subgroup $P_{J}\index{$P_{J}$}$ of $G(F)$ as \cite[3.7]{MR1235019}.
We also have an open normal subgroup $U_{J}\index{$U_{J}$}$ of $P_{J}$ called the radical of $P_{J}$ such that the quotient $P_{J}/U_{J}$ is isomorphic to the group of $k_{F}$-valued points of a connected reductive group $\mathbf{M}_{J}\index{$\mathbf{M}_{J}$}$ defined over $k_{F}$ \cite[3.13]{MR1235019}. 
We note that $P_{J}$ and $U_{J}$ depend not only on $J$ but also on $B$.
When we emphasis the dependence on $B$, we write $P_{J} = P_{J, B}$ and $U_{J} = U_{J, B}$.
Moreover, for another basis $B'$ of $\Phi_{\aff}$ containing $J$, we write $P_{J, B'}$ and $U_{J, B'}$ for the corresponding subgroups of $G(F)$.
%On the other hand, the group $\mathbf{M}_{J}$ does not depend on the choice of $B$.
In \cite[3.15]{MR1235019}, Morris defined a reductive subgroup $\mathfrak{M}$ of $G$ and its parahoric subgroup 
\[
\mathcal{M}_{J} = P_{J} \cap \mathfrak{M}(F)
\]
that only depend on $S$ and $J$.
More precisely, $\mathcal{M}_{J}\index{$\mathcal{M}_{J}$}$ is the group generated by 
\[
\left\{
Z_{G}(S)(F)_{0}, U_{a} \mid a \in \Phi_{\aff} \cap A'_{J}
\right\},
\]
where $A'_{J}\index{$A'_{J}$}$ denotes the subspace of  $A'$ spanned by $J$, and $U_{a}\index{$U_{a}$}$ denotes the group defined in \cite[3.12]{MR1235019}, that is a compact open subgroup of the root subgroup $U_{Da}$ associated with $Da$.
%We note that $\mathcal{M}_{J} \subset P_{J}$.
We write $\mathcal{U}_{J}\index{$\mathcal{U}_{J}$}$ for the radical of $\mathcal{M}_{J}$.
Then, according to \cite[3.15]{MR1235019}, the inclusion map
\[
\mathcal{M}_{J} \rightarrow P_{J}
\]
induces an isomorphism
\[
\mathcal{M}_{J}/\mathcal{U}_{J} \rightarrow P_{J}/U_{J} \simeq \mathbf{M}_{J}(k_{F}).
\]
%Indeed, there exists a parahoric subgroup $\mathcal{M}_{J}$ of a reductive subgroup $\mathfrak{M}$ of $G$ that only depends on $S$ and $J$ 
%We also note that $\mathbf{M}_{J}$ does not depend on the choice of $B$, and the quotient $P_{J, B'}/U_{J, B'}$ is canonically isomorphic to $P_{J, B}/U_{J, B}$ (see \cite[3.15]{MR1235019}).
We identify $\mathbf{M}_{J}(k_{F})$ with $\mathcal{M}_{J}/\mathcal{U}_{J}$, that does not depend on $B$.
In particular, we may identify $P_{J, B}/U_{J, B}$ with $P_{J, B'}/U_{J, B'}$ for another basis $B'$ canonically.
%hence we can regard $\rho$ as a representation of  $P_{J, B'}/U_{J, B'}$ or $P_{J, B'}$.

Let $(\rho, V_{\rho})\index{$\rho$}$ be an irreducible cuspidal representation of $\mathbf{M}_{J}(k_{F})$. We also regard $\rho$ as an irreducible smooth representation of $P_{J}$ via inflation.
We will explain the description of the endomorphism algebra $\End_{G(F)}\left(\ind_{P_{J}}^{G(F)} (\rho)\right)$ in \cite{MR1235019}.

Morris defined a subgroup $W(J, \rho)$\index{$W(J, \rho)$} of $W$ in \cite[4.16]{MR1235019} as
\[
W(J, \rho) = \{
w \in W \mid w J = J, \ ^{\dot{w}}\!\rho \simeq \rho
\}.
\]
Here, we regard $\rho$ as an irreducible representation of $\mathcal{M}_{J}$ via inflation.
% (see \cite[3.15]{MR1235019} and \cite[4.14]{MR1235019}).
The group $W(J, \rho)$ has a subgroup $R(J, \rho)$ that is isomorphic to the affine Weyl group of an affine root system $\Gamma'(J, \rho)$.
We will explain the definition of $\Gamma'(J, \rho)$ and $R(J, \rho)$.

When $\abs{B \backslash J} = 1$, we set $\Gamma'(J, \rho) = \emptyset$ and $R(J, \rho) = \{1\}$.
In this case, all of our results become trivial.
Hence, in the rest of paper, we assume that $\abs{B \backslash J} > 1$.
Let $a \in \Phi_{\aff} \backslash A'_{J}$ such that $J \cup \{a\}$ is contained in a basis $B'$ of $\Phi_{\aff}$.
For $* = J$ or $* = J \cup \{a\}$, let $W_{*}\index{$W_{J}$}$ denote the subgroup of $W$ generated by $s_{b} \ (b \in *)$.
Let $u$ denote the unique element of $W_{J \cup \{a\}}$ satisfying
\[
u(J \cup \{a\}) = -(J \cup \{a\}).
\]
We also define $t \in W_{J}$ as the element satisfying
\[
tJ = -J.
\]
We define
\[
v[a, J] = ut.\index{$v[a, J]$}
\]
For an element $a \in \Phi_{\aff} \backslash A'_{J}$ such that $v[a, J] \in W(J, \rho)$, we define a number $p_{a} \ge 1$ as follows (see \cite[Subsection~7.1]{MR1235019}).
We have the parahoric subgroup $P_{J \cup \{a\}, B'}$ with radical $U_{J \cup \{a\}, B'}$ associated with $J \cup \{a\} \subset B'$.
We also have a connected reductive group $\mathbf{M}_{J \cup \{a\}}$ defined over $k_{F}$ such that $P_{J \cup \{a\}, B'}/U_{J \cup \{a\}, B'}$ is isomorphic to $\mathbf{M}_{J \cup \{a\}}(k_{F})$.
Then, we have 
\[
U_{J \cup \{a\}, B'} \subset U_{J, B'} \subset P_{J, B'} \subset P_{J \cup \{a\}, B'}.
\]
Moreover, according to \cite[Th\'eor\`eme~4.6.33]{MR756316}, the quotient $P_{J, B'} / U_{J \cup \{a\}, B'}$ can be identified with the group of $k_{F}$-valued points of a parabolic subgroup of $\mathbf{M}_{J \cup \{a\}}$ with Levi factor $\mathbf{M}_{J}$.
We consider the parabolically induced representation
\[
\ind_{P_{J, B'} / U_{J \cup \{a\}, B'}} ^{\mathbf{M}_{J \cup \{a\}}(k_{F})} (\rho)
\]
of $\mathbf{M}_{J \cup \{a\}}(k_{F})$.
The assumption $v[a, J] \in W(J, \rho)$ implies that this representation splits into two inequivalent irreducible representations $\rho_{1}$ and $\rho_{2}$.
We may assume that $\dim(\rho_{1}) \le \dim(\rho_{2})$, and we define $p_{a}$ as
\[
p_{a} = \frac{\dim(\rho_{2})}{\dim(\rho_{1})}.\index{$p_{a}$}\index{$p_{a}$}
\]
According to \cite[Subsection~7.1]{MR1235019}, $p_{a}$ does not depend on the choice of $B'$.
We define
\[
\Gamma(J, \rho) = \{
a \in \Phi_{\aff} \backslash A'_{J} \mid v[a, J] \in W(J, \rho), \ p_{a} >1
\},\index{$\Gamma(J, \rho)$}
\]
and let $R(J, \rho)\index{$R(J, \rho)$}$ be the subgroup of $W(J, \rho)$ generated by $v[a, J]$ for all $a \in \Gamma(J, \rho)$. 
The definition of $v[a, J]$ implies that any lift of $v[a, J]$ is contained in a parahoric subgroup of $G(F)$.
Hence, we have
\[
R(J, \rho) \subset W_{G'} = W_{\aff}. 
\]
In particular, we have
\begin{align}
\label{RjrhocontainedinG1}
R(J, \rho) \subset W_{G^{1}}.
\end{align}
According to \cite[Lemma~7.2]{MR1235019}, $\Gamma(J, \rho)$ is $W(J, \rho)$-invariant.
Hence, $R(J, \rho)$ is a normal subgroup of $W(J, \rho)$.
%Let $A'_{J}$ denote the subspace of  $A'$ spanned by $J$.
We define $\Gamma'(J, \rho)\index{$\Gamma'(J, \rho)$}$ as the image of $\Gamma(J, \rho)$ on $A'/A'_{J}$ via the natural projection.
We also define 
\[
\Gamma(J, \rho)^{+} = \Gamma(J, \rho) \cap \Phi_{\aff}^{+}\index{$\Gamma(J, \rho)^{+}$}
\]
and $\Gamma'(J, \rho)^{+}\index{$\Gamma'(J, \rho)^{+}$}$ as the projection of $\Gamma(J, \rho)^{+}$ on $A'/A'_{J}$.

We may regard $\Gamma'(J, \rho)$ as a set of affine-linear functions on a Euclidean space $\mathcal{A}^{J}_{\Gamma}$ as follows.
Let
\[
\mathcal{A}^{J} = \{
x \in \mathcal{A} \mid a(x) = 0 \ (a \in J)
\}.\index{$\mathcal{A}^{J}$}
\]
Then, $\mathcal{A}^{J}$ is an affine space with the vector space of translations
\[
V^{J} = \{
y \in V \mid \alpha(y) = 0 \ (\alpha \in DJ)
\}.\index{$V^{J}$}
\]
We also define $V^{\Gamma}$ as
\[
V^{\Gamma} = \{
y \in V \mid \alpha(y) = 0 \ (\alpha \in D\Gamma(J, \rho))
\},\index{$V^{\Gamma}$}
\]
and 
\[
V^{J, \Gamma} = V^{J} \cap V^ {\Gamma}.\index{$V^{J, \Gamma}$}
\]
Finally, we define
\[
\mathcal{A}^{J}_{\Gamma} = \mathcal{A}^{J}/V^{J, \Gamma},\index{$\mathcal{A}^{J}_{\Gamma}$}
\]
that is an affine space with the vector space of translations
\[
V^{J}_{\Gamma} = V^{J}/V^{J, \Gamma}.\index{$V^{J}_{\Gamma}$}
\]
Let $( , )_{0}$ be a $W_{0}$-invariant inner product on $V$.
We also use the same notion $( , )_{0}$ for the restriction of it to a subspace of $V$.
Let $(V^{J, \Gamma})^{\perp}\index{$(V^{J, \Gamma})^{\perp}$}$ denote the orthogonal complement of $V^{J, \Gamma}$ in $V^{J}$ with respect to $( , )_{0}$.
Then, the natural projection $V^{J} \to V^{J}_{\Gamma}$ restricts to an isomorphism
\begin{align}
\label{orthogonalcomplementmorris}
(V^{J, \Gamma})^{\perp} \rightarrow V^{J}_{\Gamma}.
\end{align}
We define an inner product on $V^{J}_{\Gamma}$ by transporting the inner product $(,)_{0}$ on $(V^{J, \Gamma})^{\perp}$ via \eqref{orthogonalcomplementmorris}.
Thus, the affine space $\mathcal{A}^{J}_{\Gamma}$ is a Euclidean space, and we can canonically regard $\Gamma'(J, \rho^{M}_{0})$ as a set of affine-linear functions on $\mathcal{A}^{J}_{\Gamma}$.
Moreover, we obtain the following:
\begin{proposition}[{\cite[Proposition~7.3 (a)]{MR1235019}}]
\label{proposition7.3ofmorris}
The set $\Gamma'(J, \rho)$ is an affine root system on $\mathcal{A}^{J}_{\Gamma}$, and $\Gamma'(J, \rho)^{+}$ is a set of positive affine roots of $\Gamma'(J, \rho)$.
For $a \in \Gamma(J, \rho)$, let $s_{a+A'_{J}}$ denote the reflection on $\mathcal{A}^{J}_{\Gamma}$ corresponding to $a + A'_{J} \in \Gamma'(J, \rho)$, and let $W_{\aff}\left(\Gamma'(J, \rho)\right)\index{$W_{\aff}\left(\Gamma'(J, \rho)\right)$}$ denote the affine Weyl group of the affine root system $\Gamma'(J, \rho)$.
Then, the action of $v[a, J] \in R(J, \rho)$ on $\mathcal{A}$ preserves $\mathcal{A}^{J}$ and induces a well-defined action on $\mathcal{A}^{J}_{\Gamma}$ that coincides with $s_{a + A'_{J}}$.
Moreover, the map
\[
v[a, J] \mapsto s_{a + A'_{J}}
\]
defines an isomorphism
\[
R(J, \rho) \rightarrow W_{\aff}\left(\Gamma'(J, \rho)\right).
\]
\end{proposition}
We also note the following:
\begin{lemma}[cf. \!{\cite[Theorem~6]{MR576184}}]
\label{reducibiityofgamma}
The affine root system $\Gamma'(J, \rho)$ is reduced.
\end{lemma}
\begin{proof}
Let $a_1, a_2 \in \Gamma(J, \rho)$ and $\lambda \in \mathbb{R}$ such that
\begin{align}
\label{reducibilityofaffineroot}
a_1 + A'_{J} = \lambda \left(
a_2 + A'_{J}
\right).
\end{align}
For $i=1,2$, let $A'_{i}$ denote the subspace of $A'$ spanned by $J \cup \{a_i\}$.
Then, assumption~\eqref{reducibilityofaffineroot} implies $A'_{1} = A'_{2}$.
Hence, we obtain
\[
a_{2} \in A'_{2} = A'_{1}.
\]
Since $J \cup \{a_1\}$ is contained in a basis of $\Phi_{\aff}$, we can write
\[
a_{2} = m_{1}a_{1} + \sum_{b \in J}m_{b}b
\]
with rational integer coefficients $m_{1}, m_{b}$.
Similarly, we obtain that
\[
a_{1} \in A'_{1} = A'_{2},
\]
and we can write
\[
a_{1} = m_{2}a_{2} + \sum_{b \in J}m'_{b}b
\]
with rational integer coefficients $m_{2}, m'_{b}$.
Now, we have
\[
a_{1} = m_{1}m_{2}a_{1} + \sum_{b \in J}(m_{2}m_{b} + m'_{b})b.
\]
Using the assumption that $J \cup \{a_1\}$ is contained in a basis of $\Phi_{\aff}$ again, we conclude $m_{1}m_{2}=1$, hence $m_1, m_2 \in \{\pm 1\}$.
Thus, we conclude that
\[
a_{1} = m_{2}a_{2} + \sum_{b \in J}m'_{b}b \in \pm (a_{2} + A'_{J}),
\]
hence $\lambda = \pm 1$.
Thus, $\Gamma'(J, \rho)$ is a reduced affine root system.
\end{proof}
\begin{corollary}
\label{corollaryofreducibiityofgamma}
Let $a_1, a_2 \in \Gamma(J, \rho)^{+}$ such that
\[
a_1 + A'_{J} = \lambda \left(
a_2 + A'_{J}
\right)
\]
for some $\lambda \in \mathbb{R}$.
Then, we obtain $a_{1} = a_{2}$.
In particular, the map
\[
\Gamma(J, \rho)^{+} \rightarrow \Gamma'(J, \rho)^{+}
\]
defined as 
\[
a \mapsto a + A'_{J}
\]
is injective.
\end{corollary}
\begin{proof}
Since $a_1$ and $a_2$ are positive, the coefficients $m_1, m_2, m_{b}, m'_b$ in the proof of Lemma~\ref{reducibiityofgamma} are all non-negative.
Then, the equation
\[
a_{1} = m_{1}m_{2}a_{1} + \sum_{b \in J}(m_{2}m_{b} + m'_{b})b.
\]
implies that
\[
m_1 = m_2 =1, \ m_{b} = m'_{b} =0 \ (b \in J).
\]
Hence, we have $a_1 = a_2$.
\end{proof}

For $a' \in \Gamma'(J, \rho)$, let $D_{J}(a')\index{$D_{J}(a')$}$ denote the gradient of $a'$, that is a linear function on $V^{J}_{\Gamma}$.
Hence, for $a \in \Gamma(J, \rho)$, we obtain
\[
D_{J}(a + A'_{J}) = (Da)\restriction_{V^{J}}.
\]
Here, we identify a linear function $D_{J}(a + A'_{J})$ on $V^{J}_{\Gamma}$ with a linear function on $V^{J}$ that vanishes on $V^{J, \Gamma}$.
\begin{comment}
For $w \in W_{\aff}\left(\Gamma'(J, \rho)\right)$, let $D_{J}(w)$ denote the derivative of $w$.
Hence,
\[
D_{J}(w) \colon V^{J}_{\Gamma} \rightarrow V^{J}_{\Gamma}
\]
is the linear map such that
\[
w(x + y) = w(x) + (D_{J}(w))(y)
\]
for all $x \in \mathcal{A}^{J}_{\Gamma}$ and $y \in V^{J}_{\Gamma}$.
According to \cite[(1.5)]{MR357528}, for $a' \in \Gamma'(J, \rho)$, we have
\[
D_{J}(s_{a'}) = s_{D_{J}(a')}.
\]
\end{comment}

Let $B(J, \rho)\index{$B(J, \rho)$}$ denote the basis of $\Gamma'(J, \rho)$ with respect to the positive system $\Gamma'(J, \rho)^{+}$, and we define a subset $S(J, \rho)\index{$S(J, \rho)$}$ of $W_{\aff}\left(\Gamma'(J, \rho)\right)$ as 
\[
S(J, \rho) = \{
s_{a'} \mid a' \in B(J, \rho)
\}.
\]

Let
\[
C(J, \rho) = \{
w \in W(J, \rho) \mid w(\Gamma(J, \rho)^{+}) \subset \Gamma(J, \rho)^{+}
\}.\index{$C(J, \rho)$}
\]
According to \cite[Proposition~7.3 (b)]{MR1235019}, we have
\[
W(J, \rho) = R(J, \rho) \rtimes C(J, \rho).
\]

For $w \in W(J, \rho)$, Morris defined an element $\Phi_{w} \in \End_{G(F)}\left(\ind_{P_{J}}^{G(F)} (\rho)\right)\index{$\Phi_{w}$}$ such that the corresponding element $\phi_{w} \in \mathcal{H}(G(F), \rho)$ via isomorphism~\eqref{heckevsend} is supported on $P_{J} \dot{w} P_{J}$. 
We note that the element $\Phi_{w}$ here is written as $T_{w}$ in \cite[Section~7]{MR1235019}.
The following theorem is the main theorem of \cite{MR1235019}: 
\begin{theorem}[{\cite[Theorem~7.12]{MR1235019}}]
\label{theorem7.12ofmorris}
The endomorphism algebra $\End_{G(F)}\left(\ind_{P_{J}}^{G(F)} (\rho)\right)$ has a basis 
\[
\{
\Phi_{w} \mid w \in W(J, \rho)
\}
\]
as a vector space.
Moreover, the multiplication for this algebra can be described as follows:
Let $w \in W(J, \rho)$, $t \in C(J, \rho)$, and $v=v[a, J]$ for an element $a \in \Gamma(J, \rho)^{+}$ such that $a + A'_{J} \in B(J, \rho)$.
Then,
\begin{enumerate}
\item \[
\Phi_{w}\Phi_{t} = \chi(w, t)\Phi_{wt},
\]
\item \[
\Phi_{t}\Phi_{w} = \chi(t, w)\Phi_{tw},
\]
\item \[
\Phi_{v}\Phi_{w} =
\begin{cases}
\Phi_{vw} \ &(w^{-1} (a) \in \Gamma(J, \rho)^{+}), \\
p_{a}\Phi_{vw} + (p_{a}-1)\Phi_{w}  \ &(w^{-1} (a) \in - \Gamma(J, \rho)^{+}),
\end{cases}
\]
\item \[
\Phi_{w}\Phi_{v} =
\begin{cases}
\Phi_{wv} \ &(wa \in \Gamma(J, \rho)^{+}), \\
p_{a}\Phi_{wv} + (p_{a}-1)\Phi_{w}  \ &(wa \in - \Gamma(J, \rho)^{+}).
\end{cases}
\]
\end{enumerate}
Here, $\chi\index{$\chi$}$ denotes the $2$-cocycle on $W(J, \rho) \times W(J, \rho)$ defined in \cite[7.11]{MR1235019} (denoted as $\mu$ there).
%, that is trivial on $R(J, \rho) \times R(J, \rho)$.
\end{theorem}

We define a parameter function $q$ on $S(J, \rho)$ as
\begin{align}
\label{definitionofqmorris}
q_{s_{a+A'_{J}}} = p_{a}
\end{align}
for $a \in \Gamma(J, \rho)^{+}$ such that $a + A'_{J} \in B(J, \rho)$.
According to Corollary~\ref{corollaryofreducibiityofgamma}, any element $b \in \Gamma(J, \rho)^{+}$ with
\[
s_{a+ A'_{J}} = s_{b + A'_{J}}
\]
is equal to $a$. Hence, the parameter $q_{s_{a+A'_{J}}}$ is well-defined.
Moreover, according to \cite[Lemma~7.2 (b)]{MR1235019}, the function $q$ satisfies condition~\eqref{winvarianceofqparameter} in Appendix~\ref{Iwahori-Hecke algebras and affine Hecke algebras}.
\begin{comment}
Moreover, according to \cite[Lemma~7.2 (b)]{MR1235019}, the function $q$ satisfies condition~\eqref{winvarianceofqparameter} in appendix~\ref{Iwahori-Hecke algebras and affine Hecke algebras}.
Hence, we may define 
\[
q_{w} = q_{s_1} q_{s_2} \cdots q_{s_{r}}
\]
for $w \in W_{\aff}\left(\Gamma'(J, \rho)\right)$ with a reduced expression
\[
w = s_{1}s_{2} \cdots s_{r}.
\]
\end{comment}
Let $\mathcal{H}(W_{\aff}\left(\Gamma'(J, \rho)\right), q)\index{$\mathcal{H}(W_{\aff}\left(\Gamma'(J, \rho)\right), q)$}$ denote the Iwahori-Hecke algebra associated with the Coexter system $(W_{\aff}\left(\Gamma'(J, \rho)\right), S(J, \rho))$ and the parameter function $q$.
% defined in Appendix~\ref{Iwahori-Hecke algebras and affine Hecke algebras}.
We write the standard basis of $\mathcal{H}(W_{\aff}\left(\Gamma'(J, \rho)\right), q)$ as
\[
\{
T^{\Mor}_{w} \mid w \in W_{\aff}\left(\Gamma'(J, \rho)\right)
\}.\index{$T^{\Mor}_{w}$}
\]
Then, we obtain:
\begin{corollary}
\label{corollaryoftheorem7.12ofmorris}
Let $\mathcal{H}(R(J, \rho))\index{$\mathcal{H}(R(J, \rho))$}$ denote the subspace of $\End_{G(F)}\left(\ind_{P_{J}}^{G(F)} (\rho)\right)$ spanned by 
\[
\left\{
\Phi_{w} \mid w \in R(J, \rho)
\right\}.
\]
Then, $\mathcal{H}(R(J, \rho))$ is a subslgebra of $\End_{G(F)}\left(\ind_{P_{J}}^{G(F)} (\rho)\right)$.
Moreover, there exists an isomorphism
\[
I^{\Mor} \colon \mathcal{H}(R(J, \rho)) \rightarrow \mathcal{H}(W_{\aff}\left(\Gamma'(J, \rho)\right), q)\index{$I^{\Mor}$}
\]
such that
\[
I^{\Mor} \left(
\Phi_{v[a, J]} 
\right) = T^{\Mor}_{s_{a + A'_{J}}}
\]
for all $a \in \Gamma(J, \rho)$.
\end{corollary}

We rewrite Corollary~\ref{corollaryoftheorem7.12ofmorris} in terms of an affine Hecke algebra.
We use the same notation as Appendix~\ref{Iwahori-Hecke algebras and affine Hecke algebras}.
The affine root system $\Gamma'(J, \rho)$ is not necessarily irreducible. However, we can apply the results of Appendix~\ref{Iwahori-Hecke algebras and affine Hecke algebras} to this case by arguing component by component. 
Fix a spacial point $e$ for $\Gamma'(J, \rho)$ in the closure of the chamber corresponding to the basis $B(J, \rho)$.
According to Theorem~\ref{theorem1.8ofsol21}, the Iwahori-Hecke algebra $\mathcal{H}(W_{\aff}\left(\Gamma'(J, \rho)\right), q)$ is isomorphic to the affine Hecke algebra 
\[
\mathcal{H}^{\Mor} = \mathcal{H} \left(
\mathcal{R}^{\Mor}, \lambda^{\Mor}, (\lambda^{*})^{\Mor}, q_{F}
\right)\index{$\mathcal{H}^{\Mor}$}
\]
associated with a based root datum
\[
\mathcal{R}^{\Mor} = \left(
\Hom_{\mathbb{Z}} \left(
\mathbb{Z} (R^{\Mor})^{\vee}, \mathbb{Z}
\right), R^{\Mor},
\mathbb{Z} (R^{\Mor})^{\vee},  (R^{\Mor})^{\vee}, \Delta^{\Mor}
\right),
\]
label functions
\[
\lambda^{\Mor}, (\lambda^{*})^{\Mor} \colon \Delta^{\Mor} \rightarrow \mathbb{R}_{>0},
\]
and the parameter $q_{F}$.
We explain the definitions of the based root datum $\mathcal{R}^{\Mor}$ and the label functions $\lambda^{\Mor}, (\lambda^{*})^{\Mor}$ (for more details, see the last part of Appendix~\ref{Iwahori-Hecke algebras and affine Hecke algebras}).
Let $\Gamma'(J, \rho)_{e}\index{$\Gamma'(J, \rho)_{e}$}$ denote the set of affine roots in $\Gamma'(J, \rho)$ that vanish at $e$, and we write
\[
\Gamma'(J, \rho)^{+}_{e} =
\Gamma'(J, \rho)_{e} \cap \Gamma'(J, \rho)^{+}\index{$\Gamma'(J, \rho)^{+}_{e}$}
\]
and
\[
B(J, \rho)_{e} = \Gamma'(J, \rho)_{e} \cap B(J, \rho).\index{$B(J, \rho)_{e}$}
\] 
We define
\[
R^{\Mor} = \left\{
D_{J}(a')/k_{a'} \mid a' \in \Gamma'(J, \rho)_{e}
\right\}\index{$R^{\Mor}$}
\]
and
\[
(R^{\Mor})^{\vee} = \left\{
k_{a'} (D_{J} (a'))^{\vee} \mid a' \in \Gamma'(J, \rho)_{e}
\right\},\index{$(R^{\Mor})^{\vee}$}
\]
where $k_{a'}\index{$k_{a'}$}$ is the smallest positive real number such that $a' + k_{a'} \in \Gamma'(J, \rho)$, and $(D_{J} (a'))^{\vee}$ denotes the coroot in the dual root system $\left(D_{J}\left(\Gamma'(J, \rho)_{e}\right)\right)^{\vee}\index{$\left(D_{J}\left(\Gamma'(J, \rho)_{e}\right)\right)^{\vee}$}$ of the root system 
\[
D_{J}\left(\Gamma'(J, \rho)_{e}\right) = \left\{
D_{J}(a') \mid a' \in \Gamma'(J, \rho)_{e}
\right\}\index{$D_{J}\left(\Gamma'(J, \rho)_{e}\right) $}
\]
corresponding to the root $D_{J}(a') \in D_{J}\left(\Gamma'(J, \rho)_{e}\right)$. 
We write $W_{0}\left(R^{\Mor}\right)\index{$W_{0}\left(R^{\Mor}\right)$}$ for the Weyl group of $R^{\Mor}$, and $W_{\aff}\left(R^{\Mor}\right)\index{$W_{\aff}\left(R^{\Mor}\right)$}$ for the affine Weyl group of $R^{\Mor}$.
We also define
\[
\Delta^{\Mor} = \left\{
D_{J}(a')/k_{a'} \mid a' \in B(J, \rho)_{e}
\right\}.\index{$\Delta^{\Mor}$}
\]
Finally, we define label functions 
\[
\lambda^{\Mor}, (\lambda^{*})^{\Mor} \colon \Delta^{\Mor} \rightarrow \mathbb{C},\index{$\lambda^{\Mor}$}\index{$(\lambda^{*})^{\Mor}$}
\]
as
\[
\lambda^{\Mor} \left(
D_{J}(a')/k_{a'}
\right) = \log(q_{s_{a'}})/\log(q_{F})
\]
and
\[
(\lambda^{*})^{\Mor} \left(
D_{J}(a')/k_{a'}
\right) =
\begin{cases}
\log(q_{s_{a'}})/\log(q_{F}) & \left( D_{J}(a')/k_{a'} \not \in 2 \Hom_{\mathbb{Z}}(\mathbb{Z}(R^{\Mor})^{\vee}, \mathbb{Z}) \right), \\
\log(q_{s_{b'}})/\log(q_{F}) & \left( D_{J}(a')/k_{a'} \in 2 \Hom_{\mathbb{Z}}(\mathbb{Z}(R^{\Mor})^{\vee}, \mathbb{Z}) \right)
\end{cases}
\]
for $a' \in B(J, \rho)_{e}$. 
Here, $b'$ is the unique element of $B(J, \rho) \backslash B(J, \rho)_{e}$ that is contained in the same irreducible component of $\Gamma'(J, \rho)$ as $a'$.
For $a \in \Gamma(J, \rho)^{+}$ such that $a' := a + A'_{J} \in B(J, \rho)_{e}$, we write
\[
r(a) = D_{J}(a')/k_{a'}.\index{$r(a)$}
%=  k_{a + A'_{J}} \left(
%(D a)\restriction_{V^{J}}
%\right)^{\vee}.
\]
Then, the definition of the parameter function $q$ \eqref{definitionofqmorris} implies that
\begin{align}
\label{lambdaofmorris}
\lambda^{\Mor} \left(
r(a)
\right) = \log(p_{a})/\log(q_{F})
\end{align}
and
\begin{align}
\label{lambdastarofmorris}
(\lambda^{*})^{\Mor} \left(
r(a)
\right) =
\begin{cases}
\log(p_{a})/\log(q_{F}) & \left( r(a) \not \in 2 \Hom_{\mathbb{Z}}(\mathbb{Z}(R^{\Mor})^{\vee}, \mathbb{Z}) \right), \\
\log(p_{b})/\log(q_{F}) & \left( r(a) \in 2 \Hom_{\mathbb{Z}}(\mathbb{Z}(R^{\Mor})^{\vee}, \mathbb{Z}) \right),
\end{cases}
\end{align}
where $b \in \Gamma(J, \rho)^{+}$ denotes the unique element such that $b + A'_{J} = b'$.
In particular, $\lambda^{\Mor}, (\lambda^{*})^{\Mor}$ are $\mathbb{R}_{>0}$-valued.

We rewrite \eqref{lambdastarofmorris} as follows.
Let $a \in \Gamma(J, \rho)^{+}$ such that $a' := a + A'_{J} \in B(J, \rho)_{e}$.
Suppose that 
\[
D_{J}(a')/k_{a'} \in 2 \Hom_{\mathbb{Z}}(\mathbb{Z}(R^{\Mor})^{\vee}, \mathbb{Z}).
\]
In this case, the irreducible component of $R^{\Mor}$ containing $D_{J}(a')/k_{a'}$ has type $C_{n} \ (n \ge 1)$, and if $n \ge 2$, $D_{J}(a')/k_{a'}$ is a long root.
Hence, $D_{J}(a')/k_{a'}$ is $W_{0}\left(R^{\Mor}\right)$-associate with the highest root $\phi$ of the irreducible component of $R^{\Mor}$ containing $D_{J}(a')/k_{a'}$ with respect to the basis $\Delta^{\Mor}$.
%\[
%\left(
%\Delta^{\Mor}
%\right)^{\vee} = \{
%D_{J}(a')/k_{a'} \mid a' \in B(J, \rho)_{e}
%\}.
%\]
In particular, the reflection $s_{1 - (D_{J}(a')/k_{a'})}$ is $W_{0}\left(R^{\Mor}\right)$-conjugate to the reflection $s_{1 - \phi}$.
On the other hand, the reflection $s_{1 - (D_{J}(a')/k_{a'})}$ corresponds to the reflection 
\[
s_{k_{a'} - a'} \in W_{\aff}\left(\Gamma'(J, \rho)\right),
\] 
and the reflection $s_{1 - \phi}$ corresponds to the reflection 
\[
s_{b'} \in W_{\aff}\left(\Gamma'(J, \rho)\right)
\]
via isomorphism~\eqref{affinesplitting} (see the paragraph following isomorphism~\eqref{affinesplitting}).
Hence, we obtain that $s_{k_{a'} - a'}$ and $s_{b'}$ are $W_{\aff}\left(\Gamma'(J, \rho)\right)$-conjugate.
Thus, we obtain that $k_{a'} - a'$ and $b'$ are $W_{\aff}\left(\Gamma'(J, \rho)\right)$-associate.
Let $\widetilde{(k_{a'} - a')} \in \Gamma(J, \rho)^{+}$ denote the unique element such that
\[
\widetilde{(k_{a'} - a')} + A'_{J} = k_{a'} - a',
\]
and we write $p'_{a} = p_{\widetilde{(k_{a'} - a')}}\index{$p'_{a}$}$.
Then, according to Proposition~\ref{proposition7.3ofmorris} and Corollary~\ref{corollaryofreducibiityofgamma}, we obtain that $\widetilde{(k_{a'} - a')}$ and $b$ are $R(J, \rho)$-associate, hence
\cite[Lemma~7.2 (b)]{MR1235019} implies that $p_{b} = p'_{a}$.
On the other hand, for $a \in \Gamma(J, \rho)^{+}$ such that $a' = a + A'_{J} \in B(J, \rho)_{e}$, and
\[
D_{J}(a')/k_{a'} \not \in 2 \Hom_{\mathbb{Z}}(\mathbb{Z}(R^{\Mor})^{\vee}, \mathbb{Z}),
\]
we can define $\widetilde{(k_{a'}- a')}$ and $p'_{a}$ in the same way as the case of
\[
D_{J}(a')/k_{a'} \in 2 \Hom_{\mathbb{Z}}(\mathbb{Z}(R^{\Mor})^{\vee}, \mathbb{Z}).
\]
In this case, there exists $t \in \mathbb{Z}(R^{\Mor})^{\vee} \subset W_{\aff}\left(R^{\Mor}\right)$ such that 
\[
D_{J}(a')(t)/k_{a'} = 1.
\]
Then, we obtain that 
\[
\left(t \cdot s_{D_{J}(a')/k_{a'}}\right)\left(D_{J}(a')/k_{a'}\right) = 1 - D_{J}(a')/k_{a'},
\]
hence $D_{J}(a')/k_{a'}$ and $1 - (D_{J}(a')/k_{a'})$ are $W_{\aff}\left(R^{\Mor}\right)$-associate.
Thus, $a'$ and $k_{a'} - a'$ are $W_{\aff}\left(\Gamma'(J, \rho)\right)$-associate.
Therefore, in this case, $a$ and $\widetilde{(k_{a'} - a')}$ are $R(J, \rho)$-associate, hence $p_{a} = p'_{a}$.
Now, we can rewrite \eqref{lambdastarofmorris} as
\begin{align}
\label{lambdastarofmorrisrefined}
(\lambda^{*})^{\Mor} \left(
r(a)
\right) = \log(p'_{a})/\log(q_{F}).
\end{align}

We define a subgroup $T(J, \rho)$ of $R(J, \rho)$ as
\[
T(J, \rho) = \{
t \in R(J, \rho) \mid (Dt)\restriction_{V^{J}} = \id
\}.\index{$T(J, \rho)$}
\]
The definition of $T(J, \rho)$ implies that for any $t \in T(J, \rho)$, there exists $\widetilde{v(t)} \in (V^{J, \Gamma})^{\perp}\index{$\widetilde{v(t)}$}$ such that
\[
t(x) = x+\widetilde{v(t)}
\]
for all $x \in \mathcal{A}^{J}$.
Let $v(t)\index{$v(t)$}$ denote the projection of $\widetilde{v(t)}$ on $V^{J}_{\Gamma}$.
Hence, 
\[
t(x) = x+v(t)
\]
for all $x \in \mathcal{A}^{J}_{\Gamma}$.
An element $t \in T(J, \rho)$ maps to $v(t) \in \mathbb{Z}(R^{\Mor})^{\vee} \subset V^{J}_{\Gamma}$ via isomorphism
\[
R(J, \rho) \rightarrow W_{\aff}\left(\Gamma'(J, \rho)\right)
\]
of Proposition~\ref{proposition7.3ofmorris}, and the map
\[
t \mapsto v(t)
\]
defines an isomorphism
\[
T(J, \rho) \rightarrow \mathbb{Z}(R^{\Mor})^{\vee}.
\]

Combining Corollary~\ref{corollaryoftheorem7.12ofmorris} with Theorem~\ref{theorem1.8ofsol21}, we obtain:
\begin{corollary}
\label{corollaryoftheorem7.12ofmorrisaffinever}
There exists an isomorphism
\[
I^{\Mor} \colon \mathcal{H}(R(J, \rho)) \rightarrow \mathcal{H}^{\Mor}\index{$I^{\Mor}$}
\]
such that
\[
I^{\Mor} \left(
\Phi_{v[a, J]} 
\right) = T^{\Mor}_{s_{r(a)}}
\]
for $a \in \Gamma(J, \rho)^{+}$ such that $a + A'_{J} \in B(J, \rho)_{e}$, and
\[
I^{\Mor} \left(
\Phi_{t}
\right) = q_{v(t)}^{1/2} \cdot \theta_{v(t)}
\]
for $t \in T(J, \rho)$ such that $\left(D_{J}(a')\right)(v(t)) \ge 0$ for all $a' \in B(J, \rho)_{e}$.
%$a \in  \Gamma(J, \rho)^{+}$ such that $a + A'_{J} \in B(J, \rho)_{e}$.
\end{corollary}
\section{A review of Solleveld's results}
\label{A review of Solleveld's results}
In this section, we review the results in \cite{MR4432237}.
Let $P=MU$ be a parabolic subgroup of $G$ with Levi factor $M$ and unipotent radical $U$.
Let $N_{G}(M)(F)\index{$N_{G}(M)(F)$}$ denote the normalizer of $M$ in $G(F)$.
Let $(\sigma, E)$ be an irreducible supercusidal representation of $M(F)$, and let $\mathfrak{s}_{M}$ denote the inertial equivalence class of the pair $(M, \sigma)$ in $M$.
We take an irreducible subrepresentation $\sigma_{1}$ of $\sigma\restriction_{M^1}$.
We define
\[
M_{\sigma} = 
I_{M(F)}(\sigma_{1}) =
\{
m \in M(F) \mid ^m\!\sigma_{1} \simeq \sigma_{1}
\}.\index{$M_{\sigma}$}
\]
Since $M^1$ is a normal subgroup of $M(F)$, and the quotient group $M(F)/M^1$ is abelian, $M_{\sigma}$ is independent of the choice of $\sigma_{1}$.
We assume:
\begin{assumption}
\label{multiplicity1}
The restriction of $\sigma$ to $M^1$ is multiplicity free (see \cite[Working hypothesis~10.2]{MR4432237}).
\end{assumption}
\begin{remark}
\label{remarkaboutassumptionmultiplicity1}
According to \cite[Remark~1.6.1.3]{MR2508719}, assumption~\ref{multiplicity1} holds in many cases (see also the paragraph following \cite[Working hypothesis~10.2]{MR4432237}):
\begin{itemize}
\item when the maximal split central torus of $M$ has dimension $\le 1$,
\item when $M$ is quasi-split and $(\sigma, E)$ is generic,
\item when $M$ is a direct product of groups as in the previous two bullets. 
\end{itemize}
Moreover, according to \cite[Proposition~1.6.1.2]{MR2508719}, assumption~\ref{multiplicity1} holds if and only if the endomorphism algebra 
\[
\End_{M(F)}\left(\ind_{M^1}^{M(F)} (\sigma_{1})\right)
\]
is commutative.
The latter condition holds if $\sigma$ is a regular supercuspidal representation defined in \cite{MR4013740}, for instance (see \cite[Corollary~5.5]{2021arXiv210101873O}).
\end{remark}

Let $A_{M}\index{$A_{M}$}$ denote the maximal split central torus of $M$, and let $X^{*}(A_{M})\index{$X^{*}(A_{M})$}$ (resp.\,$X_{*}(A_{M})\index{$X_{*}(A_{M})$}$) denote the character lattice (resp.\,cocharacter lattice) of $A_M$.
We write
\[
a_{M}
=
X_{*}(A_{M}) \otimes_{\mathbb{Z}} \mathbb{R}\index{$a_{M}$}
\]
and
\[
a_{M}^{*} = X^{*}(A_{M}) \otimes_{\mathbb{Z}} \mathbb{R}.\index{$a_{M}^{*}$}
\]
Let $\langle, \rangle$ denote the canonical perfect pairing on
\[
a_{M}^{*} \times a_{M}.
\]
We define an injective map
\[
H_{M} \colon M(F)/M^{1} \rightarrow a_{M}\index{$H_{M}$}
\]
as
\[
\langle \gamma, H_{M}(m) \rangle = \ord_{F}(\gamma(m))
\]
for $m \in M(F)$ and a rational character $\gamma$ of $M$.
We note that $H_{M}(M_{\sigma}/M^{1})$ is a lattice of full rank in $a_{M}$.
We regard $M_{\sigma}/M^1$ as a subset of $a_{M}$ via $H_{M}$.
We also write
\[
\left(
M_{\sigma}/M^1
\right)^{\vee}
=
\Hom_{\mathbb{Z}}(
M_{\sigma}/M^1, \mathbb{Z}).\index{$\left(
M_{\sigma}/M^1
\right)^{\vee}$}
\]
We define an embedding 
\[
H_{M}^{\vee} \colon \left(
M_{\sigma}/M^1
\right)^{\vee} \rightarrow a_{M}^{*},\index{$H_{M}^{\vee}$}
\]
as
\[
\langle
H_{M}^{\vee}(x), H_{M}(m)
\rangle
= x(m)
\]
for $x \in \left(
M_{\sigma}/M^1
\right)^{\vee}$ and $m \in M_{\sigma}/M^{1}$.
Then, the image of $H_{M}^{\vee}$ is a lattice of full rank in $a_{M}^{*}$.
We also regard $\left(
M_{\sigma}/M^1
\right)^{\vee}$ as a subset of $a_{M}^{*}$ via $H_{M}^{\vee}$.

Let $\Sigma(G, A_M)\index{$\Sigma(G, A_M)$}$ denote the set of nonzero weights occurring in the adjoint representation of $A_{M}$ on the Lie algebra of $G$, and let $\Sigma_{\red}(A_M)\index{$\Sigma_{\red}(A_M)$}$ denote the set of indivisible elements of $\Sigma(G, A_M)$. 
For $\alpha \in \Sigma_{\red}(A_M)$, let $M_{\alpha}$ denote the Levi subgroup of $G$ that contains $M$ and the root subgroup $U_{\alpha}$ associated with $\alpha$, and whose semisimple rank is one greater than that of $M$.
Let $\alpha^{\vee} \in a_{M}$ denote the unique element that is orthogonal to the characters of the maximal split central torus $A_{M_{\alpha}}$ of $M_{\alpha}$, and satisfies
\[
\langle \alpha, \alpha^{\vee} \rangle =2.
\]
We define a subset $\Sigma_{\mathfrak{s}_{M}, \mu}$ of $\Sigma_{\red}(A_M)$ as follows.
Let $\mu^{M_{\alpha}}\index{$\mu^{M_{\alpha}}$}$ denote the Harish-Chandra's $\mu$-function defined in \cite[V.2]{MR1989693}, that is a rational function on $\mathfrak{s}_{M}$.
\begin{remark}
In \cite[V.2]{MR1989693}, the function $\mu^{M_{\alpha}}$ is only defined on a subset 
\[
\mathfrak{s}_{M, 0} =
\{
\sigma \otimes \chi
\mid \chi \in X_{\unr}(M), \, \text{$\chi$ is unitary}
\}
\]
of $\mathfrak{s}_{M}$.
However, we can define $\mu^{M_{\alpha}}$ on $\mathfrak{s}_{M}$ exactly in the same way as on $\mathfrak{s}_{M, 0}$.
\end{remark}
We define $\Sigma_{\mathfrak{s}_{M}, \mu}$ as
\[
\Sigma_{\mathfrak{s}_{M}, \mu} = \{
\alpha \in \Sigma_{\red}(A_M) \mid \text{$\mu^{M_{\alpha}}$ has a zero on $\mathfrak{s}_{M}$}
\}.\index{$\Sigma_{\mathfrak{s}_{M}, \mu}$}
\]
For $\alpha \in \Sigma_{\mathfrak{s}_{M}, \mu}$, let $s_{\alpha}$ denote the unique nontrivial element of 
\[
W(M_{\alpha}, M) =
\left(
N_{G}(M)(F) \cap M_{\alpha}(F)
\right)/M(F).
\]
Since $\mu^{M_{\alpha}}$ has a zero on $\mathfrak{s}_{M}$, \cite[5.4.2]{MR544991} implies that $s_{\alpha}$ normalizes $\sigma \otimes \chi$ for some $\chi \in X_{\unr}(M)$ (see also \cite[(3.4)]{MR4432237}).
Let $W(\Sigma_{\mathfrak{s}_{M}, \mu})\index{$W(\Sigma_{\mathfrak{s}_{M}, \mu})$}$ denote the subgroup of
\[
W(G, M, \mathfrak{s}_{M}) = \{
g \in N_{G}(M)(F) \mid \text{${^g\!\sigma} \simeq \sigma \otimes \chi$ for some $\chi \in X_{\unr}(M)$}
\}/M(F)\index{$W(G, M, \mathfrak{s}_{M}) $}
\]
generated by $s_{\alpha} \ (\alpha \in \Sigma_{\mathfrak{s}_{M}, \mu})$.
Then, according to \cite[Proposition~1.3]{MR2827179}, $\Sigma_{\mathfrak{s}_{M}, \mu}$ is a reduced root system with Weyl group $W(\Sigma_{\mathfrak{s}_{M}, \mu})$.
Let $\Sigma(P, A_M)\index{$\Sigma(P, A_M)$}$ denote the set of nonzero weights occurring in the adjoint representation of $A_{M}$ on the Lie algebra of $P$.
Then, 
\[
\Sigma_{\mathfrak{s}_{M}, \mu}(P) := \Sigma_{\mathfrak{s}_{M}, \mu} \cap \Sigma(P, A_M)\index{$\Sigma_{\mathfrak{s}_{M}, \mu}(P)$}
\]
is a set of positive roots of $\Sigma_{\mathfrak{s}_{M}, \mu}$.
We write $\Delta_{\mathfrak{s}_{M}, \mu}(P)\index{$\Delta_{\mathfrak{s}_{M}, \mu}(P)$}$ for the basis of $\Sigma_{\mathfrak{s}_{M}, \mu}$ corresponding to $\Sigma_{\mathfrak{s}_{M}, \mu}(P)$.
We also define
\[
\Sigma_{\red}(P, A_{M}) = \Sigma_{\red}(A_{M}) \cap \Sigma(P, A_M).\index{$\Sigma_{\red}(P, A_{M})$}
\]
We write
\[
R(\mathfrak{s}_{M}) = \{
w \in W(G, M, \mathfrak{s}_{M}) \mid w \left(
\Sigma_{\mathfrak{s}_{M}, \mu}(P)
\right) \subset
\Sigma_{\mathfrak{s}_{M}, \mu}(P)
\}.\index{$R(\mathfrak{s}_{M})$}
\]
According to \cite[(3.2)]{MR4432237}, we obtain
\[
W(G, M, \mathfrak{s}_{M}) = W(\Sigma_{\mathfrak{s}_{M}, \mu}) \rtimes R(\mathfrak{s}_{M}).
\]
For $\alpha \in \Sigma_{\mathfrak{s}_{M}, \mu}$, we define an element $h_{\alpha}^{\vee} \in M_{\sigma}/M^1\index{$h_{\alpha}^{\vee}$}$ as the unique generator of
\[
(M_{\sigma} \cap {M_{\alpha}^1})/M^{1} \simeq \mathbb{Z}
\]
such that $H_{M}(h_{\alpha}^{\vee}) \in \mathbb{R}_{>0} \cdot \alpha^{\vee}$.
\begin{proposition}[{\cite[Proposition~3.1]{MR4432237}}]
For any $\alpha \in \Sigma_{\mathfrak{s}_{M}, \mu}$, there exists a unique 
\[
\alpha^{\#} \in (M_{\sigma}/M^{1})^{\vee}\index{$\alpha^{\#}$}
\]
such that $H_{M}^{\vee}(\alpha^{\#}) \in \mathbb{R} \cdot \alpha$ and
\[
\langle H_{M}^{\vee}(\alpha^{\#}), H_{M}(h_{\alpha}^{\vee}) \rangle = 2.
\]
Moreover, if we write
\begin{align*}
\Sigma_{\mathfrak{s}_{M}}^{\vee} &= \{
h_{\alpha}^{\vee} \mid \alpha \in \Sigma_{\mathfrak{s}_{M}, \mu}
\},\\
\Sigma_{\mathfrak{s}_{M}} &= \{
\alpha^{\#} \mid \alpha \in \Sigma_{\mathfrak{s}_{M}, \mu}
\}, \\
\left(
\alpha^{\#} 
\right)^{\vee} &= h_{\alpha}^{\vee}\index{$\Sigma_{\mathfrak{s}_{M}}^{\vee}$}\index{$\Sigma_{\mathfrak{s}_{M}}$},
\end{align*}
then, 
\[
\left(
\left( M_{\sigma}/M^{1} \right)^{\vee}, \Sigma_{\mathfrak{s}_{M}}, M_{\sigma}/M^{1}, \Sigma_{\mathfrak{s}_{M}}^{\vee}
\right)
\]
is a reduced root datum with Weyl group $W(\Sigma_{\mathfrak{s}_{M}, \mu})$.
\end{proposition}
The parabolic subgroup $P$ also determines a set of positive roots $\Sigma_{\mathfrak{s}_{M}}(P)$ and a basis $\Delta_{\mathfrak{s}_{M}}(P)$ of $\Sigma_{\mathfrak{s}_{M}}$ as
\[
\Sigma_{\mathfrak{s}_{M}}(P) = \{
\alpha^{\#} \mid \alpha \in \Sigma_{\mathfrak{s}_{M}, \mu}(P)
\}\index{$\Sigma_{\mathfrak{s}_{M}}(P)$}
\]
and
\[
\Delta_{\mathfrak{s}_{M}}(P) = \{
\alpha^{\#} \mid \alpha \in \Delta_{\mathfrak{s}_{M}, \mu}(P)
\}.\index{$\Delta_{\mathfrak{s}_{M}}(P)$}
\]

From now on, we assume that the representation $(\sigma, E)$ satisfies \cite[Condition~3.2]{MR4432237}, that is, 
$(\sigma, E)$ is a unitary supercuspidal representation, and
\[
\mu^{M_{\alpha}}(\sigma) = 0
\]
for all $\alpha \in \Delta_{\mathfrak{s}, \mu}(P)$.
We also assume that
\[
q_{\alpha} \ge q_{\alpha *},
\]
for all $\alpha \in \Sigma_{\mathfrak{s}_{M}, \mu}$, where
$q_{\alpha}, q_{\alpha *} \ge 1$ are real numbers appearing in \cite[(3.7)]{MR4432237} (see also \cite[(3.8)]{MR4432237} and \cite[(3.11)]{MR4432237}).
We note that we can always take such a representation $\sigma$ in $\mathfrak{s}_{M}$.  

We define label functions
\[
\lambda, \lambda^{*} \colon \Sigma_{\mathfrak{s}_{M}} \rightarrow \mathbb{C}
\]
as
\begin{align}
\label{lambdaofsolleveld}
\lambda (\alpha^{\#}) = \log(q_{\alpha}q_{\alpha *})/\log(q_{F}), \ \lambda^{*}(\alpha^{\#}) = \log(q_{\alpha}q_{\alpha *}^{-1})/\log(q_{F}).
\end{align}
According to \cite[Lemma~3.3]{MR4432237} and \cite[Lemma~3.4]{MR4432237}, the restrictions of the functions $\lambda, \lambda^{*}$ to $\Delta_{\mathfrak{s}_{M}}(P)$ satisfy conditions~\eqref{w-equiv} and \eqref{lambda=lambda*} in Appendix~\ref{Iwahori-Hecke algebras and affine Hecke algebras}.
We write the affine Hecke algebra associated with the based root datum
\[
\mathcal{R}(G, \mathfrak{s}_{M}) = 
\left(
\left( M_{\sigma}/M^{1} \right)^{\vee}, \Sigma_{\mathfrak{s}_{M}}, M_{\sigma}/M^{1}, \Sigma_{\mathfrak{s}_{M}}^{\vee}, \Delta_{\mathfrak{s}_{M}}(P)
\right),\index{$\mathcal{R}(G, \mathfrak{s}_{M})$}
\]
the parameter $q_{F}$, and the label functions $\lambda, \lambda^{*}$ as
\[
\mathcal{H}(G, \mathfrak{s}_{M}) = 
\mathcal{H}\left(
\mathcal{R}(G, \mathfrak{s}_{M}), \lambda, \lambda^{*}, q_{F}
\right).\index{$\mathcal{H}(G, \mathfrak{s}_{M})$}
\]
For the definition of affine Hecke algebras, see Appendix~\ref{Iwahori-Hecke algebras and affine Hecke algebras}.
We will explain the description of the endomorphism algebra $\End_{G(F)}\left(I_{P}^{G}\left(\ind_{M^1}^{M(F)}(\sigma_{1})\right)\right)$ in terms of the affine Hecke algebra $\mathcal{H}(G, \mathfrak{s}_{M})$ \cite[Section~10]{MR4432237}.

First, we define an injection
\[
\mathbb{C}[M_{\sigma}/M^{1}] \rightarrow \End_{G(F)}\left(I_{P}^{G}\left(\ind_{M^1}^{M(F)}(\sigma_{1})\right)\right)
\]
as follows.
We consider the left regular representation of $M(F)$ on $\mathbb{C}[M(F)/M^1]$.
Then, according to \cite[(2.3)]{MR4432237}, there exists an $M(F)$-equivariant isomorphism
\begin{align}
\label{groupalgebraisomind}
\ind_{M^1}^{M(F)} (\sigma) \rightarrow \sigma \otimes \mathbb{C}[M(F)/M^1].
\end{align}
We regard 
\[\sigma \otimes \mathbb{C}[M(F)/M^1]
\]
as a $\mathbb{C}[M(F)/M^1]$-module via the left multiplication on the second factor.
Since $M(F)/M^1$ is commutative, this action commutes with the $M(F)$-action on 
\[\sigma \otimes \mathbb{C}[M(F)/M^1].
\]
We transport the $\mathbb{C}[M(F)/M^1]$-module structure to $\ind_{M^1}^{M(F)} (\sigma)$ via isomorphism~\eqref{groupalgebraisomind}.
The explicit structure is as follows \cite[(2.6)]{MR4432237}:

For $\theta_{m} \in \mathbb{C}[M(F)/M^1]$ and $f \in \ind_{M^1}^{M(F)} (\sigma)$, the element $\theta_{m} \cdot f \in \ind_{M^1}^{M(F)} (\sigma)$ is defined as
\[
(\theta_{m} \cdot f)(m') = \sigma(m^{-1}) \cdot f(mm')
\]
for $m' \in M(F)$.
We also define an action of $\mathbb{C}[M(F)/M^1]$ on $I_{P}^{G}\left(\ind_{M^1}^{M(F)} (\sigma)\right)$ by using the functoriality of $I_{P}^{G}$.
The action of $\mathbb{C}[M(F)/M^1]$ does not preserve the subspace $\ind_{M^1}^{M(F)} (\sigma_{1})$ of $\ind_{M^1}^{M(F)} (\sigma)$.
However, according to \cite[Subsection~10.1]{MR4432237} and Assumption~\ref{multiplicity1}, the restriction of the action to $\mathbb{C}[M_{\sigma}/M^1]$ preserves $\ind_{M^1}^{M(F)} (\sigma_{1})$.
We consider $\ind_{M^1}^{M(F)} (\sigma_{1})$ as a $\mathbb{C}[M_{\sigma}/M^1]$-module via this action.
Then, we obtain a map
\begin{align}
\label{isomgroupalgebraheckealgebra}
\mathbb{C}[M_{\sigma}/M^1] \rightarrow \End_{M(F)}\left(\ind_{M^1}^{M(F)} (\sigma_{1})\right).
\end{align}
According to \cite[Lemma~10.1]{MR4432237} and Assumption~\ref{multiplicity1}, this map is an isomorphism (see also the last paragraph of \cite[Subsection~10.1]{MR4432237}).
Combining \eqref{isomgroupalgebraheckealgebra} with the injection
\[
I_{P}^{G} \colon \End_{M(F)}\left(\ind_{M^1}^{M(F)} (\sigma_{1})\right) \rightarrow \End_{G(F)}\left(I_{P}^{G}\left(\ind_{M^1}^{M(F)}(\sigma_{1})\right)\right)
\]
provided by the faithful functor $I_{P}^{G}$, we obtain an injection
\begin{align}
\label{isomgroupalgebraheckealgebraparabolicinduction}
\mathbb{C}[M_{\sigma}/M^{1}] \rightarrow \End_{G(F)}\left(I_{P}^{G}\left(\ind_{M^1}^{M(F)}(\sigma_{1})\right)\right).
\end{align}
We regard $\mathbb{C}[M_{\sigma}/M^{1}]$ as a subalgebra of $\End_{G(F)}\left(I_{P}^{G}\left(\ind_{M^1}^{M(F)}(\sigma_{1})\right)\right)$ via \eqref{isomgroupalgebraheckealgebraparabolicinduction}.

We prepare a variant of the results above for later use.
Let $\mathbb{C}(M(F)/M^1)$ denote the quotient field of $\mathbb{C}[M(F)/M^1]$ and $\mathbb{C}(M_{\sigma}/M^1)$ denote the quotient field of $\mathbb{C}[M_{\sigma}/M^1]$.
We consider the left regular representation of $M(F)$ on $\mathbb{C}(M(F)/M^1)$.
Then, the left multiplication action on the second factor of
\[
\sigma \otimes \mathbb{C}(M(F)/M^1)
\]
induces an injection
\[
\mathbb{C}(M(F)/M^1) \rightarrow \End_{M(F)}\left(
\sigma \otimes \mathbb{C}(M(F)/M^1)
\right).
\]
Isomorphism~\eqref{groupalgebraisomind} induces an isomorphism of $\mathbb{C}(M(F)/M^1)$-vector spaces
\begin{align}
\label{quotofgroupalgebraisomind}
\ind_{M^1}^{M(F)} (\sigma) \otimes_{\mathbb{C}[M(F)/M^1]} \mathbb{C}(M(F)/M^1) \rightarrow \sigma \otimes \mathbb{C}(M(F)/M^1).
\end{align}
Hence, we also obtain an injection
\[
\mathbb{C}(M(F)/M^1) \rightarrow \End_{M(F)}\left(
\ind_{M^1}^{M(F)} (\sigma) \otimes_{\mathbb{C}[M(F)/M^1]} \mathbb{C}(M(F)/M^1)
\right).
\]
Here, we regard
\[
\ind_{M^1}^{M(F)} (\sigma) \otimes_{\mathbb{C}[M(F)/M^1]} \mathbb{C}(M(F)/M^1)
\]
as an $M(F)$-representation by transporting the $M(F)$-action on $\sigma \otimes \mathbb{C}(M(F)/M^1)$ via isomorphism~\eqref{quotofgroupalgebraisomind}.
According to \cite[(10.11)]{MR4432237}, the subspace
\[
\ind_{M^1}^{M(F)} (\sigma_{1}) \otimes_{\mathbb{C}[M_{\sigma}/M^1]} \mathbb{C}(M_{\sigma}/M^1)
\]
of 
\[
\ind_{M^1}^{M(F)} (\sigma) \otimes_{\mathbb{C}[M(F)/M^1]} \mathbb{C}(M(F)/M^1)
\]
is an $M(F)$-subspace that is preserved by the action of $\mathbb{C}(M_{\sigma}/M^1)$.
Thus,
we have an injection
\begin{align}
\label{quotofisomgroupalgebraheckealgebra}
\mathbb{C}(M_{\sigma}/M^1) \rightarrow \End_{M(F)}\left(\ind_{M^1}^{M(F)} (\sigma_{1}) \otimes_{\mathbb{C}[M_{\sigma}/M^1]} \mathbb{C}(M_{\sigma}/M^1)\right).
\end{align}
extending \eqref{isomgroupalgebraheckealgebra} and an injection
\begin{align}
\label{quotofisomgroupalgebraheckealgebraparabolicinduction}
\mathbb{C}(M_{\sigma}/M^{1}) \rightarrow \End_{G(F)}\left(I_{P}^{G}\left(\ind_{M^1}^{M(F)}(\sigma_{1}) \otimes_{\mathbb{C}[M_{\sigma}/M^1]} \mathbb{C}(M_{\sigma}/M^1) \right)\right)
\end{align}
by using the functoriality of $I_{P}^{G}$.
We also have an injection
\begin{align}
\label{quotofisomgroupalgebraheckealgebraparabolicinductionrestriction}
\mathbb{C}(M_{\sigma}/M^{1}) \rightarrow \Hom_{G(F)}\left(
I_{P}^{G}\left(
\ind_{M^1}^{M(F)} (\sigma_{1})
\right), 
I_{P}^{G}\left(
\ind_{M^1}^{M(F)} (\sigma_{1}) \otimes_{\mathbb{C}[M_{\sigma}/M^1]} \mathbb{C}(M_{\sigma}/M^1)
\right)
\right)
\end{align}
by restricting the image of \eqref{quotofisomgroupalgebraheckealgebraparabolicinduction} to
$
I_{P}^{G}\left(
\ind_{M^1}^{M(F)} (\sigma_{1})
\right)$.

Now, we state the main result of \cite[Section~10]{MR4432237}:
\begin{theorem}[{\cite[Theorem~10.9]{MR4432237}}]
\label{theorem10.9ofsolleveld}
The endomorphism algebra $\End_{G(F)}\left(I_{P}^{G}\left(\ind_{M^1}^{M(F)}(\sigma_{1})\right)\right)$ has a $\mathbb{C}[M_{\sigma}/M^{1}]$-basis
\[
\{
J_{r}T'_{w} \mid r \in R(\mathfrak{s}_{M}), w \in W(\Sigma_{\mathfrak{s}_{M}, \mu})
\},
\]
where $J_{r}$ and $T'_{w}$ are elements of $\End_{G(F)}\left(I_{P}^{G}\left(\ind_{M^1}^{M(F)}(\sigma_{1})\right)\right)$ defined in \cite[Subsection~10.2]{MR4432237}.\index{$J_{r}$}
Let $\mathcal{H}\left(W(\Sigma_{\mathfrak{s}_{M}, \mu})\right)\index{$\mathcal{H}\left(W(\Sigma_{\mathfrak{s}_{M}, \mu})\right)$}$ denote the subspace
\[
\bigoplus_{w \in W(\Sigma_{\mathfrak{s}_{M}, \mu})} \mathbb{C}[M_{\sigma}/M^{1}] T'_{w}
\]
of $\End_{G(F)}\left(I_{P}^{G}\left(\ind_{M^1}^{M(F)}(\sigma_{1})\right)\right)$.
Then, $\mathcal{H}\left(W(\Sigma_{\mathfrak{s}_{M}, \mu})\right)$ is a subalgebra of $\End_{G(F)}\left(I_{P}^{G}\left(\ind_{M^1}^{M(F)}(\sigma_{1})\right)\right)$, and there exists an isomorphism
\[
\mathcal{H}\left(W(\Sigma_{\mathfrak{s}_{M}, \mu})\right) \rightarrow \mathcal{H}(G, \mathfrak{s}_{M})
\]
that is identity on $\mathbb{C}[M_{\sigma}/M^{1}]$ and sends an element $T'_{w}$ of $\mathcal{H}\left(W(\Sigma_{\mathfrak{s}_{M}, \mu})\right)$ to the element $T_{w}$ of $\mathcal{H}(G, \mathfrak{s}_{M})$ for all $w \in W(\Sigma_{\mathfrak{s}_{M}, \mu})$.
\end{theorem}
We explain the definition of $T'_{s_{\alpha}}$ for $\alpha \in \Delta_{\mathfrak{s}_{M}, \mu}(P)$.
\begin{lemma}
There exists a lift $\widetilde{s_{\alpha}}$ of $s_{\alpha}$ in $I_{M_{\alpha}^{1}}(\sigma_{1})$.
\end{lemma}
\begin{proof}
Take a lift $\widetilde{s_{\alpha}}'$ of $s_{\alpha}$ contained in a maximal compact subgroup of $G(F)$ (that is possible, see \cite[Subsection~4.1]{MR4432237}).
The definition of $s_{\alpha}$ implies $\widetilde{s_{\alpha}}' \in M_{\alpha}$.
Since $\widetilde{s_{\alpha}}'$ is contained in a compact subgroup of $G(F)$, $\widetilde{s_{\alpha}}'$ is also contained in $M_{\alpha}^{1}$.
Moreover, since we assume that
\[
\mu^{M_{\alpha}}(\sigma) = 0,
\]
\cite[5.4.2]{MR544991} implies $\widetilde{s_{\alpha}}'$ normalizes the representation $\sigma$ (see also \cite[(3.4)]{MR4432237}).
Since $\sigma_{1}$ is an irreducible subrepresentation of $\sigma\restriction_{M^1}$, we can take $m_{\alpha} \in M(F)$ such that $m_{\alpha} \widetilde{s_{\alpha}}'$ normalizes the representation $\sigma_{1}$.
The proof of \cite[Lemme~4.5]{MR2827179} implies that we can take $m_{\alpha}$ in $M(F) \cap M_{\alpha}^{1}$.
Then, $\widetilde{s_{\alpha}} := m_{\alpha} \widetilde{s_{\alpha}}'$ is a lift of $s_{\alpha}$ in $I_{M_{\alpha}^{1}}(\sigma_{1})$.
\end{proof}
Fix a lift $\widetilde{s_{\alpha}}$ of $s_{\alpha}$ in $I_{M_{\alpha}^{1}}(\sigma_{1})$.
To define $T'_{s_{\alpha}}$, we prepare some operators:
\begin{itemize}
\item 
For another parabolic subgroup $P'$ of $G$ with Levi factor $M$, let 
\[
J_{P' \mid P}(\sigma \otimes \cdot) \colon I_{P}^{G}\left(\sigma \otimes \mathbb{C}[M(F)/M^1]\right) \rightarrow I_{P'}^{G}\left(\sigma \otimes \mathbb{C}(M(F)/M^1)\right)\index{$J_{P' \mid P}(\sigma \otimes \cdot)$}
\]
denote the Harish-Chandra's intertwining operator \cite[Subsection~4.1]{MR4432237}, \cite[IV.1]{MR1989693}.
In particular, we consider the map
\[
J_{s_{\alpha}^{-1}(P) \mid P}(\sigma \otimes \cdot) \colon I_{P}^{G}\left(\sigma \otimes \mathbb{C}[M(F)/M^1]\right) \rightarrow I_{s_{\alpha}^{-1}(P)}^{G}\left(\sigma \otimes \mathbb{C}(M(F)/M^1)\right),
\]
where $s_{\alpha}^{-1}(P)$ denotes the parabolic subgroup $s_{\alpha}^{-1} P s_{\alpha}$.
\item
We define
\[
\lambda(s_{\alpha}) \colon I_{s_{\alpha}^{-1}(P)}^{G}\left(\sigma \otimes \mathbb{C}(M(F)/M^1)\right) \rightarrow I_{P}^{G}\left(\widetilde{s_{\alpha}}\left(\sigma \otimes \mathbb{C}(M(F)/M^1)\right)\right)\index{$\lambda(s_{\alpha})$}
\]
as
\[
f \mapsto [g \mapsto f((\widetilde{s_{\alpha}})^{-1} g)].
\]
\item
We define
\[
\tau_{s_{\alpha}} \colon \widetilde{s_{\alpha}} \left(\sigma \otimes \mathbb{C}(M(F)/M^1)\right) \rightarrow \widetilde{s_{\alpha}} \sigma \otimes \mathbb{C}(M(F)/M^1)\index{$\tau_{s_{\alpha}}$}
\]
as
\[
e \otimes \theta_{m} \mapsto e \otimes \theta_{\widetilde{s_{\alpha}}m{\widetilde{s_{\alpha}}^{-1}}}.
\]
\item 
Since we assume that
\[
\mu^{M_{\alpha}}(\sigma) = 0,
\]
\cite[5.4.2]{MR544991} implies $\widetilde{s_{\alpha}}$ normalizes $\sigma$ (see also \cite[(3.4)]{MR4432237}).
Hence, there exists an isomorphism
\[
\rho_{\sigma, s_{\alpha}} \colon \widetilde{s_{\alpha}} \sigma \simeq \sigma,\index{$\rho_{\sigma, s_{\alpha}}$}
\]
that is unique up to a scaler multiple.
We can choose the isomorphism canonically as \cite[Lemma~4.3]{MR4432237}.
%Note that the normalization of $\rho_{\sigma, s_{\alpha}}$ depends on $P$.
\begin{comment}
%We define
%\[
%\rho_{s_{\alpha}} = \rho_{P. s_{\alpha}} \colon I_{P}^{G}\left(
%\widetilde{s_{\alpha}} \sigma \otimes \mathbb{C}(M(F)/M^1)
%\right)
%\rightarrow
I_{P}^{G}\left(
\sigma \otimes \mathbb{C}(M(F)/M^1)
\right)
\]
as
\[
\rho_{s_{\alpha}} = I_{P}^{G} \left(
\rho_{\sigma, s_{\alpha}} \otimes \id
\right).
\]
\end{comment}
\end{itemize}
We define an element
\[
J_{s_{\alpha}} \in \Hom_{G(F)}\left(
I_{P}^{G}\left(
\sigma \otimes \mathbb{C}[M(F)/M^1]
\right), 
I_{P}^{G}\left(
\sigma \otimes \mathbb{C}(M(F)/M^1)
\right)
\right)
\]
as
\[
J_{s_{\alpha}} = I_{P}^{G} \left(
\rho_{\sigma, s_{\alpha}} \otimes \id
\right) \circ I_{P}^{G} (\tau_{s_{\alpha}}) \circ \lambda(s_{\alpha}) \circ J_{s_{\alpha}^{-1}(P) \mid P}(\sigma \otimes \cdot).\index{$J_{s_{\alpha}} $}
\]
According to isomorphism~\eqref{groupalgebraisomind} and isomorphism~\eqref{quotofgroupalgebraisomind}, we can regard $J_{s_{\alpha}}$ as an element of
\[
\Hom_{G(F)}\left(
I_{P}^{G}\left(
\ind_{M^1}^{M(F)} (\sigma)
\right), 
I_{P}^{G}\left(
\ind_{M^1}^{M(F)} (\sigma) \otimes_{\mathbb{C}[M(F)/M^1]} \mathbb{C}(M(F)/M^1)
\right)
\right).
\]
Moreover, since we take the lift $\widetilde{s_{\alpha}}$ of $s_{\alpha}$ in $I_{M_{\alpha}^{1}}(\sigma_{1})$, Assumption~\ref{multiplicity1} and the proof of \cite[Lemme~4.5]{MR2827179} imply that $J_{s_{\alpha}}$ sends the subspace
\[
I_{P}^{G}\left(
\ind_{M^1}^{M(F)} (\sigma_{1})
\right)
\]
of
\[
I_{P}^{G}\left(
\ind_{M^1}^{M(F)} (\sigma)
\right),
\]
to the subspace
\[
I_{P}^{G}\left(
\ind_{M^1}^{M(F)} (\sigma_{1}) \otimes_{\mathbb{C}[M_{\sigma}/M^1]} \mathbb{C}(M_{\sigma}/M^1)
\right)
\]
of
\[
I_{P}^{G}\left(
\ind_{M^1}^{M(F)} (\sigma) \otimes_{\mathbb{C}[M(F)/M^1]} \mathbb{C}(M(F)/M^1)
\right)
\]
(see also \cite[Lemma~10.3]{MR4432237}).
Thus, we can regard $J_{s_{\alpha}}$ as an element of
\[
\Hom_{G(F)}\left(
I_{P}^{G}\left(
\ind_{M^1}^{M(F)} (\sigma_{1})
\right),
I_{P}^{G}\left(
\ind_{M^1}^{M(F)} (\sigma_{1}) \otimes_{\mathbb{C}[M_{\sigma}/M^1]} \mathbb{C}(M_{\sigma}/M^1)
\right)
\right).
\]
Although the definition of $J_{s_{\alpha}}$ here looks different from the definition of $J_{s_{\alpha}}$ in \cite[Subsection~10.2]{MR4432237}, according to \cite[3.1]{MR2827179}, these two definitions coincide.

We define an element $f_{\alpha} \in \mathbb{C}(M_{\sigma}/M^1)$ as 
\[
f_{\alpha} = \frac{
(\theta_{h_{\alpha}^{\vee}})^2(q_{\alpha} q_{\alpha*} - 1) + \theta_{h_{\alpha}^{\vee}}(q_{\alpha} - q_{\alpha*})
}{
(\theta_{h_{\alpha}^{\vee}})^2 -1
}.\index{$f_{\alpha}$}
\]
We consider $f_{\alpha}$ as an element of
\[
\Hom_{G(F)}\left(
I_{P}^{G}\left(
\ind_{M^1}^{M(F)} (\sigma_{1})
\right), 
I_{P}^{G}\left(
\ind_{M^1}^{M(F)} (\sigma_{1}) \otimes_{\mathbb{C}[M_{\sigma}/M^1]} \mathbb{C}(M_{\sigma}/M^1)
\right)
\right)
\]
via injection~\eqref{quotofisomgroupalgebraheckealgebraparabolicinductionrestriction}.

We define an element $T'_{s_{\alpha}}$ of
\[
\Hom_{G(F)}\left(
I_{P}^{G}\left(\ind_{M^1}^{M(F)} (\sigma_{1})\right), 
I_{P}^{G}\left(\ind_{M^1}^{M(F)} (\sigma_{1}) \otimes_{\mathbb{C}[M_{\sigma}/M^1]} \mathbb{C}(M_{\sigma}/M^1) \right)
\right)
\]
as
\[
T'_{s_{\alpha}} = \frac{
(q_{\alpha}-1)(q_{\alpha*}+1)
}{
2
}
(\theta_{h_{\alpha}^{\vee}})^{\epsilon_{\alpha}} \circ J_{s_{\alpha}}
+ f_{\alpha},\index{$T'_{s_{\alpha}}$}
\]
where $\epsilon_{\alpha} \in \{0,1\}\index{$\epsilon_{\alpha}$}$ denotes the number defined in \cite[Lemma~10.7 (b)]{MR4432237}.
We note that $\epsilon_{\alpha} =0$ unless $q_{\alpha*} > 1$.
In particular, $\epsilon_{\alpha} =0$ unless $\alpha^{\#}$ is the unique simple root in a type $A_{1}$ irreducible component of $\Sigma_{\mathfrak{s}_{M}}$ or a long root in a type $C_{n} \ (n \ge 2)$ irreducible component of $\Sigma_{\mathfrak{s}_{M}}$ (see \cite[Lemma~3.3]{MR4432237}). 
Here, $(\theta_{h_{\alpha}^{\vee}})^{\epsilon_{\alpha}} \circ J_{s_{\alpha}}$ denotes the element of
\[
\Hom_{G(F)}\left(
I_{P}^{G}\left(
\ind_{M^1}^{M(F)} (\sigma_{1})
\right),
I_{P}^{G}\left(
\ind_{M^1}^{M(F)} (\sigma_{1}) \otimes_{\mathbb{C}[M_{\sigma}/M^1]} \mathbb{C}(M_{\sigma}/M^1)
\right)
\right)
\]
obtained by composing $J_{s_{\alpha}}$ with 
\[
(\theta_{h_{\alpha}^{\vee}})^{\epsilon_{\alpha}} \in \mathbb{C}[M_{\sigma}/M^1] \subset \mathbb{C}(M_{\sigma}/M^1)
\]
considered as an element of
\[
\End_{G(F)}\left(I_{P}^{G}\left(\ind_{M^1}^{M(F)}(\sigma_{1}) \otimes_{\mathbb{C}[M_{\sigma}/M^1]} \mathbb{C}(M_{\sigma}/M^1) \right)\right)
\]
via injection~\eqref{quotofisomgroupalgebraheckealgebraparabolicinduction}.
Solleveld proved the following:
\begin{lemma}[{\cite[Lemma~10.8]{MR4432237}}]
\label{lemmasolleveld10.8}
The element $T'_{s_{\alpha}}$ lies in $\End_{G(F)}\left(I_{P}^{G}\left(\ind_{M^1}^{M(F)}(\sigma_{1})\right)\right)$.
\end{lemma}

Recall that $M_{\alpha}$ denotes the Levi subgroup of $G$ that contains $M$ and the root subgroup $U_{\alpha}$ associated with $\alpha$, and whose semisimple rank is one greater than that of $M$. 
We say that $M_{\alpha}$ is a standard Levi subgroup of $G$ with respect to $P$ if there exists a parabolic subgroup $P_{\alpha}$ with Levi factor $M_{\alpha}$ such that $P_{\alpha}$ contains $P$.
We note that $P_{\alpha} = PM_{\alpha}$ in this case. 
Replacing $G$ with $M_{\alpha}$ in the construction of $T'_{s_{\alpha}}$ above, we obtain the corresponding element
\[
(T'_{s_{\alpha}})^{M_{\alpha}}
 \in \End_{M_{\alpha}(F)}\left(I_{P \cap M_{\alpha}}^{M_{\alpha}}\left(\ind_{M^1}^{M(F)}(\sigma_{1})\right)\right).\index{$(T'_{s_{\alpha}})^{M_{\alpha}}$}
\]
\begin{lemma}
\label{transitivityofts}
Suppose that $M_{\alpha}$ is a standard Levi subgroup of $G$ with respect to $P$.
Then, we obtain
\[
T'_{s_{\alpha}} = I_{PM_{\alpha}}^{G} \left(
(T'_{s_{\alpha}})^{M_{\alpha}}
\right).
\]
\end{lemma}
\begin{proof}
Let $J^{M_{\alpha}}_{s_{\alpha}}$ denote the element of
\[
\Hom_{M_{\alpha}(F)}\left(
I_{P \cap M_{\alpha}}^{M_{\alpha}}\left(
\ind_{M^1}^{M(F)} (\sigma_{1})
\right),
I_{P \cap M_{\alpha}}^{M_{\alpha}}\left(
\ind_{M^1}^{M(F)} (\sigma_{1}) \otimes_{\mathbb{C}[M_{\sigma}/M^1]} \mathbb{C}(M_{\sigma}/M^1)
\right)
\right)
\]
obtained by replacing $G$ with $M_{\alpha}$ in the construction of $J_{s_{\alpha}}$.
Then, we have 
\[
(T'_{s_{\alpha}})^{M_{\alpha}} = \frac{
(q_{\alpha}-1)(q_{\alpha*}+1)
}{
2
}
(\theta_{h_{\alpha}^{\vee}})^{\epsilon_{\alpha}} \circ J^{M_{\alpha}}_{s_{\alpha}}
+ f_{\alpha}.
\]
Here, 
\[
\theta_{h_{\alpha}^{\vee}}, \ f_{\alpha} \in \mathbb{C}(M_{\sigma}/M^1)
\]
are
considered as elements of
\[
\End_{M_{\alpha}(F)}\left(I_{P \cap M_{\alpha}}^{M_{\alpha}}\left(\ind_{M^1}^{M(F)}(\sigma_{1}) \otimes_{\mathbb{C}[M_{\sigma}/M^1]} \mathbb{C}(M_{\sigma}/M^1) \right)\right)
\]
or
\[
\Hom_{M_{\alpha}(F)}\left(
I_{P \cap M_{\alpha}}^{M_{\alpha}}\left(
\ind_{M^1}^{M(F)} (\sigma_{1})
\right),
I_{P \cap M_{\alpha}}^{M_{\alpha}}\left(
\ind_{M^1}^{M(F)} (\sigma_{1}) \otimes_{\mathbb{C}[M_{\sigma}/M^1]} \mathbb{C}(M_{\sigma}/M^1)
\right)
\right)
\]
via $M_{\alpha}$-versions of injections~\eqref{quotofisomgroupalgebraheckealgebraparabolicinduction} and \eqref{quotofisomgroupalgebraheckealgebraparabolicinductionrestriction}.
Since these injections are obtained by composing injection~\eqref{quotofisomgroupalgebraheckealgebra} with $I_{P \cap M_{\alpha}}^{M_{\alpha}}$, the transitivity of the parabolic induction
\[
I_{P}^{G} \simeq I_{P M_{\alpha}}^{G} \circ I_{P \cap M_{\alpha}}^{M_{\alpha}}
\]
implies
\[
I_{PM_{\alpha}}^{G} \left(
(T'_{s_{\alpha}})^{M_{\alpha}}
\right) =  \frac{
(q_{\alpha}-1)(q_{\alpha*}+1)
}{
2
}
(\theta_{h_{\alpha}^{\vee}})^{\epsilon_{\alpha}} \circ I_{PM_{\alpha}}^{G} \left(
J_{s_{\alpha}}
\right)
+ f_{\alpha}.
\]
Note that $\theta_{h_{\alpha}^{\vee}}$ and $f_{\alpha}$ here are considered as elements of
\[
\End_{G(F)}\left(I_{P}^{G}\left(\ind_{M^1}^{M(F)}(\sigma_{1}) \otimes_{\mathbb{C}[M_{\sigma}/M^1]} \mathbb{C}(M_{\sigma}/M^1) \right)\right)
\]
or
\[
\Hom_{G(F)}\left(
I_{P}^{G}\left(
\ind_{M^1}^{M(F)} (\sigma_{1})
\right),
I_{P}^{G}\left(
\ind_{M^1}^{M(F)} (\sigma_{1}) \otimes_{\mathbb{C}[M_{\sigma}/M^1]} \mathbb{C}(M_{\sigma}/M^1)
\right)
\right).
\]
Thus, it suffices to show that
\[
J_{s_{\alpha}} =  I_{PM_{\alpha}}^{G} \left(
J^{M_{\alpha}}_{s_{\alpha}}
\right).
\]
Recall that $J_{s_{\alpha}}$ is defined as
\[
J_{s_{\alpha}} = I_{P}^{G} \left(
\rho_{\sigma, s_{\alpha}} \otimes \id
\right) \circ I_{P}^{G} (\tau_{s_{\alpha}}) \circ \lambda(s_{\alpha}) \circ J_{s_{\alpha}^{-1}(P) \mid P}(\sigma \otimes \cdot).
\]
According to \cite[IV.1.(14)]{MR1989693}, we obtain
\[
J_{s_{\alpha}^{-1}(P) \mid P}(\sigma \otimes \cdot) = I_{PM_{\alpha}}^{G} \left(
J_{s_{\alpha}^{-1}(P \cap M_{\alpha}) \mid P \cap M_{\alpha}}(\sigma \otimes \cdot)
\right).
\]
Moreover, since $\widetilde{s_{\alpha}}$ is contained in $M_{\alpha}$, the definition of $\lambda(s_{\alpha})$ implies that it is parabolically induced from the morphism
\[
\lambda^{M_{\alpha}}(s_{\alpha}) \colon I_{s_{\alpha}^{-1}(P \cap M_{\alpha})}^{M_{\alpha}}\left(\sigma \otimes \mathbb{C}(M(F)/M^1)\right) \rightarrow I_{P \cap M_{\alpha}}^{M_{\alpha}}\left(\widetilde{s_{\alpha}}\left(\sigma \otimes \mathbb{C}(M(F)/M^1)\right)\right)
\]
corresponding to $\lambda(s_{\alpha})$.
%Finally, since $I_{P}^{G} (\tau_{s_{\alpha}})$ and $I_{P}^{G} \left(
%\rho_{\sigma, s_{\alpha}} \otimes \id
%\right)$ are defined by using the functoriality of $I_{P}^{G}$, these morphisms are parabolically induced from $M_{\alpha}$-versions of them.
Thus, we obtain
\begin{align*}
I_{PM_{\alpha}}^{G} \left(
J^{M_{\alpha}}_{s_{\alpha}}
\right)
&=
I_{PM_{\alpha}}^{G} \left(
I_{P \cap M_{\alpha}}^{M_{\alpha}} \left(
\rho_{\sigma, s_{\alpha}} \otimes \id
\right). \circ I_{P \cap M_{\alpha}}^{M_{\alpha}} (\tau_{s_{\alpha}}) \circ \lambda^{M_{\alpha}}(s_{\alpha}) \circ J_{s_{\alpha}^{-1}(P \cap M_{\alpha}) \mid P \cap M_{\alpha}}(\sigma \otimes \cdot)
\right) \\
&= I_{P}^{G} \left(
\rho_{\sigma, s_{\alpha}} \otimes \id
\right) \circ I_{P}^{G} (\tau_{s_{\alpha}}) \circ \lambda(s_{\alpha}) \circ J_{s_{\alpha}^{-1}(P) \mid P}(\sigma \otimes \cdot)\\
&= J_{s_{\alpha}}.
\end{align*}
\end{proof}

At the end of this section, we modify the root datum and the label functions used to define the affine Hecke algebra $\mathcal{H}(G, \mathfrak{s}_{M})$.
Since 
$q_{\alpha}, q_{\alpha *} \ge 1$,
and we are assuming that  $q_{\alpha} \ge q_{\alpha *}$, the label functions $\lambda$ and $\lambda^{*}$ are $\mathbb{R}_{\ge 0}$-valued.
Moreover, as explained in the sentence following \cite[(3.8)]{MR4432237}, $q_{\alpha}$ is greater than $1$, hence we obtain $\lambda(\alpha^{\#}) > 0$ for any $\alpha \in \Sigma_{\mathfrak{s}_{M}, \mu}$.
However, $\lambda^{*}(\alpha^{\#}) = 0$ may occur, that is inconvenient for our purpose.
We define another based root datum $\mathcal{R}'(G, \mathfrak{s}_{M})$ and label functions $\lambda', (\lambda^{*})'$ as follows.
For $\alpha \in \Sigma_{\mathfrak{s}_{M}, \mu}$ with $\lambda^{*}(\alpha^{\#}) > 0$, we define
\begin{align}
\label{lambdaofsolleveldrefinedpositive}
\begin{cases}
(h_{\alpha}^{\vee})' &= h_{\alpha}^{\vee}, \\
(\alpha^{\#})' &= \alpha^{\#}, \\
\lambda'((\alpha^{\#})') &= \lambda(\alpha^{\#}), \\
(\lambda^{*})'((\alpha^{\#})') &= \lambda^{*}(\alpha^{\#}).
\end{cases}
\end{align}
Let $\alpha \in \Sigma_{\mathfrak{s}_{M}, \mu}$ such that $\lambda^{*}(\alpha^{\#}) = 0$, that is, $q_{\alpha} = q_{\alpha *}$.
According to \cite[Lemma~3.3]{MR4432237}, it occurs only when 
\[
\alpha^{\#} \in 2 \left( M_{\sigma}/M^{1} \right)^{\vee}.
\]
We define
\[
(h_{\alpha}^{\vee})' = 2 h_{\alpha}^{\vee}, \index{$(h_{\alpha}^{\vee})'$}
\]
and
\[
(\alpha^{\#})' = \frac{\alpha^{\#}}{2}. \index{$(\alpha^{\#})'$}
\]
We note that 
\[
(\alpha^{\#})' \in \left( M_{\sigma}/M^{1} \right)^{\vee}.
\]
We define
\begin{align*}
\left(\Sigma'_{\mathfrak{s}_{M}}\right)^{\vee} &= 
\{
(h_{\alpha}^{\vee})' \mid \alpha \in \Sigma_{\mathfrak{s}_{M}, \mu}
\},\\
\Sigma'_{\mathfrak{s}_{M}} &= \{
(\alpha^{\#})' \mid \alpha \in \Sigma_{\mathfrak{s}_{M}, \mu}
\},\\
\Delta'_{\mathfrak{s}_{M}}(P) &=
\{
(\alpha^{\#})' \mid \alpha \in \Delta_{\mathfrak{s}_{M}, \mu}(P)
\}, \\
\left(
(\alpha^{\#})' 
\right)^{\vee} &= (h_{\alpha}^{\vee})', \index{$\left(\Sigma'_{\mathfrak{s}_{M}}\right)^{\vee}$} \index{$\Sigma'_{\mathfrak{s}_{M}}$} \index{$\Delta'_{\mathfrak{s}_{M}}(P)$}
\end{align*}
and
\[
\mathcal{R}'(G, \mathfrak{s}_{M})
=
\left(
\left( M_{\sigma}/M^{1} \right)^{\vee}, \Sigma'_{\mathfrak{s}_{M}}, M_{\sigma}/M^{1}, \left(\Sigma'_{\mathfrak{s}_{M}}\right)^{\vee}, \Delta'_{\mathfrak{s}_{M}}(P)
\right). \index{$\mathcal{R}'(G, \mathfrak{s}_{M})$}
\]
\begin{lemma}
The tuple $\mathcal{R}'(G, \mathfrak{s}_{M})$ is a based root datum with Weyl group $W(\Sigma_{\mathfrak{s}_{M}, \mu})$.
\end{lemma}
\begin{proof}
Let $\alpha \in \Sigma_{\mathfrak{s}_{M}, \mu}$.
We define
\[
s'_{\alpha} \colon M_{\sigma}/M^{1} \rightarrow M_{\sigma}/M^{1}
\]
as
\[
s'_{\alpha}(m) = m-\langle (\alpha^{\#})', m \rangle (h_{\alpha}^{\vee})'.
\]
Then, our definition of $(h_{\alpha}^{\vee})'$ and $(\alpha^{\#})'$ implies that $s'_{\alpha}$ coincides with the reflection
\[
s_{\alpha} \colon M_{\sigma}/M^{1} \rightarrow M_{\sigma}/M^{1}
\]
defined as
\[
s_{\alpha}(m) = m-\langle \alpha^{\#}, m \rangle h_{\alpha}^{\vee}.
\]
We will prove that
\[
s'_{\alpha} \left(
\left(\Sigma'_{\mathfrak{s}_{M}}\right)^{\vee}
\right)
\subset \left(\Sigma'_{\mathfrak{s}_{M}}\right)^{\vee}.
\]
%It suffices to show that for any $\alpha \in \Sigma_{\mathfrak{s}_{M}, \mu}$, the reflection $s'_{\alpha}$ corresponding to $(\alpha^{\#})'$ preserves the subset $\left(\Sigma_{\mathfrak{s}_{M}}^{\vee}\right)'$ of $M_{\sigma}/M^{1}$.
%and the reflection $(s_{\alpha}^{\vee})'$ associated with $()'$ preserves the subset $\Sigma'_{\mathfrak{s}_{M}}$ of $\left( M_{\sigma}/M^{1} \right)^{\vee}$.
%We will prove the first claim.
%Since $(\alpha^{\#})'$ is $\alpha^{\#}$ or $\alpha^{\#}/2$, the reflection $s'_{\alpha}$ is same as the reflection $s_{\alpha}$ corresponding to $\alpha^{\#}$.
Let $\beta \in \Sigma_{\mathfrak{s}_{M}, \mu}$.
We write $(h_{\beta}^{\vee})' = c \cdot h_{\beta}^{\vee}$ for $c \in \{1,2\}$.
Then, we obtain
\begin{align*}
s'_{\alpha}((h_{\beta}^{\vee})') &= s_{\alpha}((h_{\beta}^{\vee})')\\
&= s_{\alpha}(c \cdot h_{\beta}^{\vee})\\
&= c \cdot s_{\alpha}(h_{\beta}^{\vee}).
\end{align*}
Since $\mathcal{R}(G, \mathfrak{s}_{M})$ is a root datum, $s_{\alpha}$ preserves $\Sigma_{\mathfrak{s}_{M}}^{\vee}$, hence
$s_{\alpha}(h_{\beta}^{\vee}) \in \Sigma_{\mathfrak{s}_{M}}^{\vee}$.
Moreover, according to \cite[Lemma~3.4]{MR4432237}, the set
\[
\{
\alpha \in \Sigma_{\mathfrak{s}_{M}, \mu} \mid q_{\alpha} = q_{\alpha*}
\}
\]
is $W(G, M, \mathfrak{s}_{M})$-invariant.
Thus, 
\[
(h_{\beta}^{\vee})' = 2 h_{\beta}^{\vee}
\]
if and only if
\[
(s_{\alpha}(h_{\beta}^{\vee}))' = 2  s_{\alpha}(h_{\beta}^{\vee}).
\]
Hence, we obtain that
\[
(s_{\alpha}(h_{\beta}^{\vee}))' = c  \cdot s_{\alpha}(h_{\beta}^{\vee}).
\]
Therefore, we obtain that 
\[
s'_{\alpha}((h_{\beta}^{\vee})') = (s_{\alpha}(h_{\beta}^{\vee}))' \in \left(\Sigma'_{\mathfrak{s}_{M}}\right)^{\vee}.
\]
Similarly, we can prove that the action
\[
s'_{\alpha} \colon \left( M_{\sigma}/M^{1} \right)^{\vee} \rightarrow \left( M_{\sigma}/M^{1} \right)^{\vee}
\]
defined as
\[
s'_{\alpha}(x) = x - \langle x, (h_{\alpha}^{\vee})' \rangle (\alpha^{\#})'
\]
preserves $\Sigma'_{\mathfrak{s}_{M}}$.
Thus, $\mathcal{R}'(G, \mathfrak{s}_{M})$ is a root datum.
The last claim follows from the fact $s'_{\alpha} = s_{\alpha}$.
\end{proof}
For $\alpha \in \Sigma_{\mathfrak{s}_{M}, \mu}$ with $\lambda^{*}(\alpha^{\#}) = 0$, we define
\begin{align}
\label{lambdaofsolleveldrefinedzero}
\lambda'((\alpha^{\#})') = (\lambda^{*})'((\alpha^{\#})') = \lambda(\alpha^{\#}).
\end{align}
Then, the label functions $\lambda', (\lambda^{*})'$ satisfy conditions~\eqref{w-equiv} and \eqref{lambda=lambda*} in Appendix~\ref{Iwahori-Hecke algebras and affine Hecke algebras}.
We also note that $\lambda'((\alpha^{\#})'), (\lambda^{*})'((\alpha^{\#})') > 0$ for any $\alpha \in \Sigma_{\mathfrak{s}_{M}, \mu}$.

Let
\[
\mathcal{H}'(G, \mathfrak{s}_{M}) = 
\mathcal{H}\left(
\mathcal{R}'(G, \mathfrak{s}_{M}), \lambda', (\lambda^{*})', q_{F}
\right)\index{$\mathcal{H}'(G, \mathfrak{s}_{M})$}
\]
be the affine Hecke algebra associated with the based root datum $\mathcal{R}'(G, \mathfrak{s}_{M})$, the parameter $q_{F}$, and the label functions $\lambda', (\lambda^{*})'$.
Since the reflection $s'_{\alpha}$ corresponding to $(\alpha^{\#})'$ is same as the reflection $s_{\alpha}$ corresponding to $\alpha^{\#}$, we obtain
\[
\mathcal{H}(G, \mathfrak{s}_{M}) = \mathcal{H}'(G, \mathfrak{s}_{M})
\]
as vector spaces.
\begin{proposition}
The identity map as vector spaces
\[
\mathcal{H}(G, \mathfrak{s}_{M}) \rightarrow \mathcal{H}'(G, \mathfrak{s}_{M})
\]
is an isomorphism of $\mathbb{C}$-algebras.
\end{proposition}
\begin{proof}
It suffices to show that the map is compatible with relation \eqref{bernsteinrel} of Definition~\ref{affinehecke}.
Let $\alpha \in \Delta_{\mathfrak{s}_{M}, \mu}(P)$ and $m \in M_{\sigma}/M^{1}$.
Relation~\eqref{bernsteinrel} for $\mathcal{H}(G, \mathfrak{s}_{M})$ is
\begin{align}
\label{havetoshow}
\theta_{m}T_{s_{\alpha}} - T_{s_{\alpha}}\theta_{s_{\alpha}(m)} =
\left(
({q_{F}}^{\lambda(\alpha^{\#})} -1) + \theta_{-h_{\alpha}^{\vee}}(
{q_{F}}^{(\lambda(\alpha^{\#}) + \lambda^{*}(\alpha^{\#}))/2} - {q_{F}}^{(\lambda(\alpha^{\#}) - \lambda^{*}(\alpha^{\#}))/2}
)
\right)
\frac{
\theta_{m} - \theta_{s_{\alpha}(m)}
}{
\theta_{0} - \theta_{-2h_{\alpha}^{\vee}}
},
\end{align}
and relation~\eqref{bernsteinrel} for $\mathcal{H}'(G, \mathfrak{s}_{M})$ is
\begin{multline}
\label{bernsteinrelforh'}
\theta_{m}T_{s_{\alpha}} - T_{s_{\alpha}}\theta_{s_{\alpha}(m)}\\
 =
\left(
({q_{F}}^{\lambda'((\alpha^{\#})')} -1) + \theta_{-(h_{\alpha}^{\vee})'}(
{q_{F}}^{(\lambda'((\alpha^{\#})') + (\lambda^{*})'((\alpha^{\#})'))/2} - {q_{F}}^{(\lambda'((\alpha^{\#})') - (\lambda^{*})'((\alpha^{\#})'))/2}
)
\right)
\frac{
\theta_{m} - \theta_{s_{\alpha}(m)}
}{
\theta_{0} - \theta_{-2(h_{\alpha}^{\vee})'}
}.
\end{multline}
If $\lambda^{*}(\alpha^{\#}) >0$, 
\[
\begin{cases}
(h_{\alpha}^{\vee})' &= h_{\alpha}^{\vee}, \\
\lambda'((\alpha^{\#})') &= \lambda(\alpha^{\#}), \\
(\lambda^{*})'((\alpha^{\#})') &= \lambda^{*}(\alpha^{\#}).
\end{cases}
\]
Hence, equation~\eqref{havetoshow} is the same as equation~\eqref{bernsteinrelforh'}.

We consider the case $\lambda^{*}(\alpha^{\#}) = 0$. 
Then, the right hand side of equation~\eqref{havetoshow} is equal to
\[
({q_{F}}^{\lambda(\alpha^{\#})} -1) 
\frac{
\theta_{m} - \theta_{s_{\alpha}(m)}
}{
\theta_{0} - \theta_{-2h_{\alpha}^{\vee}}
}.
\]
On the other hand, in this case, we have
\[
\begin{cases}
(h_{\alpha}^{\vee})' &= 2 h_{\alpha}^{\vee}, \\
\lambda'((\alpha^{\#})') &= \lambda(\alpha^{\#}), \\
(\lambda^{*})'((\alpha^{\#})') &= \lambda(\alpha^{\#}).
\end{cases}
\]
Hence, the right hand side of equation~\eqref{bernsteinrelforh'} is equal to
\begin{align*}
\left(
({q_{F}}^{\lambda(\alpha^{\#})} -1) + \theta_{-2h_{\alpha}^{\vee}}(
{q_{F}}^{\lambda(\alpha^{\#})} -1
)
\right)
\frac{
\theta_{m} - \theta_{s_{\alpha}(m)}
}{
\theta_{0} - \theta_{-4(h_{\alpha}^{\vee})}
}
&=
({q_{F}}^{\lambda(\alpha^{\#})} -1) (
1 + \theta_{-2h_{\alpha}^{\vee}}
)
\frac{
\theta_{m} - \theta_{s_{\alpha}(m)}
}{
\theta_{0} - \theta_{-4(h_{\alpha}^{\vee})}
}\\
&=
({q_{F}}^{\lambda(\alpha^{\#})} -1) 
\frac{
\theta_{m} - \theta_{s_{\alpha}(m)}
}{
\theta_{0} - \theta_{-2h_{\alpha}^{\vee}}
}.
\end{align*}
Thus, equation~\eqref{havetoshow} is the same as equation~\eqref{bernsteinrelforh'} in this case too.
\end{proof}
Now, we obtain a modification of Theorem~\ref{theorem10.9ofsolleveld}:
\begin{theorem}
\label{modificationoftheorem10.9ofsolleveld}
The endomorphism algebra $\End_{G(F)}\left(I_{P}^{G}\left(\ind_{M^1}^{M(F)}(\sigma_{1})\right)\right)$ has a $\mathbb{C}[M_{\sigma}/M^{1}]$-basis
\[
\{
J_{r}T'_{w} \mid r \in R(\mathfrak{s}_{M}), w \in W(\Sigma_{\mathfrak{s}_{M}, \mu})
\},
\]
and there exists an isomorphism
\[
I^{\Sol} \colon \mathcal{H}\left(W(\Sigma_{\mathfrak{s}_{M}, \mu})\right) \rightarrow \mathcal{H}'(G, \mathfrak{s}_{M})\index{$I^{\Sol}$}
\]
that is identity on $\mathbb{C}[M_{\sigma}/M^{1}]$ and sends an element $T'_{w}$ of $\mathcal{H}\left(W(\Sigma_{\mathfrak{s}_{M}, \mu})\right)$ to the element $T_{w}$ of $\mathcal{H}'(G, \mathfrak{s}_{M})$.
\end{theorem}

We write
\[
\begin{cases}
R^{\Sol} &= \Sigma'_{\mathfrak{s}_{M}}, \\
\left(
R^{\Sol}
\right)^{\vee} &= \left(\Sigma'_{\mathfrak{s}_{M}}\right)^{\vee}, \\
\Delta^{\Sol} &= \Delta'_{\mathfrak{s}_{M}}(P), \\
\mathcal{R}^{\Sol} &= \mathcal{R}'(G, \mathfrak{s}_{M}), \\
\lambda^{\Sol} &= \lambda', \\
(\lambda^{*})^{\Sol} &= (\lambda^{*})', \\
\mathcal{H}^{\Sol} &= \mathcal{H}'(G, \mathfrak{s}_{M}).\index{$R^{\Sol} $}\index{$\left(
R^{\Sol}
\right)^{\vee}$}\index{$\Delta^{\Sol}$}\index{$\mathcal{R}^{\Sol}$}\index{$\lambda^{\Sol} $}\index{$(\lambda^{*})^{\Sol}$}\index{$\mathcal{H}^{\Sol} $}
\end{cases}
\]
We also write an element $T_{w} \in \mathcal{H}^{\Sol}$ as $T_{w} = T^{\Sol}_{w}\index{$T^{\Sol}_{w}$}$ for $w \in W(\Sigma_{\mathfrak{s}_{M}, \mu})$.
\section{Statements of main results}
\label{Statements of main results}
In this section, we state the main results of this paper.
Let $S$ be a maximal split torus of $G$.
We use the same notation and assumptions as Section~\ref{The case of depth-zero types}.
Recall that $J$ is a subset of a fixed basis $B$ of the irreducible affine root system $\Phi_{\aff}$ associated with $(G, S)$ such that $\abs{B \backslash J} > 1$. 
Let $\mathcal{F}_{J}\index{$\mathcal{F}_{J}$}$ denote the facet of the reduced Bruhat-Tits building of $G$ over $F$ contained in the closure of $C$ that corresponds to $J$ in the sense of \cite[1.8]{MR1235019}.
We assume the following:
\begin{assumption}
\label{parahoric=stabilizer}
The parahoric subgroup $P_{J}$ coincides with the stabilizer of $\mathcal{F}_{J}$ in $G^1$.
\end{assumption}
The assumption holds when $G$ is semisimple and simply connected, for instance (see \cite[3.1]{MR546588}).
In general, $P_{J}$ is a subgroup of the stabilizer of $\mathcal{F}_{J}$ in $G^1$ of finite index. We write $K = P_{J}\index{$K$}$.

We define a semi-standard Levi subgroup $M\index{$M$}$ of $G$ as the centralizer of the subtorus
\[
\left(
\bigcap_{\alpha \in DJ} \ker(\alpha)
\right)^{\circ}
\]
of $S$.
According to \cite[3.5.1]{MR546588}, the Levi subgroup $M$ above is the same as the Levi subgroup attached with the parahoric subgroup $P_{J}$ as in \cite[6.3]{MR1371680}.
Since we are assuming that $\abs{B \backslash J} > 1$, $M$ is a proper Levi subgroup of $G$.
Let $K_{M} = K \cap M(F)\index{$K_{M}$}$.
According to \cite[2.1 Theorem~(i)]{MR1713308}, $K_{M}$ is a maximal parahoric subgroup of $M(F)$ associated with a vertex $x_{J}$ of the reduced Bruhat-Tits building $\mathcal{B}^{\red}(M, F)$ of $M$ over $F$.
More precisely, $x_{J}$ is the vertex such that
\[
\left\{
x \in \mathcal{A}(M, S) \mid a(x) = 0 \ (a \in J)
\right\} = \{x_{J}\},\index{$x_{J}$}
\]
where $\mathcal{A}(M, S)$ denotes the reduced apartment of $S$ in $\mathcal{B}^{\red}(M, F)$.
Moreover, Assumption~\ref{parahoric=stabilizer} and \cite[2.1 Theorem~(i)]{MR1713308} imply that $K_{M}$ coincides with the stabilizer of $x_{J}$ in $M^{1}$.
Let $\widetilde{K_{M}}\index{$\widetilde{K_{M}}$}$ denote the stabilizer of $x_{J}$ in $M(F)$, that is a compact-mod-center open subgroup of $M(F)$.
Since $K_{M}$ coincides with the stabilizer of $x_{J}$ in $M^1$, we obtain that
\[
K_{M} = \widetilde{K_{M}} \cap M^1.
\]
We define $\rho_{M}\index{$\rho_{M}$}$ as the restriction of $\rho$ to $K_{M}$. 
\begin{remark}
\label{remarkmp2prop6.6}
The proof of \cite[Proposition~6.6]{MR1371680} implies that
\[
I_{M(F)}(\rho_{M}) \subset \widetilde{K_{M}}.
\]
Hence, any element of $I_{M(F)}(\rho_{M})$ normalizes $K_{M}$ and $\rho_{M}$.
In particular, $I_{M(F)}(\rho_{M})$ is a group.
\end{remark}
\begin{comment}
\begin{lemma}
Let $N_{G}(M)(F)$ denote the normalizer of $M$ in $G(F)$.
Then, any element of 
\[
N_{G}(M)(F) \cap I_{G(F)}(\rho)
\]
also normalizes $K_{M}$ and $\rho_{M}$.
In particular, the set
\[
N_{G}(M)(F) \cap I_{G(F)}(\rho)
\]
is a group.
\end{lemma}
\begin{proof}
Let 
\[
g \in N_{G}(M)(F) \cap I_{G(F)}(\rho).
\]
Since $g \in I_{G(F)}(\rho)$, we have
\[
\Hom_{K \cap ^g\!K}(^g\!\rho, \rho) \neq \{0\}.
\]
In particular, 
\[
\Hom_{K_{M} \cap ^g\!K_{M}}(^g\!\rho_{M}, \rho_{M}) \neq \{0\}.
\]
Since $K_{M}$ is the  maximal parahoric subgroup associated with $x_{J}$,
\end{proof}
\end{comment}
According to \cite[Proposition~6.4]{MR1371680} and \cite[Proposition~6.7]{MR1371680}, $(K, \rho)$ is a $G$-cover of $(K_{M}, \rho_{M})$.
Thus, there exists an isomorphism
\[
I_{U} \colon \End_{G(F)}\left(\ind_{K}^{G(F)} (\rho)\right) \rightarrow \End_{G(F)}\left(I^{G}_{P} \left( \ind_{K_{M}}^{M(F)} (\rho_{M}) \right)\right)
\]
constructed in Section~\ref{An explicit isomorphism} for any parabolic subgroup $P$ with Levi factor $M$.
We will fix a suitable parabolic subgroup $P$ later.
The left hand side of the isomorphism above is studied in Section~\ref{The case of depth-zero types}.
On the other hand, the right hand side of the isomorphism can be connected with an object studied in Section~ \ref{A review of Solleveld's results} as follows.
Take an irreducible smooth representation $\widetilde{\rho_{M}}$ of $\widetilde{K_{M}}\index{$\widetilde{K_{M}}$}$ such that $\widetilde{\rho_{M}}\restriction_{K_M}$ contains $\rho_{M}$.
Corresponding to Assumption~\ref{multiplicity1}, we suppose:
\begin{assumption}
\label{parahoricvermultiplicity1}
The multiplicity of $\rho_{M}$ in $\widetilde{\rho_{M}}\restriction_{K_M}$ is equal to $1$.
\end{assumption}
We fix an injection
\begin{align}
\label{injectionrhoM}
\rho_{M} \rightarrow \widetilde{\rho_{M}}\restriction_{K_M}.
\end{align}
According to \cite[Proposition~6.6]{MR1371680}, the representation
\[
\sigma := \ind_{\widetilde{K_{M}}}^{M(F)} (\widetilde{\rho_{M}})\index{$\sigma$}
\]
is an irreducible supercuspidal representation of $M(F)$.
Let $\mathfrak{s}_{M}$ denote the inertial equivalence class of the pair $(M, \sigma)$ in $M$.
Replacing $\widetilde{\rho_{M}}$ with $\widetilde{\rho_{M}} \otimes \chi\restriction_{\widetilde{K_{M}}}$ for suitable $\chi \in X_{\unr}(M)$ if necessary, we may assume that $\sigma$ satisfies \cite[Condition~3.2]{MR4432237}, and
\[
q_{\alpha} \ge q_{\alpha *},
\]
for all $\alpha \in \Sigma_{\mathfrak{s}_{M}, \mu}$.
Moreover, according to Remark~\ref{remarkmp2prop6.6}, we obtain that
\[
I_{M^1}(\rho_{M}) = I_{M(F)}(\rho_{M}) \cap M^1 \subset \widetilde{K_{M}} \cap M^1 = K_{M}.
\]
Hence, the representation
\[
\sigma_{1} := \ind_{K_{M}}^{M^1} (\rho_{M})\index{$\sigma_{1}$}
\]
is also irreducible.
For $v \in \widetilde{\rho_{M}}$, we define an element $\widetilde{f_{v}} \in \ind_{\widetilde{K_{M}}}^{M(F)} (\widetilde{\rho_{M}})\index{$\widetilde{f_{v}} $}$ as
\[
\widetilde{f_v}(m) =
\begin{cases}
\widetilde{\rho_{M}}(m) \cdot v & (m \in \widetilde{K_{M}} ), \\
0 & (\text{otherwise}).
\end{cases} 
\]
Identifying $v \in \widetilde{\rho_{M}}$ with $\widetilde{f_{v}}$, we regard $\widetilde{\rho_{M}}$ as a $\widetilde{K_{M}}$-subrepresentation of $\ind_{\widetilde{K_{M}}}^{M(F)} (\widetilde{\rho_{M}})\restriction_{\widetilde{K_M}}$.
For $v \in \rho_{M}$, we also define an element $f_{v, 1} \in \ind_{K_M}^{M^1} (\rho_{M})\index{$f_{v, 1}$}$ as
\[
f_{v, 1}(m) =
\begin{cases}
\rho_{M}(m) \cdot v & (m \in K_{M} ), \\
0 & (\text{otherwise}).
\end{cases} 
\]
\begin{lemma}
\label{rhovssigmamultiplicity}
We have
\[
\Hom_{K_{M}}(\rho_{M}, \widetilde{\rho_{M}}) = \Hom_{K_{M}}\left(
\rho_{M}, \ind_{\widetilde{K_{M}}}^{M(F)} (\widetilde{\rho_{M}})
\right).
\]
\end{lemma}
\begin{proof}
The right hand side is calculated as
\begin{align*}
\Hom_{K_{M}}\left(
\rho_{M}, \ind_{\widetilde{K_{M}}}^{M(F)} (\widetilde{\rho_{M}})
\right)
&= 
\Hom_{K_{M}} \left(
\rho_{M}, \bigoplus_{m \in K_{M} \backslash M(F) / \widetilde{K_{M}}} \ind_{K_{M} \cap ^m\widetilde{K_{M}}}^{K_{M}} (^m\!\widetilde{\rho_{M}})
\right)\\
&=
\Hom_{K_{M}} \left(
\rho_{M}, \bigoplus_{m \in K_{M} \backslash M(F) / \widetilde{K_{M}}} \ind_{K_{M} \cap ^mK_{M}}^{K_{M}} (^m\!\widetilde{\rho_{M}})
\right)\\
&=\bigoplus_{m \in K_{M} \backslash M(F) / \widetilde{K_{M}}}
\Hom_{K_{M}}\left(
\rho_{M}, \ind_{K_{M} \cap ^mK_{M}}^{K_{M}} (^m\!\widetilde{\rho_{M}})
\right)\\
&\simeq \bigoplus_{m \in K_{M} \backslash M(F) / \widetilde{K_{M}}}
\Hom_{K_{M} \cap ^mK_{M}}\left(
\rho_{M}, ^m\!\widetilde{\rho_{M}}
\right).
\end{align*}
Since any irreducible subrepresentation of $\widetilde{\rho_{M}}\restriction_{K_{M}}$ is isomorphic to some $\widetilde{K_{M}}$-conjugate of $\rho_{M}$, Remark~\ref{remarkmp2prop6.6} implies that
\[
\Hom_{K_{M} \cap ^mK_{M}}\left(
\rho_{M}, ^m\!\widetilde{\rho_{M}}
\right) = \{0\}
\]
unless $m \in \widetilde{K_{M}}$.
Thus, we obtain
\[
\Hom_{K_{M}}\left(
\rho_{M}, \ind_{\widetilde{K_{M}}}^{M(F)} (\widetilde{\rho_{M}})
\right)
=  \Hom_{K_{M}}(\rho_{M}, \widetilde{\rho_{M}}).
\]
\end{proof}
We regard $\sigma_{1}$ as an irreducible subrepresentation of $\sigma\restriction_{M^1}$ by using the injection 
\[
\sigma_{1} \rightarrow \sigma\restriction_{M^1}
\]
corresponding to injection~\eqref{injectionrhoM} via
\begin{align*}
\Hom_{K_{M}}(\rho_{M}, \widetilde{\rho_{M}}) &= \Hom_{K_{M}}\left(
\rho_{M}, \ind_{\widetilde{K_{M}}}^{M(F)} (\widetilde{\rho_{M}})
\right)\\
&\simeq \Hom_{M^1}\left(
\ind_{K_{M}}^{M^1} (\rho_{M}), \ind_{\widetilde{K_{M}}}^{M(F)} (\widetilde{\rho_{M}})
\right)\\
&= \Hom_{M^1}(\sigma_{1}, \sigma).
\end{align*}
Hence, for 
\[
v \in \rho_{M} \subset \widetilde{\rho_{M}}\restriction_{K_M},
\]
the element $f_{v, 1}$ of $\sigma_1$ is identified with the element $\widetilde{f_{v}}$ of $\sigma$.
Assumption~\ref{parahoricvermultiplicity1} implies that $\sigma$ satisfies Assumption~\ref{multiplicity1}.
We also note that the multiplicity of $\rho_{M}$ in $\sigma\restriction_{K_{M}}$ is equal to $1$. 
%Moreover, Claim~\ref{rhovssigmamultiplicity} also implies that the multiplicity of $\rho_{M}$ in $\sigma\restriction_{K_M}$ is $1$.

The transitivity of the compact induction implies that we have an isomorphism
\[
T_{\rho_{M}} \colon \ind_{K_{M}}^{M(F)} (\rho_{M}) \rightarrow \ind_{M^1}^{M(F)} (\sigma_{1})\index{$T_{\rho_{M}}$}
\]
defined as 
\[
\left(
\left(
T_{\rho_{M}}(f)
\right)
(m)
\right)(m') = f(m'm)
\]
for $f \in \ind_{K_{M}}^{M(F)} (\rho_{M})$, $m \in M(F)$, and $m' \in M^1$.
For $v \in \rho_{M}$, we define an element $f_v \in~ \ind_{K_{M}}^{M(F)} (\rho_{M})$ as
\[
f_{v}(m) =
\begin{cases}
\rho_{M}(m) \cdot v & (m \in K_{M} ), \\
0 & (\text{otherwise}).\index{$f_{v}$}
\end{cases} 
\]
Then, the definition of $T_{\rho_{M}}$ implies that $T_{\rho_{M}}(f_v) $ is supported on $M^1$, and 
\[
\left(
T_{\rho_{M}}(f_v) 
\right)(1) = f_{v, 1}.
\]
We use the same symbol $T_{\rho_{M}}$ for isomorphisms
\[
T_{\rho_{M}} \colon I^{G}_{P} \left( \ind_{K_{M}}^{M(F)} (\rho_{M}) \right) \rightarrow I^{G}_{P} \left( \ind_{M^1}^{M(F)} (\sigma_{1}) \right),
\]
\[
T_{\rho_{M}} \colon \End_{M(F)}\left( \ind_{K_{M}}^{M(F)} (\rho_{M}) \right) \rightarrow \End_{M(F)}\left(\ind_{M^1}^{M(F)} (\sigma_{1})\right),
\]
and
\[
T_{\rho_{M}} \colon \End_{G(F)}\left(I^{G}_{P} \left( \ind_{K_{M}}^{M(F)} (\rho_{M}) \right)\right) \rightarrow \End_{G(F)}\left(I^{G}_{P} \left( \ind_{M^1}^{M(F)} (\sigma_{1}) \right)\right)\index{$T_{\rho_{M}}$}
\]
induced by $T_{\rho_{M}}$.
Combining $I_{U}$ with $T_{\rho_{M}}$, we obtain an isomorphism
\[
T_{\rho_{M}} \circ I_{U} \colon \End_{G(F)}\left(\ind_{K}^{G(F)} (\rho)\right) \rightarrow \End_{G(F)}\left(I^{G}_{P} \left( \ind_{M^1}^{M(F)} (\sigma_{1}) \right)\right).
\]
We will compare the description of the left hand side of $T_{\rho_{M}} \circ I_{U}$ in Section~\ref{The case of depth-zero types} with the description of the right hand side of $T_{\rho_{M}} \circ I_{U}$ in Section~ \ref{A review of Solleveld's results}.

We take a parabolic subgroup $P$ that is compatible with the positive system $D_{J}\left(
\Gamma'(J, \rho)^{+}_{e}
\right)$ of the root system $D_{J}\left(
\Gamma'(J, \rho)_{e}
\right)$
as follows.
The definition of the Levi subgroup $M$ implies that the vector space
\[
V^{J} = \{
y \in V \mid \alpha(y) = 0 \ (\alpha \in DJ)
\}
\]
in Section~\ref{The case of depth-zero types} is equal to the subspace of
\[
a_{M}
=
X_{*}(A_{M}) \otimes_{\mathbb{Z}} \mathbb{R}
\]
spanned by $\alpha^{\vee} \, (\alpha \in \Sigma_{\red}(A_M))$.
Hence, we can identify a linear function on $V^{J}$ as an element of the subspace of 
\[
a_{M}^{*} = X^{*}(A_{M}) \otimes_{\mathbb{Z}} \mathbb{R}
\]
spanned by $\Sigma_{\red}(A_M)$.
In particular, we can consider the root system
\[
R^{\Mor}
= \{
D_{J}(a')/k_{a'} \mid a' \in \Gamma'(J, \rho)_{e}
\}
\]
as a subset of $a_{M}^{*}$.
Under this identification, we obtain that 
\[
\frac{D_J(a + A'_{J})}{k_{a+A'_{J}}} = \frac{(Da)\restriction_{A_{M}}} {k_{a+A'_{J}}}
\]
for $a \in \Gamma(J, \rho)$ such that $a +A'_{J} \in \Gamma'(J, \rho)_{e}$.
Since $D_{J}\left(
\Gamma'(J, \rho)^{+}_{e}
\right)$ is a positive system of the root system $D_{J}\left(
\Gamma'(J, \rho)_{e}
\right)$, we can take a parabolic subgroup $P$ of $G$ with Levi factor $M$ such that
\[
D_{J}\left(
\Gamma'(J, \rho)^{+}_{e} 
\right) = 
D_{J} \left(
\Gamma'(J, \rho)_{e}
\right) \cap \left(-\Sigma(P, A_{M})\right).
\]
\begin{comment}
 in $(a_{M}^{*})_{\Gamma}$,
we can take an element $\lambda \in (a_{M})_{\Gamma}$
such that
\[
D_{J}\left(
\left(\Gamma'(J, \rho)\right)_{e} \cap \Gamma'(J, \rho)^{+}
\right) = 
\left\{
\alpha \in D_{J}\left(
\left(\Gamma'(J, \rho)\right)_{e}
\right) \mid 
\langle \lambda, \alpha \rangle >0
\right\}.
\]
We can also take an element $\lambda' \in a_{M}$ such that
\[
\langle \lambda', \alpha \rangle \neq 0
\]
for all $\alpha \in \Sigma(G, A_{M})$, and
\[
\abs{
\langle \lambda, \alpha \rangle
} > 
\abs{
\langle \lambda' - \lambda, \alpha \rangle
}
\]
for all $\alpha \in \Sigma(G, A_{M})$ with 
\[
\langle \lambda, \alpha \rangle \neq 0.
\]
Our choice of $\lambda'$ implies that if $\alpha \in \Sigma(G, A_{M})$ satisfies 
\[
\langle \lambda', \alpha \rangle > 0,
\]
then we obtain
\[
\langle \lambda, \alpha \rangle \ge 0.
\]
We fix a parabolic subgroup $P$ of $G$ with Levi factor $M$ such that 
\[
\Sigma(P, A_{M}) = \{
\alpha \in \Sigma(G, A_{M}) \mid \langle \lambda', \alpha \rangle < 0,
\}
\]
where $\Sigma(P, A_{M})$ denote the set of nonzero weight occurring in the adjoint representation of $A_M$ on the Lie algebra of $P$.
\end{comment}
%We fix such a $P$ and write $U$ for the unipotent radical of $P$.
There are several choices of parabolic subgroups $P$ that satisfy this property.
However, the injection
\[
t_{P} \colon \End_{M(F)}\left( \ind_{K_{M}}^{M(F)} (\rho_{M}) \right) \rightarrow \End_{G(F)}\left(\ind_{K}^{G(F)} (\rho)\right)
\]
in Section~\ref{Hecke algebra injections} does not depend on the choice of $P$:
\begin{lemma}
\label{independencyofP}
Let $P'$ be another parabolic subgroup of $G$ with Levi factor $M$ such that
\[
D_{J}\left(
\Gamma'(J, \rho)^{+}_{e} 
\right) = 
D_{J} \left(
\Gamma'(J, \rho)_{e}
\right) \cap \left(-\Sigma(P', A_{M})\right).
\]
Then, we obtain $t_{P} = t_{P'}$.
\end{lemma}
\begin{proof}
We define $d = d(P, P')$ as
\[
d(P, P') = \abs{
\Sigma_{\red}(P, A_{M}) \cap \left(-\Sigma_{\red}(P', A_{M})\right)
}.
\]
Then, we can take parabolic subgroups
\[
P=P_{0}, P_{1}, \cdots , P_{d}= P'
\]
with Levi factor $M$ such that
\[
D_{J}\left(
\Gamma'(J, \rho)^{+}_{e} 
\right) = 
D_{J} \left(
\Gamma'(J, \rho)_{e}
\right) \cap \left(-\Sigma(P_{i}, A_{M})\right),
\]
for all $0 \le i \le d$, and
\[
\abs{
\Sigma_{\red}(P_{i}, A_{M}) \cap \left(-\Sigma_{\red}(P_{i+1}, A_{M})\right)
} = 1
\]
for all $0 \le i \le d-1$.
It suffices to show that $t_{P_{i}} = t_{P_{i+1}}$ for all $0 \le i \le d-1$. Hence, we may assume that $d=1$.
We write
\[
\Sigma_{\red}(P, A_{M}) \cap \left(-\Sigma_{\red}(P', A_{M})\right) = \{
\alpha
\}.
\]
Recall that $M_{\alpha}$ denotes the semi-standard Levi subgroup of $G$ containing $M$ and the root subgroup $U_{\alpha}$ associated with $\alpha$, and whose semisimple rank is one greater than that of $M$.
We write $K_{\alpha} = K \cap M_{\alpha}(F)$ and $\rho_{\alpha} = \rho\restriction_{K_{\alpha}}$.
Replacing $G$ with $M_{\alpha}$ in the construction of $t_{P}$ and $t_{P'}$, we obtain injections
\[
t_{P \cap M_{\alpha}}, t_{P' \cap M_{\alpha}} \colon \mathcal{H}(M(F), \rho_{M}) \rightarrow \mathcal{H}(M_{\alpha}(F), \rho_{\alpha}).
\]
Replacing $M$ with $M_{\alpha}$ in the construction of $t_{P}$ and $t_{P'}$, we also obtain injections
\[
t_{PM_{\alpha}}, t_{P' M_{\alpha}} \colon \mathcal{H}(M_{\alpha}(F), \rho_{\alpha}) \rightarrow \mathcal{H}(G(F), \rho).
\]
According to \cite[(8.7)]{MR1643417}, we have
\[
t_{P} = t_{PM_{\alpha}} \circ t_{P \cap M_{\alpha}}
\]
and
\[
t_{P'} = t_{P' M_{\alpha}} \circ t_{P' \cap M_{\alpha}}.
\]
Since $P M_{\alpha} = P' M_{\alpha}$, it suffices to show that $t_{P \cap M_{\alpha}} = t_{P' \cap M_{\alpha}}$.
Let $\Phi^{M_{\alpha}}$ denote the set of relative roots with respect to $S$ in $M_{\alpha}$, and let $\Phi_{\aff}^{M_{\alpha}}$ denote the affine root system associated with $(M_{\alpha}, S)$ by the work of \cite{MR327923}.
Hence, 
\[
\Phi_{\aff}^{M_{\alpha}} = \{
a \in \Phi_{\aff} \mid Da \in \Phi^{M_{\alpha}}
\}.
\]
We also define $\Phi^{M}$ and $\Phi_{\aff}^{M}$, similarly.
Then, the definition of $M$ implies that
\[
\Phi^{M} = \Phi \cap \mathbb{R} \cdot (DJ),
\]
where $\mathbb{R} \cdot (DJ)$ denotes the $\mathbb{R}$-span of $DJ$.
Hence, we obtain that
\[
\Phi_{\aff}^{M} = \{
a \in \Phi_{\aff} \mid Da \in \mathbb{R} \cdot (DJ)
\},
\]
that is written as $(\Phi_{\aff})_{J}$ in Appendix~\ref{Subsets of a set of simple affine roots}.
Since $\Phi_{\aff}^{M_{\alpha}}$ contains $(\Phi_{\aff})_{J}$,
according to Corollary~\ref{corollarySubsets of a set of simple affine roots}, we can take a basis $B^{M_{\alpha}}$ of $\Phi_{\aff}^{M_{\alpha}}$ containing $J$.
Thus, we can define $W^{M_{\alpha}}(J, \rho_{\alpha})$, $\Gamma^{M_{\alpha}}(J, \rho_{\alpha})$, and $R^{M_{\alpha}}(J, \rho_{\alpha})$ by replacing $G$ with $M_{\alpha}$ and $\rho$ with $\rho_{\alpha}$ in the construction of $W(J, \rho)$, $\Gamma(J, \rho)$, and $R(J, \rho)$, respectively.
\begin{claim}
The group $R^{M_{\alpha}}(J, \rho_{\alpha})$ is trivial.
\end{claim}
\begin{proof}
It suffices to show that $\Gamma^{M_{\alpha}}(J, \rho_{\alpha})$ is empty.
Suppose that $\Gamma^{M_{\alpha}}(J, \rho_{\alpha})$ is non-empty, and take an element $a \in \Gamma^{M_{\alpha}}(J, \rho_{\alpha})$.
According to Lemma~\ref{Gammareduction} below, $a$ is also contained in $\Gamma(J, \rho)$.
\begin{comment}
We note that
$Da\restriction_{A_{M}}$
is a scalar multiple of $\alpha$.
The definition of $M_{\alpha}$ implies that
\[
\Phi_{\aff}^{M_{\alpha}} = \{
a \in \Phi_{\aff} \mid
Da \in \mathbb{R} \cdot D(J \cup \{a\})
\},
\]
where $\mathbb{R} \cdot D(J \cup \{a\})$ denotes the $\mathbb{R}$-span of $D(J \cup \{a\})$.
The definition of $\Gamma^{M_{\alpha}}(J, \rho_{\alpha})$ implies that there exists a basis of $\Phi_{\aff}^{M_{\alpha}}$ containing $J \cup \{a\}$.
Hence, taking $J' = J \cup \{a\}$ in Lemma~\ref{converselemmabasisofsimplerootsoflevi}, we obtain that there exists a basis of $\Phi_{\aff}$ containing $J \cup \{a\}$.
Moreover, since other conditions in the definition of $R(J, \rho)$ can be checked in $M_{\alpha}$, we obtain that $a \in \Gamma(J, \rho)$.
\end{comment}
Since any element of $D_{J}\left(\Gamma'(J, \rho)\right)$ is a scalar multiple of an element of 
$D_{J}
\left(
\Gamma'(J, \rho)^{+}_{e}
\right)$, there exists $c \in \mathbb{R}^{\times}$ such that
\[
c \cdot D_{J}(a + A'_{J}) \in D_{J}
\left(
\Gamma'(J, \rho)^{+}_{e}
\right).
\]
Then, our assumptions
\[
D_{J}\left(
\Gamma'(J, \rho)^{+}_{e}
\right)  = 
D_{J} \left(
\Gamma'(J, \rho)_{e}
\right) \cap \left( - \Sigma(P, A_{M}) \right)
= D_{J} \left(
\Gamma'(J, \rho)_{e}
\right) \cap \left( - \Sigma(P', A_{M}) \right)
\]
imply that
\[
c \cdot D_{J}(a + A'_{J}) \in \left( - \Sigma(P, A_{M}) \right) \cap \left( - \Sigma(P', A_{M}) \right).
\]
On the other hand, since $a \in \Phi_{\aff}^{M_{\alpha}}$, we have
\[
D_{J}(a + A'_{J}) = Da \restriction_{A_{M}} \in \mathbb{R} \cdot \alpha.
\]
Thus, there exists $c' \in \mathbb{R}^{\times}$ such that
\[
c' \cdot \alpha = -c \cdot D_{J}(a + A'_{J}) \in \Sigma(P, A_{M}) \cap \Sigma(P', A_{M}).
\]
However, since
\[
\alpha \in \Sigma(P, A_{M}) \cap \left( - \Sigma(P', A_{M}) \right),
\]
we obtain that
\[
c' \cdot \alpha \in
\begin{cases}
\mathbb{R}_{<0} \cdot \Sigma(P', A_{M}) \cap \Sigma(P', A_{M}) = \emptyset & (c'>0), \\
\mathbb{R}_{<0} \cdot \Sigma(P, A_{M}) \cap \Sigma(P, A_{M}) = \emptyset & (c' <0),
\end{cases}
\]
a contradiction.
\end{proof}
Now, the equation $t_{P \cap M_{\alpha}} = t_{P' \cap M_{\alpha}}$ follows from Corollary~\ref{rtrivimpliessupportpreserving} below.
\end{proof}
We fix a parabolic subgroup $P$ of $G$ with Levi factor $M$ such that
\[
D_{J}\left(
\Gamma'(J, \rho)^{+}_{e} 
\right) = 
D_{J} \left(
\Gamma'(J, \rho)_{e}
\right) \cap \left(-\Sigma(P, A_{M})\right).
\]

First, we study the upper row of the commutative diagram
\[
\xymatrix{
\End_{M(F)}\left( \ind_{K_{M}}^{M(F)} (\rho_{M}) \right) \ar[d]_-{t_{P}} \ar[r]^-{\id} \ar@{}[dr]|\circlearrowleft & \End_{M(F)}\left( \ind_{K_{M}}^{M(F)} (\rho_{M}) \right) \ar[d]^-{I_{P}^{G}} \ar[r]^-{T_{\rho_{M}}} \ar@{}[dr]|\circlearrowleft&
\End_{M(F)}\left(\ind_{M^1}^{M(F)} (\sigma_{1})\right) \ar[d]^-{I_{P}^{G}}
\\
\End_{G(F)}\left(\ind_{K}^{G(F)} (\rho)\right) \ar[r]^-{I_{U}} & \End_{G(F)}\left(I^{G}_{P} \left( \ind_{K_{M}}^{M(F)} (\rho_{M}) \right)\right) \ar[r]^-{T_{\rho_{M}}} & \End_{G(F)}\left(I_{P}^{G}\left(\ind_{M^1}^{M(F)}(\sigma_{1})\right)\right)
}
\]
obtained from Proposition~\ref{compatibility}.
Recall that we have isomorphism~\eqref{Mverofheckevsend}
\[
\mathcal{H}(M(F), \rho_{M}) \simeq \End_{M(F)}\left( \ind_{K_{M}}^{M(F)} (\rho_{M}) \right)
\]
and isomorphism~\eqref{isomgroupalgebraheckealgebra}
\[
\mathbb{C}[M_{\sigma}/M^1] \simeq \End_{M(F)}\left(\ind_{M^1}^{M(F)} (\sigma_{1})\right).
\]
We study the composition
\begin{align}
\label{compositionofMverofheckevsendandisomgroupalgebraheckealgebra}
\mathcal{H}(M(F), \rho_{M}) \xrightarrow{\eqref{Mverofheckevsend}} \End_{M(F)}\left( \ind_{K_{M}}^{M(F)} (\rho_{M}) \right) \xrightarrow{T_{\rho_{M}}} \End_{M(F)}\left(\ind_{M^1}^{M(F)} (\sigma_{1})\right) \xrightarrow{\eqref{isomgroupalgebraheckealgebra}} \mathbb{C}[M_{\sigma}/M^1].
\end{align}
Let $m \in M_{\sigma}$.
We use the same symbol $m$ for the image of $m$ in $M_{\sigma}/M^{1}$ by abuse of notation.
The element $\theta_{m} \in \mathbb{C}[M_{\sigma}/M^1]$ corresponds to the element 
\[
(\Phi^{M}_{m^{-1}})' \in \End_{M(F)}\left(\ind_{M^1}^{M(F)} (\sigma_{1})\right)
\]
defined as
\[
\left(
(\Phi^{M}_{m^{-1}})'(f)
\right) (m') = \sigma(m^{-1}) \cdot f(mm')
\]
for $f \in \ind_{M^1}^{M(F)} (\sigma)$ and $m' \in M(F)$ via isomorphism~\eqref{isomgroupalgebraheckealgebra}.
Since we defined $\sigma_{1}$ as
\[
\sigma_{1} = \ind_{K_{M}}^{M^1} (\rho_{M}),
\]
the natural inclusion
\[
I_{M(F)}(\rho_{M}) \subset M_{\sigma}
\]
induces an isomorphism
\[
I_{M(F)}(\rho_{M}) / K_{M} = I_{M(F)}(\rho_{M}) / I_{M^1}(\rho_{M}) \simeq M_{\sigma}/M^{1}.
\]
%Here, we note that $K_{M} = I_{M^1}(\rho_{M})$ is a normal subgroup of $I_{M(F)}(\rho_{M})$.
%We also note that it implies that every element in $I_{M(F)}(\rho_{M})$ normalizes $K_{M}$ and $\rho_{M}$.
We identify the group algebra $\mathbb{C}[M_{\sigma}/M^1]$ with $\mathbb{C}[I_{M(F)}(\rho_{M}) / K_{M} ]$.
Let
\[
m \in I_{M(F)}(\rho_{M}) \subset \widetilde{K_{M}}.
\]
Since $m$ normalizes the representation $\rho_{M}$, and the representation $\rho_{M}$ appears in $\widetilde{\rho_{M}}\restriction_{K_M}$ with multiplicity $1$, $\widetilde{\rho_{M}}(m^{-1})$ preserves the subspace $\rho_{M}$ of $\widetilde{\rho_{M}}$.
We define
\[
\Phi^{M}_{m^{-1}} \in \End_{M(F)}\left( \ind_{K_{M}}^{M(F)} (\rho_{M}) \right)
\]
as
\[
\left(
\Phi^{M}_{m^{-1}}(f)
\right) (m') = \widetilde{\rho_{M}}(m^{-1}) \cdot f(mm')
\]
for $f \in \ind_{K_{M}}^{M(F)} (\rho_{M})$ and $m' \in M(F)$.
\begin{lemma}
For  $m \in I_{M(F)}(\rho_{M})$, we have
\[
T_{\rho_{M}} (\Phi^{M}_{m^{-1}}) = (\Phi^{M}_{m^{-1}})'.
\]
\end{lemma}
\begin{proof}
It suffices to show that
\begin{align}
\label{phimvsphi'm}
\left(
T_{\rho_{M}} \circ \Phi^{M}_{m^{-1}}
\right)(f) =
\left(
(\Phi^{M}_{m^{-1}})' \circ T_{\rho_{M}}
\right)(f)
\end{align}
for all $f \in \ind_{K_{M}}^{M(F)} (\rho_{M})$.
Since $T_{\rho_{M}} \circ \Phi^{M}_{m^{-1}}$ and $(\Phi^{M}_{m^{-1}})' \circ T_{\rho_{m}}$ are $M(F)$-equivariant, and $\ind_{K_{M}}^{M(F)} (\rho_{M})$ is generated by 
\[
\left\{
f_{v} \mid v \in \rho_{M}
\right\}
\]
as an $M(F)$-representation, we may suppose that $f = f_{v}$ for some $v \in \rho_{M}$.
For $m' \in M(F)$ and $m'' \in M^1$, we have
\begin{align*}
\left(
\left(
\left(
T_{\rho_{M}} \circ \Phi^{M}_{m^{-1}}
\right)(f_v)
\right)
(m')
\right)
(m'') &= \left(
\Phi^{M}_{m^{-1}}(f_v)
\right)(m'' m')\\
&= \widetilde{\rho_{M}}(m^{-1}) \cdot f_{v}(m m'' m'),
\end{align*}
and
\begin{align*}
\left(
\left(
\left(
(\Phi^{M}_{m^{-1}})' \circ T_{\rho_{M}}
\right)(f_v)
\right)
(m')
\right)
(m'')
&=
\left(\sigma(m^{-1}) \cdot 
\left(
\left(
T_{\rho_{M}}(f_v)
\right)
(m m')
\right)
\right)(m'').
\end{align*}
Since $f_v$ is supported on $K_M$, and $T_{\rho_{M}}(f_v)$ is supported on $M^1$, both sides of \eqref{phimvsphi'm} vanish unless $m' \in m^{-1}M^1$. 
Let $m' = m^{-1}m_{1}$ for some $m_1 \in M^1$.
Then, we obtain
\begin{align*}
\left(
\left(
\left(
T_{\rho_{M}} \circ \Phi^{M}_{m^{-1}}
\right)(f_v)
\right)
(m^{-1}m_{1})
\right)
(m'')
&= \widetilde{\rho_{M}}(m^{-1}) \cdot f_{v}(m m'' m^{-1}m_{1})\\
&= 
\begin{cases}
\left(
\widetilde{\rho_{M}}(m^{-1}) \circ \rho_{M}(m m'' m^{-1} m_{1})
\right)
\cdot v & (m m'' m^{-1} m_{1} \in K_{M}),\\
0 & (\text{otherwise})
\end{cases}\\
&= 
\begin{cases}
\widetilde{\rho_{M}}(m'' m^{-1} m_{1}) \cdot v & (m m'' m^{-1} m_{1} \in K_{M}), \\
0 & (\text{otherwise}),
\end{cases}
\end{align*}
and
\begin{align*}
\left(
\left(
\left(
(\Phi^{M}_{m^{-1}})' \circ T_{\rho_{M}}
\right)(f_v)
\right)
(m^{-1} m_1)
\right)
(m'')
&=
\left(\sigma(m^{-1}) \cdot \left(
\left(
T_{\rho_{M}}(f_v)
\right)
(m_1)
\right)
\right)(m'')\\
&=
\left(
\left(
\sigma(m^{-1}) \circ \sigma_{1}(m_1)
\right) \cdot
\left(
\left(
T_{\rho_{M}}(f_v)
\right)
(1)
\right)
\right)(m'')\\
&=
\left(\sigma(m^{-1}m_1) \cdot f_{v, 1} \right)(m'')\\
&= \left(\sigma(m^{-1}m_1) \cdot \widetilde{f_{v}} \right)(m'')\\
&= \widetilde{f_{v}}(m''m^{-1}m_1)\\
&= 
\begin{cases}
\widetilde{\rho_{M}}(m'' m^{-1} m_{1}) \cdot v & (m'' m^{-1} m_{1} \in \widetilde{K_{M}}), \\
0 & (\text{otherwise}).
\end{cases}
\end{align*}
Since $m'', m_{1} \in M^1$, $K_{M} = \widetilde{K_{M}} \cap M^1$, and $m \in I_{M(F)}(\rho_{M}) \subset \widetilde{K_{M}}$, we have
\begin{align*}
m m'' m^{-1} m_{1} \in K_{M} &\iff m m'' m^{-1} m_{1} \in \widetilde{K_{M}}\\
&\iff m'' m^{-1} m_{1} \in \widetilde{K_{M}}.
\end{align*}
Thus, we obtain \eqref{phimvsphi'm}.
\end{proof}
Finally, for $m \in I_{M(F)}(\rho_{M})$, the element $\phi^{M}_{m^{-1}}$ of $\mathcal{H}(M(F), \rho_{M})$ corresponding to $\Phi^{M}_{m^{-1}}$ via isomorphism~~\eqref{Mverofheckevsend} is defined as
\begin{align*}
\left(
\phi^{M}_{m^{-1}}(m')
\right)
(v) &= 
\left(
\Phi^{M}_{m^{-1}}(f_{v})
\right)
(m')\\
&= \widetilde{\rho_{M}}(m^{-1}) \cdot f_{v}(mm')\\
&=
\begin{cases}
\widetilde{\rho_{M}}(m') \cdot v & (m' \in m^{-1}K_{M} ), \\
0 & (\text{otherwise})
\end{cases} 
\end{align*}
for $m' \in M(F)$ and $v \in \rho_{M}$.
Thus, we obtain:
\begin{lemma}
\label{lemmasupportofimageofm}
For $m \in I_{M(F)}(\rho_{M})$, let $\phi^{M}_{m^{-1}}$ denote the element of $\mathcal{H}(M(F), \rho_{M})$ corresponding to $\theta_{m} \in \mathbb{C}[M_{\sigma}/M^1]$ via isomorphism~\eqref{compositionofMverofheckevsendandisomgroupalgebraheckealgebra}.
Then, $\phi^{M}_{m^{-1}}$ is supported on $m^{-1}K_{M}$.
\end{lemma}

Recall that we defined a subgroup $T(J, \rho)$ of $R(J, \rho)$ as
\[
T(J, \rho) = \{
t \in R(J, \rho) \mid (Dt)\restriction_{V^{J}} = \id
\}.
\]
For $t \in T(J, \rho)$, $\widetilde{v(t)}$ denotes the element of $(V^{J, \Gamma})^{\perp} \subset V^{J} \subset a_{M}$ such that
\[
t(x) = x+\widetilde{v(t)}
\]
for all $x \in \mathcal{A}^{J}$, and $v(t)$ denotes the projection of $\widetilde{v(t)}$ on $V^{J}_{\Gamma}$.
The definition of the Levi subgroup $M$ implies that
\[
T(J, \rho) = R(J, \rho) \cap W_{M(F)}.
\]
\begin{lemma}
\label{WcapMsimeqImrho}
The canonical projection induces an isomorphism
\[
W(J, \rho) \cap W_{M(F)} \rightarrow I_{M(F)}(\rho_{M})/K_{M}.
\]
\end{lemma}
\begin{proof}
Since $\dot{w}$ intertwines $\rho$ for all $w \in W(J, \rho)$, we have
\[
W(J, \rho) \cap W_{M(F)}  \subset 
W_{I_{M(F)}(\rho_{M})}.
\]
According to the definition of $K = P_{J}$ \cite[3.7]{MR1235019} and general theory of BN-pair, we have
\[
W_{K_{M}}
= W_{K} = W_{J}.
\]
Since any element of $W(J, \rho)$ fixes $J$, \cite[Lemma~2.2]{MR1235019} implies that
\begin{align*}
W(J, \rho) \cap W_{K_{M}}
= W(J, \rho) \cap W_{J} = \{1\}.
\end{align*}
Hence, the natural projection induces an injection
\[
W(J, \rho) \cap W_{M(F)} \rightarrow I_{M(F)}(\rho_{M})/K_{M}.
\]
Moreover, according to \cite[Theorem~4.15]{MR1235019}, the map is surjective.
\end{proof}
%For the meaning of these inclusions, see Remark~\ref{identificationofWwithitslift}.
\begin{comment}
\begin{remark}
Since $T(J, \rho)$ is a subgroup of $W= N_{G}(S)(F)/Z_{G}(S)_{0}$, the precise statement is that any lift $\widetilde{t}$ of an element $t \in T(J, \rho)$ is contained in $I_{M(F)}(\rho_{M})$.
We omit the sentence ``any lift of'' and identify an element of $W$ with its lift if a property does not depend on the choice of lifts.
\end{remark}
\end{comment}
We regard 
\[
T(J, \rho) = R(J, \rho) \cap W_{M(F)} \subset W(J, \rho) \cap W_{M(F)}
\]
as a subgroup of 
\[
I_{M(F)}(\rho_{M})/K_{M} \simeq M_{\sigma}/M^1
\]
via Lemma~\ref{WcapMsimeqImrho}.
\begin{lemma}
\label{vvshm}
Let $t \in T(J, \rho)$. Then, we have
\[
\widetilde{v(t)} = - H_{M}(t),
\]
where 
\[
H_{M} \colon M(F)/M^1 \rightarrow a_{M}
\]
denotes the map defined in Section~\ref{A review of Solleveld's results}.
\end{lemma}
\begin{proof}
Note that equation~\eqref{RjrhocontainedinG1} implies that the image of an element of $T(J, \rho)$ via $H_{M}$ is contained in the subspace of $a_{M}$ spanned by $\alpha^{\vee} \, (\alpha \in \Sigma_{\red}(A_M))$.
Let $Z_{G}(S)$ denote the minimal semi-standard Levi subgroup of $G$ with respect to $S$.
For $t \in T(J, \rho) \cap 
W_{Z_{G}(S)(F)}$, the definition of $H_{M}$ and \cite[1.2 (1)]{MR546588} imply that
\[
\widetilde{v(t)} = - H_{M}(t).
\]
Let $t \in T(J, \rho)$.
Since the image of $Z_{G}(S)(F)$ on $M(F)/M^1$ via the natural projection
\[
Z_{G}(S)(F) \subset M(F) \rightarrow M(F)/M^1
\]
is of finite index, there exists $n \in \mathbb{Z}_{>0}$ such that 
\[
t^{n} \in T(J, \rho) \cap 
W_{Z_{G}(S)(F)}.
\]
Hence,
\begin{align*}
n \cdot \widetilde{v(t)} &= \widetilde{v(t^{n})} \\
&= -H_{M}(t^{n}) \\
&= - n \cdot H_{M}(t).
\end{align*}
Thus, we obtain $\widetilde{v(t)} = - H_{M}(t)$.
\end{proof}
\begin{comment}
\begin{remark}
\label{remarkoflemmavvsm}
The proof of Lemma~\ref{vvshm} implies that for any
\[
m \in W(J, \rho) \cap W_{M(F)} \simeq M_{\sigma}/M^{1}
\]
and $x \in \mathcal{A}^{J} \subset a_{M}$, we have
\[
m(x) = x - H_{M}(m).
\]
\end{remark}
\end{comment}
We also recall that an element $z \in M(F)$ is called positive relative to $K$ and $U$, if it satisfies the conditions
\[
zK_{U}z^{-1} \subset K_{U}, \ z^{-1}K_{\overline{U}}z \subset K_{\overline{U}}.
\]
\begin{lemma}
\label{positivitycompativility}
Let $t \in T(J, \rho)$ such that $\left(D_{J}(a')\right)(v(t)) \ge 0$ for all $a' \in B(J, \rho)_{e}$.
%$(Da)(v(t)) \ge 0$ for all $a \in  \Gamma(J, \rho)^{+}$ such that $a + A'_{J} \in B(J, \rho)_{e}$.
Then, the lift $\dot{t}$ of $t$ is positive relative to $K$ and $U'$ for some parabolic subgroup $P'$ of $G$ with Levi factor $M$ and unipotent radical $U'$ such that
\[
D_{J}\left(
\Gamma'(J, \rho)^{+}_{e}
\right)  = 
D_{J} \left(
\Gamma'(J, \rho)_{e}
\right) \cap \left( - \Sigma(P', A_{M}) \right).
\]
\end{lemma}
\begin{proof}
Let $t \in T(J, \rho)$ such that $\left(D_{J}(a')\right)(v(t)) \ge 0$ for all $a' \in B(J, \rho)_{e}$.
Take an element $\lambda \in a_{M}$ such that
\[
\langle \alpha, \lambda \rangle \neq 0
\]
for all $\alpha \in \Sigma(G, A_{M})$, and
\begin{align}
\label{positivityforgamma}
\langle \alpha, \lambda \rangle > 0
\end{align}
for all $\alpha \in D_{J} \left(B(J, \rho)_{e}\right)$.
The assumption of $t$ implies that
$\widetilde{v(t)}$ lies in the closure of the set
\[
\{
x \in a_{M} \mid \langle \alpha, x \rangle > 0 \ (\alpha \in D_{J} \left(B(J, \rho)_{e}\right))
\},
\]
hence we can take $\lambda$ sufficiently close to $\widetilde{v(t)}$.
More precisely, we may assume
\begin{align}
\label{sufficientlysmall}
\langle \alpha, \widetilde{v(t)} \rangle \le 0
%\abs{
%\langle \widetilde{v(t)} - \lambda, \alpha \rangle
%}
% < \abs{
%\langle \widetilde{v(t)}, \alpha \rangle
%}
\end{align}
for all $\alpha \in \Sigma(G, A_{M})$ such that
\[
\langle \alpha, \lambda \rangle < 0.
\]
Let $P'$ be a parabolic subgroup of $G$ with Levi factor $M$ and unipotent radical $U'$ such that
\[
\Sigma(P', A_{M}) = \{
\alpha \in \Sigma(G, A_{M}) \mid \langle \alpha, \lambda \rangle < 0
\}.
\]
Condition~\eqref{positivityforgamma} implies that
\[
D_{J}\left(
\Gamma'(J, \rho)^{+}_{e}
\right)  \subset 
D_{J} \left(
\Gamma'(J, \rho)_{e}
\right) \cap \left( - \Sigma(P', A_{M}) \right).
\]
Since $D_{J}\left(
\Gamma'(J, \rho)^{+}_{e}
\right)$ and $D_{J} \left(
\Gamma'(J, \rho)_{e}
\right) \cap \left( - \Sigma(P', A_{M}) \right)$ are sets of positive roots in $D_{J} \left(
\Gamma'(J, \rho)_{e}
\right)$, we obtain that
\[
D_{J}\left(
\Gamma'(J, \rho)^{+}_{e}
\right)  = 
D_{J} \left(
\Gamma'(J, \rho)_{e}
\right) \cap \left( - \Sigma(P', A_{M}) \right).
\]
%Moreover, condition~\eqref{sufficientlysmall} implies that
%\[
%\langle \widetilde{v(t)}, \alpha \rangle \le 0
%\]
%for all $\alpha \in \Sigma(P', A_{M})$.

We will prove that $\dot{t}$ is positive relative to $K$ and $U'$.
According to \cite[3.1]{MR546588}, we can take $y \in \mathcal{F}_{J} \subset \mathcal{A}^{J}$ such that
\[
P_{J} = G(F)_{y, 0},
\]
where $G(F)_{y, 0}$ denotes the parahoric subgroup of $G(F)$ associated with $y$ \cite[3.1, 3.2]{MR1371680}.
Then, we obtain
\begin{align*}
\dot{t} K_{U'} \dot{t}^{-1} &= \dot{t}(G(F)_{y, 0} \cap U'(F)) \dot{t}^{-1} \\
&= G(F)_{t \cdot y, 0} \cap U'(F) \\
&= G(F)_{y + \widetilde{v(t)}, 0} \cap U'(F).
\end{align*}
Condition~\eqref{sufficientlysmall} implies that
\[
\langle \alpha, \widetilde{v(t)} \rangle \le 0
\]
for all $\alpha \in \Sigma(P', A_{M})$.
Thus, the definition of the parahoric subgroup \cite[3.1, 3.2]{MR1371680} implies that $\dot{t} K_{U'} \dot{t}^{-1} \subset K_{U'}$.
Similarly, we can prove that $\dot{t}^{-1} K_{\overline{U}'} \dot{t} \subset K_{\overline{U}'}$.
\end{proof}
\begin{remark}
\label{remarkoflemmapositivitycompativility}
Let $P'$ be a parabolic subgroup of $G$ with Levi factor $M$ and unipotent radical $U'$.
Let $m \in W(J, \rho) \cap W_{M(F)}$ such that
\[
\langle \alpha, H_{M}(m) \rangle \ge 0
\]
for all $\alpha \in \Sigma(P', A_{M})$.
Then, the proof of Lemma~\ref{vvshm} and Lemma~\ref{positivitycompativility} imply that the lift $\dot{m}$ of $m$ is positive relative to $K$ and $U'$.
\end{remark}

%First, we study the translation parts of isomorphism $T_{\rho_{M}} \circ I_{U}$\

Combining Lemma~\ref{independencyofP} and Lemma~\ref{lemmasupportofimageofm}, with Lemma~\ref{positivitycompativility}, we obtain the following Corollaries.
\begin{corollary}
\label{corollaryimageofphiM}
Let $t \in T(J, \rho)$ such that $\left(D_{J}(a')\right)(v(t)) \ge 0$ for all $a' \in B(J, \rho)_{e}$.
Let $\phi^{M}_{t}$ denote the element of $\mathcal{H}(M(F), \rho_{M})$ corresponding to 
\[
\theta_{t^{-1}} \in \mathbb{C}[I_{M(F)}(\rho_{M})/K_{M}] = \mathbb{C}[M_{\sigma}/M^1]
\]
via isomorphism~\eqref{compositionofMverofheckevsendandisomgroupalgebraheckealgebra}.
We also write $\Phi^{M}_{t}$ for the element of $\End_{M(F)}\left( \ind_{K_{M}}^{M(F)} (\rho_{M}) \right)$ corresponding to $\phi^{M}_{t}$ via isomorphism~\eqref{Mverofheckevsend}.
Then, there exists $c(t) \in \mathbb{C}^{\times}$ such that
\[
(I^{\Mor} \circ t_{P})(\Phi^{M}_{t}) = c(t) \cdot \theta_{v(t)}.
\]
\end{corollary}
\begin{proof}
According to Lemma~\ref{positivitycompativility}, $\dot{t}$ is positive relative to $K$ and $U'$ for some parabolic subgroup $P'$ of $G$ with Levi factor $M$ and unipotent radical $U'$ such that
\[
D_{J}\left(
\Gamma'(J, \rho)^{+}_{e}
\right)  = 
D_{J} \left(
\Gamma'(J, \rho)_{e}
\right) \cap \left( - \Sigma(P', A_{M}) \right).
\]
According to Lemma~\ref{independencyofP}, we have $t_{P} = t_{P'}$.
Hence, by replacing $P$ with $P'$, we may assume that $\dot{t}$ is positive relative to $K$ and $U$.
According to Lemma~\ref{lemmasupportofimageofm}, $\phi^{M}_{t}$ is supported on $\dot{t} K_{M}$.
Hence, the definition of $t_{P}$ implies that $t_{P}(\phi^{M}_{t})$ is supported on $K \dot{t} K$.
Thus, there exists $c'(t) \in \mathbb{C}^{\times}$ such that
\[
t_{P} \left(
\Phi^{M}_{t}
\right)
 = c'(t) \cdot \Phi_{t}.
\]
Here, $\Phi_{t}$ denotes the element of $\End_{G(F)}\left(\ind_{K}^{G(F)} (\rho)\right)$ appearing in Theorem~\ref{theorem7.12ofmorris}.
On the other hand, since $t$ satisfies $\left(D_{J}(a')\right)(v(t)) \ge 0$ for all $a' \in B(J, \rho)_{e}$, Corollary~\ref{corollaryoftheorem7.12ofmorrisaffinever} implies that
\[
I^{\Mor} \left(
\Phi_{t}
\right) = q_{v(t)}^{1/2} \cdot \theta_{v(t)}.
\]
Hence, we obtain
\[
\left(
I^{\Mor} \circ t_{P})(\Phi^{M}_{t}
\right) = c(t) \cdot \theta_{v(t)}
\]
for
\[
c(t)= c'(t) \cdot q_{v(t)}^{1/2}.
\]
\end{proof}
\begin{corollary}
\label{corollaryvectorpart}
The image of $\mathbb{C}[\mathbb{Z}(R^{\Mor})^{\vee}]$ via the map
\[
\mathcal{H}^{\Mor} \xrightarrow{(I^{\Mor})^{-1}} \mathcal{H}(R(J, \rho)) \xrightarrow{T_{\rho_{M}} \circ I_{U}} \End_{G(F)}\left(I_{P}^{G}\left(\ind_{M^1}^{M(F)}(\sigma_{1})\right)\right)
\]
is contained in the image of $\End_{M(F)}\left(\ind_{M^1}^{M(F)} (\sigma_{1})\right)$ via $I_{P}^{G}$.
Moreover, for $t \in T(J, \rho)$ there exists $c(t) \in \mathbb{C}^{\times}$ such that
\[
\left(
I^{\Sol} \circ T_{\rho_{M}} \circ I_{U} \circ (I^{\Mor})^{-1}
\right)\left(
\theta_{v(t)} 
\right)
= c(t)^{-1} \cdot \theta_{t^{-1}}.
\]
If $t$ satisfies $\left(D_{J}(a')\right)(v(t)) \ge 0$ for all $a' \in B(J, \rho)_{e}$, 
the number $c(t)$ coincides with the number appearing in Corollary~\ref{corollaryimageofphiM}.
\end{corollary}
\begin{proof}
Let $t \in T(J, \rho)$ such that
\begin{align}
\label{positivityoftjrho}
\left(D_{J}(a')\right)(v(t)) \ge 0
\end{align}
for all $a' \in B(J, \rho)_{e}$.
We use the same notation as Corollary~\ref{corollaryimageofphiM}.
Proposition~\ref{compatibility} and Corollary~\ref{corollaryimageofphiM} imply that
\begin{align*}
\left(
T_{\rho_{M}} \circ I_{U} \circ (I^{\Mor})^{-1}
\right) (\theta_{v(t)}) &= c(t)^{-1} \cdot \left(
T_{\rho_{M}} \circ I_{U} \circ t_{P}
\right)(\Phi^{M}_{t})\\
&= c(t)^{-1} \cdot \left(
T_{\rho_{M}} \circ I_{P}^{G}
\right)
(\Phi^{M}_{t})\\
&= c(t)^{-1} \cdot \left(
I_{P}^{G} \circ T_{\rho_{M}}
\right)
(\Phi^{M}_{t}).
\end{align*}
Moreover, since $\phi^{M}_{t} \in \mathcal{H}(M(F), \rho_{M})$ corresponds to $\theta_{t^{-1}} \in \mathbb{C}[M_{\sigma}/M^1]$ via isomorphism~\eqref{compositionofMverofheckevsendandisomgroupalgebraheckealgebra} and corresponds to $\Phi^{M}_{t}$ via isomorphism~\eqref{Mverofheckevsend}, $T_{\rho_{M}}(\Phi^{M}_{t})$ corresponds to $\theta_{t^{-1}}$ via \eqref{isomgroupalgebraheckealgebra}.
Hence, Theorem~\ref{modificationoftheorem10.9ofsolleveld} implies that
\begin{align*}
\left(
I^{\Sol} \circ I_{P}^{G} \circ T_{\rho_{M}}
\right)
(\Phi^{M}_{t}) &= \theta_{t^{-1}}.
\end{align*}
Thus, we obtain
\begin{align*}
\left(
I^{\Sol} \circ  T_{\rho_{M}} \circ I_{U} \circ (I^{\Mor})^{-1}
\right) (\theta_{v(t)}) &= c(t)^{-1} \cdot \left(
I^{\Sol} \circ  I_{P}^{G} \circ T_{\rho_{M}}
\right)
(\Phi^{M}_{t})\\
&= c(t)^{-1} \cdot \theta_{t^{-1}}.
\end{align*}
Since any element of $T(J, \rho)$ can be written as a difference of elements of $T(J, \rho)$ satisfying \eqref{positivityoftjrho}, and
\[
t \mapsto v(t)
\]
defines an isomorphism
\[
T(J, \rho) \rightarrow \mathbb{Z}(R^{\Mor})^{\vee},
\]
we obtain the claim.
\end{proof}
We will prove that $c(t) = 1$ for all $t \in T(J, \rho)$ in Corollary~\ref{maintheoremc=1}.

%Moreover, we will prove the following:
%Recall that we regard $R^{\Mor}$ is a root system in $(a_{M})_{\Gamma} \subset a_{M}$. 
%We also regard the root system $R^{\Sol}$ in $M_{\sigma}/M^1$ as a subset of $a_{M}$ via 
Next, we compare the ``finite parts'' of the both sides of $T_{\rho_{M}} \circ I_{U}$.
The following theorem is one of the main results of this paper.
\begin{theorem}
\label{maintheoremrootsystem}
As subsets of $a_{M}^{*}$, we have
\[
\begin{cases}
R^{\Mor} &= R^{\Sol}, \\
\Delta^{\Mor} &= - \Delta^{\Sol}.
\end{cases}
\]
Here, we regard $R^{\Sol}$ as a subset of $a_{M}^{*}$ via
\[
H_{M}^{\vee} \colon \left(
M_{\sigma}/M^1
\right)^{\vee} \rightarrow a_{M}^{*}.
\]
\end{theorem}
\begin{remark}
\label{dualofmaintheoremrootsystem}
We also have the dual of Theorem~\ref{maintheoremrootsystem}.
By using isomorphism~\eqref{orthogonalcomplementmorris},
we consider the root system 
\[
(R^{\Mor})^{\vee} = \left\{
k_{a'} (D_{J} (a'))^{\vee} \mid a' \in \Gamma'(J, \rho)_{e}
\right\}
\]
in $V_{\Gamma}^{J}$ as a root system in $(V^{J, \Gamma})^{\perp}$.
In particular, we consider $(R^{\Mor})^{\vee}$ as a subset of $a_{M}$.
We also regard $(R^{\Sol})^{\vee}$ as a subset of $a_{M}$ via
\[
H_{M} \colon M_{\sigma}/M^1 \rightarrow a_{M}.
\]
Then, according to Theorem~\ref{maintheoremrootsystem}, as subsets of $a_{M}$, we have 
\[
\begin{cases}
(R^{\Mor})^{\vee} &= (R^{\Sol})^{\vee}, \\
(\Delta^{\Mor})^{\vee} &= - (\Delta^{\Sol})^{\vee},
\end{cases}
\]
where $(\Delta^{\Mor})^{\vee}$ and $(\Delta^{\Sol})^{\vee}$ are defined as
\[
(\Delta^{\Mor})^{\vee} = \left\{
k_{a'} (D_{J} (a'))^{\vee} \mid a' \in B(J, \rho)_{e}
\right\}\index{$(\Delta^{\Mor})^{\vee}$}
\] 
and
\[
(\Delta^{\Sol})^{\vee} =
\{
(h_{\alpha}^{\vee})' \mid \alpha \in \Delta_{\mathfrak{s}_{M}, \mu}(P)
\}.\index{$(\Delta^{\Sol})^{\vee}$}
\]
\end{remark}
\begin{comment}
\begin{lemma}
The following diagram commutes
\[
\xymatrix{
\mathcal{H}(M(F), \rho_{M}) \ar[d]_-{t_{P}} \ar[r] \ar@{}[dr]|\circlearrowleft & \mathbb{C}[\mathbb{Z}R^{\Mor}] \ar[d] \\
\mathcal{H}(G(F), \rho) \ar[r]^-{I^{\Mor}} & \mathcal{H}^{\Mor}.
}
\]
\end{lemma}
We identify the Weyl group $W_{0}(R^{\Mor})$ of $R^{\Mor}$ with the Weyl group $W_{0}(R^{\Sol})$ of $R^{\Sol}$.
Then, the second claim of Theorem~\ref{maintheoremrootsystem} implies that the set of simple reflections in $W_{0}(R^{\Mor})$ with respect to the basis $\Delta^{\Mor}$ coincides with the set of simple reflections in $W_{0}(R^{\Sol})$ with respect to the basis $\Delta^{\Sol}$. 
lift
\end{comment}

We identify the Weyl group $W_{0}(R^{\Mor})$ of $R^{\Mor}$ with the Weyl group $W_{0}(R^{\Sol})\index{$W_{0}(R^{\Sol})$}$ of $R^{\Sol}$.
Then, the second claim of Theorem~\ref{maintheoremrootsystem} implies that the set of simple reflections in $W_{0}(R^{\Mor})$ with respect to the basis $\Delta^{\Mor}$ coincides with the set of simple reflections in $W_{0}(R^{\Sol})$ with respect to the basis $\Delta^{\Sol}$. 

Let $\alpha \in \Delta_{\mathfrak{s}_{M}, \mu}(P)$.
For simplicity, we write 
\[
\alpha' = (\alpha^{\#})' \in \Delta^{\Sol} = - \Delta^{\Mor}\index{$\alpha' $}
\]
and
\[
(\alpha')^{\vee} = (h_{\alpha}^{\vee})' \in (\Delta^{\Sol})^{\vee} = -(\Delta^{\Mor})^{\vee}.\index{$(\alpha')^{\vee}$}
\]
%Let $s_{\alpha} \in W_{0}(R^{\Sol})$ denote the corresponding simple reflection.
Recall that we have to choose a lift $\widetilde{s_{\alpha}}$ in $I_{M_{\alpha}^{1}}(\sigma_{1})$ of the reflection $s_{\alpha} \in W_{0}(R^{\Sol})$ to define $T'_{s_{\alpha}} \in \End_{G(F)}\left(I^{G}_{P} \left( \ind_{M^1}^{M(F)} (\sigma_{1}) \right)\right)$.
We fix the lift $\widetilde{s_{\alpha}}$ as follows.
Let $a \in \Gamma(J, \rho)^{+}$ such that $a + A'_{J} \in B(J, \rho)_{e}$, and
\[
r(a) = - \alpha' \in \Delta^{\Mor}.
\]
We fix a  lift $\widetilde{s}$ of $v[a, J]$ in $N_{G}(S)(F)$.
Since $v[a, J] \in W(J, \rho)$ fixes $J$, the definition of $M$ implies that $\widetilde{s} \in N_{G}(M)(F)$. 
The definition of $W(J, \rho)$ also implies that $\widetilde{s}$ intertwines $\rho_{M} = \rho\restriction_{K_{M}}$, hence normalizes the representation
\[
\sigma_{1} = \ind_{K_{M}}^{M^1} (\rho_{M}).
\]
Since 
\begin{align*}
Da\restriction_{A_{M}} &= D_{J}(a + A'_{J}) \\
&= k_{a + A'_{J}} \cdot r(a) \\
&= - k_{a + A'_{J}} \cdot \alpha' \\
&\in \mathbb{R}^{\times} \cdot \alpha,
\end{align*}
we have
\[
W_{J \cup \{a\}} \subset W_{M_{\alpha}}(F) \backslash W_{M(F)}.
\]
Hence, the definition of $v[a, J]$ implies that $\widetilde{s}$ is contained in $M_{\alpha}(F) \backslash M(F)$.
Thus, we obtain that the image of $\widetilde{s}$ on $N_{G}(M)(F)/M(F)$ is equal to $s_{\alpha}$, that is
the unique nontrivial element of 
\[
W(M_{\alpha}, M) =
\left(
N_{G}(M)(F) \cap M_{\alpha}(F)
\right)/M(F).
\]
Moreover, the definition of $v[a, J]$ implies that the element $\widetilde{s}$ is contained in a parahoric subgroup of $M_{\alpha}(F)$.
In particular, we have $\widetilde{s} \in M_{\alpha}^{1}$.
Thus, we can take the lift $\widetilde{s_{\alpha}}$ of the reflection $s_{\alpha}$ in $I_{M_{\alpha}^{1}}(\sigma_{1})$ as $\widetilde{s_{\alpha}} = \widetilde{s}$.
%According to Proposition~\ref{proposition7.3ofmorris}, the action of $v[a, J]$ on $\mathcal{A}$ preserves $\mathcal{A}^{J}$ and induces a well-defined action on $\mathcal{A}^{J}_{\Gamma}$ that coincides with 
%\[
%s_{a + A'_{J}} \in W_{\aff}\left(\Gamma'(J, \rho)\right).
%\]
%Moreover, since 
%$
%a + A'_{J} \in B(J, \rho)_{e}
%$,
%the element $s_{a + A'_{J}}$ corresponds to the reflection 
%\[
%s_{r(a)} \in W_{0}\left(
%R^{\Mor}
%\right) \subset W_{\aff}\left(
%R^{\Mor}
%\right)
%\]
%via isomorphism~\ref{affinesplitting}.
%We fix a lift $\widetilde{s}$ of $s$ such that $\widetilde{s}$ is also a lift of the element $v[a, J] \in R(J, \rho)$ that maps to 
%\[
%s \in W_{0}(R^{\Sol}) = W_{0}(R^{\Mor})
%\]
%via isomorphism
%\[
%R(J, \rho) \rightarrow W_{\aff}\left(\Gamma'(J, \rho)\right)
%\] 
%of Proposition~\ref{proposition7.3ofmorris}.
%Here, we regard $W_{0}(R^{\Mor})$ as a subgroup of $W_{\aff}\left(\Gamma'(J, \rho)\right)$ via isomorphism~\eqref{affinesplitting}.

%We also regard $\alpha'$ as an element of $R^{\Mor}$ via Theorem~\ref{maintheoremrootsystem}.
%According to Theorem~\ref{maintheoremrootsystem}, we have $- \alpha' \in \Delta^{\Mor}$.
For a simple reflection $s = s_{\alpha}$ associated with an element $\alpha \in \Delta_{\mathfrak{s}_{M}, \mu}(P)$,
we define $T^{\Sol}_{s, 0} \in \mathcal{H}^{\Sol}$ as
\[
T^{\Sol}_{s, 0} = q_{F}^{\left(-\lambda^{\Sol}(\alpha') + (\lambda^{*})^{\Sol}(\alpha')\right)/2} \cdot
\left(
\theta_{(\alpha')^{\vee}} T^{\Sol}_{s} - (q_{F}^{\lambda^{\Sol}(\alpha')} - 1) \theta_{(\alpha')^{\vee}}
\right)\index{$T^{\Sol}_{s, 0}$}
\]
(see Appendix~\ref{Homomorphism between Affine Hecke algebras of type}).

Now, we state the second main theorem.
Let $\epsilon = \epsilon_{\alpha} \in \{0, 1\}$ denote the number defined in \cite[Lemma~10.7 (b)]{MR4432237}.
\begin{theorem}
\label{maintheoremisomofaffinehecke}
The image of $\mathcal{H}(R(J, \rho))$ via isomorphism
\[
T_{\rho_{M}} \circ I_{U} \colon \End_{G(F)}\left(\ind_{K}^{G(F)} (\rho)\right) \rightarrow \End_{G(F)}\left(I^{G}_{P} \left( \ind_{M^1}^{M(F)} (\sigma_{1}) \right)\right)
\]
is contained in $\mathcal{H}\left(W(\Sigma_{\mathfrak{s}_{M}, \mu})\right)$.
Moreover, for a simple reflection 
\[
s = s_{\alpha} \in W_{0}(R^{\Mor}) = W_{0}(R^{\Sol})
\]
associated with an element $\alpha \in \Delta_{\mathfrak{s}_{M}, \mu}(P)$,
the image of $T^{\Mor}_{s}$ via the composition
\[
\mathcal{H}^{\Mor} \xrightarrow{(I^{\Mor})^{-1}} \mathcal{H}(R(J, \rho)) \xrightarrow{T_{\rho_{M}} \circ I_{U}} \mathcal{H}\left(W(\Sigma_{\mathfrak{s}_{M}, \mu})\right)
\xrightarrow{I^{\Sol}} \mathcal{H}^{\Sol}
\]
is equal to $\iota\left(T^{\Sol}_{s}\right)$ if $\epsilon_{\alpha} = 0$, and equal to
$\iota\left(
T^{\Sol}_{s, 0}
\right)$ if $\epsilon_{\alpha} =1$, where
\[
\iota \colon \mathcal{H}^{\Sol} \rightarrow \mathcal{H}^{\Sol}
\]
denotes the involution defined in Appendix~\ref{An involution of an affine Hecke algebra}.
We also obtain that
\[
\begin{cases}
\lambda^{\Mor}(- \alpha') &= \lambda^{\Sol}(\alpha'), \\
(\lambda^{*})^{\Mor}(- \alpha') &= (\lambda^{*})^{\Sol} (\alpha')
\end{cases}
\]
if $\epsilon_{\alpha} = 0$, and
\[
\begin{cases}
\lambda^{\Sol}(\alpha') &> (\lambda^{*})^{\Sol}(\alpha'), \\
\lambda^{\Mor}(- \alpha') &= (\lambda^{*})^{\Sol} (\alpha'), \\
(\lambda^{*})^{\Mor}(- \alpha') &= \lambda^{\Sol}(\alpha')
\end{cases}
\]
if $\epsilon_{\alpha} = 1$.
\end{theorem}
\begin{comment}
\begin{remark}
In \cite{MR4432237}, $M_{\sigma}/M^1$ is regarded as a subset of $a_{M}$ via
\[
H_{M} \colon M_{\sigma}/M^1 \rightarrow a_{M}.
\]
On the other hand, an element $t \in T(J, \rho)$ corresponds to $v(t) \in (\mathbb{Z}R^{\Mor})^{\vee}$ via isomorphism
\[
R(J, \rho) \rightarrow W_{\aff}\left(\Gamma'(J, \rho)\right)
\]
of Proposition~\ref{proposition7.3ofmorris}.
According to Lemma~\ref{vvshm}, these two identification differs in sign.
This is why $\Delta^{\Mor}$ and $\Delta^{\Sol}$ also differs in sign, and the convolution $\iota$ is needed in Theorem~\ref{maintheoremisomofaffinehecke}.
If we use $- H_{M}$ instead of $H_{M}$ to define $\mathcal{H}^{\Sol}$ and $I^{\Sol}$, the statement of Theorem~\ref{maintheoremrootsystem} and Theorem~\ref{maintheoremisomofaffinehecke} become simpler.
\end{remark}
\end{comment}
\begin{comment}
By using Theorem~\ref{maintheoremrootsystem}, Remark~\ref{dualofmaintheoremrootsystem} and Theorem~\ref{maintheoremisomofaffinehecke}, we may refine Corollary~\ref{corollaryvectorpart} as follows.
By using isomorphism~\eqref{orthogonalcomplementmorris},
we consider the root system 
\[
(R^{\Mor})^{\vee} = \left\{
k_{a'} (D_{J} (a'))^{\vee} \mid a' \in \Gamma'(J, \rho)_{e}
\right\}
\]
in $V_{\Gamma}^{J}$ as a root system in $(V^{J, \Gamma})^{\perp}$.
In particular, we consider $(R^{\Mor})^{\vee}$ as a subset of $V^{J} = a_{M}$.
Then, Theorem~\ref{maintheoremrootsystem} implies that 
\[
(R^{\Mor})^{\vee} = (R^{\Sol})^{\vee} 
\]
as subsets of $a_{M}$.
\end{comment}
By using Theorem~\ref{maintheoremrootsystem}, Remark~\ref{dualofmaintheoremrootsystem}, and Theorem~\ref{maintheoremisomofaffinehecke}, we may refine Corollary~\ref{corollaryvectorpart} as follows.
\begin{corollary}
\label{maintheoremc=1}
The number $c(t) \in \mathbb{C}^{\times}$ appearing in Corollary~\ref{corollaryvectorpart} is equal to $1$.
Hence, for any $t \in T(J, \rho)$, we have
\[
\left(
\iota \circ I^{\Sol} \circ T_{\rho_{M}} \circ I_{U} \circ (I^{\Mor})^{-1}
\right)\left(
\theta_{v(t)} 
\right)
= \theta_{t}.
\]
\end{corollary}
\begin{proof}
Recall that we are regarding $(R^{\Mor})^{\vee}$ as a subset of $a_{M}$ by using isomorphism~\eqref{orthogonalcomplementmorris}.
Hence, for $t \in T(J, \rho)$, 
\[
v(t) \in \mathbb{Z}(R^{\Mor})^{\vee}
\]
is identified with
\[
-H_{M}(t) = \widetilde{v(t)} \in a_{M},
\]
and
\[
t \mapsto \widetilde{v(t)}
\]
defines an isomorphism
\[
T(J, \rho) \rightarrow (\mathbb{Z}R^{\Mor})^{\vee} = (\mathbb{Z}R^{\Sol})^{\vee}.
\]
Since
\[
(\Delta^{\Sol})^{\vee} = \{
(\alpha')^{\vee} \mid \alpha \in \Delta_{\mathfrak{s}_{M}, \mu}(P)
\}
\]
is a basis of $(R^{\Sol})^{\vee}$,
it suffices to show that
\[
\left(
\iota \circ I^{\Sol} \circ T_{\rho_{M}} \circ I_{U} \circ (I^{\Mor})^{-1}
\right)\left(
\theta_{- (\alpha')^{\vee}} 
\right)
= \theta_{(\alpha')^{\vee}}
\]
for all $\alpha \in \Delta_{\mathfrak{s}_{M}, \mu}(P)$.
We write $c= c((\alpha')^{\vee})$.
%We identify $v(t)$ with $- \alpha' \in \Delta^{\Mor}$.
Let
\[
s = s_{\alpha} \in W_{0}(R^{\Mor}) = W_{0}(R^{\Sol})
\]
denote the simple reflection associated with $\alpha$.
Then, the element $T^{\Mor}_{s}$ satisfies
\begin{align}
\label{bernsteinreloftmorris}
\theta_{-(\alpha')^{\vee}}T^{\Mor}_{s} - T^{\Mor}_{s} \theta_{(\alpha')^{\vee}} = (q_{F}^{\lambda^{\Mor}(- \alpha')} - 1)\theta_{- (\alpha')^{\vee}} + q_{F}^{(\lambda^{\Mor}(- \alpha') + (\lambda^{*})^{\Mor}(- \alpha'))/2} - q_{F}^{(\lambda^{\Mor}(- \alpha') - (\lambda^{*})^{\Mor}(- \alpha'))/2}.
\end{align}
First, we consider the case that $\epsilon_{\alpha} = 0$.
According to Corollary~\ref{corollaryvectorpart} and Theorem~\ref{maintheoremisomofaffinehecke}, comparing the images of both sides of \eqref{bernsteinreloftmorris} via
\[
\iota \circ I^{\Sol} \circ T_{\rho_{M}} \circ I_{U} \circ (I^{\Mor})^{-1},
\]
we obtain
\begin{multline*}
c^{-1} \cdot \theta_{(\alpha')^{\vee}}T^{\Sol}_{s} - c \cdot T^{\Sol}_{s} \theta_{- (\alpha')^{\vee}}\\
 = c^{-1} \cdot (q_{F}^{\lambda^{\Mor}(- \alpha')} - 1)\theta_{(\alpha')^{\vee}} + q_{F}^{(\lambda^{\Mor}(- \alpha') + (\lambda^{*})^{\Mor}(- \alpha'))/2} - q_{F}^{(\lambda^{\Mor}(- \alpha') - (\lambda^{*})^{\Mor}(- \alpha'))/2}.
\end{multline*}
Hence, we have
\begin{align}
\label{iotaofbernsteinrelofmorris}
\theta_{(\alpha')^{\vee}}T^{\Sol}_{s}
= c^{2} \cdot T^{\Sol}_{s} \theta_{- (\alpha')^{\vee}} + (q_{F}^{\lambda^{\Mor}(- \alpha')} - 1)\theta_{(\alpha')^{\vee}} + c \left(
q_{F}^{(\lambda^{\Mor}(- \alpha') + (\lambda^{*})^{\Mor}(- \alpha'))/2} - q_{F}^{(\lambda^{\Mor}(- \alpha') - (\lambda^{*})^{\Mor}(- \alpha'))/2}
\right).
\end{align}
On the other hand, the element $T^{\Sol}_{s}$ satisfies 
\begin{align*}
\theta_{(\alpha')^{\vee}}T^{\Sol}_{s} - T^{\Sol}_{s} \theta_{- (\alpha')^{\vee}} = (q_{F}^{\lambda^{\Sol}(\alpha')} - 1)\theta_{(\alpha')^{\vee}} + q_{F}^{(\lambda^{\Sol}(\alpha') + (\lambda^{*})^{\Sol}(\alpha'))/2} - q_{F}^{(\lambda^{\Sol}(\alpha') - (\lambda^{*})^{\Sol}(\alpha'))/2}, 
\end{align*}
hence we have
\begin{align}
\label{bernsteinrelofsolleveld}
\theta_{(\alpha')^{\vee}}T^{\Sol}_{s} = T^{\Sol}_{s} \theta_{- (\alpha')^{\vee}} + (q_{F}^{\lambda^{\Sol}(\alpha')} - 1)\theta_{(\alpha')^{\vee}} + q_{F}^{(\lambda^{\Sol}(\alpha') + (\lambda^{*})^{\Sol}(\alpha'))/2} - q_{F}^{(\lambda^{\Sol}(\alpha') - (\lambda^{*})^{\Sol}(\alpha'))/2}.
\end{align}
According to Theorem~\ref{maintheoremisomofaffinehecke}, we have
\[
\begin{cases}
\lambda^{\Mor}(- \alpha') &= \lambda^{\Sol}(\alpha') >0, \\
(\lambda^{*})^{\Mor}(- \alpha') &= (\lambda^{*})^{\Sol} (\alpha') > 0.
\end{cases}
\]
Then, comparing the constant terms of the right hand sides of \eqref{iotaofbernsteinrelofmorris} and \eqref{bernsteinrelofsolleveld}, we have $c=1$.

Next, we consider the case that $\epsilon_{\alpha} = 1$.
According to Corollary~\ref{corollaryvectorpart} and Theorem~\ref{maintheoremisomofaffinehecke}, comparing the images of both sides of \eqref{bernsteinreloftmorris} via
\[
\iota \circ I^{\Sol} \circ T_{\rho_{M}} \circ I_{U} \circ (I^{\Mor})^{-1},
\]
we obtain
\begin{multline*}
c^{-1} \cdot \theta_{(\alpha')^{\vee}} T^{\Sol}_{s, 0} - c \cdot T^{\Sol}_{s, 0} \theta_{- (\alpha')^{\vee}} \\ = c^{-1} \cdot (q_{F}^{\lambda^{\Mor}(- \alpha')} - 1)\theta_{(\alpha')^{\vee}} + q_{F}^{(\lambda^{\Mor}(- \alpha') + (\lambda^{*})^{\Mor}(- \alpha'))/2} - q_{F}^{(\lambda^{\Mor}(- \alpha') - (\lambda^{*})^{\Mor}(- \alpha'))/2}
\end{multline*}
Hence, we have
\begin{align}
\label{iotaofbernsteinrelofmorrisepsilon=1}
\theta_{(\alpha')^{\vee}} T^{\Sol}_{s, 0} = c^{2} \cdot T^{\Sol}_{s, 0} \theta_{- (\alpha')^{\vee}} + (q_{F}^{\lambda^{\Mor}(- \alpha')} - 1) \theta_{(\alpha')^{\vee}} + c \left(
q_{F}^{(\lambda^{\Mor}(- \alpha') + (\lambda^{*})^{\Mor}(- \alpha'))/2} - q_{F}^{(\lambda^{\Mor}(- \alpha') - (\lambda^{*})^{\Mor}(- \alpha'))/2}
\right).
\end{align}
On the other hand, the definition of $T^{\Sol}_{s, 0}$ implies that
\begin{align*}
& \quad q_{F}^{\left(\lambda^{\Sol}(\alpha') - (\lambda^{*})^{\Sol}(\alpha')\right)/2} \cdot 
\left(
\theta_{(\alpha')^{\vee}}T^{\Sol}_{s, 0} - T^{\Sol}_{s, 0} \theta_{- (\alpha')^{\vee}}
\right)\\
&=
\theta_{(\alpha')^{\vee}}
\left(
\theta_{(\alpha')^{\vee}} T^{\Sol}_{s} - (q_{F}^{\lambda^{\Sol}(\alpha')} - 1) \theta_{(\alpha')^{\vee}}
\right) - 
\left(
\theta_{(\alpha')^{\vee}} T^{\Sol}_{s} - (q_{F}^{\lambda^{\Sol}(\alpha')} - 1) \theta_{(\alpha')^{\vee}}
\right) \theta_{- (\alpha')^{\vee}}\\
&= \theta_{(\alpha')^{\vee}}\left(
\theta_{(\alpha')^{\vee}} T^{\Sol}_{s} - T^{\Sol}_{s} \theta_{- (\alpha')^{\vee}}
\right) - (q_{F}^{\lambda^{\Sol}(\alpha')} - 1)(\theta_{2 (\alpha')^{\vee}} - 1)\\
&= \theta_{(\alpha')^{\vee}}\left(
(q_{F}^{\lambda^{\Sol}(\alpha')} - 1) \theta_{(\alpha')^{\vee}} + q_{F}^{(\lambda^{\Sol}(\alpha') + (\lambda^{*})^{\Sol}(\alpha'))/2} - q_{F}^{(\lambda^{\Sol}(\alpha') - (\lambda^{*})^{\Sol}(\alpha'))/2}
\right) - (q_{F}^{\lambda^{\Sol}(\alpha')} - 1)(\theta_{2 (\alpha')^{\vee}} - 1)\\
&= 
\left(
q_{F}^{(\lambda^{\Sol}(\alpha') + (\lambda^{*})^{\Sol}(\alpha'))/2} - q_{F}^{(\lambda^{\Sol}(\alpha') - (\lambda^{*})^{\Sol}(\alpha'))/2}
\right) \theta_{(\alpha')^{\vee}} + q_{F}^{\lambda^{\Sol}(\alpha')} - 1.
\end{align*}
Thus, we have
\begin{align}
\label{bernsteinrelofsolleveldepsilon=1}
\theta_{(\alpha')^{\vee}} T^{\Sol}_{s, 0}  = T^{\Sol}_{s, 0} \theta_{- (\alpha')^{\vee}} + (q_{F}^{(\lambda^{*})^{\Sol}(\alpha')} - 1) \theta_{(\alpha')^{\vee}} + q_{F}^{(\lambda^{\Sol}(\alpha') + (\lambda^{*})^{\Sol}(\alpha'))/2} - q_{F}^{(- \lambda^{\Sol}(\alpha') + (\lambda^{*})^{\Sol}(\alpha'))/2}.
\end{align}
According to Theorem~\ref{maintheoremisomofaffinehecke}, we have
\[
\begin{cases}
\lambda^{\Mor}(- \alpha') &= (\lambda^{*})^{\Sol} (\alpha') > 0, \\
(\lambda^{*})^{\Mor}(- \alpha') &= \lambda^{\Sol}(\alpha') > 0.
\end{cases}
\]
Then, comparing the constant terms of the right hand sides of \eqref{iotaofbernsteinrelofmorrisepsilon=1} and \eqref{bernsteinrelofsolleveldepsilon=1}, we have $c=1$.
\end{proof}

We also have the following Corollary from Theorem~\ref{maintheoremisomofaffinehecke}.
\begin{corollary}
\label{comparisonofparameterqpower}
Let $\alpha \in \Delta_{\mathfrak{s}_{M}, \mu}(P)$.
Let $a \in \Gamma(J, \rho)^{+}$ such that $r(a) = - \alpha'$.
Then, the number $\epsilon_{\alpha}$ and the parameters $q_{\alpha}$ and $q_{\alpha*}$ can be calculated as
\[
\begin{cases}
\epsilon_{\alpha} &= 0, \\
q_{\alpha} &= p_{a}^{1/2} \cdot (p'_{a})^{1/2}, \\
q_{\alpha*} &= p_{a}^{1/2} \cdot (p'_{a})^{-1/2}
\end{cases}
\]
if $p_{a} > p'_{a}$, and
\[
\begin{cases}
\epsilon_{\alpha} &= 1, \\
q_{\alpha} &= p_{a}^{1/2} \cdot (p'_{a})^{1/2}, \\
q_{\alpha*} &= p_{a}^{- 1/2} \cdot (p'_{a})^{1/2}
\end{cases}
\]
if $p_{a} < p'_{a}$.
If $p_{a} = p'_{a}$, we have $\epsilon_{\alpha} = 0$, and there are two possibilities:
\[
\begin{cases}
q_{\alpha} &= p_{a}, \\
q_{\alpha*} &= 1
\end{cases}
\]
and
\[
q_{\alpha} = q_{\alpha*} = p_{a}^{1/2}.
\]
\end{corollary}
\begin{remark}
According to \cite[Lemma~3.3]{MR4432237}, $q_{\alpha*} = 1$ unless $\alpha^{\#}$ is the unique simple root in a type $A_{1}$ irreducible component of $\Sigma_{\mathfrak{s}_{M}}$ or a long root in a type $C_{n} \ (n \ge 2)$ irreducible component of $\Sigma_{\mathfrak{s}_{M}}$
If $q_{\alpha*} = 1$, Corollary~\ref{comparisonofparameterqpower} implies that
\[
p_{a} = p'_{a} = q_{\alpha}.
\]
\end{remark}
\begin{proof}[Proof of Corollary~\ref{comparisonofparameterqpower}]
First, we assume that $\epsilon_{\alpha} = 0$.
Then, Theorem~\ref{maintheoremisomofaffinehecke} implies that
\begin{align}
\label{maintheoremisomofaffineheckeparameter}
\begin{cases}
\lambda^{\Mor}(- \alpha') &= \lambda^{\Sol}(\alpha'), \\
(\lambda^{*})^{\Mor}(- \alpha') &= (\lambda^{*})^{\Sol} (\alpha')
\end{cases}
\end{align}
If $q_{\alpha} > q_{\alpha*}$, \eqref{lambdaofsolleveld} and \eqref{lambdaofsolleveldrefinedpositive} imply that
\begin{align*}
q_{\alpha} &= q_{F}^{\left(\lambda^{\Sol}(\alpha') + (\lambda^{*})^{\Sol}(\alpha') \right)/2}\\
&= q_{F}^{\left(\lambda^{\Mor}(- \alpha') + (\lambda^{*})^{\Mor}(- \alpha') \right)/2}
\end{align*}
and 
\begin{align*}
q_{\alpha*} &= q_{F}^{\left(\lambda^{\Sol}(\alpha') - (\lambda^{*})^{\Sol}(\alpha') \right)/2}\\
&= q_{F}^{\left(\lambda^{\Mor}(- \alpha') - (\lambda^{*})^{\Mor}(- \alpha') \right)/2}.
\end{align*}
Then, according to \eqref{lambdaofmorris} and \eqref{lambdastarofmorrisrefined}, we have
\begin{align}
\label{epsilon=0q>q*}
\begin{cases}
q_{\alpha} &= p_{a}^{1/2} \cdot (p'_{a})^{1/2}, \\
q_{\alpha*} &= p_{a}^{1/2} \cdot (p'_{a})^{-1/2}.
\end{cases}
\end{align}
If $q_{\alpha} = q_{\alpha*}$, \eqref{lambdaofsolleveld} and \eqref{lambdaofsolleveldrefinedzero} imply that
\[
\lambda^{\Sol}(\alpha') = (\lambda^{*})^{\Sol}(\alpha')
\]
and
\[
q_{\alpha} = q_{\alpha*} = q_{F}^{\lambda^{\Sol}(\alpha')/2} = q_{F}^{\lambda^{\Mor}(- \alpha')/2}.
\]
We note that in this case, we also have
\[
\lambda^{\Mor}(- \alpha') = (\lambda^{*})^{\Mor}(- \alpha')
\]
Then, according to \eqref{lambdaofmorris} and \eqref{lambdastarofmorrisrefined}, we have
\begin{align}
\label{epsilon=0q=q*}
q_{\alpha} = q_{\alpha*} = p_{a}^{1/2} = (p'_{a})^{1/2}.
\end{align}

Next, we consider the case $\epsilon_{\alpha} = 1$.
According to Theorem~\ref{maintheoremisomofaffinehecke}, we have
\[
\begin{cases}
\lambda^{\Sol}(\alpha') &> (\lambda^{*})^{\Sol}(\alpha'), \\
\lambda^{\Mor}(- \alpha') &= (\lambda^{*})^{\Sol} (\alpha'), \\
(\lambda^{*})^{\Mor}(- \alpha') &= \lambda^{\Sol}(\alpha').
\end{cases}
\]
If  $q_{\alpha} = q_{\alpha*}$, \eqref{lambdaofsolleveldrefinedzero} implies that
\[
\lambda^{\Sol}(\alpha') = (\lambda^{*})^{\Sol}(\alpha'),
\]
a contradiction. Hence, we have $q_{\alpha} > q_{\alpha*}$.
Then, \eqref{lambdaofsolleveld} and \eqref{lambdaofsolleveldrefinedpositive} imply that
\begin{align*}
q_{\alpha} &= q_{F}^{\left(\lambda^{\Sol}(\alpha') + (\lambda^{*})^{\Sol}(\alpha') \right)/2}\\
&= q_{F}^{\left((\lambda^{*})^{\Mor}(- \alpha') + \lambda^{\Mor}(- \alpha') \right)/2}
\end{align*}
and 
\begin{align*}
q_{\alpha*} &= q_{F}^{\left(\lambda^{\Sol}(\alpha') - (\lambda^{*})^{\Sol}(\alpha') \right)/2}\\
&= q_{F}^{\left(\lambda^{*})^{\Mor}(- \alpha') - \lambda^{\Mor}(- \alpha')  \right)/2}.
\end{align*}
Hence, \eqref{lambdaofmorris} and \eqref{lambdastarofmorrisrefined} imply that
\begin{align}
\label{epsilon=1}
\begin{cases}
q_{\alpha} &= p_{a}^{1/2} \cdot (p'_{a})^{1/2}, \\
q_{\alpha*} &= p_{a}^{- 1/2} \cdot (p'_{a})^{1/2}.
\end{cases}
\end{align}

Now, we prove the corollary.
There are three possibilities \eqref{epsilon=0q>q*}, \eqref{epsilon=0q=q*}, and \eqref{epsilon=1}.
We note that $p_{a}, p'_{a} >1$ and $q_{\alpha} \ge q_{\alpha*} \ge 1$.
Hence, only \eqref{epsilon=0q>q*} can happen when $p_{a} > p'_{a}$, and only \eqref{epsilon=1} can happen when $p_{a} < p'_{a}$.
Suppose that $p_{a} = p'_{a}$.
If $\epsilon_{\alpha} = 1$, we have
\[
\lambda^{\Mor}(- \alpha') = (\lambda^{*})^{\Sol}(\alpha') < \lambda^{\Sol}(\alpha') = (\lambda^{*})^{\Mor}(- \alpha'),
\]
hence $p_{a} < p'_{a}$, a contradiction.
Thus, we obtain that $\epsilon_{\alpha} = 0$ and there are two possibilities \eqref{epsilon=0q>q*} and \eqref{epsilon=0q=q*}.
\end{proof}
The parameter $p_{a}$ are studied in \cite[Section~8]{MR742472}.
In particular, according to \cite[Theorem~8.6]{MR742472}, the parameter $p_{a}$ is a powers of $q_{F}$ if the center of $\mathbf{M}_{J \cup \{a\}}$ is connected.
In this case, Corollary~\ref{comparisonofparameterqpower} implies that $q_{\alpha}$ and $q_{\alpha*}$ are powers of $q_{F}^{1/2}$ (see \cite[1.(a)]{2020arXiv200608535L} and \cite[Conjecture~A]{2021arXiv210313113S}).
\begin{comment}
\[
\left(
\iota \circ I^{\Sol} \circ T_{\rho_{M}} \circ (I^{\Mor})
\right)\left(T_{s}^{\Mor}\right) = T_{s}^{\Sol}.
\]
Comparing two quadratic relations
\[
\left(T_{s}^{\Mor}\right)^2 = (q_{F}^{\lambda^{\Mor}}(-\alpha') - 1)T_{s}^{\Mor} + q_{F}^{\lambda^{\Mor}}(-\alpha')
\]
and
\[
\left(T_{s}^{\Sol}\right)^2 = (q_{F}^{\lambda^{\Sol}}(\alpha') - 1)T_{s}^{\Sol} + q_{F}^{\lambda^{\Sol}}(\alpha'),
\]
we obtain that
\begin{align*}
\lambda^{\Mor}(- \alpha') = \lambda^{\Sol}(\alpha').
\end{align*}
Moreover, according to Theorem~\ref{maintheoremc=1} and Theorem~\ref{maintheoremrootsystem}, we also have
\[
\left(
\iota \circ I^{\Sol} \circ T_{\rho_{M}} \circ (I^{\Mor})
\right)\left(
\theta_{- \alpha'}
\right) = \theta_{\alpha'}.
\]
Then, comparing
\[
\theta_{- \alpha'}T^{\Mor}_{s} - T^{\Mor}_{s}\theta_{\alpha'} \\
=
(q_{F}^{\lambda^{\Mor}(-\alpha')} -1) \theta_{- \alpha'} + q_{F}^{\lambda^{\Mor}(- \alpha')/2} \left(
q_{F}^{(\lambda^{*})^{\Mor}(-\alpha'/2)} - q_{F}^{- (\lambda^{*})^{\Mor}(-\alpha'/2)}
\right)
\]
with
\[
\theta_{- \alpha'}T^{\Mor}_{s} - T^{\Mor}_{s}\theta_{\alpha'} \\
=
(q_{F}^{\lambda^{\Mor}(-\alpha')} -1) \theta_{- \alpha'} + 
q_{F}^{(\lambda^{\Mor}(- \alpha') + (\lambda^{*})^{\Mor}(-\alpha'))/2} - q_{F}^{(\lambda^{\Mor}(- \alpha') - (\lambda^{*})^{\Mor}(- \alpha'))/2}
\]
\end{comment}
\section{Some lemmas for main theorems}
\label{Some lemmas about elements of}
In this section, we prepare some lemmas that will be used to prove Theorem~\ref{maintheoremrootsystem} and Theorem~\ref{maintheoremisomofaffinehecke} in the following sections.
We use the same notation as Section~\ref{Statements of main results}.
Let $a \in \Gamma(J, \rho)$, and we fix a lift $s$ of $v[a, J]$ in $N_{G}(S)(F)$.
We note that the definition of the Levi subgroup $M$ implies that $W_{J \cup \{a\}} \not \subset W_{M(F)}$.
Hence, we obtain $s \not \in M(F)$.
\begin{lemma}
\label{snormalizesKM}
The element $s$ normalizes $M$, $K_{M}$, and $\rho_{M}$.
\end{lemma}
\begin{proof}
Since $v[a, J] \in R(J, \rho) \subset W(J, \rho)$, the element $s$ intertwines $\rho$ and fixes $J$.
Hence, it also fixes the Levi subgroup $M$ and the subset
\[
\left\{
x \in \mathcal{A}(G, S) \mid a(x) = 0 \ (a \in J)
\right\}
\]
of $\mathcal{A}(G, S)$.
The definition of $K_{M}$ implies that for any element $x$ of the set above, 
we have
\[
K_{M} = G(F)_{x, 0} \cap M(F),
\]
where $G(F)_{x, 0}$ denotes the parahoric subgroup of $G(F)$ associated with $x$ \cite[3.1, 3.2]{MR1371680}.
Hence, we obtain that $s$ normalizes $K_{M}$.
Since $s$ intertwines $\rho$, we obtain the claim.
\end{proof}
\begin{lemma}
\label{snotinpk}
The element $s$ is not contained in $P(F) \cdot K$.
\end{lemma}
\begin{proof}
Suppose that $s \in P(F) \cdot K$.
We write
\[
s= muk \ (m \in M(F), u \in U(F), k \in K).
\]
Let $s' = m^{-1}s$.
Then, for any $m_{1} \in M(F)$, we have
\[
(s')^{-1} m_{1} s' = k^{-1}u^{-1}m_{1}uk = k^{-1} u^{-1} (m_{1} u m_{1}^{-1})m_{1} k.
\]
We write
\[
k = k_{U} \cdot k_{\overline{U}} \cdot k_{M} \ (k_{U} \in K_{U}, k_{M} \in K_{M}, k_{\overline{U}} \in K_{\overline{U}}),
\]
and let
\[
m' = (s')^{-1} m_{1} s' \in M(F)
\]
and
\[
u' = u^{-1} (m_{1} u m_{1}^{-1}) \in U(F).
\]
Then, we obtain
\[
k_{U} \cdot k_{\overline{U}} \cdot (k_{M} m') = \left(u' (m_{1} k_{U} m_{1}^{-1})\right) \cdot \left(m_{1} k_{\overline{U}} m_{1}^{-1}\right) \cdot m_{1} k_{M}.
\]
Hence,
\[
\begin{cases}
k_{U} &= u' (m_{1} k_{U} m_{1}^{-1}), \\
k_{\overline{U}} &= m_{1} k_{\overline{U}} m_{1}^{-1}, \\
k_{M} m' &= m_{1} k_{M}.
\end{cases}
\]
The last equation implies that $s' k_{M}^{-1}$ commutes with any element $m_{1}$ of $M(F)$.
In particular, $s' k_{M}^{-1}$ commutes with any element of $A_{M}$.
Hence, we obtain
$
s' k_{M}^{-1} \in M(F).
$
Thus, we conclude 
\[
s= m s' = m (s' k_{M}^{-1}) k_{M} \in M(F),
\]
a contradiction.
\end{proof}
We write $P_{s} = s^{-1} Ps$ and $U_{s} = s^{-1} U s$.
%If $M$ is maximal, we have $P_{s} = \overline{P}$ and $U_{s} = \overline{U}$.
\begin{lemma}
\label{lemmaminkm}
Let $m \in M(F)$ such that
\[
U(F)ms \cap KsK \neq \emptyset.
\]
Then, we have $m \in K_{M}$.
Moreover, let $m \in K_{M}$, $u \in U(F)$, and $k, k' \in K$ such that
\[
ums = k^{-1} s k'.
\]
We write
\[
k = k_{\overline{U}} \cdot k_{M} \cdot k_{U} \ (k_{U} \in K_{U}, k_{M} \in K_{M}, k_{\overline{U}} \in K_{\overline{U}})
\]
and
\[
k' = k'_{\overline{U_{s}}} \cdot k'_{M} \cdot k'_{U_{s}} \ (k'_{U_{s}} \in K_{U_{s}}, k'_{M} \in K_{M}, k'_{\overline{U_{s}}} \in K_{\overline{U_{s}}}).
\]
Then, we have
\[
k_{M}m = s k'_{M} s^{-1}.
\]
\end{lemma}
\begin{proof}
Let $u \in U(F)$ and $k, k' \in K$ such that
\[
ums = k^{-1} s k' \in U(F)ms \cap KsK.
\]
Then, we have
\[
kum = s k' s^{-1}.
\]
We note that
\[
s k' s^{-1} = (s k'_{\overline{U_{s}}} s^{-1}) \cdot (s k'_{M} s^{-1}) \cdot (s k'_{U_{s}} s^{-1})
\]
and 
\[
s k'_{\overline{U_{s}}} s^{-1} \in \overline{U}(F), \ s k'_{U_{s}} s^{-1} \in U(F).
\]
Then, we have
\[
k_{\overline{U}} \cdot k_{M}m \cdot (m^{-1}(k_{U}u)m) = (s k'_{\overline{U_{s}}} s^{-1}) \cdot (s k'_{M} s^{-1}) \cdot (s k'_{U_{s}} s^{-1}),
\]
hence
\[
\begin{cases}
k_{\overline{U}}
 &= s k'_{\overline{U_{s}}} s^{-1}, \\
k_{M}m &= s k'_{M} s^{-1}, \\
m^{-1}(k_{U}u)m &= s k'_{U_{s}} s^{-1}.
\end{cases}
\]
The last claim follows from the second equation.
According to Lemma~\ref{snormalizesKM}, $s$ normalizes $K_{M}$.
Hence, we obtain that
\[
m= k_{M}^{-1} \cdot (s k'_{M} s^{-1}) \in K_{M}.
\]
%Here, we regard $x_{J}$ as a point of the reduced Bruhat-Tits building $\mathcal{B}^{\red}(G, F)$ of $G$ over $F$ via an $M(F)$-equivariant embedding
%\[
%\mathcal{B}^{\red}(M, F) \hookrightarrow \mathcal{B}^{\red}(G, F).
%\]
\end{proof}

Now, we suppose that there exists $\alpha \in \Delta_{\mathfrak{s}_{M}, \mu}(P)$ such that $s$ is also a lift in $I_{M_{\alpha}^{1}}(\sigma_{1})$ of the simple reflection 
\[
s_{\alpha} \in N_{G}(M)(F) / M(F)
\]
associated with $\alpha$.
We identify $s_{\alpha}$ and $v[a, J]$ with $s$.
Let $\epsilon = \epsilon_{\alpha} \in \{0, 1\}$ denote the number defined in \cite[Lemma~10.7 (b)]{MR4432237}.
Let $\Phi_{s}$ denote the element of $\End_{G(F)}\left(\ind_{K}^{G(F)} (\rho)\right)$ appearing in Theorem~\ref{theorem7.12ofmorris} and $T'_{s}$ denote the element of $\End_{G(F)}\left(I_{P}^{G}\left(\ind_{M^1}^{M(F)}(\sigma_{1})\right)\right)$ appearing in Lemma~\ref{lemmasolleveld10.8}.
\begin{lemma}
\label{lemmaforcomparisonofmorrisandsolleveldkeypropositionrank1}
Suppose that there exist $b_{0}, b' \in \mathbb{C}[M_{\sigma}/M^1]$ such that
\begin{align}
\label{notsubstitutelemma}
\left(
T_{\rho_{M}} \circ I_{U}
\right)
(\Phi_{s}) =
b_{0} \cdot T'_{s} + b'.
\end{align}
Then, there exists $c' \in \mathbb{C}^{\times}$ such that
\[
b_{0} = c' \cdot (\theta_{h_{\alpha}^{\vee}})^{-\epsilon}.
\]
\end{lemma}
We prove Lemma~\ref{lemmaforcomparisonofmorrisandsolleveldkeypropositionrank1}.
Let $v \in V_{\rho} = V_{\rho_{M}}$.
We define $f^{G}_{v} \in \ind_{K}^{G(F)}(\rho)$ and $f^{M}_{v} \in \ind_{K_{M}}^{M(F)} (\rho_{M})$ as
\[
f^{G}_{v}(g) =
\begin{cases}
\rho(g) \cdot v & (g \in K), \\
0 & (\text{otherwise}),
\end{cases}
\]
and
\[
f^{M}_{v}(m) =
\begin{cases}
\rho_{M}(m) \cdot v & (m \in K_{M}), \\
0 & (\text{otherwise}),
\end{cases}
\]
respectively.
We write
\[
F_{v, U} = I_{U}\left(
f^{G}_{v}
\right) \in I^{G}_{P} \left( \ind_{K_{M}}^{M(F)} (\rho_{M}) \right) 
\]
and
\[
F'_{v, U} = \left(
T_{\rho_{M}} \circ I_{U}
\right)
\left(
f^{G}_{v}
\right) \in I_{P}^{G}\left(\ind_{M^1}^{M(F)}(\sigma_{1})\right).
\]
Substituting $F'_{v, U}$ to equation~\eqref{notsubstitutelemma}, and comparing the values at $s$, we obtain that
\begin{align}
\label{substitutefvands}
\left(
\left(
\left(
T_{\rho_{M}} \circ I_{U}
\right)
(\Phi_{s})
\right)\left(
F'_{v, U}
\right)
\right)(s) = b_{0} \cdot 
\left(
\left(
T'_{s}\left(
F'_{v, U}
\right)
\right)(s)
\right) + b' \cdot \left(F'_{v, U}(s)\right).
\end{align}
Here, we regard $b_{0}, b$ as elements of $\End_{M(F)}\left(\ind_{M^1}^{M(F)} (\sigma_{1})\right)$ via isomorphism~\eqref{isomgroupalgebraheckealgebra}.

First, we calculate the left hand side of equation~\eqref{substitutefvands}.
Let $\phi_{s}$ denote the element of $\mathcal{H}(G(F), \rho)$ that corresponds to $\Phi_{s}$ via isomorphism~\eqref{heckevsend}.
We note that $\phi_{s}$ is supported on $K s K$.
Moreover, the definition of isomorphism~\eqref{heckevsend} implies that
\begin{align}
\label{lemmasupportofphisf}
\left(
\Phi_{s}\left(
f^{G}_{v}
\right)
\right)(x) = \phi_{s}(x) \cdot v
\end{align}
for all $x \in G(F)$.
We write $\phi_{s}(s) \cdot v = v_{s}$, and define $f^{M}_{v_{s}} \in \ind_{K_{M}}^{M(F)} (\rho_{M})$ as
\[
f^{M}_{v_{s}}(m) =
\begin{cases}
\rho_{M}(m) \cdot v_{s} & (m \in K_{M}), \\
0 & (\text{otherwise}).
\end{cases}
\]
Then, the left hand side of equation~\eqref{substitutefvands} is calculated as follows:
\begin{lemma}
\label{lemmalefthandsidetrhoiuphis}
There exists $c_{1} \in \mathbb{C}^{\times}$ such that
\[
\left(
 \left(
  \left(
  T_{\rho_{M}} \circ I_{U}
  \right)
  (\Phi_{s})
 \right)
 (
 F'_{v, U}
 )
\right)(s)
= c_{1} \cdot T_{\rho_{M}}\left(
f^{M}_{v_{s}}
\right).
\]
\end{lemma}
\begin{proof}
Since
\[
\left(
 \left(
  \left(
  T_{\rho_{M}} \circ I_{U}
  \right)
  (\Phi_{s})
 \right)
 (
 F'_{v, U}
 )
\right)(s) = 
T_{\rho_{M}}
\left(
\left(
I_{U}\left(
\Phi_{s}(f^{G}_{v})
\right)
\right)(s)
\right),
\]
it suffices to show that
\[
\left(
\left(
I_{U}
\left(
\Phi_{s}
\right)
\right)(F_{v, U})
\right)
(s) = c_{1} \cdot f^{M}_{v_{s}}.
\]
for some $c_{1} \in \mathbb{C}^{\times}$.

For $m \in M(F)$, we have
\begin{align*}
\left(
\left(
\left(
I_{U}
\left(
\Phi_{s}
\right)
\right)(F_{v, U})
\right)
(s)
\right)(m) &=
\left(
I_{U}
\left(
\Phi_{s}
\left(
f^{G}_{v}
\right)
\right)(s)
\right)(m)\\
&= 
\delta_{P}(m)^{1/2}
\int_{U(F)} 
\left(
\Phi_{s}
\left(
f^{G}_{v}
\right)
\right)
(ums) du.
\end{align*}
According to equation~\eqref{lemmasupportofphisf}, the integrand vanishes unless
\[
U(F)ms \cap KsK \neq \emptyset.
\]
Hence, Lemma~\ref{lemmaminkm} implies that
\[
\left(
\left(
\left(
I_{U}
\left(
\Phi_{s}
\right)
\right)(F_{v, U})
\right)
(s)
\right)(m) = 0
\]
unless $m \in K_{M}$.
Let $m \in K_{M}$ and $u \in U(F)$ such that
\[
ums \in KsK.
\]
We write
\[
ums = k^{-1} s k'
\]
for some $k, k' \in K$ with factorizations
\[
k = k_{\overline{U}} \cdot k_{M} \cdot k_{U} \ (k_{U} \in K_{U}, k_{M} \in K_{M}, k_{\overline{U}} \in K_{\overline{U}})
\]
and
\[
k' = k'_{\overline{U_{s}}} \cdot k'_{M} \cdot k'_{U_{s}} \ (k'_{U_{s}} \in K_{U_{s}}, k'_{M} \in K_{M}, k'_{\overline{U_{s}}} \in K_{\overline{U_{s}}}).
\]
Then, Lemma~\ref{lemmaminkm} also implies that
\[
k_{M}m = s k'_{M} s^{-1}.
\]
Since $K_{U}$, $K_{\overline{U}}$, $K_{U_{s}}$, and $K_{\overline{U_{s}}}$ are contained in the kernel of $\rho$, we have
\begin{align*}
\left(
\Phi_{s}
\left(
f^{G}_{v}
\right)
\right)
(ums) &=
\left(
 \Phi_{s}
\left(
f^{G}_{v}
\right)
\right)
(k^{-1} s k') \\
&= \phi_{s}(k^{-1} s k') \cdot v\\
&= \left(
\rho(k^{-1}) \circ \phi_{s}(s) \circ \rho(k')
\right) \cdot v\\
&= \left(
\rho(k_{M}^{-1}) \circ \phi_{s}(s) \circ \rho(k'_{M})
\right) \cdot v\\
&= \phi_{s}(k_{M}^{-1} s k'_{M}) \cdot v\\
&= \phi_{s}(ms) \cdot v\\
&= \rho_{M}(m) \cdot \left(
\phi_{s}(s) \cdot v
\right)\\
&= \rho_{M}(m) \cdot v_{s}\\
&= f^{M}_{v_{s}}(m).
\end{align*}
Thus, we obtain that
\begin{align*}
\left(
\left(
\left(
I_{U}
\left(
\Phi_{s}
\right)
\right)(F_{v, U})
\right)
(s)
\right)(m)
&= 
\delta_{P}(m)^{1/2}
\int_{U(F)} \Phi_{s}
\left(
f^{G}_{v}
\right)
(ums) du \\
&= \delta_{P}(m)^{1/2} \cdot c(m) \cdot f^{M}_{v_{s}}(m) \\
&= c(m) \cdot f^{M}_{v_{s}}(m),
\end{align*}
where $c(m)$ denotes the volume of the set
\[
U(F) \cap K s K s^{-1} m^{-1}.
\]
According to Lemma~\ref{snormalizesKM}, for $m \in K_{M}$, we have
\[
Ks K s^{-1} m^{-1} = K s K (s^{-1} m^{-1} s) s^{-1} = K s K s^{-1},
\]
hence $c(m)$ does not depend on $m$.
We write $c_{1} = c(m)$.
Then, we obtain that
\[
\left(
\left(
I_{U}
\left(
\Phi_{s}
\right)
\right)(F_{v, U})
\right)
(s) = c_{1} \cdot f^{M}_{v_{s}}.
\]
\end{proof}

Next, we calculate the right hand side of equation~\eqref{substitutefvands}.
The definition of $I_{U}$ implies that for $g \in G(F)$ and $m \in M(F)$, we have
\begin{align*}
 \left(
F_{v, U}(g)
 \right)(m) &=
\delta_{P}(m)^{1/2}
\int_{U(F)}
f^{G}_{v}
(umg) du.
\end{align*}
Since $f^{G}_{v}$ is supported on $K$, the integrand vanishes unless
\[
g \in P(F) \cdot K.
\]
Hence, we have
\begin{align}
\label{supportofFvsandFvs'}
\supp\left(
F_{v, U}
\right), \supp\left(
F'_{v, U}
\right) \subset P(F) \cdot K.
\end{align}
In particular, Lemma~\ref{snotinpk} implies that 
\begin{align}
\label{vanishingoff'vus}
F'_{v, U}(s) = 0.
\end{align}
\begin{comment}
For $m \in M(F)$, we have
\begin{align*}
  \left(
 \left(
I_{U}
\left(
f^{G}_{v}
\right)
 \right)(s)
  \right)(m) &=
\delta_{P}(m)^{1/2}
\int_{U(F)} 
f^{G}_{v}
(ums) du.
\end{align*}
Since $f^{G}_{v}$ is supported on $K$, Lemma~\ref{snotinpk} implies that the integrand vanishes for all $m \in M(F)$ and $u \in U(F)$.
Thus, 
\[
\left(
I_{U}
\left(
f^{G}_{v}
\right)
 \right)(s) = 0,
\]
and
\begin{align*}
F'_{v, U}(s) &= 
   \left(
\left(
T_{\rho_{M}} \circ I_{U}
\right)
\left(
f^{G}_{v}
\right)
 \right)(s)\\
&=
  T_{\rho_{M}}
  \left(
 \left(
I_{U}
\left(
f^{G}_{v}
\right)
 \right)(s)
  \right)\\ &= 0.
\end{align*}
\end{comment}
Thus, the second term of equation~\eqref{substitutefvands} vanishes.
We calculate the first term of equation~\eqref{substitutefvands}.
\begin{lemma}
\label{lemmarighthandsidet'f'}
There exists $c_{2} \in \mathbb{C}^{\times}$ such that
\[
\left(
T'_{s}\left(
F'_{v, U}
\right)
 \right)(s) = c_{2} \cdot (\theta_{h_{\alpha}^{\vee}})^{\epsilon} \cdot T_{\rho_{M}}\left(
f^{M}_{v_{s}}
\right).
\]
\end{lemma}
\begin{proof}
Recall that $T'_{s}$ is defined as
\[
T'_{s} = \frac{
(q_{\alpha}-1)(q_{\alpha*}+1)
}{
2
}
(\theta_{h_{\alpha}^{\vee}})^{\epsilon} \circ J_{s}
+ f_{\alpha},
\]
for some $f_{\alpha} \in \mathbb{C}(M_{\sigma}/M^1)$.
According to \eqref{vanishingoff'vus}, we have
\begin{align}
\label{relationoftandj}
 \left(
T'_{s}\left(
F'_{v, U}
\right)
 \right)(s) = \frac{
(q_{\alpha}-1)(q_{\alpha*}+1)
}{
2
}
(\theta_{h_{\alpha}^{\vee}})^{\epsilon} \cdot 
  \left(
 \left( 
J_{s}
\left(
F'_{v, U}
\right)
 \right)(s)
   \right).
\end{align}
We also recall that $J_{s}$ is defined as the composition
\[
J_{s} = I_{P}^{G} \left(
\rho_{\sigma, s} \otimes \id
\right) \circ I_{P}^{G} (\tau_{s}) \circ \lambda(s) \circ J_{s^{-1}(P) \mid P}(\sigma \otimes \cdot).
\]
%We note that $sPs^{-1} = \overline{P}$.
According to \cite[(4.3)]{MR4432237}, we have
\begin{align}
\label{lambdajoff'}
 \left(
\left(
\lambda(s) \circ J_{s^{-1}(P) \mid P}(\sigma \otimes \cdot)
\right)
\left(
F'_{v, U}
\right)
 \right)(s) &=
\int_{\left(U(F) \cap U_{s}(F)\right) \backslash U_{s}(F)} F'_{v, U}(u') du'\\
&= \int_{\overline{U}(F) \cap U_{s}(F)} F'_{v, U}(u') du'.\notag
\end{align}
%We note that
%\[
%sPs^{-1} = s^{-1}Ps = \overline{P}.
%\]
%Here, we fix a Haar measure $du'$ on $\overline{U}(F)$.
According to equation~\eqref{supportofFvsandFvs'}, the integrand vanishes unless
\begin{align*}
u' &\in P(F) \cdot K \cap \overline{U}(F) \cap U_{s}(F)\\
&= P(F) \cdot K_{\overline{U}} \cap \overline{U}(F) \cap U_{s}(F)\\ 
&= K_{\overline{U}} \cap U_{s}(F).
\end{align*}
We calculate $F'_{v, U}(u')$ for $u' \in K_{\overline{U}} \cap U_{s}(F)$.
Let $u' \in K_{\overline{U}} \cap U_{s}(F)$ and $m \in M(F)$.
Then, we have
\[
 \left(
F_{v, U}\left(u'\right)
 \right)(m) =
\delta_{P}(m)^{1/2}
\int_{U(F)}
f^{G}_{v}
\left(umu'\right) du.
\]
The integrand vanishes unless
\[
umu' \in K = K_{U} \cdot K_{M} \cdot K_{\overline{U}},
\]
that is equivalent to $u \in K_{U}$ and $m \in K_{M}$.
Then, the definition of $f^{G}_{v}$ implies that
\[
 \left(
F_{v, U}\left(u'\right)
 \right)(m) =
\begin{cases}
\rho_{M}(m) \cdot v & (m \in K_{M}), \\
0 & (\text{otherwise}).
\end{cases}
\]
Thus, we obtain that $F_{v, U}\left(u'\right) = f^{M}_{v}$ and 
$
F'_{v, U}\left(u'\right)
= T_{\rho_{M}} \left(
f^{M}_{v}
\right)
$
for any $u' \in K_{\overline{U}} \cap U_{s}(F)$.
Then, equation~\eqref{lambdajoff'} implies that
\[
 \left(
\left(
\lambda(s) \circ J_{s^{-1}(P) \mid P}(\sigma \otimes \cdot)
\right)
\left(
F'_{v, U}
\right)
 \right)(s) =
c_{3} \cdot T_{\rho_{M}} \left(
f^{M}_{v}
\right),
\]
where $c_{3}$ denotes the volume of $K_{\overline{U}} \cap U_{s}(F)$.
Now, we have
\begin{align*}
 \left(
J_{s}\left(
F'_{v, U}
\right)
 \right)(s) &= 
 \left(
\left(
I_{P}^{G} \left(
\rho_{\sigma, s} \otimes \id
\right) \circ I_{P}^{G} (\tau_{s}) \circ \lambda(s) \circ J_{s^{-1}(P) \mid P}(\sigma \otimes \cdot)
\right)\left(
F'_{v, U}
\right)
 \right)(s) \\
&= 
 \left(
(\rho_{\sigma, s} \otimes \id) \circ \tau_{s}
 \right)
  \left(
 \left(
\left(
\lambda(s) \circ J_{s^{-1}(P) \mid P}(\sigma \otimes \cdot)
\right)
\left(
F'_{v, U}
\right)
 \right)(s)
  \right) \\
&= c_{3} \left(
(\rho_{\sigma, s} \otimes \id) \circ \tau_{s}
 \right)
\left(
T_{\rho_{M}} \left(
f^{M}_{v}
\right)
\right).
\end{align*}
To calculate this, we have to recall the way to regard $T_{\rho_{M}} \left(
f^{M}_{v}
\right)$ as an element of $\sigma \otimes \mathbb{C}[M(F)/M^1]$.
The definition of $T_{\rho_{M}}$ implies that $T_{\rho_{M}}\left(f^{M}_{v}\right)$ is supported on $M^1$, and satisfies
\[
\left(
T_{\rho_{M}}\left(f^{M}_{v}\right)
\right)(1) = f^{M}_{v,1},
\]
where $f^{M}_{v, 1}$ is the element of 
\[
\sigma_{1} = \ind_{K_{M}}^{M^1} (\rho_{M})
\]
defined as
\[
f^{M}_{v, 1}(m) =
\begin{cases}
\rho_{M}(m) \cdot v & (m \in K_{M}), \\
0 & (\text{otherwise}).
\end{cases}
\]
Recall that we are regarding $\sigma_{1}$ as an irreducible $M^{1}$-subrepresentation of 
\[
\sigma = \ind_{\widetilde{K_{M}}}^{M(F)} (\widetilde{\rho_{M}}),
\]
and the element $f^{M}_{v, 1} \in \sigma_{1}$ is identified with the element $\widetilde{f^{M}_{v}} \in \sigma$
defined as
\[
\widetilde{f^{M}_{v}}(m) =
\begin{cases}
\widetilde{\rho_{M}}(m) \cdot v & (m \in \widetilde{K_{M}}), \\
0 & (\text{otherwise}).
\end{cases}
\]
Then, we may regard $T_{\rho_{M}}\left(f^{M}_{v}\right)$ as the element of
$
\ind_{M^1}^{M(F)} (\sigma)
$
supported on $M^1$ and satisfies
\[
\left(
T_{\rho_{M}}\left(f^{M}_{v}\right)
\right)(1) = \widetilde{f^{M}_{v}}.
\]
Moreover, to define $J_{s}$, we identified $\ind_{M^1}^{M(F)} (\sigma)$ with $\sigma \otimes \mathbb{C}[M(F)/M^1]$ via isomorphism~\eqref{groupalgebraisomind}.
According to \cite[(2.3)]{MR4432237}, $T_{\rho_{M}}\left(f^{M}_{v}\right)$ is identified with the element 
\[
\widetilde{f^{M}_{v}} \otimes \theta_{1} \in \sigma \otimes \mathbb{C}[M(F)/M^1].
\]
Thus, we obtain that 
\[
\left(
(\rho_{\sigma, s} \otimes \id) \circ \tau_{s}
 \right)
\left(
T_{\rho_{M}} \left(
f^{M}_{v}
\right)
\right) = \rho_{\sigma, s}\left(
\widetilde{f^{M}_{v}}
\right) \otimes \theta_{1}.
\]
Recall that $\rho_{\sigma, s}$ is an element of the one-dimensional vector space
\[
\Hom_{M(F)}(^s\!\sigma, \sigma).
\]
Since $s$ normalizes $\rho_{M}$ and the multiplicity of $\rho_{M}$ in $\sigma\restriction_{K_{M}}$ is equal to $1$, the restriction of $\rho_{\sigma, s}$ to $^s\!\rho_{M}$ is contained in the one-dimensional space
\[
\Hom_{K_{M}}(^s\!\rho_{M}, \rho_{M}).
\]
Here, we identify $^s\!\rho_{M}$ and $\rho_{M}$ as $K_{M}$-subrepresentations of $^s\!\sigma$ and $\sigma$ via the map
\[
v' \mapsto \widetilde{f^{M}_{v'}}
\]
for $v' \in V_{\rho_{M}} = V_{^s\!\rho_{M}}$, respectively.
Hence, there exists $c_{4} \in \mathbb{C}^{\times}$ such that 
\[
\rho_{\sigma, s}\restriction_{^s\!\rho_{M}} = c_{4} \cdot \phi_{s}(s).
\]
In particular, we obtain that
\[
\rho_{\sigma, s}\left(
\widetilde{f^{M}_{v}}
\right) = c_{4} \cdot \widetilde{f^{M}_{v_{s}}}.
\]
Then, our way of identification~\eqref{groupalgebraisomind} implies that the element
\[
\rho_{\sigma, s}\left(
\widetilde{f^{M}_{v}}
\right) \otimes \theta_{1} = c_{4} \cdot \left(
\widetilde{f^{M}_{v_{s}}} \otimes \theta_{1}
\right) \in \sigma \otimes \mathbb{C}[M(F)/M^1]
\]
is identified with the element
\[
c_{4} \cdot T_{\rho_{M}}\left(
f^{M}_{v_{s}}
\right) \in \ind_{M^1}^{M(F)} (\sigma_{1}) \subset \ind_{M^1}^{M(F)} \left(\sigma\right).
\]
Thus, we obtain that
\begin{align*}
 \left(
J_{s}\left(
F'_{v, U}
\right)
 \right)(s) &= c_{3} \left(
(\rho_{\sigma, s} \otimes \id) \circ \tau_{s}
 \right)
\left(
T_{\rho_{M}} \left(
f^{M}_{v}
\right)
\right)\\
&= c_{3} \cdot \left(
\rho_{\sigma, s}\left(
\widetilde{f^{M}_{v}}
\right) \otimes \theta_{1}
\right)\\
&= c_{3} c_{4} \cdot T_{\rho_{M}}\left(
f^{M}_{v_{s}}
\right).
\end{align*}
Substituting it to equation~\eqref{relationoftandj}, we obtain
\begin{align*}
\left(
T'_{s}\left(
F'_{v, U}
\right)
 \right)(s) = c_{2} \cdot (\theta_{h_{\alpha}^{\vee}})^{\epsilon} \cdot T_{\rho_{M}}\left(
f^{M}_{v_{s}}
\right),
\end{align*}
where
\[
c_{2} = c_{3} c_{4} \cdot \frac{
(q_{\alpha}-1)(q_{\alpha*}+1)
}{
2
}.
\]
\end{proof}
Substituting equation~\eqref{vanishingoff'vus}, Lemma~\ref{lemmalefthandsidetrhoiuphis}, and Lemma~\ref{lemmarighthandsidet'f'} to equation~\eqref{substitutefvands}, we obtain
\[
c_{1} \cdot T_{\rho_{M}}\left(
f^{M}_{v_{s}}
\right) = b_{0} \cdot c_{2} \cdot (\theta_{h_{\alpha}^{\vee}})^{\epsilon} \cdot T_{\rho_{M}}\left(
f^{M}_{v_{s}}
\right).
\]
We note that $c_{1}$ and $c_{2}$ are independent of $v \in V_{\rho_{M}}$.
Since 
\[
\left\{
T_{\rho_{M}}\left(
f^{M}_{v_{s}}
\right) \mid v \in V_{\rho_{M}}
\right\}
\]
generates $\ind_{M^1}^{M(F)} (\sigma_{1})$ as an $M(F)$-representation, we have
\[
c_{1} = b_{0} \cdot c_{2} \cdot (\theta_{h_{\alpha}^{\vee}})^{\epsilon} \in \mathbb{C}[M_{\sigma}/M^1] \simeq \End_{M(F)}\left(
\ind_{M^1}^{M(F)} (\sigma_{1})
\right),
\]
hence 
\[
b_{0} = c' \cdot (\theta_{h_{\alpha}^{\vee}})^{-\epsilon}
\]
for
\[
c' = c_{1} \cdot c_{2}^{-1}.
\]

\section{Comparison of Morris and Solleveld's endomorphism algebras: maximal case}
\label{Comparison of Morris and Solleveld's endomorphism algebras : maximal case}
In this section, we prove Theorem~\ref{maintheoremrootsystem} and Theorem~\ref{maintheoremisomofaffinehecke} when $M$ is a maximal proper Levi subgroup of $G$.
Suppose that $M$ is a maximal proper Levi subgroup of $G$, that is, we suppose $\abs{B \backslash J} = 2$.
\begin{proposition}
\label{comparisonaffineweylandfiniteweyl}
The group $R(J, \rho)$ is trivial if and only if the group $W(\Sigma_{\mathfrak{s}_{M}, \mu})$ is trivial.
\end{proposition}
First, we prove the following:
\begin{lemma}
\label{rtrivimpliessubalgebra}
Suppose that $R(J, \rho)$ is trivial.
Then, the subspace
\[
\mathcal{H}(G(F), \rho)_{M} = \left\{
\phi \in \mathcal{H}(G(F), \rho) \mid \supp(\phi) \subset K \cdot M(F) \cdot K
\right\}
\]
is a subalgebra of $\mathcal{H}(G(F), \rho)$.
\end{lemma}
\begin{proof}
Since $R(J, \rho)$ is trivial,
\[
W(J, \rho) = C(J, \rho).
\]
Then, Theorem~\ref{theorem7.12ofmorris} implies that $\End_{G(F)}\left(\ind_{K}^{G(F)} (\rho)\right)$ is isomorphic to the twisted group algebra $\mathbb{C}[W(J, \rho), \chi]$.
We identify $\End_{G(F)}\left(\ind_{K}^{G(F)} (\rho)\right)$ with $\mathcal{H}(G(F), \rho)$ via isomorphism~\eqref{heckevsend}, and for $w \in W(J, \rho)$, let $\phi_{w}$ denote the element of $\mathcal{H}(G(F), \rho)$ corresponding to $\Phi_{w}$ appearing in Theorem~\ref{theorem7.12ofmorris}.
Hence, $\phi_{w}$ is supported on $K \dot{w} K$.
Let $w_{1}, w_{2} \in W(J, \rho)$ such that
\[
\dot{w_{i}} \in K \cdot M(F) \cdot K \ (i=1,2).
\]
It suffices to show that 
\[
\supp(\phi_{w_1} * \phi_{w_2}) \subset K \cdot M(F) \cdot K.
\]
The Iwahori decomposition for $M(F)$ implies that
\[
K \cdot M(F) \cdot K = K \cdot (N_{G}(S)(F) \cap M(F)) \cdot K.
\]
Hence, there exists 
\[
w_{i}^{M} \in W_{M(F)} \, (i=1,2)
\]
such that 
\[
\dot{w_{i}} \in K \cdot \dot{(w_{i}^{M})} \cdot K.
\]
Recall that $W_{J}$ denotes the subgroup of $W$ generated by $s_{b} \ (b \in J)$. 
The definition of $M$ implies that $W_{J}$ is contained in $W_{M(F)}$.
According to \cite[3.11]{MR1235019} (see also \cite[3.22]{MR1235019}), the canonical inclusion 
\[
N_{G}(S)(F) \rightarrow G(F)
\]
induces a bijection
\[
W_{J} \backslash W / W_{J} \rightarrow K \backslash G(F) / K .
\]
Hence, we obtain
\begin{align*}
w_{i} & \in W_{J} \cdot w_{i}^{M} \cdot W_{J} \\
& \subset 
W_{J} \cdot W_{M(F)}
 \cdot W_{J} \\
& = W_{M(F)}.
\end{align*}
Thus, 
\begin{align*}
\supp(\phi_{w_1} * \phi_{w_2}) &= \supp (\chi(w_1, w_2) \phi_{w_1 w_2})\\
&\subset K \cdot \dot{w_1} \dot{w_2} \cdot K \\
&\subset K \cdot M(F) \cdot K.
\end{align*}
\end{proof}
\begin{corollary}
\label{rtrivimpliessupportpreserving}
Suppose that $R(J, \rho)$ is trivial.
Let $\phi$ be an element of $\mathcal{H}(M(F), \rho_{M})$ whose support is contained in $K_{M} z K_{M}$ for some $z \in I_{M(F)} (\rho_{M})$.
Then, we obtain
\[
t_{P}(\phi) = \frac{
\abs{K_{M} / \left(
K_{M} \cap zK_{M}z^{-1}
\right)}^{1/2}
}{
\abs{K / \left(
K \cap zKz^{-1}
\right)}^{1/2}
}T(\phi).
\]
In particular, $t_{P}$ does not depend on the choice of $P$.
\end{corollary}
\begin{proof}
It follows from Lemma~\ref{tPwhenmissubalgebra} and Lemma~\ref{rtrivimpliessubalgebra}.
\end{proof}
\begin{proof}[Proof of Proposition~\ref{comparisonaffineweylandfiniteweyl}]
Suppose that $R(J, \rho)$ is trivial and $W(\Sigma_{\mathfrak{s}_{M}, \mu})$ is non-trivial.
We identify $\End_{G(F)}\left(\ind_{K}^{G(F)} (\rho)\right)$ with $\mathcal{H}(G(F), \rho)$ via isomorphism~\eqref{heckevsend} and use the same notation as in the proof of Lemma~\ref{rtrivimpliessubalgebra}.

Since $M$ is a maximal proper Levi subgroup of $G$, the order of $W(G, M, \mathfrak{s}_{M})$ is at most $2$.
Hence, the assumption $W(\Sigma_{\mathfrak{s}_{M}, \mu})$ is non-trivial implies that the order of $W(\Sigma_{\mathfrak{s}_{M}, \mu})$ is $2$, and $R(\mathfrak{s}_{M})$ is trivial.
We write 
\[
W(\Sigma_{\mathfrak{s}_{M}, \mu}) = \{
1, s
\}.
\]
Hence, $s$ is the reflection associated with the unique root $\alpha$ in $\Delta_{\mathfrak{s}_{M}, \mu}(P)$.
We write 
\[
\alpha' = (\alpha^{\#})' \in \Delta^{\Sol}
\]
and
\[
(\alpha')^{\vee} = (h_{\alpha}^{\vee})' \in (\Delta^{\Sol})^{\vee}.
\]
According to Theorem~\ref{modificationoftheorem10.9ofsolleveld}, there exists an isomorphism
\[
I^{\Sol} \colon \End_{G(F)}\left(I_{P}^{G}\left(\ind_{M^1}^{M(F)}(\sigma_{1})\right)\right) = \mathcal{H}\left(W(\Sigma_{\mathfrak{s}_{M}, \mu})\right) \rightarrow \mathcal{H}^{\Sol}.
\]
Now we have the following commutative diagram:
\[
\xymatrix@R+1pc@C+1pc{
\mathcal{H}(M(F), \rho_{M}) \ar[d]_-{t_{P}} \ar[r]^-{\eqref{compositionofMverofheckevsendandisomgroupalgebraheckealgebra}} \ar@{}[dr]|\circlearrowleft & \mathbb{C}[M_{\sigma}/M^{1}] \ar[d] \\
\mathcal{H}(G(F), \rho) \ar[r]^-{I^{\Sol} \circ T_{\rho_{M}}\circ I_{U}} & \mathcal{H}^{\Sol}.
}
\]
We identify $\mathcal{H}(M(F), \rho_{M})$ with its image via $t_{P}$.
Then, Corollary~\ref{rtrivimpliessupportpreserving} implies that
\[
\mathcal{H}(M(F), \rho_{M}) = \bigoplus_{w \in W(J, \rho) \cap W_{M(F)}} \mathbb{C} \cdot \phi_{w}.
\]
We note that $\mathcal{H}(M(F), \rho_{M})$ is commutative since it is isomorphic to $\mathbb{C}[M_{\sigma}/M^{1}]$. 
Since $\mathcal{H}^{\Sol}$ is free of rank $2$ as a $\mathbb{C}[M_{\sigma}/M^{1}]$-module in this case, $\mathcal{H}(G(F), \rho)$ is also free of rank $2$ as an $\mathcal{H}(M(F), \rho_{M})$-module.
In particular, 
\[
\mathcal{H}(M(F), \rho_{M}) \subsetneq \mathcal{H}(G(F), \rho).
\]
Take an element 
\[
\widetilde{s} \in W(J, \rho) \backslash W_{M(F)}.
\] 
Then, the description of $\mathcal{H}(G(F), \rho)$ in Theorem~\ref{theorem7.12ofmorris} implies that
\[
\mathcal{H}(M(F), \rho_{M}) * \phi_{\widetilde{s}} = \phi_{\widetilde{s}} * \mathcal{H}(M(F), \rho_{M}),
\]
and
\begin{align}
\label{directdecompositionofhecke}
\mathcal{H}(G(F), \rho) = \mathcal{H}(M(F), \rho_{M}) \oplus \mathcal{H}(M(F), \rho_{M}) * \phi_{\widetilde{s}}.
\end{align}
We write 
\[
\left(
I^{\Sol} \circ T_{\rho_{M}} \circ I_{U}
\right)^{-1}
\left(
T^{\Sol}_{s}
\right)
= \phi^{M}_{0} + \phi^{M}_{1} * \phi_{\widetilde{s}}
\]
for some
\[
\phi^{M}_{0}, \phi^{M}_{1} \in \mathcal{H}(M(F), \rho_{M}).
\]
We also write 
\[
\left(
I^{\Sol} \circ T_{\rho_{M}} \circ I_{U}
\right)^{-1}
(\theta_{(\alpha')^{\vee}}) 
= \phi_{+}
\in \mathcal{H}(M(F), \rho_{M})
\]
and
\[
\left(
I^{\Sol} \circ T_{\rho_{M}} \circ I_{U}
\right)^{-1}
(\theta_{-(\alpha')^{\vee}}) 
= \phi_{-}
\in \mathcal{H}(M(F), \rho_{M}),
\]
respectively.
Relation~\eqref{bernsteinrel} of Definition~\ref{affinehecke} for $\mathcal{H}^{\Sol}$ implies
\begin{multline}
\label{bernsteincomparisonweylgroup}
\theta_{(\alpha')^{\vee}}T^{\Sol}_{s} - T^{\Sol}_{s}\theta_{-(\alpha')^{\vee}} \\
=
\left(
(q_{F}^{\lambda^{\Sol}(\alpha')} -1) + \theta_{-(\alpha')^{\vee}}(
q_{F}^{(\lambda^{\Sol}(\alpha') + (\lambda^{*})^{\Sol}(\alpha'))/2} - q_{F}^{(\lambda^{\Sol}(\alpha') - (\lambda^{*})^{\Sol}(\alpha'))/2}
)
\right)
\frac{
\theta_{(\alpha')^{\vee}} - \theta_{-(\alpha')^{\vee}}
}{
\theta_{0} - \theta_{-2(\alpha')^{\vee}}
}.
\end{multline}
The left hand side maps to 
\[
\phi_{+} * (\phi^{M}_{0} + \phi^{M}_{1} * \phi_{\widetilde{s}}) - (\phi^{M}_{0} + \phi^{M}_{1} * \phi_{\widetilde{s}}) * \phi_{-} 
=
\phi^{M}_{0} * (\phi_{+} - \phi_{-}) + \phi^{M}_{1} * (\phi_{+}*\phi_{\widetilde{s}} - \phi_{\widetilde{s}} * \phi_{-})
\]
via $\left(
I^{\Sol} \circ T_{\rho_{M}} \circ I_{U}
\right)^{-1}$.
On the other hand, the right hand side of \eqref{bernsteincomparisonweylgroup} maps to
\[
\left(
(q_{F}^{\lambda^{\Sol}(\alpha')} -1) + \phi_{-}(
q_{F}^{(\lambda^{\Sol}(\alpha') + (\lambda^{*})^{\Sol}(\alpha'))/2} - q_{F}^{(\lambda^{\Sol}(\alpha') - (\lambda^{*})^{\Sol}(\alpha'))/2}
)
\right)
\frac{
\phi_{+} - \phi_{-}
}{
1 - (\phi_{-})^{2}
} \in \mathcal{H}(M(F), \rho_{M})
\]
via $\left(
I^{\Sol} \circ T_{\rho_{M}} \circ I_{U}
\right)^{-1}$.
Comparing the $\mathcal{H}(M(F), \rho_{M})$-factor of the decomposition~\eqref{directdecompositionofhecke}, we obtain
\begin{multline*}
\phi^{M}_{0} * (\phi_{+} - \phi_{-}) \\
=
\left(
(q_{F}^{\lambda^{\Sol}(\alpha')} -1) + \phi_{-}(
q_{F}^{(\lambda^{\Sol}(\alpha') + (\lambda^{*})^{\Sol}(\alpha'))/2} - q_{F}^{(\lambda^{\Sol}(\alpha') - (\lambda^{*})^{\Sol}(\alpha'))/2}
)
\right)
\frac{
\phi_{+} - \phi_{-}
}{
1 - (\phi_{-})^{2}
},
\end{multline*}
hence
\begin{align}
\label{regeqpole}
\phi^{M}_{0}
 =
\left(
(q_{F}^{\lambda^{\Sol}(\alpha')} -1) + \phi_{-}(
q_{F}^{(\lambda^{\Sol}(\alpha') + (\lambda^{*})^{\Sol}(\alpha'))/2} - q_{F}^{(\lambda^{\Sol}(\alpha') - (\lambda^{*})^{\Sol}(\alpha'))/2}
)
\right)
\frac{
1
}{
1 - (\phi_{-})^{2}
}
\end{align}
However, since $\lambda^{\Sol}(\alpha') > 0$ and $(\lambda^{*})^{\Sol}(\alpha') > 0$, regarding 
\[
\mathcal{H}(M(F), \rho_{M}) \simeq \mathbb{C}[M_{\sigma}/M^{1}]
\]
as the ring of regular functions on an algebraic torus over $\mathbb{C}$, the right hand side of \eqref{regeqpole} has a pole at $\phi_{-} = 1$,
hence it is not contained in $\mathcal{H}(M(F), \rho_{M})$, a contradiction.

Conversely, assume that $R(J, \rho)$ is non-trivial and $W(\Sigma_{\mathfrak{s}_{M}, \mu})$ is trivial.
According to Corollary~\ref{corollaryvectorpart}, the image of $\mathbb{C}[\mathbb{Z} (R^{\Mor})^{\vee}]$ via the composition 
\[
T_{\rho_{M}} \circ I_{U} \circ (I^{\Mor})^{-1} \colon \mathcal{H}^{\Mor} \rightarrow \mathcal{H}(R(J, \rho)) \subset \End_{G(F)}\left(\ind_{K}^{G(F)} (\rho)\right)
\rightarrow \End_{G(F)}\left(I_{P}^{G}\left(\ind_{M^1}^{M(F)}(\sigma_{1})\right)\right)
\]
is contained in
\[
\mathbb{C}[M_{\sigma}/M^{1}] \subset \End_{G(F)}\left(I_{P}^{G}\left(\ind_{M^1}^{M(F)}(\sigma_{1})\right)\right).
\]
Hence, we obtain a commutative diagram
\[
\xymatrix{
\mathbb{C}[\mathbb{Z} (R^{\Mor})^{\vee}] \ar[d] \ar[r] \ar@{}[dr]|\circlearrowleft & \mathbb{C}[M_{\sigma}/M^{1}] \ar[d] \\
\mathcal{H}^{\Mor} \ar[r] & \End_{G(F)}\left(I_{P}^{G}\left(\ind_{M^1}^{M(F)}(\sigma_{1})\right)\right).
}
\]
Since $R(J, \rho)$ is non-trivial, $\mathcal{H}^{\Mor}$ is non-commutative.
Hence, $\End_{G(F)}\left(I_{P}^{G}\left(\ind_{M^1}^{M(F)}(\sigma_{1})\right)\right)$ is also non-commutative.
In particular, we have
\[
\mathbb{C}[M_{\sigma}/M^{1}] \subsetneq \End_{G(F)}\left(I_{P}^{G}\left(\ind_{M^1}^{M(F)}(\sigma_{1})\right)\right).
\]
Thus, $R(\mathfrak{s}_{M})$ is non-trivial.
We write
\[
R(\mathfrak{s}_{M}) = \{1, r\}.
\]
According to Theorem~\ref{modificationoftheorem10.9ofsolleveld},  
\begin{align*}
\End_{G(F)}\left(I_{P}^{G}\left(\ind_{M^1}^{M(F)}(\sigma_{1})\right)\right) = \mathbb{C}[M_{\sigma}/M^1] \oplus \mathbb{C}[M_{\sigma}/M^1] J_{r}.
\end{align*}
Moreover, according to \cite[Theorem~10.6 (a)]{MR4432237}, we obtain
\[
\mathbb{C}[M_{\sigma}/M^1] J_{r}
= J_{r} \mathbb{C}[M_{\sigma}/M^1].
\]
Then, replacing $\mathcal{H}^{\Sol}$ with $\mathcal{H}^{\Mor}$ and $\mathcal{H}(G(F), \rho)$ with $\End_{G(F)}\left(I_{P}^{G}\left(\ind_{M^1}^{M(F)}(\sigma_{1})\right)\right)$ in the proof of the case $R(J, \rho)$ is trivial and $W(\Sigma_{\mathfrak{s}_{M}, \mu})$ is non-trivial, we obtain a contradiction.
\begin{comment}
Since $R(J, \rho)$ is non-trivial, the root system $R^{\Mor}$ is non-trivial, and the affine Hecke algebra $\mathcal{H}^{\Mor}$ has an element $T^{\Mor}_{}$
In this case, we have
\[
I^{\Mor} \colon \mathcal{H}(R(J, \rho)) \rightarrow \mathcal{H}(W_{\aff}\left(\Gamma'(J, \rho)\right), q)
\]
According to~\cite[Theorem~1.8]{MR4310011}, 
According to \cite[Theorem~7.2 (ii)]{MR1643417}, $\mathcal{H}(R(\rho))$ is an affine Hecke algebra.
On the other hand, according to \cite[Theorem~10.9]{MR4432237} (see also \cite[Theorem~10.6 (a)]{MR4432237}), 
\begin{align*}
\End_{G(F)}\left(I_{P}^{G}\left(\ind_{M^1}^{M(F)}(\sigma_{1})\right)\right) &\simeq \bigoplus_{r \in R(\mathfrak{s}_{M})} \mathbb{C}[M_{\sigma}/M^1] J_{r}\\
&= \bigoplus_{r \in R(\mathfrak{s}_{M})} J_{r} \mathbb{C}[M_{\sigma}/M^1].
\end{align*}
\end{comment}
\end{proof}
If 
\[
R(J, \rho) = W(\Sigma_{\mathfrak{s}_{M}, \mu}) = \{1\},
\]
Theorem~\ref{maintheoremrootsystem} and Theorem~\ref{maintheoremisomofaffinehecke} are trivial.
Suppose that both of $R(J, \rho)$ and $W(\Sigma_{\mathfrak{s}_{M}, \mu})$ are non-trivial.
Let $\alpha$ denote the unique element of $\Delta_{\mathfrak{s}_{M}, \mu}(P)$, and let $s_{\alpha}$ denote the simple reflection associated with $\alpha$.
Then, we have 
\[
N_{G}(M)(F) / M(F) = W(G, M, \mathfrak{s}_{M}) = W(\Sigma_{\mathfrak{s}_{M}, \mu}) = \{
1, s_{\alpha}
\}.
\]
We also write $a$ for the unique element of $\Gamma(J, \rho)^{+}$ such that 
\[
B(J, \rho)_{e} = \{
a + A'_{J}
\}.
\]
Then, 
\[
D_{J}(a + A'_{J}) = Da\restriction_{A_{M}}
\]
is a scalar multiple of $\alpha$.
%Hence, the element $s_{r(a)}$ coincides with $s_{\alpha}$ as reflections on $a_{M}$.
We fix a lift $s$ of $v[a, J]$ in $N_{G}(S)(F)$.
According to Lemma~\ref{snormalizesKM}, $s$ normalizes $M$.
Since $s \not \in M(F)$, the projection of $s$ on $N_{G}(M)(F)/M(F)$ is the element $s_{\alpha}$.
Hence, the element $s$ is also a lift of $s_{\alpha}$.
We identify $s_{\alpha}$ and $v[a, J]$ with $s$.
Let $\epsilon = \epsilon_{\alpha} \in \{0, 1\}$ denote the number defined in \cite[Lemma~10.7 (b)]{MR4432237}.
Let $\Phi_{s}$ denote the element of $\End_{G(F)}\left(\ind_{K}^{G(F)} (\rho)\right)$ appearing in Theorem~\ref{theorem7.12ofmorris} and $T'_{s}$ denote the element of $\End_{G(F)}\left(I_{P}^{G}\left(\ind_{M^1}^{M(F)}(\sigma_{1})\right)\right)$ appearing in Lemma~\ref{lemmasolleveld10.8}.
According to Theorem~\ref{modificationoftheorem10.9ofsolleveld},
we can write
\begin{align}
\label{notsubstitute}
\left(
T_{\rho_{M}} \circ I_{U}
\right)
(\Phi_{s}) =
b_{0} \cdot T'_{s} + b'
\end{align}
for some $b_{0}, b' \in \mathbb{C}[M_{\sigma}/M^1]$.
Then, Lemma~\ref{lemmaforcomparisonofmorrisandsolleveldkeypropositionrank1} implies the following:
\begin{proposition}
\label{comparisonofmorrisandsolleveldkeypropositionrank1}
There exists $c' \in \mathbb{C}^{\times}$ and $b' \in \mathbb{C}[M_{\sigma}/M^1]$ such that
\[
\left(
T_{\rho_{M}} \circ I_{U}
\right)
(\Phi_{s}) =
c' \cdot (\theta_{h_{\alpha}^{\vee}})^{-\epsilon} \cdot T'_{s} + b'.
\]
\end{proposition}
%\begin{remark}
%In the proof of Proposition~\ref{comparisonofmorrisandsolleveldkeypropositionrank1}, we will use the assumption that $M$ is a proper maximal Levi subgroup of $G$ only to write $\left(
%T_{\rho_{M}} \circ I_{U}
%\right)
%(\Phi_{s})$ as \eqref{notsubstitute}.
%Thus, once we obtain equation~\eqref{notsubstitute} in general case, we can apply Proposition~\ref{comparisonofmorrisandsolleveldkeypropositionrank1} to the case.
%\end{remark}
Now, we prove Theorem~\ref{maintheoremrootsystem} and Theorem~\ref{maintheoremisomofaffinehecke} when $M$ is maximal and both of $R(J, \rho)$ and $W(\Sigma_{\mathfrak{s}_{M}, \mu})$ are non-trivial.
We consider the map
\[
\iota \circ I^{\Sol} \circ T_{\rho_{M}} \circ I_{U} \circ (I^{\Mor})^{-1} \colon \mathcal{H}^{\Mor} \rightarrow \mathcal{H}^{\Sol}.
\]
According to Corollary~\ref{corollaryvectorpart}, for $t \in T(J, \rho)$, there exists $c(t) \in \mathbb{C}^{\times}$ such that
\[
\left(
\iota \circ I^{\Sol} \circ T_{\rho_{M}} \circ I_{U} \circ (I^{\Mor})^{-1} 
\right)
\left(
\theta_{v(t)}
\right) = c(t)^{-1} \cdot \theta_{t}.
\]
In particular, we take $t = t_{0}$ such that 
\[
v(t_{0}) = k_{a + A'_{J}} \left(D_{J}(a + A'_{J})\right)^{\vee}.
\]
We identify $v(t_{0})$ with $\widetilde{v(t_{0})}$ via isomorphism~\eqref{orthogonalcomplementmorris}.
%We write $c= c(t_{0})^{-1}$.
According to Lemma~\ref{vvshm}, we have 
\[
v(t_{0}) = -H_{M}(t_{0}).
\]
Since $a + A'_{J} \in B(J, \rho)_{e} \subset \Gamma'(J, \rho)^{+}_{e}$, the element $\alpha$ is contained in $\Delta_{\mathfrak{s}_{M}, \mu}(P) \subset \Sigma(P, A_{M})$, and the parabolic subgroup $P$ satisfies
\[
D_{J}\left(
\Gamma'(J, \rho)^{+}_{e} 
\right) = 
D_{J} \left(
\Gamma'(J, \rho)_{e}
\right) \cap \left(-\Sigma(P, A_{M})\right),
\]
we have
\begin{align*}
H_{M}(t_{0}) 
& = - v(t_{0}) \\
&= - k_{a + A'_{J}} \left(D_{J}(a + A'_{J})\right)^{\vee} \\
& \in \mathbb{R}_{>0} \cdot \alpha^{\vee}.
\end{align*}
Since $h_{\alpha}^{\vee}$ is the unique generator of $(M_{\sigma} \cap G^{1} )/M^1$ such that $H_{M}(h_{\alpha}^{\vee}) \in \mathbb{R}_{>0} \cdot \alpha^{\vee}$, there exists a positive integer $m$ such that $t_{0} = m \cdot h_{\alpha}^{\vee}$.
Then, there exists $n \in (1/2) \cdot \mathbb{Z}$ such that $t_{0} = n \cdot (h_{\alpha}^{\vee})'$.
We note that Theorem~\ref{maintheoremrootsystem} is equivalent to the claim $n=1$.

First, we assume that $\epsilon = 0$. Then, according to Corollary~\ref{corollaryvectorpart} and Proposition~\ref{comparisonofmorrisandsolleveldkeypropositionrank1} together with Corollary~\ref{corollaryoftheorem7.12ofmorrisaffinever} and Theorem~\ref{modificationoftheorem10.9ofsolleveld}, we obtain that the map
\[
\iota \circ I^{\Sol} \circ T_{\rho_{M}} \circ I_{U} \circ (I^{\Mor})^{-1} \colon \mathcal{H}^{\Mor} \rightarrow \mathcal{H}^{\Sol}
\]
satisfies the conditions of Corollary~\ref{generalizationkeypropositionappendixinteger} for $k=0$.
Therefore, Theorem~\ref{maintheoremrootsystem} and Theorem~\ref{maintheoremisomofaffinehecke} follow from Corollary~\ref{generalizationkeypropositionappendixinteger} in this case.

Next, we assume that $\epsilon =1$.
We note that
\[
\begin{cases}
(h_{\alpha}^{\vee})' = h_{\alpha}^{\vee} & (q_{\alpha} > q_{\alpha*}), \\
(h_{\alpha}^{\vee})' = 2h_{\alpha}^{\vee} & (q_{\alpha} = q_{\alpha*}).
\end{cases}
\]
Then, according to Corollary~\ref{corollaryvectorpart} and Proposition~\ref{comparisonofmorrisandsolleveldkeypropositionrank1} together with Corollary~\ref{corollaryoftheorem7.12ofmorrisaffinever} and Theorem~\ref{modificationoftheorem10.9ofsolleveld}, we obtain that the map
\[
\iota \circ I^{\Sol} \circ T_{\rho_{M}} \circ I_{U} \circ (I^{\Mor})^{-1} \colon \mathcal{H}^{\Mor} \rightarrow \mathcal{H}^{\Sol}
\]
satisfies the condition of Corollary~\ref{generalizationkeypropositionappendixinteger} for $k=1$ or $k= 1/2$.
According to Corollary~\ref{generalizationkeypropositionappendixinteger}, we have $k= 1$, hence $(h_{\alpha}^{\vee})' = h_{\alpha}^{\vee}$ and $q_{\alpha} > q_{\alpha*}$.
Then, \eqref{lambdaofsolleveld} and \eqref{lambdaofsolleveldrefinedpositive} imply that
\begin{align*}
q_{\alpha*} = q_{F}^{\left(\lambda^{\Sol}\left((h_{\alpha}^{\vee})'\right) - (\lambda^{*})^{\Sol}\left((h_{\alpha}^{\vee})'\right) \right)/2}.
\end{align*}
Since $\epsilon = 0$ unless $q_{\alpha*} > 1$, we have 
\[
\lambda^{\Sol}\left((h_{\alpha}^{\vee})'\right) > (\lambda^{*})^{\Sol}\left((h_{\alpha}^{\vee})'\right).
\]
Now, the rest claims of Theorem~\ref{maintheoremrootsystem} and Theorem~\ref{maintheoremisomofaffinehecke} follow from Corollary~\ref{generalizationkeypropositionappendixinteger} too. 
\section{Comparison of Morris and Solleveld's endomorphism algebras: general case}
\label{Comparison of Morris and Solleveld's endomorphism algebras: general case}
In this section, we prove  Theorem~\ref{maintheoremrootsystem} and Theorem~\ref{maintheoremisomofaffinehecke} for general case.
First, we prove Theorem~\ref{maintheoremrootsystem}.
Recall that $R^{\Mor}$ and $R^{\Sol}$ are root systems in $a_{M}^{*}$ defined as
\[
R^{\Mor}
= \{
D_{J}(a')/k_{a'} \mid a' \in \Gamma'(J, \rho)_{e}
\}
\]
and
\[
R^{\Sol} = \{
(\alpha^{\#})' \mid \alpha \in \Delta_{\mathfrak{s}_{M}, \mu}(P)
\}.
\]
We also recall that for $\alpha \in \Sigma_{\red}(A_{M})$, $M_{\alpha}$ denotes the Levi subgroup of $G$ that contains $M$ and the root subgroup $U_{\alpha}$ associated with $\alpha$, and whose semisimple rank is one greater than that of $M$.
We write $K_{\alpha} = K \cap M_{\alpha}(F)$ and $\rho_{\alpha} = \rho\restriction_{K_{\alpha}}$.
Let $\Phi^{M_{\alpha}}$ denote the set of relative roots with respect to $S$ in $M_{\alpha}$, and let $\Phi_{\aff}^{M_{\alpha}}$ denote the affine root system associated with $(M_{\alpha}, S)$ by the work of \cite{MR327923}.
%Hence, 
%\[
%\Phi_{\aff}^{M_{\alpha}} = \{
%a \in \Phi_{\aff} \mid Da \in \Phi^{M_{\alpha}}
%\}.
%\]
%We also define $\Phi^{M}$ and $\Phi_{\aff}^{M}$ similarly.
%Then, the definition of $M$ implies that
%\[
%\Phi^{M} = \Phi \cap \mathbb{R} \cdot (DJ),
%\]
%$where $\mathbb{R} \cdot (DJ)$ denotes the $\mathbb{R}$-span of $DJ$.
%Hence, we obtain that
%\[
%\Phi_{\aff}^{M} = \{
%a \in \Phi_{\aff} \mid Da \in \mathbb{R} \cdot (DJ)
%\},
%\]
%that is written as $(\Phi_{\aff})_{J}$ in Appendix~\ref{Subsets of a set of simple affine roots}.
According to Corollary~\ref{corollarySubsets of a set of simple affine roots}, we can take a basis $B^{M_{\alpha}}$ of $\Phi_{\aff}^{M_{\alpha}}$ containing $J$,
and we can define $W^{M_{\alpha}}(J, \rho_{\alpha})$ and $\Gamma^{M_{\alpha}}(J, \rho_{\alpha})$ by replacing $G$ with $M_{\alpha}$ and $\rho$ with $\rho_{\alpha}$ in the definition of $W(J, \rho)$ and $\Gamma(J, \rho)$, respectively (see the proof of Lemma~\ref{independencyofP}).
Since $J \subset \Phi_{\aff}^{M_{\alpha}}$, we have $\mathcal{M}_{J} \subset M_{\alpha}(F)$.
Hence, the definition of $W(J, \rho)$ implies that
\[
W^{M_{\alpha}}(J, \rho_{\alpha}) = W(J, \rho) \cap W_{M_{\alpha}(F)}.
\]
We also have the following:
\begin{lemma}
\label{Gammareduction}
We have
\[
\Gamma^{M_{\alpha}}(J, \rho_{\alpha}) = \Gamma(J, \rho) \cap \Phi_{\aff}^{M_{\alpha}}.
\]
\end{lemma}
\begin{proof}
Let $a \in \Phi_{\aff}^{M_{\alpha}}$ such that $Da\restriction_{A_{M}}$ is non-trivial.
Then, $Da\restriction_{A_{M}}$
is a scalar multiple of $\alpha$, and the definition of $M_{\alpha}$ implies that
\[
\Phi^{M_{\alpha}} = \Phi \cap \mathbb{R} \cdot D(J \cup \{a\}),
\]
hence
\[
\Phi_{\aff}^{M_{\alpha}} = \{
b \in \Phi_{\aff} \mid
Db \in \mathbb{R} \cdot D(J \cup \{a\})
\},
\]
where $\mathbb{R} \cdot D(J \cup \{a\})$ denotes the $\mathbb{R}$-span of $D(J \cup \{a\})$.
According to Lemma~\ref{lemmabasisofsimplerootsoflevi} and Lemma~\ref{converselemmabasisofsimplerootsoflevi}, there exists a basis $B'$ of $\Phi_{\aff}$ containing $J \cup \{a\}$ if and only if there exists a basis $B^{M_{\alpha}}$ of $\Phi_{\aff}^{M_{\alpha}}$ containing $J \cup \{a\}$.
We assume that $a$ satisfies these conditions.
%Since $\abs{B \backslash J} > 1$, and $\Phi_{\aff}$ is irreducible, $D(J \cup \{a\})$ is linearly independent.
Then, we can define the element 
\[
v[a, J] \in W_{M_{\alpha}(F)}.
\]
%We note that the conditions of $\Gamma(J, \rho)$ can be checked in $M_{\alpha}$.
Since 
\[
W^{M_{\alpha}}(J, \rho_{\alpha}) = W(J, \rho) \cap W_{M_{\alpha}(F)},
\]
$v[a, J] \in W(J, \rho)$ if and only if $v[a, J] \in W^{M_{\alpha}}(J, \rho_{\alpha})$.
Moreover, replacing $J$ with $J \cup \{a\}$ in \cite[3.15]{MR1235019}, we obtain a parahoric subgroup $\mathcal{M}_{J \cup \{a\}}$ of a reductive subgroup of $M_{\alpha}$ with radical $\mathcal{U}_{J \cup \{a\}}$ such that the canonical inclusion 
\[
\mathcal{M}_{J \cup \{a\}} \rightarrow P_{J \cup \{a\}, B'}
\]
induces an isomorphism
\[
\mathcal{M}_{J \cup \{a\}}/\mathcal{U}_{J \cup \{a\}} \rightarrow P_{J \cup \{a\}, B'}/U_{J \cup \{a\}, B'} \simeq \mathbf{M}_{J \cup \{a\}}(k_{F}).
\]
Hence, we can calculate $p_{a}$ in $\mathcal{M}_{J \cup \{a\}} \subset M_{\alpha}(F)$.
Thus, the definition of $\Gamma(J, \rho)$ implies that
\[
\Gamma^{M_{\alpha}}(J, \rho_{\alpha}) = \Gamma(J, \rho) \cap \Phi_{\aff}^{M_{\alpha}}.
\]
\begin{comment}
Since $a \in \Gamma^{M_{\alpha}}(J, \rho_{\alpha})$, we obtain $p_{a}>1$, hence $a$ is also contained in $\Gamma(J, \rho)$.
Let $a \in \Gamma^{M_{\alpha}}(J, \rho_{\alpha})$. 
The definition of $\Phi_{\aff}^{M_{\alpha}}$ implies that $Da\restriction_{A_{M}}$
is a scalar multiple of $\alpha$.
Then, the definition of $M_{\alpha}$ implies that

Conversely, let $a \in \Gamma(J, \rho) \cap \Phi_{\aff}^{M_{\alpha}}$.
Since $a \in \Phi_{\aff}^{M_{\alpha}}$, we have
\[
\Phi_{\aff}^{M_{\alpha}} = \{
b \in \Phi_{\aff} \mid
Db \in \mathbb{R} \cdot D(J \cup \{a\})
\}
\]
in this case too.
Since $a \in \Gamma(J, \rho)$, there exists a basis of $\Phi_{\aff}$ containing $J \cup \{a\}$.
Replacing $J$ with $J \cup \{a\}$ in Lemma~\ref{lemmabasisofsimplerootsoflevi}, we obtain that there exists a basis of $\Phi_{\aff}^{M_{\alpha}}$ containing $J \cup \{a\}$.
Then, since the conditions of $\Gamma(J, \rho)$ can be checked in $M_{\alpha}$ as above, we obtain that $a \in \Gamma^{M_{\alpha}}(J, \rho)$.
\end{comment}
\end{proof}

Let $\alpha \in \Sigma_{\mathfrak{s}_{M}, \mu}$.
We will prove that $(\alpha^{\#})' \in R^{\Mor}$.
%We define $\Gamma^{M_{\alpha}}(J, \rho_{\alpha})$ as above.
Let $(\Gamma^{M_{\alpha}})'(J, \rho_{\alpha})$ denote the image of $\Gamma^{M_{\alpha}}(J, \rho_{\alpha})$ on $A'/A'_{J}$.
We define
\[
\begin{cases}
V^{\Gamma^{M_{\alpha}}} &= \{
y \in V \mid \alpha(y) = 0 \ (\alpha \in D\Gamma^{M_{\alpha}}(J, \rho))
\}, \\
V^{J, \Gamma^{M_{\alpha}}} &= V^{J} \cap V^ {\Gamma^{M_{\alpha}}}, \\
\mathcal{A}^{J}_{\Gamma^{M_{\alpha}}} &= \mathcal{A}^{J}/ V^{J, \Gamma^{M_{\alpha}}}.
\end{cases}
\]
Replacing $G$ with $M_{\alpha}$ in Proposition~\ref{proposition7.3ofmorris}, we obtain that $(\Gamma^{M_{\alpha}})'(J, \rho_{\alpha})$ is an affine root system on $\mathcal{A}^{J}_{\Gamma^{M_{\alpha}}} $.
Let $[e]$ denote the image of $e$ on $\mathcal{A}^{J}_{\Gamma^{M_{\alpha}}} $ via the natural projection
\[
\mathcal{A}^{J}_{\Gamma} \rightarrow \mathcal{A}^{J}_{\Gamma^{M_{\alpha}}}.
\]
Since $e$ is a special point for $\Gamma'(J, \rho)$, $[e]$ is a special point for $(\Gamma^{M_{\alpha}})'(J, \rho_{\alpha})$.
Let $(\Gamma^{M_{\alpha}})'(J, \rho_{\alpha})_{[e]}$ denote the set of affine roots in $(\Gamma^{M_{\alpha}})'(J, \rho_{\alpha})$ that vanish at $[e]$.
Then, we have
\[
(\Gamma^{M_{\alpha}})'(J, \rho_{\alpha})_{[e]} = \Gamma'(J, \rho)_{e} \cap (\Gamma^{M_{\alpha}})'(J, \rho_{\alpha}).
\]
Since $M$ is a maximal Levi subgroup of $M_{\alpha}$, Theorem~\ref{maintheoremrootsystem} holds if we replace $G$ with $M_{\alpha}$.
We define 
\[
(R^{\Mor})^{M_{\alpha}} = \{
D_{J}(a')/k_{a'} \mid a' \in (\Gamma')^{M_{\alpha}}(J, \rho)_{[e]}
\}
\]
and
\[
(R^{\Sol})^{M_{\alpha}} = \{
(\alpha^{\#})' \mid \alpha \in \Sigma_{\mathfrak{s}_{M}, \mu} \cap \Sigma(M_{\alpha}, A_{M})
\},
\]
where $\Sigma(M_{\alpha}, A_{M})$ denotes the set of nonzero weights occurring in the adjoint representation of $A_{M}$ on the Lie algebra of $M_{\alpha}$. 
According to Theorem~\ref{maintheoremrootsystem} for $M_{\alpha}$, we have
\[
(\alpha^{\#})' \in (R^{\Sol})^{M_{\alpha}} = (R^{\Mor})^{M_{\alpha}} \subset R^{\Mor}.
\]

On the other hand, let $a \in \Gamma(J, \rho)$ such that $a' = a+A'_{J} \in \Gamma'(J, \rho)_{e}$.
We will prove that 
\[
D_{J}(a')/k_{a'} \in R^{\Sol}.
\]
We write 
\[
\alpha = 
\begin{cases}
D_{J}(a') & (D_{J}(a') \in \Sigma_{\red}(A_{M})), \\
D_{J}(a')/2 & (D_{J}(a') \not \in \Sigma_{\red}(A_{M})).
\end{cases}
\]
%We write $M_{\alpha}$ for the Levi subgroup of $G$ that contains $M$ and the root subgroup $U_{\alpha}$ associated with $\alpha$, and whose semisimple rank is one greater than that of $M$.
We use the same notation as above.
%We also define $\Phi^{M_{\alpha}}$ and $\Phi_{\aff}^{M_{\alpha}}$ as above.
Then, we have 
\[
a \in \Gamma(J, \rho) \cap \Phi_{\aff}^{M_{\alpha}}.
\]
According to Lemma~\ref{Gammareduction}, we have $a \in \Gamma^{M_{\alpha}}(J, \rho_{\alpha})$, hence
\[
a' \in (\Gamma')^{M_{\alpha}}(J, \rho)_{[e]}.
\]
Then, according to Theorem~\ref{maintheoremrootsystem} for $M_{\alpha}$, we have
\[
D_{J}(a')/k_{a'} \in (R^{\Mor})^{M_{\alpha}} = (R^{\Sol})^{M_{\alpha}} \subset R^{\Sol}.
\]
Thus, we obtain that
$R^{\Mor} = R^{\Sol}$.
Let $(R^{\Mor})^{+}$ denote the set of positive roots of $R^{\Mor}$ with respect to the basis $\Delta^{\Mor}$ and $(R^{\Sol})^{+}$ denote the set of positive roots of $R^{\Sol}$ with respect to the basis $\Delta^{\Sol}$.
Hence, we have
\[
(R^{\Mor})^{+} = \{
D_{J}(a')/k_{a'} \mid a' \in \Gamma'(J, \rho)^{+}_{e}
\}
\]
and
\[
(R^{\Sol})^{+} = \{
(\alpha^{\#})' \mid \alpha \in \Sigma_{\mathfrak{s}_{M}, \mu}(P)
\}.
\]
Our choice of the parabolic subgroup $P$ implies that if $a' \in \Gamma'(J, \rho)^{+}_{e}$, we have
\[
D_{J}(a')/k_{a'} \in - (R^{\Sol})^{+}.
\]
Thus, we obtain that $(R^{\Mor})^{+} = - (R^{\Sol})^{+}$, hence $\Delta^{\Mor} = - \Delta^{\Sol}$.

Next, we prove Theorem~\ref{maintheoremisomofaffinehecke}.
Let $\alpha \in \Delta_{\mathfrak{s}_{M}, \mu}(P)$.
We write $\alpha' = (\alpha^{\#})'$ and $(\alpha')^{\vee} = (h_{\alpha}^{\vee})'$, and let $s_{\alpha}$ denote the corresponding reflection.
We also regard $- \alpha' \in \Delta^{\Mor}$ and $s_{\alpha} \in W_{0}\left(R^{\Mor}\right)$.
Let $a \in \Gamma(J, \rho)^{+}$ such that $r(a) = -\alpha'$.
%Then, the element $v[a, J]$ maps to $s_{\alpha} \in W_{0}(R^{\Mor})$ via the isomorphism
%\[
%R(J, \rho) \rightarrow W_{\aff}\left(\Gamma'(J, \rho)\right)
%\] 
%of Proposition~\ref{proposition7.3ofmorris}.
We fix a lift $s$ of $v[a, J]$ in $N_{G}(S)(F)$, that is also a lift of $s_{\alpha}$ in $I_{M_{\alpha}^{1}}(\sigma_{1})$.
We identify $s_{\alpha}$ and $v[a, J]$ with $s$.
First, we assume that $a \in B$ and $M_{\alpha}$ is a standard Levi subgroup with respect to $P$.
Then, $P M_{\alpha}$ is a parabolic subgroup of $G$ with Levi factor $M_{\alpha}$.
Let $U^{M_{\alpha}}\index{$U^{M_{\alpha}}$}$ denote the unipotent radical of $P M_{\alpha}$ and $\overline{U^{M_{\alpha}}}\index{$\overline{U^{M_{\alpha}}}$}$ denote the unipotent radical of the opposite parabolic subgroup $\overline{P} M_{\alpha}$ of $P M_{\alpha}$.
Let $P_{J \cup \{a\}} = P_{J \cup \{a\}, B}$ denote the parahoric subgroup of $G(F)$ associated with $J \cup \{a\} \subset B$, and let $U_{J \cup \{a\}}$ denote its radical. 
\begin{lemma}
\label{pJtopJcupa}
We have
\[
K \cap U^{M_{\alpha}}(F) = P_{J \cup \{a\}} \cap U^{M_{\alpha}}(F)
\]
and
\[
K \cap \overline{U^{M_{\alpha}}}(F) = P_{J \cup \{a\}} \cap \overline{U^{M_{\alpha}}}(F).
\]
\end{lemma}
\begin{proof}
Since $r(a) = -\alpha'$, the Levi subgroup $M_{\alpha}$ coincides with the centralizer of the subtorus
\[
\left(
\bigcap_{\beta \in D(J \cup \{a\})} \ker(\beta)
\right)^{\circ}
\]
of $S$. 
Moreover, according to \cite[3.5.1]{MR546588}, the Levi subgroup $M_{\alpha}$ is same as the Levi subgroup attached with the parahoric subgroup $P_{J \cup \{a\}}$ as in \cite[6.3]{MR1371680}.
Hence, \cite[Proposition~6.4]{MR1371680} implies that the canonical inclusion
\[
P_{J \cup \{a\}} \cap M_{\alpha}(F) \rightarrow P_{J \cup \{a\}}
\]
induces an isomorphism
\[
(P_{J \cup \{a\}} \cap M_{\alpha}(F))/(U_{J \cup \{a\}} \cap M_{\alpha}(F)) \rightarrow P_{J \cup \{a\}}/U_{J \cup \{a\}}.
\]
Thus, we obtain
\begin{align}
\label{moyprasad6.4implies}
P_{J \cup \{a\}} = (P_{J \cup \{a\}} \cap M_{\alpha}(F)) \cdot U_{J \cup \{a\}}.
\end{align}
Moreover, according to \cite[6.4.48]{MR327923}, we have
\[
U_{J \cup \{a\}} = \left(
U_{J \cup \{a\}} \cap U^{M_{\alpha}}(F)
\right) \cdot \left(
U_{J \cup \{a\}} \cap M_{\alpha}(F)
\right) \cdot \left(
U_{J \cup \{a\}} \cap \overline{U^{M_{\alpha}}}(F)
\right).
\]
Combining it with \eqref{moyprasad6.4implies}, we obtain
\[
P_{J \cup \{a\}} = \left(
U_{J \cup \{a\}} \cap U^{M_{\alpha}}(F)
\right) \cdot \left(
P_{J \cup \{a\}} \cap M_{\alpha}(F)
\right) \cdot \left(
U_{J \cup \{a\}} \cap \overline{U^{M_{\alpha}}}(F)
\right).
\]
Thus, we have
\[
P_{J \cup \{a\}} \cap U^{M_{\alpha}}(F) = U_{J \cup \{a\}} \cap U^{M_{\alpha}}(F)
\]
and
\[
P_{J \cup \{a\}} \cap \overline{U^{M_{\alpha}}}(F) = U_{J \cup \{a\}} \cap \overline{U^{M_{\alpha}}}(F).
\]
Since
\[
U_{J \cup \{a\}} \subset U_{J} \subset P_{J} = K \subset P_{J \cup \{a\}},
\]
we obtain the claim.
\end{proof}
\begin{corollary}
\label{corollarysnormalizeskcapmalpha}
The element $s$ normalizes the groups $K \cap U^{M_{\alpha}}(F)$ and $K \cap \overline{U^{M_{\alpha}}}(F)$.
\end{corollary}
\begin{proof}
Since $s$ is a lift of $v[a, J] \in W_{J \cup \{a\}}$, it is contained in $P_{J \cup \{a\}} \cap M_{\alpha}(F)$.
Hence, the claim follows from Lemma~\ref{pJtopJcupa}.
\end{proof}

%According to Lemma~\ref{transitivityofts}, we have
%\begin{align*}
%T'_{s} = I^{G}_{P M_{\alpha}} \left(
%(T'_{s})^{M_{\alpha}}
%\right).
%\end{align*}
%On the other hand, 
According to \cite[(8.7)]{MR1643417}, we have
\[
t_{P} = t_{PM_{\alpha}} \circ t_{P \cap M_{\alpha}},
\]
where
\[
t_{P \cap M_{\alpha}} \colon \mathcal{H}(M(F), \rho_{M}) \rightarrow \mathcal{H}(M_{\alpha}(F), \rho_{\alpha}).
\]
denotes the injection obtained by replacing $G$ with $M_{\alpha}$ in the construction of $t_{P}$, and
\[
t_{PM_{\alpha}} \colon \mathcal{H}(M_{\alpha}(F), \rho_{\alpha}) \rightarrow \mathcal{H}(G(F), \rho).
\]
denotes the injection obtained by replacing $M$ with $M_{\alpha}$ in the construction of $t_{P}$.
\begin{comment}
Next, we prove Theorem~\ref{maintheoremc=1}.
According to Corollary~\ref{corollaryvectorpart}, it suffices to show that the number $c$ appearing Corollary~\ref{corollaryimageofphiM} is equal to $1$.
Let $t \in T(J, \rho)$ such that $\left(D_{J}(a')\right)(v(t)) \ge 0$ for all $a' \in B(J, \rho)_{e}$.
Let $\phi^{M}_{t}$ denote the element of $\mathcal{H}(M(F), \rho_{M})$ corresponding to 
\[
\theta_{t^{-1}} \in \mathbb{C}[I_{M(F)}(\rho_{M})/K_{M}] = \mathbb{C}[M_{\sigma}/M^1]
\]
via isomorphism~\eqref{compositionofMverofheckevsendandisomgroupalgebraheckealgebra}, and let $\Phi^{M}_{t}$ denote the element of $\End_{M(F)}\left( \ind_{K_{M}}^{M(F)} (\rho_{M}) \right)$ corresponding to $\phi^{M}_{t}$ via isomorphism~\eqref{Mverofheckevsend}.
It suffices to show that
\[
(I^{\Mor} \circ t_{P})(\Phi^{M}_{t}) = \theta_{v(t)}.
\]
Moreover, we may assume that $t = - \alpha'$ for some $\alpha \in \Sigma_{\mathfrak{s}_{M}, \mu}(P)$.
Replacing $G$ with $M_{\alpha}$ in the construction of $t_{P}$, we obtain an injection
\[
t_{P \cap M_{\alpha}} \colon \mathcal{H}(M(F), \rho_{M}) \rightarrow \mathcal{H}(M_{\alpha}(F), \rho_{\alpha}).
\]
Replacing $M$ with $M_{\alpha}$ in the construction of $t_{P}$, we also obtain an injection
\[
t_{PM_{\alpha}} \colon \mathcal{H}(M_{\alpha}(F), \rho_{\alpha}) \rightarrow \mathcal{H}(G(F), \rho).
\]
According to \cite[(8.7)]{MR1643417}, we have
\[
t_{P} = t_{PM_{\alpha}} \circ t_{P \cap M_{\alpha}}.
\]
\end{comment}
Moreover, replacing $G$ with $M_{\alpha}$ in the construction of $I_{U}$, we have the isomorphism
\[
I_{U \cap M_{\alpha}} \colon \ind_{K_{\alpha}}^{M_{\alpha}(F)} (\rho_{\alpha}) \rightarrow I_{P \cap M_{\alpha}}^{M_{\alpha}}\left(\ind_{K_M}^{M(F)}(\rho_{M})\right),
\]
and
replacing $M$ with $M_{\alpha}$ in the construction of $I_{U}$, we have the isomorphism
\[
I_{U^{M_{\alpha}}} \colon \ind_{K}^{G(F)}(\rho) \rightarrow I^{G}_{P M_{\alpha}}\left(\ind_{K_{\alpha}}^{M_{\alpha}(F)} (\rho_{\alpha})\right).
\]
According to Proposition~\ref{compatibility} replacing $G$ or $M$ with $M_{\alpha}$, we have the following commutative diagrams:
\begin{align}
\label{commutativeUcapM}
\xymatrix{
\End_{M(F)}\left( \ind_{K_{M}}^{M(F)} (\rho_{M}) \right) \ar[d]_-{t_{P \cap M_{\alpha}}} \ar[r]^-{\id} \ar@{}[dr]|\circlearrowleft & \End_{M(F)}\left( \ind_{K_{M}}^{M(F)} (\rho_{M}) \right) \ar[d]^-{I_{P \cap M_{\alpha}}^{M_{\alpha}}} \\
\End_{M_{\alpha}(F)}\left(\ind_{K_{\alpha}}^{M_{\alpha}(F)} (\rho_{\alpha})\right) \ar[r]^-{I_{U \cap M_{\alpha}}} & \End_{M_{\alpha}(F)}\left(I^{M_{\alpha}}_{P \cap M_{\alpha}} \left( \ind_{K_{M}}^{M(F)} (\rho_{M}) \right)\right),
}
\end{align}
\begin{align}
\label{commutative/UcapM}
\xymatrix{
\End_{M_{\alpha}(F)}\left( \ind_{K_{\alpha}}^{M_{\alpha}(F)} (\rho_{\alpha}) \right) \ar[d]_-{t_{P M_{\alpha}}} \ar[r]^-{\id} \ar@{}[dr]|\circlearrowleft & \End_{M_{\alpha}(F)}\left( \ind_{K_{\alpha}}^{M_{\alpha}(F)} (\rho_{\alpha}) \right) \ar[d]^-{I_{P M_{\alpha}}^{G}} \\
\End_{G(F)}\left(\ind_{K}^{G(F)} (\rho)\right) \ar[r]^-{I_{U^{M_{\alpha}}}} & \End_{G(F)}\left(I^{G}_{P M_{\alpha}} \left( \ind_{K_{\alpha}}^{M_{\alpha}(F)} (\rho_{\alpha}) \right)\right). 
}
\end{align}
Moreover, it follows easily from the definition of $I_{U}$ that the composition
\[
I^{G}_{P M_{\alpha}}\left(
I_{U \cap M_{\alpha}} 
\right)
\circ
I_{U^{M_{\alpha}}} \colon \ind_{K}^{G(F)}(\rho) \rightarrow I^{G}_{P} \left(\ind_{K_M}^{M(F)}(\rho_{M})\right)
\]
coincides with $I_{U}$.
Here, we use the canonical isomorphism
\[
I^{G}_{P} \left(\ind_{K_M}^{M(F)}(\rho_{M})\right) \simeq I_{P M_{\alpha}}^{G} \left(
I_{P \cap M_{\alpha}}^{M_{\alpha}}\left(\ind_{K_M}^{M(F)}(\rho_{M})\right)
\right)
\]
defined as
\[
f \mapsto [
g \mapsto [
m \mapsto \delta_{P M_{\alpha}}(m)^{1/2} \cdot f(mg)
]
]
\]
to identify them.
\begin{comment}
Combining the commutative diagram Proposition~\ref{compatibility}
\[
\xymatrix{
\End_{M_{\alpha}(F)}\left( \ind_{K_{\alpha}}^{M_{\alpha}(F)} (\rho_{\alpha}) \right) \ar[d]_-{t_{P M_{\alpha}}} \ar[r]^-{\id} \ar@{}[dr]|\circlearrowleft & \End_{M_{\alpha}(F)}\left( \ind_{K_{\alpha}}^{M(F)} (\rho_{\alpha}) \right) \ar[d]^-{I_{P M_{\alpha}}^{G}} \\
\End_{G(F)}\left(\ind_{K}^{G(F)} (\rho)\right) \ar[r]^-{I_{U/(U \cap M_{\alpha})}} & \End_{G(F)}\left(I^{G}_{P M_{\alpha}} \left( \ind_{K_{\alpha}}^{M_{\alpha}(F)} (\rho_{\alpha}) \right)\right). 
}
\]
obtained by replacing $M$ with $M_{\\alpha}$ in Proposition~\ref{compatibility} and the trivial diagram
\[
\xymatrix{
\End_{M_{\alpha}(F)}\left(\ind_{K_{\alpha}}^{M_{\alpha}(F)} (\rho_{\alpha})\right) \ar[d]_-{I^{G}_{P M_{\alpha}}} \ar[r]^-{I_{U \cap M_{\alpha}}} \ar@{}[dr]|\circlearrowleft & \End_{M_{\alpha}(F)}\left(I_{P \cap M_{\alpha}}^{M_{\alpha}}\left(\ind_{M^1}^{M(F)}(\sigma_{1})\right)\right) \ar[d]^{I^{G}_{P M_{\alpha}}}
\\
\End_{G(F)}\left(
I^{G}_{P M_{\alpha}} \left(
\ind_{K_{\alpha}}^{M_{\alpha}(F)} (\rho_{\alpha})
\right)
\right) \ar[r]^-{I^{G}_{P M_{\alpha}}\left(
I_{U \cap M_{\alpha}} 
\right)}& \End_{G(F)}\left(
I^{G}_{P}
\left(\ind_{K_M}^{M(F)}(\rho_{M})\right)
\right)
},
\]
\end{comment}
Hence, combining \eqref{commutativeUcapM} and \eqref{commutative/UcapM}, with the trivial diagram
\[
\xymatrix@R+1pc@C+1pc{
\End_{M_{\alpha}(F)}\left(\ind_{K_{\alpha}}^{M_{\alpha}(F)} (\rho_{\alpha})\right) \ar[d]_-{I^{G}_{P M_{\alpha}}} \ar[r]^-{I_{U \cap M_{\alpha}}} \ar@{}[dr]|\circlearrowleft & \End_{M_{\alpha}(F)}\left(I_{P \cap M_{\alpha}}^{M_{\alpha}}\left(\ind_{M^1}^{M(F)}(\sigma_{1})\right)\right) \ar[d]^{I^{G}_{P M_{\alpha}}}
\\
\End_{G(F)}\left(
I^{G}_{P M_{\alpha}} \left(
\ind_{K_{\alpha}}^{M_{\alpha}(F)} (\rho_{\alpha})
\right)
\right) \ar[r]^-{I^{G}_{P M_{\alpha}}\left(
I_{U \cap M_{\alpha}} 
\right)}& \End_{G(F)}\left(
I^{G}_{P}
\left(\ind_{K_M}^{M(F)}(\rho_{M})\right)
\right),
}
\]
we obtain the following commutative diagram:
\begin{align}
\label{transitivitycommutativediagram}
\xymatrix{
\End_{M(F)}\left( \ind_{K_{M}}^{M(F)} (\rho_{M}) \right) \ar[d]_-{t_{P \cap M_{\alpha}}} \ar[r]^-{\id} \ar@{}[dr]|\circlearrowleft & \End_{M(F)}\left( \ind_{K_{M}}^{M(F)} (\rho_{M}) \right) \ar[d]^-{I_{P \cap M_{\alpha}}^{M_{\alpha}}} \ar[r]^-{T_{\rho_{M}}} \ar@{}[dr]|\circlearrowleft&
\End_{M(F)}\left(\ind_{M^1}^{M(F)} (\sigma_{1})\right) \ar[d]^-{I_{P \cap M_{\alpha}}^{M_{\alpha}}}
\\
\End_{M_{\alpha}(F)}\left(\ind_{K_{\alpha}}^{M_{\alpha}(F)} (\rho_{\alpha})\right) \ar[d]_-{t_{PM_{\alpha}}} \ar[r]^-{I_{U \cap M_{\alpha}}} \ar@{}[dr]|\circlearrowleft & \End_{M_{\alpha}(F)}\left(I^{M_{\alpha}}_{P \cap M_{\alpha}} \left( \ind_{K_{M}}^{M(F)} (\rho_{M}) \right)\right) \ar[r]^-{T_{\rho_{M}}} \ar[d]^-{I_{PM_{\alpha}}^{G}} \ar@{}[dr]|\circlearrowleft & \End_{M_{\alpha}(F)}\left(I_{P \cap M_{\alpha}}^{M_{\alpha}}\left(\ind_{M^1}^{M(F)}(\sigma_{1})\right)\right) \ar[d]^-{I_{PM_{\alpha}}^{G}}
\\
\End_{G(F)}\left(\ind_{K}^{G(F)} (\rho)\right) \ar[r]^-{I_{U}} & \End_{G(F)}\left(I^{G}_{P} \left( \ind_{K_{M}}^{M(F)} (\rho_{M}) \right)\right) \ar[r]^-{T_{\rho_{M}}} & \End_{G(F)}\left(I_{P}^{G}\left(\ind_{M^1}^{M(F)}(\sigma_{1})\right)\right). 
}
\end{align}

Let $\Phi_{s}$ denote the element of $\End_{G(F)}\left(\ind_{K}^{G(F)} (\rho)\right)$ appearing in Theorem~\ref{theorem7.12ofmorris}.
We also have the similar description of $\End_{M_{\alpha}(F)}\left(\ind_{K_{\alpha}}^{M_{\alpha}(F)} (\rho_{\alpha})\right)$.
In particular, we have the element $\Phi_{s}^{M_{\alpha}}$ of $\End_{M_{\alpha}(F)}\left(\ind_{K_{\alpha}}^{M_{\alpha}(F)} (\rho_{\alpha})\right)$ such that the element $\phi_{s}^{M_{\alpha}} \in \mathcal{H}(M_{\alpha}(F), \rho_{\alpha})$ corresponding to $\Phi_{s}^{M_{\alpha}}$ via the $M_{\alpha}$-version of \eqref{heckevsend} is supported on $K_{\alpha} s K_{\alpha}$, and satisfies
\[
(\Phi_{s}^{M_{\alpha}})^2 = (p_{a} - 1)\Phi_{s}^{M_{\alpha}} + p_{a}.
\]
%According to Theorem~\ref{maintheoremisomofaffinehecke}, 
\begin{proposition}
\label{reductionmorrisphis}
We have
\[
t_{P M_{\alpha}}\left(
\Phi_{s}^{M_{\alpha}}
\right) = \Phi_{s}.
\]
\end{proposition}
\begin{proof}
According to Corollary~\ref{corollarysnormalizeskcapmalpha}, $s$ normalizes $K \cap U^{M_{\alpha}}(F)$ and $K \cap \overline{U^{M_{\alpha}}}(F)$.
In particular, $s$ is positive relative to $K$ and $U^{M_{\alpha}}$.
Hence, the definition of $t_{P M_{\alpha}}$ implies that there exists $c \in \mathbb{C}^{\times}$ such that
\[
t_{P M_{\alpha}}\left(
\Phi_{s}^{M_{\alpha}}
\right) = c \cdot \Phi_{s}.
\]
Since $\Phi_{s}^{M_{\alpha}}$ and $\Phi_{s}$ satisfy the same quadratic relation
\[
(\Phi_{s}^{M_{\alpha}})^2 = (p_{a} - 1)\Phi_{s}^{M_{\alpha}} + p_{a}
\]
and
\[
\Phi_{s}^{2} = (p_{a} - 1)\Phi_{s} + p_{a},
\]
we have $c=1$.
\end{proof}
\begin{comment}
\begin{proof}
Let $U^{M_{\alpha}}$ denote the unipotent radical of $P M_{\alpha}$.
We prove that $s$ is positive relative to $K$ and $U^{M_{\alpha}}$.
Recall that there exists $y \in \mathcal{F}_{J} \subset \mathcal{A}^{J}$ such that
\[
K = P_{J} = G(F)_{y, 0}
\]
(see the proof of Lemma~\ref{positivitycompativility}.)
\[
s \cdot y = y - a(y)\alpha^{\vee}
\]
\begin{align*}
s(K \cap U^{M_{\alpha}}(F))s^{-1}
&=
s\left(
G(F)_{y, 0} \cap U^{M_{\alpha}}(F)
\right)s^{-1}\\
&= G(F)_{s \cdot y, 0} \cap U^{M_{\alpha}}(F)\\
&= G(F)_{y - a(y) \alpha^{\vee}, 0} \cap U^{M_{\alpha}}(F).
\end{align*}
\[
\langle
\alpha^{\vee}, \beta
\rangle
\]
\end{proof}
\end{comment}
Now, we prove Theorem~\ref{maintheoremisomofaffinehecke} in case that $a \in B$ and $M_{\alpha}$ is a standard Levi subgroup with respect to $P$.
%\begin{remark}
%The number $\epsilon_{\alpha}$ of \cite[Lemma~10.7 (b)]{MR4432237} and the parameter $q_{\alpha}$ and $q_{\alpha*}$ of \cite[(3.7)]{MR4432237} can be calculated in $M_{\alpha}$.
%\end{remark}
Since the label functions can be calculated in $M_{\alpha}$, the latter claim follows from the results of Section~\ref{Comparison of Morris and Solleveld's endomorphism algebras : maximal case}. 
It suffices to show the former claim.
We rewrite it by using Corollary~\ref{corollaryoftheorem7.12ofmorrisaffinever} and Theorem~\ref{modificationoftheorem10.9ofsolleveld}:
\begin{theorem}
\label{maintheoremisomofaffineheckerewrite}
Let $\alpha \in \Delta_{\mathfrak{s}_{M}, \mu}(P)$ such that $M_{\alpha}$ is a standard Levi subgroup with respect to $P$.
We also suppose that the element $a \in \Gamma(J, \rho)^{+}$ such that $r(a) = - \alpha'$ is contained in $B$.
Let
\[
s = s_{\alpha} \in W_{0}(R^{\Mor}) = W_{0}(R^{\Sol})
\]
denote the simple reflection associated with the element $\alpha$.
Then, we have
\begin{align*}
\left(
T_{\rho_{M}} \circ I_{U}
\right)(\Phi_{s}) =
\begin{cases}
q_{F}^{\lambda^{\Sol}(\alpha')} -1 - T'_{s} & (\epsilon_{\alpha} = 0), \\
- q_{F}^{\left(-\lambda^{\Sol}(\alpha') + (\lambda^{*})^{\Sol}(\alpha')\right)/2} \cdot \theta_{- (\alpha')^{\vee}} T'_{s} & (\epsilon_{\alpha} =1).
\end{cases}
\end{align*}
\end{theorem}
\begin{proof}
We have already proved in Section~\ref{Comparison of Morris and Solleveld's endomorphism algebras : maximal case} the $M_{\alpha}$-version of Theorem~\ref{maintheoremisomofaffineheckerewrite}:
\begin{align}
\label{maintheoremisomofaffineheckerewritemalphaver}
\left(
T_{\rho_{M}} \circ I_{U \cap M_{\alpha}}
\right)(\Phi^{M_{\alpha}}_{s}) =
\begin{cases}
q_{F}^{\lambda^{\Sol}(\alpha')} -1 - (T'_{s})^{M_{\alpha}} & (\epsilon_{\alpha} = 0), \\
- q_{F}^{\left(-\lambda^{\Sol}(\alpha') + (\lambda^{*})^{\Sol}(\alpha')\right)/2} \cdot \theta_{- (\alpha')^{\vee}} (T'_{s})^{M_{\alpha}} & (\epsilon_{\alpha} =1).
\end{cases}
\end{align}
Then, according to commutative diagram~\eqref{transitivitycommutativediagram} combining with Lemma~\ref{transitivityofts} and Proposition~\ref{reductionmorrisphis}, we obtain Theorem~\ref{maintheoremisomofaffineheckerewrite} from \eqref{maintheoremisomofaffineheckerewritemalphaver}.
\end{proof}

To drop the conditions that $a \in B$ and $M_{\alpha}$ is a standard Levi subgroup with respect to $P$, we use intertwining operators.
Recall that $\alpha \in \Delta_{\mathfrak{s}_{M}, \mu}(P)$.
%Since $\alpha \in \Delta_{\mathfrak{s}_{M}, \mu}(P)$, we can take a parabolic subgroup $P'$ with Levi factor $M$ such that $M_{\alpha}$ is standard with respect to $P'$, and
%\[
%\Sigma_{\mathfrak{s}_{M}, \mu} \cap \Sigma(P, A_{M}) \cap \left(-\Sigma(P', A_{M})\right) = \emptyset.
%\]
%Equivalently, we have
%\[
%\Sigma_{\mathfrak{s}_{M}, \mu}(P) = \Sigma_{\mathfrak{s}_{M}, \mu}(P').
%\]
\begin{lemma}
\label{reductionfiniterootpositivesimple}
There exists a parabolic subgroup $P'$ with Levi factor $M$ such that $M_{\alpha}$ is a standard Levi subgroup with respect to $P'$, and
\[
\Sigma_{\mathfrak{s}_{M}, \mu}(P) = \Sigma_{\mathfrak{s}_{M}, \mu}(P').
\]
\end{lemma}
\begin{proof}
For a parabolic subgroup $P'$ with Levi factor $M$, let $\Delta(P')$ denote the basis of $\Sigma(G, A_M)$ with respect to $P'$.
Hence, any element of $\Sigma(P', A_{M})$ can be written as a linear combination of elements of $\Delta(P')$ with rational integer coefficients that are all non-negative.
We also note that if $\alpha \in \Delta(P')$, $M_{\alpha}$ is standard with respect to $P'$.
%We fix a semi-standard minimal parabolic subgroup $P_{0}$ of $G$ with respect to $S$ such that $P_{0} \subset P$.
%Then, $P_{0}$ determines a set of positive roots $\Phi^{+}$ and a basis $\Delta_{0}$ of $\Phi$.
%We define
%\[
%\Delta_{0}^{M} = \{
%\beta \in \Delta_{0} \mid \beta\restriction_{A_{M}} = 0
%\}.
%\]
%Then, the set
%\[
%\Delta(P) := \{
%\beta\restriction_{A_{M}} \mid \beta \in \Delta_{0} \backslash \Delta_{0}^{M}
%\}
%\]
%is a basis of $\Sigma(G, A_{M})$.
For $w \in W(G, M, \mathfrak{s}_{M})$, we define
\[
N(w, P') = \left\{
\beta \in \Sigma_{\red}(P', A_{M}) \mid w(\beta) \in - \Sigma_{\red}(P', A_{M})
\right\}.\index{$N(w, P')$}
\]
To prove Lemma~\ref{reductionfiniterootpositivesimple}, it suffices to show the following claim:
\begin{claim}
If $\alpha \not \in \Delta(P)$, there exists a parabolic subgroup $P'$ with Levi factor $M$ such that
\[
\Sigma_{\mathfrak{s}_{M}, \mu}(P) = \Sigma_{\mathfrak{s}_{M}, \mu}(P'),
\]
and
\[
N(s_{\alpha}, P') \subsetneq N(s_{\alpha}, P).
\]
\end{claim}
We prove the claim.
Suppose that $\alpha \not \in \Delta(P)$.
Then, there exists 
\[
\alpha \neq \beta \in \Delta(P) \cap N(s_{\alpha}, P).
\]
Since $\alpha \in \Delta_{\mathfrak{s}_{M}, \mu}(P)$,
\[
N(s_{\alpha}, P) \cap \Sigma_{\mathfrak{s}_{M}, \mu} = \{\alpha\}.
\]
In particular, we have $\beta \not \in \Sigma_{\mathfrak{s}_{M}, \mu}$.
We take the parabolic subgroup $P'$ with Levi factor $M$ such that
\[
\Sigma_{\red}(P', A_{M}) = \left(\Sigma_{\red}(P, A_{M}) \backslash \{ \beta \}\right) \cup \{-\beta\}.
\]
Since $\beta \not \in \Sigma_{\mathfrak{s}_{M}, \mu}$, we obtain that
\[
\Sigma_{\mathfrak{s}_{M}, \mu}(P) = \Sigma_{\mathfrak{s}_{M}, \mu}(P').
\]
Moreover, the definition of $P'$ implies that
\[
N(s_{\alpha}, P') = N(s_{\alpha}, P) \backslash \{\beta\}.
\]
Thus, we obtain the claim.
\begin{comment}
Since $M$ is a standard Levi subgroup of $G$ with respect to $P_{0}$, we can take a subset $\Delta_{0}^{M} \subset \Delta_{0}$ 
For $w \in W_{0}$, we define
\[
N(w) = \left\{
\beta \in \Phi^{+} \mid w \beta \in - \Phi^{+}
\right\}.
\]
Let $\alpha_{0}$ denotes an element of $\Phi$ such that
\[
\alpha_{0}\restriction_{A_{M}} = \alpha.
\]
If $\alpha_{0} \in \Delta_{0}$, we may take $P' = P$.
Suppose that $\alpha_{0} \not \in \Delta_{0}$.
It suffices to show that there exists a parabolic subgroup $P'$ with Levi factor $M$ such that
\[
\Sigma_{\mathfrak{s}_{M}, \mu}(P) = \Sigma_{\mathfrak{s}_{M}, \mu}(P')
\]
and 

Then, we can take $\beta_{0} \in \Delta_{0}$ such that 
\[
s_{\alpha_{0}}(\beta_{0}) \in - \Phi^{+}.
\]
\end{comment}
\end{proof}
We fix such a $P'$.
Then, according to \cite[Proposition~4.2 (a)]{MR4432237}, the Harish-Chandra's intertwining operator $J_{P' \mid P}$ has no poles.
Hence, it restricts to a $G(F)$-equivariant isomorphism
\[
J_{P' \mid P}(\sigma \otimes \cdot) \colon I_{P}^{G}\left(\ind_{M^1}^{M(F)}(\sigma_{1})\right) \rightarrow I_{P'}^{G}\left(\ind_{M^1}^{M(F)}(\sigma_{1})\right).
\] 
The definition of $J_{P' \mid P}(\sigma \otimes \cdot)$ \cite[Subsection~4.1]{MR4432237} implies that
%\cite[(4.3)]{MR4432237}, \cite[(4.5)]{MR4432237}
\begin{align}
\label{harishchadravselementsofB}
J_{P' \mid P}(\sigma \otimes \cdot) \circ I_{P}^{G}(b) = I_{P'}^{G}(b) \circ J_{P' \mid P}(\sigma \otimes \cdot)
\end{align}
for all $b \in \End_{M(F)}\left(\ind_{M^1}^{M(F)}(\sigma_{1}) \right)$.

We define
\[
T'_{s, P'} \in \End_{G(F)}\left(I_{P'}^{G}\left(\ind_{M^1}^{M(F)}(\sigma_{1})\right)\right)
\]
by replacing $P$ with $P'$ in the definition of
\[
T'_{s} = T'_{s, P} \in \End_{G(F)}\left(I_{P}^{G}\left(\ind_{M^1}^{M(F)}(\sigma_{1})\right)\right).
\]
appearing in Lemma~\ref{lemmasolleveld10.8}.
Then, we have the following:
\begin{lemma}
\label{lemmareductionintertwiningsolleveldside}
We have
\[
T'_{s, P'} = J_{P' \mid P}(\sigma \otimes \cdot) \circ T'_{s, P} \circ (J_{P' \mid P}(\sigma \otimes \cdot))^{-1}.
\]
\end{lemma}
\begin{proof}
Recall that $T'_{s}$ is defined as
\[
T'_{s} = \frac{
(q_{\alpha}-1)(q_{\alpha*}+1)
}{
2
}
(\theta_{h_{\alpha}^{\vee}})^{\epsilon_{\alpha}} \circ J_{s}
+ f_{\alpha},
\]
for some $f_{\alpha} \in \mathbb{C}(M_{\sigma}/M^1)$.
We write $J_{s} = J_{s, P}$, and define
\[
J_{s, P'} \in \Hom_{G(F)}\left(
I_{P'}^{G}\left(
\ind_{M^1}^{M(F)} (\sigma_{1})
\right),
I_{P'}^{G}\left(
\ind_{M^1}^{M(F)} (\sigma_{1}) \otimes_{\mathbb{C}[M_{\sigma}/M^1]} \mathbb{C}(M_{\sigma}/M^1)
\right)
\right)
\]
by replacing $P$ with $P'$ in the definition of $J_{s, P}$.
Then, we have
\[
T'_{s, P'} = \frac{
(q_{\alpha}-1)(q_{\alpha*}+1)
}{
2
}
(\theta_{h_{\alpha}^{\vee}})^{\epsilon_{\alpha}} \circ J_{s, P'}
+ f_{\alpha}.
\]
According to \eqref{harishchadravselementsofB}, it suffices to show that 
\[
J_{s, P'} = J_{P' \mid P}(\sigma \otimes \cdot) \circ J_{s, P} \circ (J_{P' \mid P}(\sigma \otimes \cdot))^{-1}.
\]
The definition of $J_{s}$ implies that
\[
J_{s, P} = I_{P}^{G}\left(\rho_{P, \sigma, s} \otimes \id\right) \circ I_{P}^{G} (\tau_{s}) \circ \lambda(s) \circ J_{s^{-1}(P) \mid P}(\sigma \otimes \cdot)
\]
and
\[
J_{s, P'} = I_{P'}^{G}\left(\rho_{P', \sigma, s} \otimes \id\right) \circ I_{P'}^{G} (\tau_{s}) \circ \lambda(s) \circ J_{s^{-1}(P') \mid P'}(\sigma \otimes \cdot).
\]
Since the normalization of
\[
\rho_{\sigma, s} \colon ^s\!\sigma \simeq \sigma,
\]
in \cite[Lemma~4.3]{MR4432237} depends on $P$, we write $\rho_{\sigma, s}$ in $J_{s, P}$ and $J_{s, P'}$ as $\rho_{P, \sigma, s}$ and $\rho_{P', \sigma, s}$, respectively.
Since the space $\Hom_{M(F)}\left(
^s\!\sigma, \sigma
\right)$ is one-dimensional, there exists $c_{1} \in \mathbb{C}^{\times}$ such that
\begin{align}
\label{dependenceonP}
\rho_{P', \sigma, s} = c_{1} \cdot \rho_{P, \sigma, s}.
\end{align}
The definition of $J_{P' \mid P}(\sigma \otimes \cdot)$ \cite[Subsection~4.1]{MR4432237} implies that
\begin{align}
\label{almostharishchandraintertwiningcommuteswithB}
J_{P' \mid P}(\sigma \otimes \cdot) \circ I_{P}^{G}\left(\rho_{P, \sigma, s} \otimes \id\right) \circ I_{P}^{G} (\tau_{s}) \circ \lambda(s) =
I_{P'}^{G}\left(\rho_{P, \sigma, s} \otimes \id\right) \circ I_{P'}^{G} (\tau_{s}) \circ \lambda(s) \circ J_{s^{-1}(P') \mid s^{-1}(P)}(\sigma \otimes \cdot).
\end{align}
Moreover, according to \cite[IV.3 (4)]{MR1989693} and \cite[V.2]{MR1989693}, there exists $c_{2} \in \mathbb{C}^{\times}$ such that
\[
J_{s^{-1}(P') \mid s^{-1}(P)}(\sigma \otimes \cdot) \circ J_{s^{-1}(P) \mid P}(\sigma \otimes \cdot) 
= c_{2} \cdot \left(
\prod_{\beta} \mu^{M_{\beta}}(\sigma \otimes \cdot)^{-1}
\right)
J_{s^{-1}(P') \mid P}(\sigma \otimes \cdot),
\]
where $\mu^{M_{\beta}}$ denotes the Harish-Chandra's $\mu$-function \cite[V.2]{MR1989693}, and $\beta$ runs over
\[
\Sigma_{\red}(P, A_{M}) \cap \left(- \Sigma_{\red}(s^{-1}(P), A_{M})\right) \cap \Sigma_{\red}(s^{-1}(P'), A_{M}).
\]
Since we are assuming
\[
\Sigma_{\mathfrak{s}_{M}, \mu}(P) = \Sigma_{\mathfrak{s}_{M}, \mu}(P'),
\]
we have
\[
\left( - \Sigma_{\mathfrak{s}_{M}, \mu}(s^{-1}(P)) \right) \cap \Sigma_{\mathfrak{s}_{M}, \mu}(s^{-1}(P')) = \emptyset.
\]
Hence, all $\beta$ appearing in the product are contained in $\Sigma_{\red}(A_{M}) \backslash \Sigma_{\mathfrak{s}_{M}, \mu}$.
According to \cite[Proposition~1.6]{MR2827179}, for such $\beta$, $\mu^{M_{\beta}}$ are constant.
Thus, we obtain that there exists $c_{3} \in \mathbb{C}^{\times}$ such that
\begin{align}
\label{transitivityofharishchandraplusmuconst}
J_{s^{-1}(P') \mid s^{-1}(P)}(\sigma \otimes \cdot) \circ J_{s^{-1}(P) \mid P}(\sigma \otimes \cdot) = c_{3} \cdot J_{s^{-1}(P') \mid P}(\sigma \otimes \cdot).
\end{align}
Similarly, we can prove that there exists $c_{4} \in \mathbb{C}^{\times}$ such that
\begin{align}
\label{transitivityofharishchandraplusmuconstsimilarly}
J_{s^{-1}(P') \mid P'}(\sigma \otimes \cdot) \circ J_{P' \mid P}(\sigma \otimes \cdot) = c_{4} \cdot J_{s^{-1}(P') \mid P}(\sigma \otimes \cdot).
\end{align}
Combining \eqref{almostharishchandraintertwiningcommuteswithB} with \eqref{transitivityofharishchandraplusmuconst}, we obtain that
\begin{align*}
J_{P' \mid P}(\sigma \otimes \cdot) \circ J_{s, P}
&= J_{P' \mid P}(\sigma \otimes \cdot) \circ I_{P}^{G}\left(\rho_{P, \sigma, s} \otimes \id\right) \circ I_{P}^{G} (\tau_{s}) \circ \lambda(s) \circ J_{s^{-1}(P) \mid P}(\sigma \otimes \cdot)\\
&= I_{P'}^{G}\left(\rho_{P, \sigma, s} \otimes \id\right) \circ I_{P'}^{G} (\tau_{s}) \circ \lambda(s) \circ  J_{s^{-1}(P') \mid s^{-1}(P)}(\sigma \otimes \cdot) \circ J_{s^{-1}(P) \mid P}(\sigma \otimes \cdot)\\
&= c_{3} \cdot I_{P'}^{G}\left(\rho_{P, \sigma, s} \otimes \id\right) \circ I_{P'}^{G} (\tau_{s}) \circ \lambda(s) \circ J_{s^{-1}(P') \mid P}(\sigma \otimes \cdot).
\end{align*}
On the other hand, equation~\eqref{dependenceonP} and equation~\eqref{transitivityofharishchandraplusmuconstsimilarly} imply that
\begin{align*}
J_{s, P'} \circ J_{P' \mid P}(\sigma \otimes \cdot)
&= I_{P'}^{G}\left(\rho_{P', \sigma, s} \otimes \id\right) \circ I_{P'}^{G} (\tau_{s}) \circ \lambda(s) \circ J_{s^{-1}(P') \mid P'}(\sigma \otimes \cdot) \circ J_{P' \mid P}(\sigma \otimes \cdot)\\
&= c_{4} \cdot I_{P'}^{G}\left(\rho_{P', \sigma, s} \otimes \id\right) \circ I_{P'}^{G} (\tau_{s}) \circ \lambda(s) \circ J_{s^{-1}(P') \mid P}(\sigma \otimes \cdot) \\
&= c_{1} c_{4} \cdot I_{P'}^{G}\left(\rho_{P, \sigma, s} \otimes \id\right) \circ I_{P'}^{G} (\tau_{s}) \circ \lambda(s) \circ J_{s^{-1}(P') \mid P}(\sigma \otimes \cdot).
\end{align*}
Now, we conclude that
\[
J_{s, P'} = c_{5} \cdot J_{P' \mid P}(\sigma \otimes \cdot) \circ J_{s, P} \circ (J_{P' \mid P}(\sigma \otimes \cdot))^{-1},
\]
where 
\[
c_{5} = c_{1} c_{4} c_{3}^{-1}.
\]
According to \cite[Lemma~10.7 (a)]{MR4432237}, comparing the residues of both sides at a point $\sigma_{+} \in \mathfrak{s}_{M}$, we obtain that $c_{5}= 1$. 
\begin{comment}
and
\[
(c')^{-1} \cdot c'' \cdot T'_{s, P'} = J_{P' \mid P}(\sigma \otimes \cdot) \circ T'_{s, P} \circ (J_{P' \mid P}(\sigma \otimes \cdot))^{-1}.
\]
Since $T'_{s, P}$ and $T'_{s, P'}$ have the same quadratic relation
\[
\left(T'_{s, P}\right)^2 = (q_{F}^{\lambda^{\Sol}(\alpha')} - 1)T'_{s, P} + q_{F}^{\lambda^{\Sol}(\alpha')}
\]
and
\[
\left(T'_{s, P'}\right)^2 = (q_{F}^{\lambda^{\Sol}(\alpha')} - 1)T'_{s, P'} + q_{F}^{\lambda^{\Sol}(\alpha')},
\]
we conclude that $(c')^{-1} \cdot c'' = 1$. 
\end{comment}
\end{proof}
%\begin{lemma}
%We have
%\[
%I_{U'}(\Phi_{s}) = J_{P' \mid P}(\sigma) \circ I_{U}(\Phi_{s}) \circ (J_{P' \mid P}(\sigma))^{-1}.
%\]
%\end{lemma}

Next, we define an intertwining operator on $\ind_{K}^{G(F)} (\rho)$.
Let $\Phi_{\aff, \red}$ denote the set of indivisible elements in $\Phi_{\aff}$, and we write
\[
\Phi_{\aff, \red}^{+} = \Phi_{\aff, \red} \cap \Phi_{\aff}^{+}.\index{$\Phi_{\aff, \red}^{+} $}
\]
Since any element of $\Gamma(J, \rho)$ is contained in a basis of $\Phi_{\aff}$, we have
\[
\Gamma(J, \rho) \subset \Phi_{\aff, \red}.
\]
%We fix lifts of elements of $W$ as \cite[Proposition~5.2]{MR1235019} and identify an element $w \in W$ with its lift $\dot{w}$.
%\begin{remark}
%For $w_{1}, w_{2} \in W$, the lift of $w_{1} w_{2}$ is not necessarily equal to the product $\dot{w_{1}} \dot{w_{2}}$.
%When we write $w_{1} w_{2}$, it is identified with the product $\dot{w_{1}} \dot{w_{2}}$, not with the lift of $w_{1} w_{2}$.
%\end{remark}
For $w \in W$, we write
\[
N(w) = \{
a \in \Phi_{\aff, \red}^{+} \mid wa \in - \Phi_{\aff, \red}^{+}
\}.\index{$N(w)$}
\]
We also define
\[
l(w) = \abs{N(w)}\index{$l(w)$}
\]
for $w \in W$.

Let $B_{1}, B_{2}$ be bases of $\Phi_{\aff}$ containing $J$.
Then, we can define the parahoric subgroup $P_{J, B_{i}}$ with radical $U_{J, B_{i}}$ associated with $J \subset B_{i}$ for $i=1,2$.
We define
\[
\theta_{B_{2} \mid B_{1}} \colon \ind_{P_{J, B_{1}}}^{G(F)} (\rho) \rightarrow \ind_{P_{J, B_{2}}}^{G(F)} (\rho)\index{$\theta_{B_{2} \mid B_{1}}$}
\]
as
\[
\left(
\theta_{B_{2} \mid B_{1}} (f) 
\right)(g) = \int_{U_{J, B_{2}}} f(u' g) du'
\]
for $f \in \ind_{P_{J, B_{1}}}^{G(F)} (\rho)$ and $g \in G(F)$.
Here, we use the Haar measure on $U_{J, B_{2}}$ such that the volume of $U_{J, B_{2}}$ is equal to $1$.
Let $w \in W$ such that $wJ \subset B$.
Then, we have
\[
P_{J, w^{-1}B} = {^{\dot{w}^{-1}}\!{P_{wJ, B}}}
\]
and 
\begin{align}
\label{intertwiningoperatorcomparisonmorris}
\theta_{w^{-1}B \mid B} = \lambda(\dot{w}^{-1}) \circ \theta_{\rho, \dot{w}},
\end{align}
where
\[
\lambda(\dot{w}^{-1}) \colon \ind_{P_{wJ, B}}^{G(F)} (^{\dot{w}}\!\rho) \rightarrow \ind_{P_{J, w^{-1}B}}^{G(F)} (\rho)
\]
is defined as
\[
f \mapsto [g \mapsto f(\dot{w}g)],
\]
and 
\[
\theta_{\rho, \dot{w}} \colon \ind_{P_{J, B}}^{G(F)} (\rho) \rightarrow \ind_{P_{wJ, B}}^{G(F)} (^{\dot{w}}\!\rho)\index{$\theta_{\rho, \dot{w}}$}
\]
denotes the map defined in \cite[Subsection~5.3]{MR1235019}.
\begin{lemma}
\label{essentiallymorrislemma7.5}
Let $w \in W$ such that $wJ \subset B$, and $v \in W(J, \rho)$.
Suppose that
\[
N(v^{-1}) \cap N(w) \cap \Gamma(J, \rho) = \emptyset.
\]
Then, there exists $c(w, v) \in \mathbb{C}^{\times}$ such that
\[
\theta_{v^{-1}w^{-1}B \mid v^{-1}B}
\circ
\theta_{v^{-1}B \mid B}
=
c(w, v) \cdot \theta_{v^{-1}w^{-1}B \mid B}.\index{$c(w, v)$}
\]
\end{lemma}
\begin{proof}
According to \eqref{intertwiningoperatorcomparisonmorris}, we have
\[
\begin{cases}
\theta_{v^{-1}B \mid B} &= \lambda(\dot{v}^{-1}) \circ \theta_{\rho, \dot{v}}, \\
\theta_{v^{-1}w^{-1}B \mid B} &= \lambda(\dot{v}^{-1}\dot{w}^{-1}) \circ \theta_{\rho, \dot{w}\dot{v}}.
\end{cases}
\]
Moreover, the definition of $\theta_{v^{-1}w^{-1}B \mid v^{-1}B}$ implies that
\[
\theta_{v^{-1}w^{-1}B \mid v^{-1}B} = \lambda(\dot{v}^{-1}\dot{w}^{-1}) \circ \theta_{\dot{v}\rho, \dot{w}} \circ \lambda(\dot{v}),
\]
where 
\[
\lambda(\dot{v}) \colon \ind_{P_{J, v^{-1}B}}^{G(F)} (\rho) \rightarrow \ind_{P_{vJ, B}}^{G(F)} (^{\dot{v}}\!\rho) 
\]
denotes the map defined as 
$
f \mapsto [g \mapsto f(\dot{v}^{-1}g)]
$,
\[
\theta_{\dot{v}\rho, \dot{w}} \colon \ind_{P_{vJ, B}}^{G(F)} (^{\dot{v}}\!\rho) \rightarrow \ind_{P_{wvJ, B}}^{G(F)} (^{\dot{w} \dot{v}}\!\rho) 
\]
denotes the map defined in \cite[Subsection~5.3]{MR1235019}, and
\[
\lambda(\dot{v}^{-1}\dot{w}^{-1}) \colon \ind_{P_{wvJ, B}}^{G(F)} (^{\dot{w} \dot{v}}\!\rho) \rightarrow \ind_{P_{J, v^{-1} w^{-1}B}}^{G(F)} (\rho) 
\]
denotes the map defined as
\[
f \mapsto [g \mapsto f(\dot{w}\dot{v} g)].
\]
Then, the claim follows from \cite[Lemma~7.5]{MR1235019}.
\end{proof}
We also have a variant of Lemma~\ref{essentiallymorrislemma7.5}.
\begin{lemma}
\label{lemmavariantofmorris}
Let $w \in W$ such that $wJ \subset B$, and $v \in W$ such that $vwJ \subset B$.
Suppose that
\[
N(v) \cap N(w^{-1}) \cap w \Gamma(J, \rho) = \emptyset.
\]
Then, there exists $c'(v, w) \in \mathbb{C}^{\times}$ such that
\[
\theta_{w^{-1}v^{-1}B \mid w^{-1}B} 
\circ
\theta_{w^{-1}B \mid B}
=
c'(v, w) \cdot \theta_{w^{-1}v^{-1}B \mid B}.\index{$c'(v, w)$}
\]
\end{lemma}
\begin{proof}
The same argument as the proof of Lemma~\ref{essentiallymorrislemma7.5} implies that
\[
\theta_{w^{-1}B \mid B} = \lambda(\dot{w}^{-1}) \circ \theta_{\rho, \dot{w}},
\]
\[
\theta_{w^{-1}v^{-1}B \mid w^{-1}B} = \lambda(\dot{w}^{-1}\dot{v}^{-1}) \circ \theta_{\dot{w}\rho, \dot{v}} \circ \lambda(\dot{w}),
\]
and 
\[
\theta_{w^{-1}v^{-1}B \mid B} = \lambda(\dot{w}^{-1}\dot{v}^{-1}) \circ \theta_{\rho, \dot{v}\dot{w}}.
\]
Hence, we can rewrite the claim as
\[
\theta_{\dot{w}\rho, \dot{v}} \circ \theta_{\rho, \dot{w}} = c'(v, w) \cdot \theta_{\rho, \dot{v}\dot{w}}.
\]
We use the induction on $l(v)$ to prove this.
If $l(v)=0$, then
\[
l(vw) = l(v) + l(w),
\]
and the claim follows from \cite[Proposition~5.10]{MR1235019}.

Suppose that $l(v)>0$.
According to \cite[1.6 (b)]{MR1235019}, we can take an element $a \in N(v) \cap B$.
We write $v_{0} = v[a, wJ]$ and $v' = v v_{0}^{-1}$.
Then, according to \cite[Lemma~2.5 (a)]{MR1235019}, 
\[
l(v) = l(v') + l(v_{0}),
\]
that is equivalent to
\[
N(v) = N(v_{0}) \cup v_{0}^{-1} N(v').
\]
According to \cite[Proposition~5.10]{MR1235019}, we have
\begin{align}
\label{decompositionofv}
\theta_{\dot{v_{0}} \dot{w} \rho, \dot{v'}} \circ \theta_{\dot{w} \rho, \dot{v_{0}}} = \theta_{\dot{w}\rho, \dot{v}}.
\end{align}
Since
\[
N(v_{0}) \subset N(v),
\]
the assumption implies
\[
N(v_{0}) \cap N(w^{-1}) \cap w \Gamma(J, \rho) = \emptyset.
\]
According to \cite[Lemma~2.4]{MR1235019}, $a \in N(v_{0})$, hence we have
\[
a \not \in N(w^{-1}) \cap w \Gamma(J, \rho).
\]
Then, according to \cite[Lemma~7.4]{MR1235019}, there exists $c'(v_{0}, w) \in \mathbb{C}^{\times}$ such that
\begin{align}
\label{lemma7.4ofmorris}
\theta_{\dot{w}\rho, \dot{v_{0}}} \circ \theta_{\rho, \dot{w}} = c'(v_{0}, w) \cdot \theta_{\rho, \dot{v_{0}} \dot{w}}.
\end{align}
Combining \eqref{decompositionofv} with \eqref{lemma7.4ofmorris}, we obtain 
\begin{align}
\label{decompositionofv+lemma7.4ofmorris}
\theta_{\dot{w} \rho,\dot{v}} \circ \theta_{\rho, \dot{w}} &=
\theta_{\dot{v_{0}} \dot{w} \rho, \dot{v'}} \circ \theta_{\dot{w}\rho, \dot{v_{0}}} \circ \theta_{\rho, \dot{w}} \\
&= c'(v_{0}, w) \cdot 
\theta_{\dot{v_{0}} \dot{w} \rho, \dot{v'}} \circ \theta_{\rho, \dot{v_{0}} \dot{w}}. \notag
\end{align}
To use the induction hypothesis, we will prove that
\[
N(v') \cap N(w^{-1}v_{0}^{-1}) \cap v_{0} w \Gamma(J, \rho) = \emptyset.
\]
Since, 
\[
N(w^{-1}v_{0}^{-1}) \subset N(v_{0}^{-1}) \cup v_{0}N(w^{-1}), 
\]
it suffices to prove 
\[
N(v') \cap N(v_{0}^{-1}) \cap v_{0} w \Gamma(J, \rho) = \emptyset
\]
and
\[
N(v') \cap v_{0}N(w^{-1}) \cap v_{0} w \Gamma(J, \rho) = \emptyset.
\]
Let $b \in N(v') \cap N(v_{0}^{-1})$.
Then, 
\[
-v_{0}^{-1}(b) > 0, \ v_{0}(-v_{0}^{-1}b) = -b < 0, \ , v(-v_{0}^{-1}(b)) = -v'(b) >0,
\]
hence
\[
-v_{0}^{-1}(b) \in N(v_{0}) \backslash N(v).
\]
However, since 
\[
N(v) = N(v_{0}) \cup v_{0}^{-1} N(v'),
\]
$N(v_{0})$ is contained in $N(v)$, a contradiction.
Thus, we conclude that 
\[
N(v') \cap N(v_{0}^{-1}) \cap v_{0} w \Gamma(J, \rho) \subset N(v') \cap N(v_{0}^{-1}) = \emptyset.
\]
Next, we will prove 
\[
N(v') \cap v_{0}N(w^{-1}) \cap v_{0} w \Gamma(J, \rho) = \emptyset.
\]
Since 
\[
v_{0}^{-1} N(v') \subset N(v_{0}) \cup v_{0}^{-1} N(v') = N(v),
\]
and we are assuming that
\[
N(v) \cap N(w^{-1}) \cap w \Gamma(J, \rho) = \emptyset,
\]
we obtain
\[
v_{0}^{-1} N(v') \cap N(w^{-1}) \cap w \Gamma(J, \rho) = \emptyset.
\]
Hence, we obtain
\[
N(v') \cap v_{0}N(w^{-1}) \cap v_{0} w \Gamma(J, \rho)
= v_{0}\left(
v_{0}^{-1} N(v') \cap N(w^{-1}) \cap w \Gamma(J, \rho) 
\right) = \emptyset.
\]

We write $w' = v_{0} w$.
Then, the induction hypothesis implies that there exists $c'(v', w') \in \mathbb{C}^{\times}$ such that
\begin{align}
\label{calculationthetainductionhypothesis}
\theta_{\dot{w'} \rho, \dot{v'}} \circ \theta_{\rho, \dot{w'}} =
c'(v', w') \cdot
\theta_{\rho, \dot{v'} \dot{w'}}.
\end{align}
Here, we note that the lift $\dot{w}'$ of $w'$ is not necessarily equal to the product of the lift $\dot{v_{0}}$ of $v_{0}$ and the lift $\dot{w}$ of $w$.
We write
\[
t= (\dot{w}')^{-1} \dot{v_{0}}\dot{w}.
\]
According to \cite[Lemma~6.3 (a)]{MR1235019}, we have
\[
\theta_{\rho, \dot{v_{0}}\dot{w}} = \theta_{\rho, \dot{w}' t} = \rho(t^{-1}) \circ \theta_{\rho, \dot{w}'},
\]
and according to \cite[Lemma~6.3 (c)]{MR1235019}, we have
\begin{align*}
\theta_{\dot{v_0} \dot{w}\rho, \dot{v}'} &= \theta_{\dot{w}' t \rho, \dot{v}'}\\ 
&= \theta_{(\dot{w}' t (\dot{w}')^{-1}) \dot{w}' \rho, \dot{v}'}\\
&= (\dot{w}' \rho)(\dot{w}' t^{-1} (\dot{w}')^{-1}) \circ \theta_{\dot{w}' \rho, \dot{v}'} \circ (\dot{w}' \rho)(\dot{w}' t (\dot{w}')^{-1})\\
&= \rho(t^{-1}) \circ \theta_{\dot{w}' \rho, \dot{v}'} \circ \rho(t).
\end{align*}
Hence, 
\begin{align*}
\theta_{\dot{v_0} \dot{w}\rho, \dot{v}'} \circ \theta_{\rho, \dot{v_{0}}\dot{w}} &= \theta_{\dot{v_0} \dot{w}\rho, \dot{v}'} \circ \rho(t^{-1}) \circ \theta_{\rho, \dot{w}'}\\
&= \rho(t^{-1}) \circ \theta_{\dot{w}' \rho, \dot{v}'} \circ \theta_{\rho, \dot{w}'}.
\end{align*}
Combining it with \eqref{calculationthetainductionhypothesis} and using \cite[Lemma~6.3 (a)]{MR1235019} again, we obtain
\begin{align}
\label{inductionhypothesis+Lemma6.3Morris}
\theta_{\dot{v_0} \dot{w}\rho, \dot{v}'} \circ \theta_{\rho, \dot{v_{0}}\dot{w}}
&= 
\rho(t^{-1}) \circ \theta_{\dot{w}' \rho, \dot{v}'} \circ \theta_{\rho, \dot{w}'} \\
&= c'(v', w') \cdot
\rho(t^{-1}) \circ \theta_{\rho, \dot{v}'\dot{w}'} \notag \\
&= c'(v', w') \cdot
\rho(t^{-1}) \circ \theta_{\rho, \dot{v}' \dot{v_{0}}\dot{w} t^{-1}} \notag \\
&= c'(v', w') \cdot
\theta_{\rho, \dot{v}' \dot{v_{0}}\dot{w}} \notag \\
&= c'(v', w') \cdot
\theta_{\rho, \dot{v} \dot{w}}. \notag
\end{align}
For the last eauality, we used the fact
\[
\dot{v} = \dot{v}' \dot{v_{0}},
\]
that follows from
\[
l(v) = l(v') + l(v_{0})
\]
and our choices of lifts (see \cite[Proposition~5.2]{MR1235019}).
Now, combining \eqref{decompositionofv+lemma7.4ofmorris} with \eqref{inductionhypothesis+Lemma6.3Morris}, we obtain
\[
\theta_{\dot{w} \rho, \dot{v}} \circ \theta_{\rho, \dot{w}} = c'(v_{0}, w) \cdot
 c'(v', w') \cdot \theta_{\rho, \dot{v}\dot{w}}.
\]
\end{proof}
\begin{corollary}
\label{corollarythetaw-1BBisom}
Let $w \in W$ such that $wJ \subset B$ and
\[
N(w) \cap \Gamma(J, \rho) = N(w) \cap -\Gamma(J, \rho)= \emptyset.
\]
Then, the map
\[
\theta_{w^{-1}B \mid B} \colon \ind_{P_{J, B}}^{G(F)}(\rho) \rightarrow \ind_{P_{J, w^{-1}B}}^{G(F)}(\rho).
\]
is an isomorphism.
\end{corollary}
\begin{proof}
Since
\[
N(w) \cap - \Gamma(J, \rho) = \emptyset,
\]
we obtain
\[
N(w^{-1}) \cap w \Gamma(J, \rho) =
-w \left(
N(w) \cap -\Gamma(J, \rho)
\right)
= \emptyset.
\]
Then, substituting $v = w^{-1}$ in Lemma~\ref{lemmavariantofmorris}, we obtain
\begin{align}
\label{leftinverseexists}
\theta_{B \mid w^{-1}B} \circ \theta_{w^{-1}B \mid B} = c'(w^{-1}, w) \cdot \theta_{B \mid B}
= c'(w^{-1}, w) \cdot \id_{\ind_{P_{J, B}}^{G(F)}(\rho) }
\end{align}
On the other hand, replacing $B$ with $w^{-1}B$ and $w$ with $w^{-1}$ in Lemma~\ref{lemmavariantofmorris}, we obtain
\[
\theta_{wv^{-1}w^{-1} B \mid B} \circ \theta_{B \mid w^{-1} B} \in \mathbb{C}^{\times} \cdot \theta_{wv^{-1} w^{-1} B \mid B}
\]
for all $v \in W$ such that $vw^{-1}J \subset w^{-1}B$ and
\[
w^{-1}N(v) \cap w^{-1}N(w) \cap w^{-1} \Gamma(J, \rho).
\]
Since 
\[
N(w) \cap \Gamma(J, \rho) = \emptyset,
\]
we can take $v = w$.
Then, we obtain
\begin{align}
\label{rightinverseexists}
\theta_{w^{-1} B \mid B} \circ \theta_{B \mid w^{-1} B} = c \cdot \theta_{w^{-1}B \mid w^{-1}B} =c \cdot \id_{\ind_{P_{J, w^{-1}B}}^{G(F)}(\rho) }
\end{align}
for some $c \in \mathbb{C}^{\times}$.
Combining \eqref{leftinverseexists} with \eqref{rightinverseexists}, we obtain the corollary.
\end{proof}
%Since $w' = v_{0} w$ and $v= v' v_{0}$, we obtain
%\[
%(\ind v_{0})^{-1} (\ind w)^{-1} (\ind v_{0}w) (\ind v')^{-1} (\ind w')^{-1} (\ind v' w')
%= (\ind v_{0})^{-1} (\ind w)^{-1} (\ind v')^{-1} (\ind vw).
%\]
%Moreover, since
%\[
%l(v) = l(v') + l(v_{0}),
%\]
%\cite[Corollary~5.8]{MR1235019} implies 
%\[
%\ind v = (\ind v') (\ind v_{0}).
%\]
%Thus, we obtain 
%\[
%(\ind v_{0})^{-1} (\ind w)^{-1} (\ind v_{0}w) (\ind v')^{-1} (\ind w')^{-1} (\ind v' w') = (\ind v)^{-1} (\ind w)^{-1} (\ind vw),
%\]
%and
%\[
%\theta_{\dot{w}\rho, \dot{v}} \circ \theta_{\rho, \dot{w}} = \{
%(\ind v)^{-1} (\ind w)^{-1} (\ind vw)
%\}^{1/2} \theta_{\rho, \dot{v}\dot{w}}.
%\]
\begin{corollary}
\label{corollaryoflemmavariantofmorris}
Let $w \in W$ such that $wJ \subset B$ and
\[
N(w) \cap \Gamma(J, \rho) = N(w) \cap -\Gamma(J, \rho)= \emptyset.
\]
Then, for any $v \in W(J, \rho)$, there exists $c''(w, v) \in \mathbb{C}^{\times}$ such that 
%We also write
%\[
%t= (\dot{w}\dot{v})^{-1} \dot{v}'\dot{w} \in H,
%\]
%and
%\[
%\theta'_{\dot{w} \rho, \dot{v}'} = 
%\{
%(\ind v') (\ind v)^{-1}
%\}^{1/2}
%\rho(t) \circ \theta_{\dot{w}\rho, \dot{v}'}.
%\]
The following diagram commutes:
\[
\xymatrix@R+2pc@C+3pc{
\ind_{P_{J, B}}^{G(F)}(\rho) \ar[d]_-{\theta_{v^{-1}B \mid B}} \ar[r]^-{\theta_{w^{-1}B \mid B}} \ar@{}[dr]|\circlearrowleft & \ind_{P_{J, w^{-1}B}}^{G(F)}(\rho) \ar[d]^-{c''(w, v) \cdot \theta_{v^{-1}w^{-1} B \mid w^{-1}B}} \\
\ind_{P_{J, v^{-1}B}}^{G(F)} (\rho) \ar[r]^-{\theta_{v^{-1}w^{-1} B \mid v^{-1}B}} & \ind_{P_{J, v^{-1} w^{-1}B }}^{G(F)} (\rho).
}\index{$c''(w, v)$}
\]
%\[
%\xymatrix{
%\ind_{P_{J}}^{G(F)} (\rho) \ar[d]_-{\theta_{\rho, \dot{v}}} \ar[r]^-{\theta_{\rho, \dot{w}}} \ar@{}[dr]|\circlearrowleft & \ind_{P_{wJ}}^{G(F)} (\dot{w} \rho) \ar[d]^-{\theta'_{\dot{w} \rho, \dot{v}'}} \\
%\ind_{P_{J}}^{G(F)} (\dot{v} \rho) \ar[r]^-{\theta_{\dot{v} \rho, \dot{w}}} & \ind_{P_{w J}}^{G(F)} (\dot{w} \dot{v} \rho).
%}
%\]
\end{corollary}
\begin{proof}
We write $v' = wvw^{-1}$.
Since 
\[
N(w) \cap \Gamma(J, \rho) = \emptyset,
\]
Lemma~\ref{essentiallymorrislemma7.5} implies that
\[
\theta_{v^{-1}w^{-1}B \mid v^{-1}B}
\circ
\theta_{v^{-1}B \mid B}
=
c(w, v) \cdot \theta_{v^{-1}w^{-1}B \mid B}.
\]
On the other hand, the assumption
\[
N(w) \cap - \Gamma(J, \rho) = \emptyset
\]
implies
\[
N(w^{-1}) \cap w \Gamma(J, \rho) =
-w \left(
N(w) \cap -\Gamma(J, \rho)
\right)
= \emptyset,
\]
hence Lemma~\ref{lemmavariantofmorris} implies that
\[
\theta_{w^{-1}(v')^{-1}B \mid w^{-1}B} 
\circ
\theta_{w^{-1}B \mid B}
=
c'(v', w) \cdot \theta_{w^{-1}(v')^{-1}B \mid B}.
\]
Since
\[
v' w = w v w^{-1} w = wv,
\]
we obtain the corollary for
\[
c''(w, v) = c(w, v) \cdot c'(v', w)^{-1}.
\]
\end{proof}
The following lemma claims that we can take $w \in W$ such that $a \in w^{-1}B$ and $w$ satisfies the assumption above: 
\begin{lemma}
\label{existenceofgoodw}
There exists $w \in W$ such that $w(J \cup \{a\}) \subset B$ and
\[
N(w) \cap \Gamma(J, \rho) = N(w) \cap - \Gamma(J, \rho) = \emptyset.
\]
\end{lemma}
\begin{proof}
According to \cite[Theorem~2.11 (b)]{MR1235019}, there exists $w \in W_{\aff}$ such that
$
w(J \cup \{a\}) \subset B,
$
and 
\[
l(v[a, J]) = l(wv[a, J]w^{-1}) + 2l(w)
\]
(see also the proof of \cite[Proposition~7.6]{MR1235019}).
Explicitly, we can take $w$ as
\[
w = v[a_{r}, J_{r}] \cdots v[a_{1}, J_{1}],
\]
where $a_{i} \in B$, $J=J_{1}$, and
\[
J_{i+1} = v[a_{i}, J_{i}] J_{i} \subset B
\]
for $1 \le i \le r-1$, and we have
\[
l(w) = \sum_{i=1}^{r} l(v[a_{i}, J_{i}]).
\]
Since $a+A'_{J} \in B(J, \rho)$, the proof of \cite[Proposition~7.6]{MR1235019} implies that
\[
N(w) \cap \Gamma(J, \rho) = \emptyset.
\]
We will prove
\begin{align*}
N(w) \cap - \Gamma(J, \rho) = \emptyset.
\end{align*}
Since $\Gamma'(J, \rho)$ is an affine root system, $\Gamma'(J, \rho) = -\Gamma'(J, \rho)$, hence
\[
- \Gamma(J, \rho) \subset \Gamma(J, \rho) + A'_{J}.
\]
Thus, if 
\[
N(w) \cap - \Gamma(J, \rho) \neq \emptyset,
\]
we have
\[
(N(w) + A'_{J}) \cap \Gamma(J, \rho) \neq \emptyset.
\]
On the other hand, according to Lemma~\ref{lemmainvariancebyJ} below, we have
\[
(N(w) + A'_{J}) \cap \Phi_{\aff, \red} = N(w).
\]
Since $\Gamma(J, \rho) \subset \Phi_{\aff, \red}$, we have
\begin{align*}
N(w) \cap \Gamma(J, \rho) &= (N(w) + A'_{J}) \cap \Phi_{\aff, \red} \cap \Gamma(J, \rho) \\
&= (N(w) + A'_{J}) \cap \Gamma(J, \rho) \\
& \neq \emptyset,
\end{align*}
a contradiction.
Thus, we conclude that
\[
N(w) \cap - \Gamma(J, \rho) = \emptyset.
\]
\end{proof}
\begin{lemma}
\label{lemmainvariancebyJ}
Suppose that an element $w \in W$ is written as
\[
w = v[a_{r}, J_{r}] \cdots v[a_{1}, J_{1}],
\]
where $J_{i} \cup \{a_{i}\} \subset B$ for all $1 \le i \le r$, and
\[
J_{i+1} = v[a_{i}, J_{i}] J_{i}
\]
for $1 \le i \le r-1$.
We also suppose that
\[
l(w) = \sum_{i=1}^{r} l(v[a_{i}, J_{i}]).
\]
Then, we have
\[
(N(w) + A'_{J_{1}}) \cap \Phi_{\aff, \red} = N(w),
\]
where $A'_{J_{1}}$ denotes the subspace of $A'$ spanned by $J_{1}$.
\end{lemma}
\begin{proof}
We prove the lemma by using the induction on $r$.
When $r=0$, $N(w) = N(1) = \emptyset$, hence the equation is trivial.
Suppose that $r \ge 1$.
Then, 
\[
w = w' v[a_{1}, J_{1}],
\]
where
\[
w' = v[a_{r}, J_{r}] \cdots v[a_{2}, J_{2}].
\]
Since we suppose
\[
l(w) = \sum_{i=1}^{r} l(v[a_{i}, J_{i}]),
\]
we have
\[
l(w) = l(w') + l(v[a_{1}, J_{1}]),
\]
hence
\begin{align}
\label{inductionw=vw'}
N(w) = N(v[a_{1}, J_{1}]) \cup v[a_{1}, J_{1}]^{-1}N(w').
\end{align}
According to \cite[Lemma~2.4]{MR1235019}, 
\[
N\left(
v[a_{1}, J_{1}]
\right) = (A'_{J_{1} \cup \{a_{1}\}} \cap \Phi_{\aff, \red}) \backslash A'_{J_{1}},
\]
where $A'_{J_{1} \cup \{a_{1}\}}$ denotes the subspace of $A'$ spanned by $J_{1} \cup \{a_{1}\}$.
Hence, we have
\begin{align}
\label{inductionmorrislemma2.4}
\left(
N\left(
v[a_{1}, J_{1}]
\right) + A'_{J_{1}}
\right) \cap \Phi_{\aff, \red} = N(v[a_{1}, J_{1}]) .
\end{align}
Moreover, the induction hypothesis implies
\[
(N(w') + A'_{J_{2}}) \cap \Phi_{\aff, \red} = N(w').
\]
Since 
\[
J_{2} = v[a_{1}, J_{1}] J_{1},
\]
we obtain
\begin{align}
\label{inductionw'capJ}
(v[a_{1}, J_{1}]^{-1}N(w') + A'_{J_{1}}) \cap \Phi_{\aff, \red}
&= v[a_{1}, J_{1}]^{-1} \left(
(N(w') + A'_{J_{2}}) \cap \Phi_{\aff, \red}
\right)\\
&= v[a_{1}, J_{1}]^{-1}N(w').\notag
\end{align}
Now, \eqref{inductionw=vw'}, \eqref{inductionmorrislemma2.4}, and \eqref{inductionw'capJ} imply
\begin{align*}
(N(w) + A'_{J_{1}}) \cap \Phi_{\aff, \red} &= \left(
(N(v[a_{1}, J_{1}]) \cup v[a_{1}, J_{1}]^{-1}N(w')) + A'_{J_{1}}
\right) \cap \Phi_{\aff, \red} \\
&= 
\left(
\left(
N(v[a_{1}, J_{1}]) + A'_{J_{1}}
\right) \cup \left(
v[a_{1}, J_{1}]^{-1}N(w') + A'_{J_{1}}
\right)
\right) \cap \Phi_{\aff, \red}\\
&=
\left(
\left(
N(v[a_{1}, J_{1}]) + A'_{J_{1}}
\right) \cap \Phi_{\aff, \red}
\right) \cup
\left(
\left(
v[a_{1}, J_{1}]^{-1}N(w') + A'_{J_{1}}
\right) \cap \Phi_{\aff, \red}
\right)\\
&= N(v[a_{1}, J_{1}]) \cup v[a_{1}, J_{1}]^{-1}N(w') = N(w).
\end{align*}
\end{proof}
We fix $w \in W$ such that $w(J \cup \{a\}) \subset B$ and
\[
N(w) \cap \Gamma(J, \rho) = N(w) \cap - \Gamma(J, \rho) = \emptyset.
\]
%Let $w \in W$ such that $wJ \subset B$ and
%\[
%N(w) \cap \pm \Gamma(J, \rho) = \emptyset.
%\]
For $v \in W(J, \rho)$, let
\[
\Phi_{v} = \Phi_{v, B} \in \End_{G(F)}\left(
\ind_{P_{J, B}}^{G(F)} (\rho)
\right)
\]
denote the element appearing in Theorem~\ref{theorem7.12ofmorris}.
%We note that $vJ =J$ in this case.
According to \cite[Subsection~5.4, Subsection~7.7]{MR1235019} and equation~\eqref{intertwiningoperatorcomparisonmorris}, there exists
\[
T(v) \in \Hom_{\mathcal{M}_{J}}(^{\dot{v}}\!\rho, \rho)
%\Hom_{P_{J, B} \cap ^v\!P_{J, B}}(^v\!\rho, \rho)
\]
such that 
\[
\Phi_{v, B} = T(v) \circ \lambda(\dot{v}) \circ \theta_{v^{-1}B \mid B}.
\]
Here, we identify $T(v)$ with the element of
\[
\Hom_{G(F)}\left(\ind_{P_{J, B}}^{G(F)} (^{\dot{v}}\!\rho), \ind_{P_{J, B}}^{G(F)} (\rho)\right)
\]
defined as
\[
(T(v)f)(g) = T(v)\left(
f(g)
\right)
\]
for $f \in \ind_{P_{J, B}}^{G(F)} (^{\dot{v}}\!\rho)$ and $g \in G(F)$.
Replacing $B$ with $w^{-1} B$, we also have the similar description of $\End_{G(F)}\left(
\ind_{P_{J, w^{-1} B}}^{G(F)} (\rho)
\right)$  and the element 
\[
\Phi_{v, w^{-1} B} \in \End_{G(F)}\left(
\ind_{P_{J, w^{-1} B}}^{G(F)} (\rho)
\right)
\]
corresponding to $\Phi_{v, B}$.
\begin{remark}
\label{remarkgamma+doesnotchange}
Since
\[
N(w) \cap \Gamma(J, \rho) = N(w) \cap - \Gamma(J, \rho) =\emptyset,
\]
we have
\[
\Gamma(J, \rho) \cap w^{-1}\left(\Phi_{\aff}^{+}\right) = \Gamma(J, \rho) \cap \Phi_{\aff}^{+}.
\]
Hence, $\Gamma(J, \rho)^{+}$ and $\Gamma'(J, \rho)^{+}$ do not change if we replace $B$ with $w^{-1} B$.
\end{remark}
We also have
\[
\Phi_{v, w^{-1} B} = T'(v) \circ \lambda(\dot{v}) \circ \theta_{v^{-1}w^{-1}B \mid w^{-1}B}
\]
for some
\[
T'(v) \in \Hom_{\mathcal{M}_{J}}(^{\dot{v}}\!\rho, \rho).
%\Hom_{P_{J, w^{-1} B} \cap ^v\!P_{J,w^{-1} B}}(^v\!\rho, \rho).
\] 
Since $T(v)$ and $T'(v)$ are elements of the vector space $\Hom_{\mathcal{M}_{J}}(^{\dot{v}}\!\rho, \rho)$
%\[
%\Hom_{P_{J, B} \cap ^v\!P_{J, B}}(^v\!\rho, \rho) = \Hom_{P_{J, w^{-1} B} \cap ^v\!P_{J,w^{-1} B}}(^v\!\rho, \rho) = \Hom_{\mathcal{M}_{J}}(^v\!\rho, \rho)
%\]
of dimension $1$, there exists $c(v) \in \mathbb{C}^{\times}$ such that
\[
T(v) = c(v) \cdot T'(v).
\]
The definition of $\theta_{w^{-1}B \mid B}$ implies that
\[
\theta_{w^{-1}B \mid B} \circ T(v) \circ \lambda(\dot{v}) = T(v) \circ \lambda(\dot{v}) \circ \theta_{v^{-1} w^{-1}B \mid v^{-1} B}
\in \Hom_{G(F)}\left(
\ind_{P_{J, v^{-1}B}}^{G(F)}(\rho), \ind_{P_{J, w^{-1}B}}^{G(F)} (\rho)
\right).
\]
Then, according to Corollary~\ref{corollaryoflemmavariantofmorris}, we have
\begin{align*}
\theta_{w^{-1}B \mid B} \circ \Phi_{v, B}
&= \theta_{w^{-1}B \mid B} \circ T(v) \circ \lambda(\dot{v}) \circ \theta_{v^{-1}B \mid B}\\
&= T(v) \circ \lambda(\dot{v}) \circ \theta_{v^{-1} w^{-1}B \mid v^{-1} B} \circ \theta_{v^{-1}B \mid B}\\
&= c''(w, v) \cdot T(v) \circ \lambda(\dot{v}) \circ \theta_{v^{-1} w^{-1}B \mid w^{-1} B} \circ \theta_{w^{-1}B \mid B}\\
&= c(v) \cdot c''(w, v) \cdot T'(v) \circ \lambda(\dot{v}) \circ \theta_{v^{-1} w^{-1}B \mid w^{-1} B} \circ \theta_{w^{-1}B \mid B}\\
&= c(v) \cdot c''(w, v) \cdot \Phi_{v, w^{-1}B} \circ \theta_{w^{-1}B \mid B}.
\end{align*}
We use the same symbol $\theta_{w^{-1}B \mid B}$ for the map
\[
\End_{G(F)}\left(
\ind_{P_{J, B}}^{G(F)} (\rho)
\right) \rightarrow \End_{G(F)}\left(
\ind_{P_{J, w^{-1} B}}^{G(F)} (\rho)
\right)
\]
induced by the isomorphism
\[
\theta_{w^{-1}B \mid B} \colon \ind_{P_{J, B}}^{G(F)}(\rho) \rightarrow \ind_{P_{J, w^{-1}B}}^{G(F)}(\rho)
\]
(see Corollary~\ref{corollarythetaw-1BBisom}).
Then, the calculation above implies the following: 
\begin{proposition}
\label{propositionintertwiningmorriswjrho}
%Let $w \in W$ such that $wJ \subset B$ and
%\[
%N(w) \cap \pm \Gamma(J, \rho) = \emptyset.
%\]
For any $v \in W(J, \rho)$, there exists $c'''(w, v) \in \mathbb{C}^{\times}$ such that 
\[
\theta_{w^{-1}B, B}\left(
\Phi_{v, B}
\right) = c'''(w, v) \cdot \Phi_{v, w^{-1}B}.\index{$c'''(w, v)$}
\]
\end{proposition}
Moreover, for an element $v \in R(J, \rho)$, we have:
\begin{corollary}
\label{corollaryforrgrouppropositionintertwiningmorriswjrho}
For any $v \in R(J, \rho)$, we have
\[
\theta_{w^{-1}B, B}\left(
\Phi_{v, B}
\right) = \Phi_{v, w^{-1}B}.
\]
\end{corollary}
\begin{proof}
%For $a \in \Gamma(J, \rho)$, $p_{a}$ is independent of the choice of $B$
Comparing the multiplication rules of $\End_{G(F)}\left(
\ind_{P_{J, B}}^{G(F)} (\rho)
\right)$ with those of $\End_{G(F)}\left(
\ind_{P_{J, w^{-1} B}}^{G(F)} (\rho)
\right)$ in Theorem~\ref{theorem7.12ofmorris}, we conclude that $c'''(w, v) = 1$ for all $v \in R(J, \rho)$.
\end{proof}
Since $K_{M}$ is the maximal parahoric subgroup associated with the vertex $x_{J}$, that does not depend on the choice of $B$, we have
\[
K_{M} = P_{J, B} \cap M(F) = P_{J, w^{-1}B} \cap M(F).
\]
Hence, we obtain the injection
\[
t_{P, w^{-1}B} \colon \End_{M(F)}\left(
\ind_{K_{M}}^{M(F)} (\rho_{M})
\right) \rightarrow \End_{G(F)}\left(
\ind_{P_{J, w^{-1} B}}^{G(F)} (\rho)
\right)
\]
by replacing $B$ with $w^{-1} B$ in the construction of
\[
t_{P} = t_{P, B} \colon \End_{M(F)}\left(
\ind_{K_{M}}^{M(F)} (\rho_{M})
\right) \rightarrow \End_{G(F)}\left(
\ind_{P_{J, B}}^{G(F)} (\rho)
\right).
\]
\begin{corollary}
\label{corollarycompativilityofintertwiningmorrisviatp}
Let $m \in M_{\sigma}/M^1$, and
let $\Phi_{m}^{M} \in \End_{M(F)}\left(
\ind_{K_{M}}^{M(F)} (\rho_{M})
\right)$ denote the element corresponding to $\theta_{m^{-1}} \in \mathbb{C}[M_{\sigma}/M^1]$ via \eqref{isomgroupalgebraheckealgebra} and $T_{\rho_{M}}$.
Then, there exists $c(w, \Phi^{M}_{m}) \in \mathbb{C}^{\times}$ such that
\[
\theta_{w^{-1}B \mid B}\left( t_{P, B}(\Phi^{M}_{m})\right)
=c(w, \Phi^{M}_{m}) \cdot t_{P, w^{-1}B}(\Phi^{M}_{m}).\index{$c(w, \Phi^{M}_{m})$}
\]
\begin{comment}
The following diagram commutes:
\[
\xymatrix{
\End_{M(F)}\left(
\ind_{K_{M}}^{M(F)} (\rho_{M})
\right) \ar[d]_-{t_{P, B}} \ar[r]^-{\id} \ar@{}[dr]|\circlearrowleft & \End_{M(F)}\left(
\ind_{K_{M}}^{M(F)} (\rho_{M})
\right) \ar[d]^-{t_{P, w^{-1}B}}
\\
\End_{G(F)}\left(
\ind_{P_{J, B}}^{G(F)} (\rho)
\right) \ar[r]^-{\theta_{w^{-1}B \mid B}} & \End_{G(F)}\left(
\ind_{P_{J, w^{-1} B}}^{G(F)} (\rho)
\right)
}.
\] 
\end{comment}
\end{corollary}
\begin{proof}
Recall that the canonical inclusion
\[
I_{M(F)}(\rho_{M}) \rightarrow M_{\sigma}
\]
induces an isomorphism
\[
I_{M(F)}(\rho_{M})/K_{M} \rightarrow M_{\sigma}/M^1.
\]
Moreover, according to Lemma~\ref{WcapMsimeqImrho}, the canonical quotient map
\[
W(J, \rho) \cap W_{M(F)} \rightarrow I_{M(F)}(\rho_{M})/K_{M}
\]
is a bijection.
We identify them.
We define
\[
W(J, \rho)^{M, +} =
\{
m \in W(J, \rho) \cap W_{M(F)} \mid \langle \alpha, H_{M}(m) \rangle \ge 0 \ (\alpha \in \Sigma(P, A_{M}))
\}.
\]
According to Remark~\ref{remarkoflemmapositivitycompativility},
% and Remark~\ref{remarkgamma+doesnotchange}, 
for any $m \in W(J, \rho)^{M, +} $, the lift $\dot{m}$ is positive relative to $(P_{J, B}, U)$ and $(P_{J, w^{-1}B}, U)$.
We identify 
\[
\mathbb{C}[W(J, \rho) \cap W_{M(F)}] =  \mathbb{C}[I_{M(F)}(\rho_{M})/K_{M}] = \mathbb{C}[M_{\sigma}/M^1]
\]
with $\mathcal{H}(M(F), \rho_{M})$ and $\End_{M(F)}\left(
\ind_{K_{M}}^{M(F)} (\rho_{M})
\right)$ via \eqref{Mverofheckevsend} and \eqref{compositionofMverofheckevsendandisomgroupalgebraheckealgebra}.
For $m \in W(J, \rho) \cap W_{M(F)}$, let $\phi^{M}_{m} \in \mathcal{H}(M(F), \rho_{M})$ and
$\Phi^{M}_{m} \in \End_{M(F)}\left(
\ind_{K_{M}}^{M(F)} (\rho_{M})
\right)$
denote the elements corresponding to $\theta_{m^{-1}} \in \mathbb{C}[W(J, \rho) \cap W_{M(F)}]$.
According to Lemma~\ref{lemmasupportofimageofm}, $\phi^{M}_{m}$ is supported on $\dot{m} K_{M}$.
Since the group $W(J, \rho) \cap W_{M(F)}$ is generated by
\[
\{
m^{-1} \mid m \in W(J, \rho)^{M, +} 
\},
\]
it suffices to show that
\[
\theta_{w^{-1}B \mid B}\left( t_{P, B}(\Phi^{M}_{m})\right)
\in \mathbb{C}^{\times} \cdot t_{P, w^{-1}B}(\Phi^{M}_{m})
\]
for all $m \in W(J, \rho)^{M, +}$.
Since $\dot{m}$ is positive relative to $(P_{J, B}, U)$ and $(P_{J, w^{-1}B}, U)$, there exists $c(m, B), c(m, w^{-1}B) \in \mathbb{C}^{\times}$ such that
\[
t_{P, B}(\Phi^{M}_{m}) = c(m, B) \cdot \Phi_{m, B}
\]
and
\[
t_{P, w^{-1}B}(\Phi^{M}_{m}) = c(m, w^{-1}B) \cdot \Phi_{m, w^{-1}B}.
\]
Then, the claim follows from Proposition~\ref{propositionintertwiningmorriswjrho}.
\begin{comment}
\[
\theta_{w^{-1}B \mid B}\left( t_{P, B}(\Phi^{M}_{m})\right)
= c \cdot t_{P, w^{-1}B}(\Phi^{M}_{m}).
\]
For $v \in \rho_{M}$, we define an element $f_{v} \in \ind_{P_{J, w^{-1}B}}^{G(F)} (\rho)$ as
\[
f_{v}(g) =
\begin{cases}
\rho(g) v & (g \in P_{J, w^{-1}B})\\
0 & (\text{otherwise}).
\end{cases}
\]
We compare the both sides of the equation
\[
\left(
\left(
\theta_{w^{-1}B \mid B} \circ t_{P, B}(\Phi^{M}_{m})
\right)(f_{v})
\right)(m)
= c \cdot
\left(
\left(
t_{P, w^{-1}B}(\Phi^{M}_{m}) \circ \theta_{w^{-1}B \mid B}
\right) (f_{v})
\right)(m).
\]
The definition of $t_{P, B}$ and $\theta_{w^{-1}B \mid B}$ imply that
\begin{align*}
\left(
\left(
\theta_{w^{-1}B \mid B} \circ t_{P, B}(\Phi^{M}_{m})
\right)(f_{v})
\right)(m)
&= \int_{U_{J, w^{-1}B}} 
\left(
\left(
t_{P, B}(\Phi^{M}_{m})
\right)(f_{v})
\right)(u'm)
du'
\\
&=
\int_{U_{J, w^{-1}B}} 
\int_{G(F)}
t_{P, B}(\phi^{M}_{m})(y) \cdot f_{v}(y^{-1}u'm)
dy
du'\\
&=
\int_{U_{J, w^{-1}B}} 
\int_{P_{J, B} m P_{J, B}}
t_{P, B}(\phi^{M}_{m})(y) \cdot f_{v}(y^{-1}u'm)
dy
du'\\
&=
\int_{U_{J, w^{-1}B}} 
\sum_{ y \in P_{J, B} m P_{J, B}/P_{J, B}}
t_{P, B}(\phi^{M}_{m})(y) \cdot f_{v}(y^{-1}u'm)
du'\\
&=
\int_{U_{J, w^{-1}B}} 
\sum_{ k \in P_{J, B}/\left(P_{J, B} \cap m P_{J, B} m^{-1} \right)}
t_{P, B}(\phi^{M}_{m})(km) \cdot f_{v}(m^{-1}k^{-1}u'm)
du'\\
\end{align*}
\[
\left(
t_{P, B}(\phi^{M}_{m})
\right)(m)
= \delta_{P}(m)^{-1/2}\phi^{M}_{m}(m)
\]
\end{comment}
\end{proof}

Recall that we fixed a parabolic subgroup $P'$ with Levi factor $M$ such that $M_{\alpha}$ is standard with respect to $P'$, and
\[
\Sigma_{\mathfrak{s}_{M}, \mu}(P) = \Sigma_{\mathfrak{s}_{M}, \mu}(P').
\]
Combining $J_{P' \mid P}(\sigma \otimes \cdot)$ with $\theta_{w^{-1}B \mid B}$, we obtain the following diagram:
\begin{proposition}
There exists $b(w, P') \in \mathbb{C}[M_{\sigma}/M^1]^{\times}$ such that the following diagram commutes:
\begin{comment}
We write $v = v[a, J]$.
According to \cite[Theorem~2.11]{MR1235019}, there exists $w \in W_{\aff}$ such that
\[
w(J \cup \{a\}) \subset B,
\]
and 
\[
l(v) = l(wvw^{-1}) + 2l(w),
\]
where $l$ denotes the length function on $W_{\aff}$ with respect to the basis $B$.
Explicitly, we can take $w$ as
\[
w = v[a_{r}, J_{r}] \cdots v[a_{1}, J_{1}],
\]
where $a_{i} \in B$, $J=J_{1}$, and
\[
J_{i+1} = v[a_{i}, J_{i}] J_{i}
\]
for $1 \le i \le r-1$.
We identify $w$ with its lift taken as \cite[Proposition~5.2]{MR1235019}.
For $w' \in W$, we write
\[
N(w') = \{
a \in \Phi_{\aff}^{+} \mid w'a \in - \Phi_{\aff}^{+}
\}.
\]
Since $a+A'_{J} \in B(J, \rho)$, the proof of \cite[Proposition~7.6]{MR1235019} implies that
\[
N(w) \cap \Gamma(J, \rho) = \emptyset.
\]
\[
\xymatrix@R+1pc@C+2pc{
\End_{G(F)}\left(
\ind_{P_{J, B}}^{G(F)} (\rho)
\right) \ar[d]_-{\theta_{w^{-1}B \mid B}} \ar[r]^-{T_{\rho_{M}} \circ I_{U}} \ar@{}[dr]|\circlearrowleft & \End_{G(F)}\left(
I_{P}^{G} \left(
\ind_{M^{1}}^{M(F)} (\sigma_{1})
\right)
\right) \ar[d]^-{b \circ J_{P' \mid P}(\sigma)}
\\
\End_{G(F)}\left(
\ind_{P_{J, B}}^{G(F)} (\rho)
\right) \ar[r]^-{T_{\rho_{M}} \circ I_{U'}} & \End_{G(F)}\left(
I_{P'}^{G} \left(
\ind_{M^{1}}^{M(F)} (\sigma_{1})
\right)
\right).
}
\]
\end{comment}
\[
\xymatrix@R+1pc@C+1pc{
\ind_{P_{J, B}}^{G(F)} (\rho)
\ar[d]_-{\theta_{w^{-1}B \mid B}} \ar[r]^-{T_{\rho_{M}} \circ I_{U}} \ar@{}[dr]|\circlearrowleft & 
I_{P}^{G} \left(
\ind_{M^{1}}^{M(F)} (\sigma_{1})
\right)
\ar[d]^-{b(w, P') \circ J_{P' \mid P}(\sigma \otimes \cdot)}
\\
\ind_{P_{J, w^{-1}B}}^{G(F)} (\rho)
\ar[r]^-{T_{\rho_{M}} \circ I_{U'}} & 
I_{P'}^{G} \left(
\ind_{M^{1}}^{M(F)} (\sigma_{1})
\right).
}\index{$b(w, P')$}
\]
Here, we regard $b(w, P')$ as an element of $\End_{G(F)}\left(
I_{P'}^{G} \left(
\ind_{M^{1}}^{M(F)} (\sigma_{1})
\right)
\right)^{\times}$ via \eqref{isomgroupalgebraheckealgebraparabolicinduction}.
\end{proposition}
\begin{proof}
We identify $\mathbb{C}[M_{\sigma}/M^1]$ with $\End_{M(F)}\left(
\ind_{M^{1}}^{M(F)} (\sigma_{1})
\right)$ and $\End_{M(F)}\left(
\ind_{K_{M}}^{M(F)} (\rho_{M})
\right)$ via \eqref{isomgroupalgebraheckealgebra} and $T_{\rho_{M}}$.
According to equation~\eqref{harishchadravselementsofB}, Proposition~\ref{compatibility}, Corollary~\ref{corollarycompativilityofintertwiningmorrisviatp}, and Lemma~\ref{independencyofP},
for any $\theta_{m} \in \mathbb{C}[M_{\sigma}/M^1]$, we have
\begin{align*}
&T_{\rho_{M}} \circ I_{U'} \circ \theta_{w^{-1}B \mid B} \circ 
\left(
T_{\rho_{M}} \circ I_{U}
\right)^{-1}
 \circ J_{P' \mid P}(\sigma \otimes \cdot)^{-1}
\circ I_{P'}^{G}(\theta_{m})\\
&= T_{\rho_{M}} \circ I_{U'} \circ \theta_{w^{-1}B \mid B} \circ 
\left(
T_{\rho_{M}} \circ I_{U}
\right)^{-1}
 \circ I_{P}^{G}(\theta_{m}) \circ J_{P' \mid P}(\sigma \otimes \cdot)^{-1}\\
&=
T_{\rho_{M}} \circ I_{U'} \circ \theta_{w^{-1}B \mid B} \circ t_{P, B}(\theta_{m}) \circ \left(
T_{\rho_{M}} \circ I_{U}
\right)^{-1}
 \circ J_{P' \mid P}(\sigma \otimes \cdot)^{-1}\\
&=
c(w, \theta_{m}) \cdot T_{\rho_{M}} \circ I_{U'} \circ t_{P, w^{-1}B}(\theta_{m}) \circ \theta_{w^{-1}B \mid B} \circ 
\left(
T_{\rho_{M}} \circ I_{U}
\right)^{-1}
 \circ J_{P' \mid P}(\sigma \otimes \cdot)^{-1}\\
&= 
c(w, \theta_{m}) \cdot T_{\rho_{M}} \circ I_{U'} \circ t_{P', w^{-1}B}(\theta_{m}) \circ \theta_{w^{-1}B \mid B} \circ 
\left(
T_{\rho_{M}} \circ I_{U}
\right)^{-1}
 \circ J_{P' \mid P}(\sigma \otimes \cdot)^{-1}\\
&=
c(w, \theta_{m}) \cdot  I_{P'}^{G}(\theta_{m}) \circ T_{\rho_{M}} \circ I_{U'} \circ \theta_{w^{-1}B \mid B} \circ 
\left(
T_{\rho_{M}} \circ I_{U}
\right)^{-1}
 \circ J_{P' \mid P}(\sigma \otimes \cdot)^{-1}.
\end{align*}
Hence, the element
\[
T_{\rho_{M}} \circ I_{U'} \circ \theta_{w^{-1}B \mid B} \circ 
\left(
T_{\rho_{M}} \circ I_{U}
\right)^{-1}
 \circ J_{P' \mid P}(\sigma \otimes \cdot)^{-1} \in \End_{G(F)}\left(
I_{P'}^{G} \left(
\ind_{M^{1}}^{M(F)} (\sigma_{1})
\right)
\right)^{\times}
\]
commutes with any element $\theta_{m} \in \mathbb{C}[M_{\sigma}/M^1]$ up to a constant.
According to \cite[Theorem~10.6 (a)]{MR4432237}, any element of
\[
\End_{G(F)}\left(
I_{P'}^{G} \left(
\ind_{M^{1}}^{M(F)} (\sigma_{1})
\right)
\right) \otimes_{\mathbb{C}[M_{\sigma}/M^1]} \mathbb{C}(M_{\sigma}/M^1)
\]
that commutes with any element $\theta_{m} \in \mathbb{C}[M_{\sigma}/M^1]$ up to a constant is contained in $\mathbb{C}(M_{\sigma}/M^1)$.
Thus, we obtain that
\begin{multline*}
T_{\rho_{M}} \circ I_{U'} \circ \theta_{w^{-1}B \mid B} \circ 
\left(
T_{\rho_{M}} \circ I_{U}
\right)^{-1}
 \circ J_{P' \mid P}(\sigma \otimes \cdot)^{-1}\\
 \in
\End_{G(F)}\left(
I_{P'}^{G} \left(
\ind_{M^{1}}^{M(F)} (\sigma_{1})
\right)
\right)^{\times} \cap \mathbb{C}(M_{\sigma}/M^1) =
 \mathbb{C}[M_{\sigma}/M^1]^{\times}.
\end{multline*}
\end{proof}
We use the same symbol $J_{P' \mid P}(\sigma \otimes \cdot)$ for the map
\[
\End_{G(F)}\left(
I_{P}^{G} \left(
\ind_{M^{1}}^{M(F)} (\sigma_{1})
\right)
\right) \rightarrow \End_{G(F)}\left(
I_{P'}^{G} \left(
\ind_{M^{1}}^{M(F)} (\sigma_{1})
\right)
\right)
\]
induced by the isomorphism
\[
J_{P' \mid P}(\sigma \otimes \cdot) \colon I_{P}^{G} \left(
\ind_{M^{1}}^{M(F)} (\sigma_{1})
\right) \rightarrow I_{P'}^{G} \left(
\ind_{M^{1}}^{M(F)} (\sigma_{1})
\right).
\]
\begin{corollary}
\label{commutativediagramintertwiningreduction}
We have the following commutative diagram:
\[
\xymatrix@R+1pc@C+2pc{
\End_{G(F)}\left(
\ind_{P_{J, B}}^{G(F)} (\rho)
\right) \ar[d]_-{\theta_{w^{-1}B \mid B}} \ar[r]^-{T_{\rho_{M}} \circ I_{U}} \ar@{}[dr]|\circlearrowleft & \End_{G(F)}\left(
I_{P}^{G} \left(
\ind_{M^{1}}^{M(F)} (\sigma_{1})
\right)
\right) \ar[d]^-{\Ad\left(b(w, P')\right) \circ J_{P' \mid P}(\sigma \otimes \cdot)}
\\
\End_{G(F)}\left(
\ind_{P_{J, w^{-1}B}}^{G(F)} (\rho)
\right) \ar[r]^-{T_{\rho_{M}} \circ I_{U'}} & \End_{G(F)}\left(
I_{P'}^{G} \left(
\ind_{M^{1}}^{M(F)} (\sigma_{1})
\right)
\right),
}
\]
where $\Ad\left(b(w, P')\right)$ denotes the conjugation by $b(w, P')$ on $\End_{G(F)}\left(
I_{P'}^{G} \left(
\ind_{M^{1}}^{M(F)} (\sigma_{1})
\right)
\right)$.
\end{corollary}
Now, we drop the conditions that $a \in B$ and $M_{\alpha}$ is a standard Levi subgroup with respect to $P$ in Theorem~\ref{maintheoremisomofaffineheckerewrite}:
\begin{theorem}
\label{maintheoremisomofaffineheckerewritewithoutassumption}
Let
\[
s = s_{\alpha} \in W_{0}(R^{\Mor}) = W_{0}(R^{\Sol})
\]
be the simple reflection associated with an element $\alpha \in \in \Delta_{\mathfrak{s}_{M}, \mu}(P)$.
Then, we have
\begin{align*}
\left(
T_{\rho_{M}} \circ I_{U}
\right)(\Phi_{s}) =
\begin{cases}
q_{F}^{\lambda^{\Sol}(\alpha')} -1 - T'_{s} & (\epsilon_{\alpha} = 0), \\
- q_{F}^{\left(-\lambda^{\Sol}(\alpha') + (\lambda^{*})^{\Sol}(\alpha')\right)/2} \cdot \theta_{- (\alpha')^{\vee}} T'_{s} & (\epsilon_{\alpha} =1).
\end{cases}
\end{align*}
\end{theorem}
\begin{comment}
\begin{align}
\label{maintheoremisomofaffineheckerewriteagain}
\left(
T_{\rho_{M}} \circ I_{U}
\right)(\Phi_{s, B}) =
\begin{cases}
q_{F}^{\lambda^{\Sol}(\alpha')} -1 - T'_{s, P} & (\epsilon_{\alpha} = 0)\\
- q_{F}^{\left(-\lambda^{\Sol}(\alpha') + (\lambda^{*})^{\Sol}(\alpha')\right)/2} \cdot T'_{s, P} \theta_{- \alpha'} & (\epsilon_{\alpha} =1).
\end{cases}
\end{align}
\end{comment}
\begin{proof}
Since $a \in w^{-1}B$, and $M_{\alpha}$ is standard with respect to $P'$, Theorem~\ref{maintheoremisomofaffineheckerewrite} implies that
\begin{align}
\label{maintheoremisomofaffineheckerewritestandardcase}
\left(
T_{\rho_{M}} \circ I_{U'}
\right)(\Phi_{s, w^{-1}B}) =
\begin{cases}
q_{F}^{\lambda^{\Sol}(\alpha')} -1 - T'_{s, P'} & (\epsilon_{\alpha} = 0), \\
- q_{F}^{\left(-\lambda^{\Sol}(\alpha') + (\lambda^{*})^{\Sol}(\alpha')\right)/2} \cdot \theta_{- (\alpha')^{\vee}} T'_{s, P'} & (\epsilon_{\alpha} =1).
\end{cases}
\end{align}
According to Corollary~\ref{commutativediagramintertwiningreduction}, Corollary~\ref{corollaryforrgrouppropositionintertwiningmorriswjrho}, equation~\eqref{maintheoremisomofaffineheckerewritestandardcase}, 
 equation~\eqref{harishchadravselementsofB}, and Lemma~\ref{lemmareductionintertwiningsolleveldside}, we have
\begin{align*}
\left(
T_{\rho_{M}} \circ I_{U}
\right)(\Phi_{s, B}) 
&= 
\left(
\left(
\Ad\left(b(w, P')\right) \circ J_{P' \mid P}(\sigma \otimes \cdot)
\right)^{-1} \circ T_{\rho_{M}} \circ I_{U'} \circ \theta_{w^{-1}B \mid B}
\right)(\Phi_{s, B}) \\
&=
\left(
\left(
\Ad\left(b(w, P')\right) \circ J_{P' \mid P}(\sigma \otimes \cdot)
\right)^{-1} \circ T_{\rho_{M}} \circ I_{U'}
\right)(\Phi_{s, w^{-1}B})\\
&=
\begin{cases}
\left(
\Ad\left(b(w, P')\right) \circ J_{P' \mid P}(\sigma \otimes \cdot)
\right)^{-1} \left(
q_{F}^{\lambda^{\Sol}(\alpha')} -1 - T'_{s, P'}
\right) & (\epsilon_{\alpha} = 0), \\
\left(
\Ad\left(b(w, P')\right) \circ J_{P' \mid P}(\sigma \otimes \cdot)
\right)^{-1} \left(
- q_{F}^{\left(-\lambda^{\Sol}(\alpha') + (\lambda^{*})^{\Sol}(\alpha')\right)/2} \cdot \theta_{- (\alpha')^{\vee}} T'_{s, P'}
\right) & (\epsilon_{\alpha} =1)
\end{cases}\\
&=
\begin{cases}
\left(
J_{P' \mid P}(\sigma \otimes \cdot) \circ \Ad\left(b(w, P')\right)
\right)^{-1} \left(
q_{F}^{\lambda^{\Sol}(\alpha')} -1 - T'_{s, P'}
\right) & (\epsilon_{\alpha} = 0), \\
\left(
J_{P' \mid P}(\sigma \otimes \cdot) \circ \Ad\left(b(w, P')\right)
\right)^{-1} \left(
- q_{F}^{\left(-\lambda^{\Sol}(\alpha') + (\lambda^{*})^{\Sol}(\alpha')\right)/2} \cdot \theta_{- (\alpha')^{\vee}} T'_{s, P'}
\right) & (\epsilon_{\alpha} =1)
\end{cases}\\
&=
\begin{cases}
\Ad\left(b(w, P')\right)^{-1} \left(
q_{F}^{\lambda^{\Sol}(\alpha')} -1 - T'_{s, P}
\right) & (\epsilon_{\alpha} = 0), \\
\Ad\left(b(w, P')\right)^{-1} \left(
- q_{F}^{\left(-\lambda^{\Sol}(\alpha') + (\lambda^{*})^{\Sol}(\alpha')\right)/2} \cdot \theta_{- (\alpha')^{\vee}} T'_{s, P} 
\right) & (\epsilon_{\alpha} =1).
\end{cases}
\end{align*}
We regard $b(w, P')$ as an element of $\End_{G(F)}\left(
I_{P}^{G} \left(
\ind_{M^{1}}^{M(F)} (\sigma_{1})
\right)
\right)^{\times}$ in the last two terms.
Since $b(w, P') \in \mathbb{C}[M_{\sigma}/M^1]^{\times}$, and $M_{\sigma}/M^1$ is a free $\mathbb{Z}$-module of finite rank, we can write 
\[
b(w, P')^{-1} = c \cdot \theta_{m}
\]
for some $c \in \mathbb{C}^{\times}$ and $m \in M_{\sigma}/M^1$.
Then, we have
\begin{align*}
\Ad\left(b(w, P')\right)^{-1} \left(T'_{s, P}\right)
&=\theta_{m} \cdot T'_{s, P} \cdot \theta_{-m}\\
&= \theta_{m}
\left(
\theta_{- s(m)} T'_{s, P} - \left(
\theta_{- s(m)} T'_{s, P} - T'_{s, P} \theta_{-m}
\right)
\right)\\
&= \theta_{m - s(m)} T'_{s, P} - \theta_{m} \left(
\theta_{-s(m)} T'_{s, P} - T'_{s, P} \theta_{- m}
\right)\\
& \in \theta_{m - s(m)} T'_{s, P} + \mathbb{C}[M_{\sigma}/M^1].
\end{align*}
Thus, we have
\[
\left(
T_{\rho_{M}} \circ I_{U}
\right)(\Phi_{s, B}) 
\in
\begin{cases}
q_{F}^{\lambda^{\Sol}(\alpha')} -1 - \theta_{m - s(m)} T'_{s, P} + \mathbb{C}[M_{\sigma}/M^1]
& (\epsilon_{\alpha} = 0), \\
- q_{F}^{\left(-\lambda^{\Sol}(\alpha') + (\lambda^{*})^{\Sol}(\alpha')\right)/2} \cdot \theta_{- (\alpha')^{\vee}}\theta_{m - s(m)} T'_{s, P} + \mathbb{C}[M_{\sigma}/M^1]
& (\epsilon_{\alpha} =1).
\end{cases}
\]
We note that when $\epsilon_{\alpha} = 1$, we have $(\alpha')^{\vee} = h_{\alpha}^{\vee}$ (see the last paragraph of Section~\ref{Comparison of Morris and Solleveld's endomorphism algebras : maximal case}).
Then, according to Lemma~\ref{lemmaforcomparisonofmorrisandsolleveldkeypropositionrank1}, we obtain that $\theta_{m - s(m)} \in \mathbb{C}^{\times}$, hence $m = s(m)$.
Therefore, we obtain
\begin{align*}
\Ad\left(b(w, P')\right)^{-1} \left(T'_{s, P}\right)
&= \theta_{m - s(m)} T'_{s, P} - \theta_{m} \left(
\theta_{-s(m)} T'_{s, P} - T'_{s, P} \theta_{- m}
\right)\\
&= T'_{s, P},
\end{align*}
and
\begin{align*}
\left(
T_{\rho_{M}} \circ I_{U}
\right)(\Phi_{s}) =
\begin{cases}
q_{F}^{\lambda^{\Sol}(\alpha')} -1 - T'_{s} & (\epsilon_{\alpha} = 0), \\
- q_{F}^{\left(-\lambda^{\Sol}(\alpha') + (\lambda^{*})^{\Sol}(\alpha')\right)/2} \cdot \theta_{- (\alpha')^{\vee}} T'_{s} & (\epsilon_{\alpha} =1).
\end{cases}
\end{align*}
\end{proof}
\appendix
\section{Subsets of a set of simple affine roots}
\label{Subsets of a set of simple affine roots}
Let $E$ be a real Euclidean space of finite dimension.
Let $V$ denote its vector space of translations and $A'$ denote the vector space of affine-linear functions on $E$.
%We define the dual space $V^{*}$ of $V$ as
%\[
%V^{*} = \Hom_{\mathbb{R}}(V, \mathbb{R}).
%\]
Let $\Phi_{\aff} \subset A'$ be an affine root system on $E$ \cite[Section~2]{MR357528}.
%For simplicity, we assume that $\Phi_{\aff}$ is irreducible.
We fix a chamber $C$ of $\Phi_{\aff}$, and let $B$ denote the corresponding basis of $\Phi_{\aff}$.
For an affine root $a$, let $Da$ denote its gradient, that is, 
\[
Da \colon V \rightarrow \mathbb{R}
\]
is a linear function such that
\[
a(x+v) = a(x) + (Da)(v)
\]
for all $x \in E$ and $v \in V$.
For a subset $\Psi \subset \Phi_{\aff}$, we write
\[
D\Psi = \{
Da \mid a \in \Psi
\}.
\]
We define
\[
H_{a} = \{
x \in E \mid a(x) = 0
\}
\]
and
\[
H_{a}^{+} = \{
x \in E \mid a(x) >0
\}.
\]
We also use similar notation as above for other affine root systems below.

Let $J \subset B$ such that $DJ$ is linearly independent.
We define a subspace $V^{J}$ of $V$ as
\[
V^{J} = \{
v \in V \mid \alpha(v) = 0 \ (\alpha \in DJ)
\}.
\]
We write 
\[
E_{J} = E/V^{J},
\]
that is an affine space with the vector space of translations $V/V^{J}$.
Since $DJ$ is linearly independent, we have
\[
\dim(E_{J}) = \abs{DJ}.
\]
Let $(V^{J})^{\perp}$ denote the orthogonal complement of $V^{J}$ in $V$.
Then, the natural projection $V \to V/V^{J}$ restricts to an isomorphism
\begin{align}
\label{orthogonalcomplement}
(V^{J})^{\perp} \rightarrow V/V^{J}.
\end{align}
We define an inner product on $(V^{J})^{\perp}$ as the restriction of the inner product on $V$.
We also define an inner product on $V/V^{J}$ by transporting the inner product on $(V_{J})^{\perp}$ via \eqref{orthogonalcomplement}.
Then, $E_{J}$ is a real Euclidean space.
Let $\mathbb{R} \cdot(DJ)$ denote the $\mathbb{R}$-span of $DJ$ in $V^{*}$.
We define
\[
(\Phi_{\aff})_{J} = \{
a \in \Phi_{\aff} \mid Da \in \mathbb{R} \cdot(DJ)
\}.
\]
Then, $(\Phi_{\aff})_{J}$ is an affine root system on $E_{J}$.
\begin{lemma}
\label{lemmabasisofsimplerootsoflevi}
There exists a basis $B_{J}$ of $(\Phi_{\aff})_{J}$ such that $J \subset B_{J}$.
\end{lemma}
%The proof of Lemma~\ref{lemmabasisofsimplerootsoflevi} is essentially same as \cite[Chapter~V, Section~3.10, Proposition~11]{MR0240238}.
\begin{proof}
It suffices to show that there exists a chamber $C_{J}$ of $E_{J}$ such that $H_{a}$ is a wall of $C_{J}$ for any $a \in J$.
Let 
\[
E^{J} = \{
x \in E \mid a(x) = 0 \ (a \in J)
\}.
\]
\begin{claim}
\label{claimbasisofkernel}
Let $a \in \Phi_{\aff}$ be an affine root such that $a(x) = 0$ for any $x \in E^{J}$.
Then, we can write
\[
a = \sum_{b_i \in J} c_{i} b_{i}
\]
with rational integer coefficients $c_{i}$ which are all non-negative or non-positive.
\end{claim}
\begin{proof}
Since $B$ is a basis of $\Phi_{\aff}$, we can write
\[
a = \sum_{b_{i} \in B} c_{i}b_{i}
\]
with rational integer coefficients $c_{i}$ which are all non-negative or non-positive.
We will prove that $c_{i}=0$ unless $b_{i} \in J$.
Assume that $c_{i} \neq 0$ for some $b_{i} \in B \backslash J$.
Since $B$ is a basis of $\Phi_{\aff}$ corresponding to the chamber $C$, there exists a vertex $x_{i}$ of $C$ such that 
\[
b_{i}(x_{i}) > 0
\]
and 
\[
b_{j}(x_{i}) = 0
\]
for any $j \neq i$.
Since $b_{i} \not\in J$, the second equation implies $x_{i} \in E^{J}$, hence $a(x_{i}) = 0$.
However, we can calculate as
\[
a(x_{i}) = \sum_{b_{j} \in B} c_{j}b_{j}(x_{i}) = c_{i}b_{i}(x_{i}) \neq 0,
\]
a contradiction.
\end{proof}
Let $E^{J}_{J}$ denote the image of $E^{J}$ on $E_{J}$, that is,
\[
E^{J}_{J} = \{
x \in E_{J} \mid a(x) = 0 \ (a \in J)
\}.
\]
Since $DJ$ is linearly independent and $\abs{DJ}$ is equal to the dimension of $E_{J}$, $E^{J}_{J}$ is a singleton.
We write
\[
E^{J}_{J} = \{
x_{J}
\}.
\]
We also write 
\[
E_{J}^{+} = \{
x \in E_{J} \mid a(x) >0 \ (a \in J)
\}.
\]
We take an open ball $U$ in $E_{J}$ whose center is $x_{J}$ such that 
\[
U \cap H_{a} = \emptyset
\]
for any $a \in (\Phi_{\aff})_{J}$ satisfying $a(x_{J}) \neq 0$.
Since $x_{J} \in \overline{E_{J}^{+}}$, we can take an element $y \in U \cap E_{J}^{+}$.
Since $y \in U$, $a(y) \neq 0$ for any $a \in (\Phi_{\aff})_{J}$ satisfying $a(x_{J}) \neq 0$.
On the other hand, we can prove that $a(y) \neq 0$ for any $a \in (\Phi_{\aff})_{J}$ satisfying $a(x_{J}) = 0$ as follows.
Let $a \in (\Phi_{\aff})_{J} \subset \Phi_{\aff}$ be an affine root such that $a(x_{J}) = 0$.
As an affine function on $E$, $a(x)=0$ for any $x \in E^{J}$.
Then, Claim~\ref{claimbasisofkernel} implies that we can write
\[
a = \sum_{b_i \in J} c_{i} b_{i}
\]
with rational integer coefficients $c_{i}$ which are all non-negative or non-positive.
Since $y \in E_{J}^{+}$, 
\[
b_{i}(y) > 0
\]
for any $b_{i} \in J$.
Hence, 
\[
a(y) = \sum_{b_i \in J} c_{i} b_{i}(y) 
\begin{cases}
>0 \ &(\text{$c_{i} \ge 0$ for all $i$}), \\
<0  \ &(\text{$c_{i} \le 0$ for all $i$}).
\end{cases}
\]
Here, we note that at least one $c_{i}$ is nonzero since $a$ is an affine root.
Thus, we conclude that $a(y) \neq 0$ for any $a \in (\Phi_{\aff})_{J}$, hence $y$ is in a chamber $C_{J}$ of $(\Phi_{\aff})_{J}$.
Since $C_{J}$ is a chamber of $(\Phi_{\aff})_{J}$ containing $y$, and $y \in E_{J}^{+}$, $C_{J}$ is contained in $E_{J}^{+}$.
On the other hand, since $U \cap E_{J}^{+}$ is a convex subset of 
\[
E_{J} \backslash \bigcup_{a \in (\Phi_{\aff})_{J}} H_{a}
\]
containing $y$,
$U \cap E_{J}^{+}$ is contained in $C_{J}$.

We will prove that $H_{a}$ is a wall of $C_{J}$ for any $a \in J$.
Let $a \in J$.
It suffices to show that there exists $z \in H_{a}$ and an open neighborhood $W$ of $z$ in $E_{J}$ such that
\[
W \cap C_{J} = W \cap H_{a}^{+}.
\]
We write
\[
H_{J, a}^{+} =
\{
x \in E_{J} \mid a(x) = 0, \ b(x) > 0 \ (b \in J \backslash\{a\})
\}.
\]
%that is an open subset of $H_{a}$.
Since $x_{J} \in \overline{H_{J, a}^{+}}$, we can take an element $z \in U \cap H_{J, a}^{+}$.
We take an open neighborhood $W$ of $z$ in $E_{J}$ as
\[
W = U \cap \left(\bigcap_{b \in J \backslash\{a\}} H_{b}^{+}\right).
\]
Then,
\[
W \cap E_{J}^{+} = W \cap H_{a}^{+}. 
\]
Since 
\[
U \cap E_{J}^{+} \subset C_{J} \subset E_{J}^{+},
\]
we conclude
\[
W \cap C_{J} = W \cap U \cap C_{J} = W \cap U \cap E_{J}^{+} = W \cap E_{J}^{+} = W \cap H_{a}^{+}.
\]
%that implies $H_{a}$ is a wall of $C_{J}$.
\end{proof}

We will prove a ``converse'' of Lemma~\ref{lemmabasisofsimplerootsoflevi}.
Let $J' \subset \Phi_{\aff}$ such that $DJ'$ is linearly independent.
Here, we do not assume that $J'$ is a subset of a basis of $\Phi_{\aff}$.
We define $V^{J'}, E_{J'}$ and $(\Phi_{\aff})_{J'}$ as above.
\begin{lemma}
\label{converselemmabasisofsimplerootsoflevi}
Suppose that there exists a basis $B_{J'}$ of $(\Phi_{\aff})_{J'}$ containing $J'$.
Then, there exists a basis $B'$ of $\Phi_{\aff}$ containing $J'$.
\end{lemma}
\begin{proof}
We will prove that there exists a chamber $C'$ of $\Phi_{\aff}$ such that $H_{b}$ is a wall of $C'$ for any $b \in J'$.
We write
\[
E^{J'} =
\{
x \in E \mid b(x) = 0 \ (b \in J')
\}.
\]
\begin{claim}
\label{claimbasisconverse}
Let $a \in \Phi_{\aff}$ be an affine root such that $a(x) = 0$ for any $x \in E^{J'}$.
Then, we can write
\[
a = \sum_{b_{i} \in J'} c_{i} b_{i}
\]
with rational integer coefficients $c_{i}$ which are all non-negative or non-positive.
\end{claim}
\begin{proof}
Since $E^{J'}$ is stable under the translation by $V^{J'}$, $Da$ vanishes on $V^{J'}$.
Hence, we obtain  $Da \in \mathbb{R} \cdot (DJ')$, equivalently, we have $a \in (\Phi_{\aff})_{J'}$.
Since $B_{J'}$ is a basis of $(\Phi_{\aff})_{J'}$, we can write
\[
a = \sum_{b_{i} \in B_{J'}} c_{i} b_{i}
\]
with rational integer coefficients $c_{i}$ which are all non-negative or non-positive.
We will prove that $c_{i}=0$ unless $b_{i} \in J'$.
Assume that $c_{i} \neq 0$ for some $b_{i} \in B_{J'} \backslash J'$.
We take a vertex $x_{i}$ in the chamber in $E_{J'}$ corresponding to the basis $B_{J'}$
such that 
\[
b_{i}(x_{i}) > 0
\]
and 
\[
b_{j}(x_{i}) = 0
\]
for any $j \neq i$.
We identify $x_{i}$ with its lift in $E$.
Since $b_{i} \not\in B_{J'}$, the second equation implies $x_{i} \in E^{J'}$, hence $a(x_{i}) = 0$.
However, we can calculate as
\[
a(x_{i}) = \sum_{b_{j} \in B_{J'}} c_{j} b_{j}(x_{i}) = c_{i}b_{i}(x_{i}) \neq 0,
\]
a contradiction.
\end{proof}
We also write
\[
E^{J', +} = \{
x \in E \mid b(x) >0 \ (b \in J')
\}.
\]
Let $x_{J'} \in E^{J'}$ be a point such that
\[
a(x_{J'}) \neq 0
\]
for any $a \in \Phi_{\aff}$ with
\[
E^{J'} \not\subset H_{a}.
\]
We take an open ball $U$ in $E$ whose center is $x_{J'}$ such that 
\[
U \cap H_{a} = \emptyset
\]
for any $a \in \Phi_{\aff}$ with
\[
E^{J'} \not\subset H_{a}.
\]
Since $x_{J'} \in \overline{E^{J', +}}$, we can take an element $y \in U \cap E^{J', +}$.
Since $y \in U$, $a(y) \neq 0$ for any $a \in \Phi_{\aff}$ with
\[
E^{J'} \not\subset H_{a}.
\]
Moreover, we can prove that $a(y) \neq 0$ for any $a \in \Phi_{\aff}$ with 
\[
E^{J'} \subset H_{a}
\]
as follows.
Let $a \in \Phi_{\aff}$ be an affine root such that $a(x) = 0$ for any $x \in E^{J'}$.
Then, Claim~\ref{claimbasisconverse} implies that we can write
\[
a = \sum_{b_{i} \in J'} c_{i} b_{i}
\]
with rational integer coefficients $c_{i}$ which are all non-negative or non-positive.
Since $y \in E^{J', +}$, 
\[
b_{i}(y) > 0
\]
for any $b_{i} \in J'$.
Hence, 
\[
a(y) = \sum_{b_{i} \in J'} c_{i} b_{i}(y) 
\begin{cases}
>0 \ &(\text{$c_{i} \ge 0$ for all $i$}), \\
<0  \ &(\text{$c_{i} \le 0$ for all $i$}).
\end{cases}
\]
Here, we note that at least one $c_{i}$ is nonzero since $a$ is an affine root.
Thus, we conclude that $a(y) \neq 0$ for any $a \in \Phi_{\aff}$, hence $y$ is contained in a chamber $C'$ of $\Phi_{\aff}$.

Then, the same argument as the proof of Lemma~\ref{lemmabasisofsimplerootsoflevi} implies that $H_{b}$ is a wall of $C'$ for any $b \in J'$.
\end{proof}
\begin{corollary}
\label{corollarySubsets of a set of simple affine roots}
Let $J \subset B$ such that $DJ$ is linearly independent, and
$\Phi'_{\aff}$ be a subsystem of $\Phi_{\aff}$ containing $(\Phi_{\aff})_{J}$.
Then, there exists a basis of $\Phi'_{\aff}$ containing $J$.
\end{corollary}
\begin{proof}
According to Lemma~\ref{lemmabasisofsimplerootsoflevi}, there exists a basis $B_{J}$ of $(\Phi_{\aff})_{J}$ containing $J$.
On the other hand, replacing $\Phi_{\aff}$ with $\Phi'_{\aff}$ and taking $J' = J$ in Lemma~\ref{converselemmabasisofsimplerootsoflevi}, we obtain that there exists a basis of $\Phi'_{\aff}$ containing $J$.
\end{proof}
\begin{comment}
\begin{corollary}
\label{corollarySubsets of a set of simple affine roots}
Let $\Phi'_{\aff}$ be subsystem of $\Phi_{\aff}$ containing $(\Phi_{\aff})_{J}$.
Let $B'$ be a basis of $\Phi'_{\aff}$, and let $J' \subset B'$ such that $D J'$ is linearly independent.
Then, there exists a basis of $\Phi_{\aff}$ containing $J'$.
\end{corollary}
\begin{proof}
Replacing $\Phi_{\aff}$ with $\Phi'_{\aff}$ and 
\end{proof}
\end{comment}
\section{Iwahori-Hecke algebras and affine Hecke algebras}
\label{Iwahori-Hecke algebras and affine Hecke algebras}
In this appendix, we explain the definitions of Iwahori-Hecke algebras and affine Hecke algebras following \cite{MR4310011}.

First, we recall the definition of Iwahori-Hecke algebras of affine type \cite[Section~1.2]{MR4310011}.
Let $E$ be a real Euclidean space of finite dimension, and let $V$ denote its vector space of translations.
Let $A'$ denote the vector space of affine-linear functions on $E$.
Let $\Phi_{\aff} \subset A'$ be an affine root system on $E$.
For simplicity, we assume that $\Phi_{\aff}$ is irreducible and reduced.
We use the same notation as Appendix~\ref{Subsets of a set of simple affine roots}.
For $a \in \Phi_{\aff}$, let $s_{a}$ denote the corresponding reflection on $E$, and let $W_{\aff} = W_{\aff}(\Phi_{\aff})$ denote the affine Weyl group of $\Phi_{\aff}$.
Hence, $W_{\aff}$ is generated by $s_{a} \ (a \in \Phi_{\aff})$.
The group $W_{\aff}$ also acts on $A'$ as
\[
(w(f))(x) = f(w^{-1}(x))
\]
for $w \in W_{\aff}$, $f \in A'$, and $x \in E$, and the action stabilizes $\Phi_{\aff}$.
We define the derivative $Dw$ of an element $w \in W_{\aff}$ as the linear map
\[
Dw \colon V \rightarrow V
\]
such that
\[
w(x+v) = w(x) + (Dw)(v)
\]
for all $x \in E$ and $v \in V$.
According to \cite[(1.5)]{MR357528}, for $a \in \Phi_{\aff}$, we have 
\[
D s_{a} = s_{Da},
\]
where $s_{Da}$ denotes the reflection on $V$ with respect to $Da$.

We fix a chamber $C$ of $\Phi_{\aff}$, and let $B$ denote the corresponding basis of $\Phi_{\aff}$.
We define a subset $S_{\aff} = S(\Phi_{\aff}, B)$ of $W_{\aff}$ as 
\[
S_{\aff} = \{
s_{b} \mid b \in B
\}.
\]
Then, $\left(W_{\aff}, S_{\aff}\right)$ is a Coxeter system of affine type.
Let $l$ denote the length function on $W_{\aff}$ with respect to $S_{\aff}$.
 
Let
\[
q \colon S_{\aff} \rightarrow \mathbb{C}
\]
be a function $s \mapsto q_{s}$ such that
\begin{align}
\label{winvarianceofqparameter}
\text{$q_{s_{1}} = q_{s_{2}}$ if $s_1, s_2 \in S_{\aff}$ are conjugate in $W_{\aff}$}.
\end{align}
For $w \in W_{\aff}$ with a reduced expression
\[
w = s_1 s_2 \cdots s_r \ (s_i \in S_{\aff}),
\]
we put
\[
q_{w} = q_{s_1} q_{s_2} \cdots q_{s_r}.
\]
Condition~\eqref{winvarianceofqparameter} implies that $q_{w}$ is well-defined.

The Iwahori-Hecke algebra $\mathcal{H}(W_{\aff}, q)$ associated with the Coxeter system $\left(W_{\aff}, S_{\aff} \right)$ and the parameter function $q$ is the unique $\mathbb{C}$-algebra with generators
\[
\left\{
T_{s} \mid s \in S_{\aff}
\right\}
\]
and relations:
\begin{description}
\item[Quadratic relations]
For all $s \in S_{\aff}$, we have
\[
(T_{s} + 1)(T_{s} - q_{s}) = 0.
\]
\item[Braid relations]
\label{Iwahoriheckebraidrelation}
For all $s, t \in S_{\aff}$ such that the order of $st$ in $W_{\aff}$ is $m < \infty$, we have
\[
\underbrace{
T_{s} T_{t} T_{s} \cdots
}_{m \terms}
= \underbrace{
T_{t} T_{s} T_{t} \cdots
}_{m \terms}.
\]
\end{description}
For $w \in W_{\aff}$ with a reduced expression
\[
w = s_1 s_2 \cdots s_r \ (s_i \in S_{\aff}),
\]
we put
\[
T_{w} = T_{s_{1}} T_{s_{2}} \cdots T_{s_{r}}.
\]
Relation~\eqref{Iwahoriheckebraidrelation} above implies that $T_{w}$ is well-defined.
Moreover, the set
\[
\{
T_{w} \mid w \in W_{\aff}
\}
\] 
is a vector space basis of $\mathcal{H}(W_{\aff}, q)$.
%For more details, see \cite[Section~1.2]{MR4310011}.

Next, we recall the definition of affine Hecke algebras \cite[Section~1.3]{MR4310011}.
Let
\[
\mathcal{R} =
(X, R, Y, R^{\vee}, \Delta)
\]
be a based root datum, that is,
\begin{itemize}
\item $X$ and $Y$ are free $\mathbb{Z}$-module of finite rank, with a perfect pairing 
\[
\langle , \rangle \colon X \times Y \rightarrow \mathbb{Z},
\]
\item $R$ is a reduced root system in $X$,
\item $R^{\vee}$ is the dual root system of $R$ in $Y$, with a bijection
\[
R \rightarrow R^{\vee}, \ \alpha \mapsto \alpha^{\vee}
\]
such that
\[
\langle \alpha, \alpha^{\vee} \rangle = 2,
\]
\item $\Delta$ is a basis of $R$.
\end{itemize}
For $\alpha \in R$, let
\[
s_{\alpha} \colon Y \rightarrow Y
\]
denote the reflection
\[
y \mapsto y - \langle \alpha, y \rangle \alpha^{\vee},
\]
that stabilizes $R^{\vee}$.
Let $W_0 = W_{0}(R)$ denote the Weyl group of $R$, that is generated by $s_{\alpha} \ (\alpha \in R)$.
The group $W_{0}$ also acts on $X$ as
\[
\langle w(x), y \rangle = \langle x, w^{-1}y \rangle
\]
for $w \in W_{0}$, $x \in X$, and $y \in Y$, and the action stabilizes $R$.
For $\alpha \in R$ and $x \in X$, we have
\[
s_{\alpha}(x) = x - \langle x, \alpha^{\vee} \rangle \alpha.
\]
The basis $\Delta$ determines a set of simple reflections
\[
S_0 = \{
s_{\alpha} \mid \alpha \in \Delta
\}
\]
in $W_{0}$.
Then, $(W_{0}, S_0)$ is a finite Coxeter system.
We fix a real number $\mathbf{q} >1$, and let
\[
\lambda, \lambda^{*} \colon \Delta \rightarrow \mathbb{C}
\]
be functions such that
\begin{align}
\label{w-equiv}
\text{
if $\alpha, \beta \in \Delta$ are $W_0$-associate, $\lambda(\alpha) = \lambda(\beta)$, and $\lambda^{*}(\alpha) = \lambda^{*}(\beta)$,
}
\end{align}
and 
\begin{align}
\label{lambda=lambda*}
\text{
if $\alpha \not \in 2X$, $\lambda(\alpha) = \lambda^{*}(\alpha)$.
}
\end{align}
For $\alpha \in \Delta$, we define
\[
q_{s_{\alpha}} = \mathbf{q}^{\lambda(\alpha)}.
\]
Then, this parameter function satisfies:
\begin{align*}
\text{$q_{s_{1}} = q_{s_{2}}$ if $s_1, s_2 \in S_{0}$ are conjugate in $W_{0}$}.
\end{align*}
We can define the Iwahori-Hecke algebra $\mathcal{H}(W_0, q)$ associated with the finite Coxeter system $(W_{0}, S_0)$ and the parameter function $q$ exactly as the affine case. 
%Let $\mathbb{C}[Y]$ denote the group algebra of $Y$ over $\mathbb{C}$.
\begin{definition}[{\cite[Definition~1.6]{MR4310011}}]
\label{affinehecke}
The affine Hecke algebra $\mathcal{H}(\mathcal{R}, \lambda, \lambda^{*}, \mathbf{q})$ associated with $\mathcal{R}, \lambda, \lambda^{*}, \mathbf{q}$ is the vector space
\[
\mathbb{C}[Y] \otimes \mathcal{H}(W_0, q)
\]
with the multiplication rules:
\begin{enumerate}
\item $\mathbb{C}[Y]$ and $\mathcal{H}(W_0, q)$ are embedded as subalgebras,
\item 
\label{bernsteinrel}
for $\alpha \in \Delta$ and $y \in Y$,
\[
\theta_{y}T_{s_{\alpha}} - T_{s_{\alpha}}\theta_{s_{\alpha}(y)} =
\left(
(\mathbf{q}^{\lambda(\alpha)} -1) + \theta_{-\alpha^{\vee}}(
\mathbf{q}^{(\lambda(\alpha) + \lambda^{*}(\alpha))/2} - \mathbf{q}^{(\lambda(\alpha) - \lambda^{*}(\alpha))/2}
)
\right)
\frac{
\theta_{y} - \theta_{s_{\alpha}(y)}
}{
\theta_{0} - \theta_{-2\alpha^{\vee}}
}.
\]
\end{enumerate}
\end{definition}
\begin{remark}
The definition of the affine Hecke algebra $\mathcal{H}(\mathcal{R}, \lambda, \lambda^{*}, \mathbf{q})$ above is different from that of \cite[Definition~1.6]{MR4310011}. Our definition of $\mathcal{H}(\mathcal{R}, \lambda, \lambda^{*}, \mathbf{q})$ denotes the affine Hecke algebra of \cite[Definition~1.6]{MR4310011} associated with the dual root datum
\[
\mathcal{R}^{\vee} =
(Y, R^{\vee}, X, R, \Delta^{\vee}),
\]
where $\Delta^{\vee}$ denotes the dual basis
\[
\Delta^{\vee} = \{
\alpha^{\vee} \mid \alpha \in \Delta
\}
\]
of $R^{\vee}$.
\end{remark}

In the end of this appendix, we explain the Bernstein presentation of an Iwahori-Hecke algebra of affine type.
Let $\Phi_{\aff}$ be an irreducible and reduced affine root system on a real Euclidean space $E$ of finite dimension.
We use the same notation as the first part of this appendix.
We defined the Iwahori-Hecke algebra $\mathcal{H}(W_{\aff}, q)$ associated with an affine Coxeter system $\left(W_{\aff}, S_{\aff} \right)$ and a parameter function $q$ there.
We will give a description of $\mathcal{H}(W_{\aff}, q)$ as an affine Hecke algebra.
From now on, we assume that the parameter function $q$ is $\mathbb{R}_{> 0}$-valued.

According to \cite[Proposition~6.1.(1)]{MR357528}, 
%\[
%D\Phi_{\aff} = \{
%Da \mid a \in \Phi_{\aff}
%\}
%\]
$D \Phi_{\aff}$ is a finite root system in $V^{*}$.
Let $(D \Phi_{\aff})^{\vee}$ denote the dual root system of $D \Phi_{\aff}$ in $V$, and let $W_0 = W_{0}(D\Phi_{\aff})$ denote the Weyl group of $D\Phi_{\aff}$.
Then, \cite[Proposition~6.1.(3)]{MR357528} implies that the map
\[
D \colon w \mapsto Dw
\]
defines a homomorphism 
\[
D \colon W_{\aff} \rightarrow W_{0},
\]
and the kernel of $D$ is the subgroup $T$ of translations in $W_{\aff}$.

Let $e$ be a special point for the affine root system $\Phi_{\aff}$ in the sense of  \cite[Section~6]{MR357528} contained in the closure of the chamber $C$.
Let $(\Phi_{\aff})_{e}$ denote the set of affine roots in $\Phi_{\aff}$ that vanish at $e$, and let $(W_{\aff})_{e}$ denote the stabilizer of $e$ in $W_{\aff}$.
We also define
\[
B_{e} = (\Phi_{\aff})_{e} \cap B.
\]
We note that 
\[
B_{e} = B \backslash \{b\},
\]
where $b = b_{e}$ is the unique element of $B$ such that
\[
b (e) > 0
\]
(see \cite[Section~4]{MR357528}).
According to \cite[Proposition~5.1]{MR357528}, $(\Phi_{\aff})_{e}$ is a finite root system with basis $B_{e}$, and $(W_{\aff})_{e}$ is the Weyl group of $(\Phi_{\aff})_{e}$.
According to \cite[Proposition~6.4]{MR357528}, $D (\Phi_{\aff})_{e}$ is the set of indivisible roots of $D \Phi_{\aff}$, and we have an isomorphism of root systems
\[
D \colon (\Phi_{\aff})_{e} \rightarrow D (\Phi_{\aff})_{e}.
\]
Moreover, \cite[Proposition~6.2 (2)]{MR357528} implies that the homomorphism
\[
D \colon W_{\aff} \rightarrow W_{0}
\]
restricts to an isomorphism
\[
D \colon (W_{\aff})_{e} \rightarrow W_{0}.
\]
Hence, we obtain 
\begin{align}
\label{splittingofaffineqeyl}
W_{\aff} = T \rtimes (W_{\aff})_{e}.
%\xrightarrow{(\id, D)} T \rtimes W_{0}.
\end{align}
%Here, we define the action of $W_{0}$ on $T \subset V$ by the natural action of $W_{0}$ on $V$.

%We will define a reduced root system $R = R\left(\Phi_{\aff} , e\right)$ in $V$ with the Weyl group $W_{0}$ such that the $\mathbb{Z}$-span $\mathbb{Z}R$ of $R$ is equal to $T$.
For $a \in (\Phi_{\aff})_{e}$, let $k_{a}$ denote the smallest positive real number such that $a + k_{a} \in \Phi_{\aff}$.
According to \cite[1.9]{MR1235019}, such $k_a$ exists, and we have
\[
\left\{
l \in \mathbb{R} \mid a+l \in \Phi_{\aff}
\right\}
= \left\{
k_{a} n \mid n \in \mathbb{Z}
\right\}.
\]
We define
\[
R = \{
Da/k_{a} \mid a \in (\Phi_{\aff})_{e}
\}.
\]
%We note that
%\[
%a \mapsto Da/k_{a}
%\]
%gives a bijection
%\[
%(\Phi_{\aff})_{e} \rightarrow R.
%\]
The proof of \cite[Chapter~VI, Section~2.5, Proposition~8]{MR0240238} implies that $R$ is a reduced root system in $V^{*}$.
Since each element of $R$ is a scalar multiple of a root in $D \Phi_{\aff}$, and $D (\Phi_{\aff})_{e}$ contains a basis of $D\Phi_{\aff}$, the Weyl group of the root system $R$ is equal to $W_0$.
We define the dual root system $R^{\vee}$ of $R$ in $V$ as
\[
R^{\vee} = \{
k_{a} (Da)^{\vee} \mid a \in (\Phi_{\aff})_{e}
\},
\]
where $(Da)^{\vee}$ denotes the coroot in $(D \Phi_{\aff})^{\vee}$ corresponding to the root $Da \in D\Phi_{\aff}$.
%We note that in \cite[Chapter~VI, Section~2.5, Proposition~8]{MR0240238}, the dual root system
%\[
%R^{\vee} = \{
%Da/k_{a} \mid a \in (\Phi_{\aff})_{e}
%\}
%\]
%of $R$ is considered instead of $R$.
Let $\mathbb{Z}R^{\vee}$ denote the $\mathbb{Z}$-span of $R^{\vee}$ in $V$.
We will prove that $T = \mathbb{Z}R^{\vee}$.
For $a \in (\Phi_{\aff})_{e}$, $a + k_{a}$ is also contained in $\Phi_{\aff}$, hence we have
\[
k_{a} (Da)^{\vee} = s_{a} \circ s_{a+k_a} \in W_{\aff},
\]
where $k_{a} (Da)^{\vee}$ denotes the translation by $k_{a} (Da)^{\vee} \subset V$.
Thus, we obtain that $R^{\vee} \subset T$.
Moreover, the last claim of \cite[Chapter~VI, Section~2.5, Proposition~8]{MR0240238} implies that 
\[
W_{\aff} = \mathbb{Z}R^{\vee} \rtimes (W_{\aff})_{e}.
\]
Comparing it with \eqref{splittingofaffineqeyl}, we obtain $T = \mathbb{Z}R^{\vee}$.

For $\alpha \in R$ and $k \in \mathbb{Z}$, we define a reflection $s_{\alpha + k}$ on $V$ as
\[
s_{\alpha + k}(x) = x - \alpha(x)\alpha^{\vee} -k \alpha^{\vee}
\]
for $x \in V$. 
Here, $\alpha^{\vee} \in R^{\vee}$ denotes the coroot corresponding to $\alpha \in R$.
\begin{comment}
We define
\begin{align*}
R_{\aff} &= \{
\alpha + k \mid \alpha \in R^{\vee}, k \in \mathbb{Z}
\}\\
&= \{

\}
\end{align*}
\end{comment}
We define the affine Weyl group $W_{\aff}(R)$ of $R$ as the group of affine transformations on $V$ generated by $s_{\alpha + k}$ for $\alpha \in R$ and $k \in \mathbb{Z}$ \cite[Chapter~VI, Section~2.1, D{\'e}finition~1]{MR0240238}.
The froup $W_{\aff}(R)$ also acts on the space $A'(V)$ of affine-linear functions on $V$ as
\[
(w(f))(v) = f(w^{-1}(v))
\]
for $w \in W_{\aff}(R)$, $f \in A'(V)$, and $v \in V$.
According to \cite[Chapter~VI, Section~2.1, Proposition~1]{MR0240238}, we obtain
\[
W_{\aff}(R) = \mathbb{Z}R^{\vee} \rtimes W_{0}.
\]
As affine spaces, we have an isomorphism
\begin{align}
\label{affinespaceisom}
E \simeq V
\end{align}
defined as
\[
e+v \mapsto v.
\]
We identify an affine transformation on $E$ with an affine transformation on $V$ via \eqref{affinespaceisom}.
Then, we obtain the isomorphism
\begin{align}
\label{affinesplitting}
W_{\aff} \simeq W_{\aff}(R).
\end{align}
More explicitly, isomorphism~\eqref{affinesplitting} is described as
\[
W_{\aff} = \mathbb{Z}R^{\vee} \rtimes (W_{\aff})_{e} \xrightarrow{(\id, D)} \mathbb{Z}R^{\vee} \rtimes W_{0} = W_{\aff}(R).
\]
We identify $W_{\aff}$ with $W_{\aff}(R)$ via isomorphism~\eqref{affinesplitting}.
In particular, we regard $W_{0}$ as a subgroup of $W_{\aff}$.
We describe the images of simple reflections $s \in S_{\aff}$ via isomorphism~\eqref{affinesplitting}.
For $a \in B_{e}$, the simple reflection $s_{a} \in S_{\aff}$ corresponds to the reflection $s_{Da/k_{a}} \in W_{0}$.
On the other hand, since $s_{b}$ is the reflection via the unique wall of the chamber $C$ that does not contain $e$, the simple reflection $s_{b}$ corresponds to the reflection $s_{1 - \phi}$, where $\phi$ is the highest root of the root system $R$ with respect to the basis
\[
\Delta =
\{
Da/k_{a} \mid a \in B_{e}
\}
\]
(see \cite[Example~4.7]{MR357528}).

We consider a based root datum 
\[
\mathcal{R} = \mathcal{R}(\Phi_{\aff}, e) 
= \left\{
\Hom_{\mathbb{Z}}(\mathbb{Z}R^{\vee}, \mathbb{Z}), R, \mathbb{Z}R^{\vee}, R^{\vee}, \Delta
\right\}.
\]
%where $\Delta$ is a basis of $R$ defined as 
%\[
%\Delta = \{
%k_{a} (Da)^{\vee} \mid a \in B_{e}
%\}.
%\]
We also fix a real number $\mathbf{q} >1$.
\begin{comment}
For $a \in (\Phi_{\aff})_{e}$, we define $s'_{a} \in W_{\aff}$ as
\[
s'_{a}(x) = x - a(x)(Da)^{\vee} + k_{a} (Da)^{\vee}
\]
for $x \in E$.
A simple calculation implies that $s'_{a} = s_{k_{a} - a}$, and it corresponds to $s_{1 - k_{a} (Da)^{\vee}}$ via isomorphism~\eqref{affinesplitting}.
\end{comment}
We define label functions
\[
\lambda, \lambda^{*} \colon \Delta \rightarrow \mathbb{C}
\]
as
\[
\lambda \left(Da/k_{a}\right) = \log(q_{s_{a}})/\log(\mathbf{q}), 
\]
and
\[
\lambda^{*} \left(
Da/k_{a}
\right) =
\begin{cases}
\log(q_{s_{a}})/\log(\mathbf{q}) & \left( Da/k_{a} \not \in 2 \Hom_{\mathbb{Z}}(\mathbb{Z}R^{\vee}, \mathbb{Z}) \right), \\
\log(q_{s_{b}})/\log(\mathbf{q}) & \left( Da/k_{a} \in 2 \Hom_{\mathbb{Z}}(\mathbb{Z}R^{\vee}, \mathbb{Z}) \right)
\end{cases}
\]
for $a \in B_{e}$. 
We note that the condition
\[
Da/k_{a} \in 2 \Hom_{\mathbb{Z}}(\mathbb{Z}R^{\vee}, \mathbb{Z})
\]
holds only when $R$ is of type $A_{1}$, or $R$ is of type $C_{n} \ (n \ge 2)$, and $Da/k_{a}$ is a long root.
The label functions $\lambda, \lambda^{*}$ satisfy condition~\eqref{w-equiv} and condition~\eqref{lambda=lambda*}.
We define the affine Hecke algebra $\mathcal{H}(\mathcal{R}, \lambda, \lambda^{*}, \mathbf{q})$ associated with $\mathcal{R}, \lambda, \lambda^{*}, \mathbf{q}$.
\begin{theorem}[{\cite[Theorem~1.8]{MR4310011}}]
\label{theorem1.8ofsol21}
There exists a unique isomorphism
\[
\mathcal{H}(W_{\aff}, q) \rightarrow \mathcal{H}(\mathcal{R}, \lambda, \lambda^{*}, \mathbf{q})
\]
such that:
\begin{itemize}
\item
that is identity on $\mathcal{H}(W_{0}, q)$,
\item
for $y \in T = \mathbb{Z}R^{\vee} \subset V$ with $(Da)(y) \ge 0$ for all $a \in B_{e}$, it sends $T_{y}$ to $q_{y}^{1/2} \cdot \theta_{y}$.
\end{itemize}
\end{theorem}
\section{An involution of an affine Hecke algebra}
\label{An involution of an affine Hecke algebra}
We use the same notation as Appendix~\ref{Iwahori-Hecke algebras and affine Hecke algebras}.
Let $\mathcal{H} = \mathcal{H}(\mathcal{R}, \lambda, \lambda^{*}, \mathbf{q})$ be the affine Hecke algebra associated with a based root datum 
\[
\mathcal{R} =
(X, R, Y, R^{\vee}, \Delta),
\]
label functions $\lambda, \lambda^{*}$, and a parameter $\mathbf{q}$.
In this appendix, we define a $\mathbb{C}$-algebra automorphism $\iota$ of $\mathcal{H}$.

We define
\[
\iota_{Y} \colon \mathbb{C}[Y] \rightarrow \mathbb{C}[Y]
\]
as
\[
\theta_{y} \mapsto \theta_{-y}.
\]
Since $Y$ is abelian, $\iota_{Y}$ is an algebra automorphism of $\mathbb{C}[Y]$.
We also define
\[
\iota_{0} \colon \mathcal{H}(W_0, q) \rightarrow \mathcal{H}(W_0, q)
\]
as
\[
T_{w} \mapsto (-1)^{l(w)} q_{w} T_{w^{-1}}^{-1}.
\]
The quadratic relation
\[
(T_{s} + 1)(T_{s} - q_{s}) = 0
\]
implies 
\[
T_{s}^{-1} = \frac{
T_{s} - (q_{s} - 1)
}{
q_{s}
},
\]
hence
\[
(-1)^{l(s)} q_{s} T_{s^{-1}}^{-1} = q_{s} -1 - T_{s}
\]
for $s \in S_0$.
The element $q_{s} -1 - T_{s}$ satisfies the quadratic relation
\[
\left((q_{s} -1 - T_{s}) + 1 \right)\left((q_{s} -1 - T_{s}) - q_{s}\right) = (T_{s} -q_{s})(T_{s} + 1) = 0.
\]
Moreover, for $w_1, w_2 \in W_0$ with $l(w_1 w_2) = l(w_1) + l(w_2)$, we have $l(w_{2}^{-1} w_{1}^{-1}) = l(w_{2}^{-1}) + l(w_{1}^{-1})$, and 
\begin{align*}
\left(
(-1)^{l(w_1)} q_{w_1} T_{w_{1}^{-1}}^{-1}
\right)
\left(
(-1)^{l(w_2)} q_{w_2} T_{w_{2}^{-1}}^{-1}
\right)
&= (-1)^{l(w_1) + l(w_2)} q_{w_1} q_{w_2} T_{w_{1}^{-1}}^{-1} T_{w_{2}^{-1}}^{-1}\\
&= (-1)^{l(w_1 w_2)} q_{w_1 w_2} (T_{w_{2}^{-1}} T_{w_{1}^{-1}})^{-1}\\
&= (-1)^{l(w_1 w_2)} q_{w_1 w_2} T_{w_{2}^{-1}w_{1}^{-1}}^{-1}\\
&= (-1)^{l(w_1 w_2)} q_{w_1 w_2} T_{(w_1 w_2)^{-1}}^{-1}.
\end{align*}
In particular, the map $\iota_{0}$ is compatible with the braid relations of $\mathcal{H}(W_{0}, q)$.
Hence, we obtain that the map $\iota_{0}$ is an algebra automorphism of $\mathcal{H}(W_{0}, q)$.

We define
\[
\iota := \iota_{Y} \otimes \iota_{0} \colon \mathcal{H} \rightarrow \mathcal{H}. 
\]
\begin{lemma}
The map $\iota$ is an algebra isomorphism of $\mathcal{H}$.
\end{lemma}
\begin{proof}
It suffices to show that $\iota$ is compatible with the relation~\eqref{bernsteinrel} of Definition~\ref{affinehecke}.
Let $\alpha \in \Delta$ and $y \in Y$.
We write $s = s_{\alpha}$ for simplicity.
The equation that we have to prove is
\[
\theta_{-y}(q_{s} -1 - T_{s}) - (q_{s} -1 - T_{s})\theta_{-s(y)} =
\left(
(\mathbf{q}^{\lambda(\alpha)} -1) + \theta_{\alpha^{\vee}}(
\mathbf{q}^{(\lambda(\alpha) + \lambda^{*}(\alpha))/2} - \mathbf{q}^{(\lambda(\alpha) - \lambda^{*}(\alpha))/2}
)
\right)
\frac{
\theta_{-y} - \theta_{-s(y)}
}{
\theta_{0} - \theta_{2\alpha^{\vee}}
},
\]
that is proved by the calculations below:
\begin{align*}
\text{(LHS)}
&= (q_{s}-1)(\theta_{-y} - \theta_{-s(y)}) - (\theta_{-y}T_{s} - T_{s}\theta_{-s(y)})\\
&= (q_{s}-1)(\theta_{-y} - \theta_{-s(y)}) - 
\left(
(\mathbf{q}^{\lambda(\alpha)} -1) + \theta_{-\alpha^{\vee}}(
\mathbf{q}^{(\lambda(\alpha) + \lambda^{*}(\alpha))/2} - \mathbf{q}^{(\lambda(\alpha) - \lambda^{*}(\alpha))/2}
)
\right)
\frac{
\theta_{-y} - \theta_{-s(y)}
}{
\theta_{0} - \theta_{-2\alpha^{\vee}}
}\\
&=\left(
(\mathbf{q}^{\lambda(\alpha)} - 1)(\theta_{0} - \theta_{-2\alpha^{\vee}}) - (\mathbf{q}^{\lambda(\alpha)} -1) - \theta_{-\alpha^{\vee}}(
\mathbf{q}^{(\lambda(\alpha) + \lambda^{*}(\alpha))/2} - \mathbf{q}^{(\lambda(\alpha) - \lambda^{*}(\alpha))/2}
)
\right)
\frac{
\theta_{-y} - \theta_{-s(y)}
}{
\theta_{0} - \theta_{-2\alpha^{\vee}}
}\\
&= 
-\left(
\theta_{-2 \alpha^{\vee}}(\mathbf{q}^{\lambda(\alpha)} - 1) +  \theta_{-\alpha^{\vee}}(
\mathbf{q}^{(\lambda(\alpha) + \lambda^{*}(\alpha))/2} - \mathbf{q}^{(\lambda(\alpha) - \lambda^{*}(\alpha))/2}
)
\right)
\frac{
\theta_{-y} - \theta_{-s(y)}
}{
\theta_{0} - \theta_{-2\alpha^{\vee}}
}\\
&=
-\left(
(\mathbf{q}^{\lambda(\alpha)} -1) + \theta_{\alpha^{\vee}}(
\mathbf{q}^{(\lambda(\alpha) + \lambda^{*}(\alpha))/2} - \mathbf{q}^{(\lambda(\alpha) - \lambda^{*}(\alpha))/2}
)
\right)
\theta_{-2\alpha^{\vee}}
\frac{
\theta_{-y} - \theta_{-s(y)}
}{
\theta_{0} - \theta_{-2\alpha^{\vee}}
}\\
&=
\left(
(\mathbf{q}^{\lambda(\alpha)} -1) + \theta_{\alpha^{\vee}}(
\mathbf{q}^{(\lambda(\alpha) + \lambda^{*}(\alpha))/2} - \mathbf{q}^{(\lambda(\alpha) - \lambda^{*}(\alpha))/2}
)
\right)
\frac{
\theta_{-y} - \theta_{-s(y)}
}{
\theta_{0} - \theta_{2\alpha^{\vee}}
}\\
&= \text{(RHS)}.
\end{align*}
\end{proof}
\section{Homomorphism between Affine Hecke algebras of type $\widetilde{A_1}$}
\label{Homomorphism between Affine Hecke algebras of type}
Let 
\[
\mathcal{R} =
(X, R=\{\pm \alpha\}, Y, R^{\vee} = \{\pm \alpha^{\vee}\}, \Delta=\{\alpha\})
\]
and 
\[
\mathcal{R}' =
(X', R'=\{\pm \alpha'\}, Y', (R')^{\vee} = \{\pm (\alpha')^{\vee}\}, \Delta'=\{\alpha'\})
\]
be based root data.
Here, we do not assume that $X$ is of rank $1$.
Let $\lambda(\alpha), \lambda^{*}(\alpha), \lambda'(\alpha'), (\lambda^{*})'(\alpha')$ be positive real numbers.
We assume that $\lambda(\alpha) = \lambda^{*}(\alpha)$ unless $\alpha \in 2X$, and $\lambda'(\alpha') = (\lambda^{*})'(\alpha')$ unless $\alpha' \in 2X'$.
We fix $\mathbf{q} > 1$, and write
\[
\begin{cases}
q_{1}&= \mathbf{q}^{\lambda(\alpha)}, \\
q_{0} &= \mathbf{q}^{\lambda^{*}(\alpha)}, \\
q'_{1} &= \mathbf{q}^{\lambda'(\alpha')}, \\
q'_{0} &= \mathbf{q}^{(\lambda^{*})'(\alpha')}.
\end{cases}
\]
We note that $q_{1}, q_{0}, q'_{1}, q'_{0} >1$.
We define the affine Hecke algebra $\mathcal{H} = \mathcal{H}(\mathcal{R}, \lambda, \lambda^{*}, \mathbf{q})$ associated with $\mathcal{R}, \lambda, \lambda^{*}, \mathbf{q}$ and the affine Hecke algebra $\mathcal{H}' = \mathcal{H}(\mathcal{R}', \lambda', (\lambda^{*})', \mathbf{q})$ associated with $\mathcal{R}', \lambda', (\lambda^{*})', \mathbf{q}$.
We write $s = s_{\alpha}$ and $s' = s_{\alpha'}$.
\begin{proposition}
\label{keypropositionappendixinteger}
Let 
\[
I \colon \mathcal{H} \rightarrow \mathcal{H}'
\]
be an algebra homomorphism such that
\[
I(T_{s}) = c' \cdot \theta_{k (\alpha')^{\vee}} T_{s'} + b'
\]
for some $c' \in \mathbb{C}^{\times}$, $k \in \mathbb{Z}$, and $b' \in \mathbb{C}[Y']$.
We also assume that 
\[
I(\theta_{\alpha^{\vee}}) = c \cdot \theta_{n (\alpha')^{\vee}}
\]
for some $c \in \mathbb{C}^{\times}$ and positive integer $n$.
Then, we obtain that
\[
\begin{cases}
I(T_{s}) &= \theta_{k (\alpha')^{\vee}/2} \cdot T_{s'} \cdot \theta_{- k (\alpha')^{\vee}/2}, \\
I(\theta_{\alpha^{\vee}}) &= \theta_{(\alpha')^{\vee}}, \\
q_{1} &= q_{1}', \\
q_{0} &= q_{0}'
\end{cases} 
\]
if $k$ is even, and
\[
\begin{cases}
I(T_{s}) &= \theta_{(k-1) (\alpha')^{\vee}/2} \cdot T_{s', 0} \cdot \theta_{- (k-1) (\alpha')^{\vee}/2}, \\
I(\theta_{\alpha^{\vee}}) &= \theta_{(\alpha')^{\vee}}, \\
q_{1} &= q'_{0}, \\
q_{0} &= q'_{1}
\end{cases} 
\]
if $k$ is odd, where $T_{s', 0} \in \mathcal{H}'$ is defined as
\[
T_{s', 0} = (q'_{1})^{-1/2} \cdot (q'_{0})^{1/2} \cdot \left(
\theta_{(\alpha')^{\vee}} T_{s'} - (q'_{1} - 1) \theta_{(\alpha')^{\vee}}
\right).
\]
\end{proposition}
\begin{remark}
Before proving Proposition~\ref{keypropositionappendixinteger}, we explain about $T_{s', 0}$.
%Let $(\alpha')^{\vee} \in (R')^{\vee}$ denote the coroot corresponding to $\alpha' \in R'$.
We define an affine root system $R'_{\aff}$ on
$
\mathbb{R} \cdot (\alpha')^{\vee}
$
as
\[
R'_{\aff} = \{
\pm \alpha' + k \mid k \in \mathbb{Z}
\}.
\]
Then, $R'_{\aff}$ has a basis
\[
B'_{\aff} = \{
\alpha', 1 - \alpha'
\}.
\]
We write $s' = s_{\alpha'}$ and $s'_{0} = s_{1-\alpha'}$.
Let $W_{\aff}\left(R'_{\aff}\right)$ denote the affine Weyl group of $R'_{\aff}$ and we define
\[
S'_{\aff} = \{
s', s'_{0}
\}.
\]
We also define a parameter function 
\[
q' \colon S'_{\aff} \rightarrow \mathbb{R}_{>1}
\]
as
\[
q'_{s'} = q'_{1} 
\]
and 
\[
q'_{s'_{0}} = q'_{0}.
\]
We define the Iwahori-Hecke algebra $\mathcal{H}\left(W_{\aff}\left(R'_{\aff}\right), q'\right)$ associated with the affine Coxeter system $\left(W_{\aff}\left(R'_{\aff}\right), S'_{\aff} \right)$ and the parameter function $q'$.
Since
\[
s'_{0} \circ s' = (\alpha')^{\vee},
\]
we have
\[
T_{s'_{0}} T_{s'} = T_{(\alpha')^{\vee}}.
\]
According to Theorem~\ref{theorem1.8ofsol21}, $\mathcal{H}\left(W_{\aff}\left(R'_{\aff}\right), q'\right)$ can be regarded as an affine Hecke algebra, and the element $T_{(\alpha')^{\vee}}$ corresponds to the element
\[
q_{(\alpha')^{\vee}}^{1/2} \cdot \theta_{(\alpha')^{\vee}} = (q'_{s'_{0}})^{1/2} \cdot (q'_{s'})^{1/2} \cdot \theta_{(\alpha')^{\vee}} = (q'_{1})^{1/2} \cdot (q'_{0})^{1/2} \cdot \theta_{(\alpha')^{\vee}}.
\]
Hence, the element 
\[
T_{s'_{0}} = T_{(\alpha')^{\vee}} (T_{s'})^{-1} \in \mathcal{H}\left(W_{\aff}\left(R'_{\aff}\right), q'\right)
\] corresponds to
\[
(q'_{1})^{1/2} \cdot (q'_{0})^{1/2} \cdot \theta_{(\alpha')^{\vee}} (T_{s'})^{-1} = (q'_{1})^{-1/2} \cdot (q'_{0})^{1/2} \cdot \left(
\theta_{(\alpha')^{\vee}} T_{s'} - (q'_{1} - 1) \theta_{(\alpha')^{\vee}}
\right),
\]
that is the element $T_{s', 0}$ we defined above.
\end{remark}
\begin{proof}[Proof of Proposition~\ref{keypropositionappendixinteger}]
We compare the images of the both sides of the equation
\begin{align}
\label{bernstein}
\theta_{\alpha^{\vee}} T_{s} - T_{s} \theta_{-\alpha^{\vee}} = \left(
(q_{1}-1) + \theta_{-\alpha^{\vee}}q_{1}^{1/2}(q_{0}^{1/2} - q_{0}^{-1/2})
 \right)\frac{\theta_{\alpha^{\vee}} - \theta_{-\alpha^{\vee}}}{\theta_{0} - \theta_{-2\alpha^{\vee}}} 
\end{align}
via $I$.
The image of the left hand side of \eqref{bernstein} via $I$ is equal to
\begin{align*}
& \quad c \cdot \theta_{n (\alpha')^{\vee}} (c' \cdot \theta_{k (\alpha')^{\vee}} T_{s'} + b') -  c^{-1} (c' \cdot \theta_{k (\alpha')^{\vee}} T_{s'} + b') \theta_{-n (\alpha')^{\vee}}\\
&=
c' \cdot \theta_{k (\alpha')^{\vee}} (c \cdot \theta_{n (\alpha')^{\vee}} T_{s'} - c^{-1} \cdot T_{s'} \theta_{-n (\alpha')^{\vee}}) + \left(\text{element of $\mathbb{C}[Y']$}\right) \\
&=
c' (c-c^{-1}) \cdot T_{s'} \theta_{-(n+k) (\alpha')^{\vee}} + \left(\text{element of $\mathbb{C}[Y']$}\right).
\end{align*}
On the other hand, the image of the right hand side of \eqref{bernstein} via $I$ is contained in $\mathbb{C}[Y']$.
Hence, we obtain that 
\[
c-c^{-1} = 0,
\]
that is,
\begin{align}
\label{c=pm1}
c= \pm 1.
\end{align}
Then, the image of the left hand side of \eqref{bernstein} via $I$ is equal to
\begin{multline}
\label{lhs}
cc' \cdot \theta_{k (\alpha')^{\vee}} \left(
\theta_{n (\alpha')^{\vee}}T_{s'} - T_{s'} \theta_{- n (\alpha')^{\vee}}
\right)
+ cb' (\theta_{n (\alpha')^{\vee}} - \theta_{-n (\alpha')^{\vee}})\\
=
cc' \cdot \theta_{k (\alpha')^{\vee}} \left(
(q_{1}'-1) + \theta_{-(\alpha')^{\vee}}(q_{1}')^{1/2}((q_{0}')^{1/2} - (q_{0}')^{-1/2})
\right)\frac{\theta_{n(\alpha')^{\vee}} - \theta_{-n(\alpha')^{\vee}}}{\theta_{0} - \theta_{-2(\alpha')^{\vee}}} 
+ cb' (\theta_{n (\alpha')^{\vee}} - \theta_{-n (\alpha')^{\vee}}).
\end{multline}
On the other hand, the right hand side of \eqref{bernstein} is equal to
\[
(q_{1} - 1)\theta_{\alpha^{\vee}} + q_{1}^{1/2}(q_{0}^{1/2} - q_{0}^{-1/2}),
\]
and the image of this term via $I$ is equal to
\begin{align}
\label{rhs}
c (q_{1} - 1)\theta_{n (\alpha')^{\vee}} + q_{1}^{1/2}(q_{0}^{1/2} - q_{0}^{-1/2}).
\end{align}
Comparing \eqref{lhs} with \eqref{rhs}, we obtain that
\begin{multline}
\label{comparison}
cc' \cdot \theta_{k (\alpha')^{\vee}} \left(
(q_{1}'-1) + \theta_{-(\alpha')^{\vee}}(q_{1}')^{1/2}((q_{0}')^{1/2} - (q_{0}')^{-1/2})
\right)\frac{\theta_{n (\alpha')^{\vee}} - \theta_{-n (\alpha')^{\vee}}}{\theta_{0} - \theta_{-2 (\alpha')^{\vee}}} 
+ cb' (\theta_{n (\alpha')^{\vee}} - \theta_{-n (\alpha')^{\vee}})\\
=
c (q_{1} - 1)\theta_{n (\alpha')^{\vee}} + q_{1}^{1/2}(q_{0}^{1/2} - q_{0}^{-1/2}).
\end{multline}
Let $\mathbb{C}(\mathbb{Z}(R')^{\vee})$ denote the quotient field of $\mathbb{C}[\mathbb{Z}(R')^{\vee}]$.
Then, according to equation~\eqref{comparison}, we obtain that
\[
b' \in \mathbb{C}[Y'] \cap \mathbb{C}(\mathbb{Z}(R')^{\vee}) = \mathbb{C}[\mathbb{Z} (R')^{\vee}].
\]
We regard equation~\eqref{comparison} as the equation
\begin{multline}
\label{comparison2}
cc' \cdot T^{k} \left(
(q_{1}'-1) + (q_{1}')^{1/2}((q_{0}')^{1/2} - (q_{0}')^{-1/2})T^{-1}
\right)\frac{T^{n} - T^{-n}}{1 - T^{-2}} 
+ cb' (T^{n} - T^{-n})\\
=
c (q_{1} - 1)T^{n} + q_{1}^{1/2}(q_{0}^{1/2} - q_{0}^{-1/2})
\end{multline}
in the ring of Laurent polynomials $\mathbb{C}[T, T^{-1}]$ via the isomorphism
\[
\mathbb{C}[\mathbb{Z} (R')^{\vee}] \rightarrow \mathbb{C}[T, T^{-1}]
\]
defined as
\[
\theta_{(\alpha')^{\vee}} \mapsto T.
\]
If $n \ge 3$, we can take $\zeta \in \mathbb{C}^{\times}$ such that
\[
\zeta^{n} = c = 1/c, \, \zeta^{2} \neq 1.
\]
Substituting $T = \zeta$ to \eqref{comparison2}, we obtain
\[
0 = (q_{1} - 1) + q_{1}^{1/2}(q_{0}^{1/2} - q_{0}^{-1/2}).
\]
However, since we are assuming that $q_1, q_0 >1$, the right hand side of the equation above is positive. Thus, it cannot happen.

Next, we consider the case $n=2$.
Substituting $T = \sqrt{-1}$ to \eqref{comparison2}, we obtain
\[
0 = -c(q_{1} - 1) + q_{1}^{1/2}(q_{0}^{1/2} - q_{0}^{-1/2}).
\]
If $c=-1$, the argument above implies a contradiction. 
Hence, we obtain $c=1$ and
\[
0 = -(q_{1} - 1) + q_{1}^{1/2}(q_{0}^{1/2} - q_{0}^{-1/2}) = -(q_{1}^{1/2} - q_{0}^{1/2})(q_{1}^{1/2} + q_{0}^{-1/2}).
\]
Thus, we have $q_1 = q_0$.
Then, equation~\eqref{comparison2} becomes
\[
c' \cdot T^{k} \left(
(q_{1}'-1) + (q_{1}')^{1/2}((q_{0}')^{1/2} - (q_{0}')^{-1/2})T^{-1}
\right)\frac{T^{2} - T^{-2}}{1 - T^{-2}} 
+ b' (T^{2} - T^{-2})
=
(q_{1} - 1)(T^{2}+1). 
\]
Dividing both sides by $T^2+1$, we obtain
\[
c' \cdot T^{k} \left(
(q_{1}'-1) + (q_{1}')^{1/2}((q_{0}')^{1/2} - (q_{0}')^{-1/2})T^{-1}
\right)
+ b'(1-T^{-2})
= q_{1} + 1.
\]
Substituting $T=1$ and $T= -1$ to both sides, we obtain
\[
\begin{cases}
c' \left(
(q_{1}'-1) + (q_{1}')^{1/2}((q_{0}')^{1/2} - (q_{0}')^{-1/2})
\right) &= q_{1} +1, \\
c' \cdot (-1)^{k} \left(
(q_{1}'-1) - (q_{1}')^{1/2}((q_{0}')^{1/2} - (q_{0}')^{-1/2})
\right) &= q_{1} +1,
\end{cases}
\]
that imply
\[
\begin{cases}
c'(q_{1}'-1) &= q_{1} +1, \\
(q_{0}')^{1/2} - (q_{0}')^{-1/2} &= 0
\end{cases}
\]
if $k$ is even, and
\[
\begin{cases}
q'_{1} - 1 &= 0, \\
c' \cdot (q_{1}')^{1/2}((q_{0}')^{1/2} - (q_{0}')^{-1/2}) &= q_{1} +1
\end{cases}
\]
if $k$ is odd.
However, since we are assuming that $q_{1}' , q'_{0}> 1$, 
\[
(q_{0}')^{1/2} - (q_{0}')^{-1/2} > 0
\]
and 
\[
q'_{1} -1>0.
\]
Thus, both cannot happen either.

Now, we conclude that $n=1$. Then, the equation~\eqref{comparison2} becomes
\begin{align}
\label{comparisonn=1}
cc' \cdot T^{k} \left(
(q_{1}'-1)T + (q_{1}')^{1/2}((q_{0}')^{1/2} - (q_{0}')^{-1/2})
\right)
+ cb' (T - T^{-1})
=
c (q_{1} - 1)T + q_{1}^{1/2}(q_{0}^{1/2} - q_{0}^{-1/2}).
\end{align}
Substituting $T=1$ and $T= -1$ to both sides, we obtain
\[
\begin{cases}
cc' \left(
(q_{1}'-1) + (q_{1}')^{1/2}((q_{0}')^{1/2} - (q_{0}')^{-1/2})
\right)
&=
c (q_{1} - 1) + q_{1}^{1/2}(q_{0}^{1/2} - q_{0}^{-1/2}),\\
cc' \cdot (-1)^{k} \left(
-(q_{1}'-1) + (q_{1}')^{1/2}((q_{0}')^{1/2} - (q_{0}')^{-1/2})
\right)
&=
-c (q_{1} - 1) + q_{1}^{1/2}(q_{0}^{1/2} - q_{0}^{-1/2}),
\end{cases}
\]
that imply
\begin{align}
\label{qcomparison}
\begin{cases}
cc' (q_{1}' - 1) &= c(q_{1} - 1), \\
cc' \cdot (q_{1}')^{1/2}((q_{0}')^{1/2} - (q_{0}')^{-1/2}) &= q_{1}^{1/2}(q_{0}^{1/2} - q_{0}^{-1/2})
\end{cases}
\end{align}
if $k$ is even,
and
\begin{align}
\label{qcomparison2}
\begin{cases}
cc' (q'_{1} - 1) &= q_{1}^{1/2}(q_{0}^{1/2} - q_{0}^{-1/2}), \\
cc' \cdot (q_{1}')^{1/2}((q_{0}')^{1/2} - (q_{0}')^{-1/2}) &= c(q_{1} - 1)
\end{cases}
\end{align}
if $k$ is odd.

First, we assume that $k$ is even.
Substituting equation~\eqref{qcomparison} to \eqref{comparisonn=1}, we obtain that
\begin{multline}
c c' \cdot (q_{1}'-1)T^{k+1} + cc' \cdot (q_{1}')^{1/2}((q_{0}')^{1/2} - (q_{0}')^{-1/2}) T^{k} + + cb' (T - T^{-1})\\
= c c' (q_{1}'-1)T + c c' \cdot (q_{1}')^{1/2}((q_{0}')^{1/2} - (q_{0}')^{-1/2}),
\end{multline}
hence 
\[
-c c' \cdot \left(
 (q_{1}'-1)T + (q_{1}')^{1/2}((q_{0}')^{1/2} - (q_{0}')^{-1/2})
\right)\left(
T^{k} - 1
\right) = cb' (T - T^{-1}). 
\]
Since $k$ is even, $T - T^{-1}$ divides $T^{k} - 1$, and we
have
\begin{align*}
b' &= -c' \cdot \left(
 (q_{1}'-1)T + (q_{1}')^{1/2}((q_{0}')^{1/2} - (q_{0}')^{-1/2})
\right)\frac{T^{k} - 1}{T - T^{-1}}\\
&= -c' \cdot T^{k/2} \cdot \left(
 (q_{1}'-1) + T^{-1} (q_{1}')^{1/2}((q_{0}')^{1/2} - (q_{0}')^{-1/2})
\right)\frac{T^{k/2} - T^{-k/2}}{1 - T^{-2}}\\
&=  -c' \cdot \theta_{k (\alpha')^{\vee}/2} \cdot \left(
(q'_{1}-1) + \theta_{- (\alpha')^{\vee}}(q'_{1})^{1/2}((q'_{0})^{1/2} - (q'_{0})^{-1/2})
 \right)\frac{\theta_{k (\alpha')^{\vee}/2} - \theta_{-k (\alpha')^{\vee}/2}}{\theta_{0} - \theta_{-2 (\alpha')^{\vee}}}\\
&= -c' \cdot \theta_{k (\alpha')^{\vee}/2} \cdot \left(
\theta_{k (\alpha')^{\vee}/2} T_{s'} - T_{s'} \theta_{- k (\alpha')^{\vee}/2}
\right)\\
&= -c' \cdot \theta_{k (\alpha')^{\vee}} T_{s'} + c' \cdot \left(\theta_{k (\alpha')^{\vee}/2} \cdot T_{s'} \cdot \theta_{-k (\alpha')^{\vee}/2}\right).
\end{align*}
Hence, we have
\[
I(T_{s}) = c' \cdot \theta_{k (\alpha')^{\vee}} T_{s'} + b' = c' \cdot \left(\theta_{k (\alpha')^{\vee}/2} \cdot T_{s'} \cdot \theta_{-k (\alpha')^{\vee}/2}\right).
\]
Since $I$ is an algebra homomorphism, and $T_s$ satisfies the quadratic relation
\[
T_{s}^{2} = (q_{1} - 1)T_{s} + q_{1},
\]
we obtain
\begin{align}
\label{quadrel}
\left(
c' \cdot \left(\theta_{k (\alpha')^{\vee}/2} \cdot T_{s'} \cdot \theta_{-k (\alpha')^{\vee}/2}\right)
\right)^2 = (q_{1} - 1) \left(
c' \cdot \left(\theta_{k (\alpha')^{\vee}/2} \cdot T_{s'} \cdot \theta_{-k (\alpha')^{\vee}/2}\right)
\right) + q_{1}.
\end{align}
On the other hand, the quadratic relation
\[
T_{s'}^{2} = (q_{1}' - 1)T_{s'} + q_{1}'
\]
of $T_{s'}$ implies that
\[
\left(
\theta_{k (\alpha')^{\vee}/2} \cdot T_{s'} \cdot \theta_{-k (\alpha')^{\vee}/2}
\right)^{2} = (q'_{1} - 1)\left(
\theta_{k (\alpha')^{\vee}/2} \cdot T_{s'} \cdot \theta_{-k (\alpha')^{\vee}/2}
\right) + q'_{1}.
\]
Substituting it to \eqref{quadrel}, we obtain
\[
(c')^{2} (q_{1}' - 1)\left(
\theta_{k (\alpha')^{\vee}/2} \cdot T_{s'} \cdot \theta_{-k (\alpha')^{\vee}/2}
\right) + (c')^{2} q_{1}' = c' (q_{1} - 1) \left(
\theta_{k (\alpha')^{\vee}/2} \cdot T_{s'} \cdot \theta_{-k (\alpha')^{\vee}/2}
\right) + q_{1}.
\]
Hence, we obtain that
\[
\begin{cases}
c' (q'_{1} -1) &= q_{1} - 1, \\
(c')^2 q'_{1} &= q_{1}
\end{cases}
\]
Combining them, we obtain
\[
0 = q'_{1} \cdot (c')^2 - (q'_{1} -1) c' -1 = (c' -1)(q'_{1} \cdot c' +1).
\]
If $q_{1}' \cdot c' +1 = 0$, we have
\[
q_{1} = (c')^2 \cdot q'_{1} = \frac{1}{q_{1}'}.
\]
However, since we are assuming that $q_{1}, q_{1}' >1$, it cannot happen.
Thus, we obtain that $c'=1$ and $q_{1} = q_{1}'$.
Substituting these equations to the second equation of \eqref{qcomparison}, we obtain
\[
c \cdot ((q_{0}')^{1/2} - (q_{0}')^{-1/2}) = q_{0}^{1/2} - q_{0}^{-1/2}.
\]
Since $c = \pm 1$ and $q_{0}, q_{0}' > 1$, we obtain $c = 1$ and $q_{0} = q_{0}'$.

We consider the case that $k$ is odd.
Substituting equation~\eqref{qcomparison2} to \eqref{comparisonn=1}, we obtain that
\begin{multline}
c c' \cdot (q_{1}'-1)T^{k+1} + cc' \cdot (q_{1}')^{1/2}((q_{0}')^{1/2} - (q_{0}')^{-1/2}) T^{k} + cb' (T - T^{-1})\\
= c c' (q_{1}'-1) + c c' \cdot (q_{1}')^{1/2}((q_{0}')^{1/2} - (q_{0}')^{-1/2})T,
\end{multline}
hence 
\[
-c c' \cdot \left(
q'_{1} -1
\right)(T^{k+1} - 1) - cc' \cdot (q_{1}')^{1/2} \left(
(q_{0}')^{1/2} - (q_{0}')^{-1/2}
\right)\left(
T^{k} - T
\right) = cb' (T - T^{-1}). 
\]
Since $k$ is odd, $T - T^{-1}$ divides $T^{k+1} - 1$ and $T^{k} - T$, and we
have
\begin{align*}
b'
&= -c' \cdot \left(
(q'_{1} - 1)\frac{T^{k+1} - 1}{T - T^{-1}} + (q_{1}')^{1/2} \left(
(q_{0}')^{1/2} - (q_{0}')^{-1/2}
\right)\frac{T^{k} - T}{T - T^{-1}}
\right)\\
&= -c' \cdot
\left(
 (q'_{1}-1)\frac{T - T^{-1}}{1 - T^{-2}} + 
T^{k+1/2} \cdot \left(
 (q_{1}'-1) + T^{-1} (q_{1}')^{1/2}((q_{0}')^{1/2} - (q_{0}')^{-1/2})
\right)\frac{T^{(k-1)/2} - T^{-(k-1)/2}}{1 - T^{-2}}
\right)\\
&= -c' \cdot
(q'_{1}-1)\frac{\theta_{(\alpha')^{\vee}} - \theta_{-(\alpha')^{\vee}}}{1 - \theta_{-2 (\alpha')^{\vee}}}\\
& \quad - c' \cdot 
\theta_{(k+1) (\alpha')^{\vee}/2} \cdot \left(
 (q_{1}'-1) + \theta_{- (\alpha')^{\vee}} (q_{1}')^{1/2}((q_{0}')^{1/2} - (q_{0}')^{-1/2})
\right)
\frac{\theta_{(k-1) (\alpha')^{\vee}/2}- \theta_{-(k-1) (\alpha')^{\vee}/2}}{1 - \theta_{-2 (\alpha')^{\vee}}}\\
&= -c' \cdot
(q'_{1} -1)\theta_{(\alpha')^{\vee}} - c' \cdot \theta_{(k+1) (\alpha')^{\vee}/2} \cdot \left(
\theta_{(k-1) (\alpha')^{\vee}/2} T_{s'} - T_{s'} \theta_{- (k-1) (\alpha')^{\vee}/2}
\right)\\
&= - c' \cdot (q'_{1} - 1) \theta_{(\alpha')^{\vee}} -c' \cdot \theta_{k (\alpha')^{\vee}} T_{s'} + c' \cdot \theta_{(k+1) (\alpha')^{\vee}/2} \cdot T_{s'} \cdot \theta_{- (k-1) (\alpha')^{\vee}/2}.
\end{align*}
Hence, we have
\begin{align*}
I(T_{s}) &= c' \cdot \theta_{k (\alpha')^{\vee}} T_{s'} + b' \\
&= c' \cdot
\left(
\theta_{(k+1) (\alpha')^{\vee}/2} \cdot T_{s'} \cdot \theta_{- (k-1) (\alpha')^{\vee}/2} - (q'_{1} - 1) \theta_{(\alpha')^{\vee}}
\right)\\
&= c' \cdot \theta_{(k-1) (\alpha')^{\vee}/2} \cdot \left(
\theta_{(\alpha')^{\vee}} T_{s'} -(q'_{1} -1) \theta_{(\alpha')^{\vee}}
\right) \cdot \theta_{-(k-1) (\alpha')^{\vee}/2}.
\end{align*}
Recall that we defined
\[
T_{s', 0} = (q'_{1})^{-1/2} \cdot (q'_{0})^{1/2} \cdot \left(
\theta_{(\alpha')^{\vee}} T_{s'} - (q'_{1} - 1) \theta_{(\alpha')^{\vee}}
\right).
\]
Thus, we have
\[
T_{s', 0} = - (q'_{1})^{-1/2} \cdot (q'_{0})^{1/2} \cdot \iota\left(
\theta_{- (\alpha')^{\vee}} T_{s'}
\right),
\]
where 
\[
\iota \colon \mathcal{H}' \rightarrow \mathcal{H}'
\]
denotes the involution defined in Appendix~\ref{An involution of an affine Hecke algebra}.
The quadratic relation of $T_{s'}$ implies the quadratic relation of $\theta_{- (\alpha')^{\vee}} T_{s'}$ as follows:
\begin{align*}
\left(
\theta_{- (\alpha')^{\vee}} T_{s'}
\right)^2 
&= \theta_{- (\alpha')^{\vee}} \cdot \left(
T_{s'} \theta_{- (\alpha')^{\vee}}
\right) \cdot T_{s'}\\
&= 
\theta_{- (\alpha')^{\vee}} \cdot \left(
\theta_{(\alpha')^{\vee}} T_{s'} - (q'_{1} - 1)\theta_{(\alpha')^{\vee}} - (q'_{1})^{1/2}((q'_{0})^{1/2} - (q'_{0})^{-1/2})
\right) \cdot T_{s'}\\
&= T_{s'}^{2} -(q'_{1} - 1)T_{s'} - (q'_{1})^{1/2}((q'_{0})^{1/2} - (q'_{0})^{-1/2}) \cdot \theta_{- (\alpha')^{\vee}} T_{s'}\\
&= - (q'_{1})^{1/2}((q'_{0})^{1/2} - (q'_{0})^{-1/2}) \cdot \theta_{- (\alpha')^{\vee}} T_{s'} + q'_{1}.
\end{align*}
Then, we obtain that $T_{s', 0}$ has the quadratic relation
\begin{align*}
T_{s', 0}^{2} &= (q'_{1})^{-1} \cdot q'_{0} \cdot \iota\left(
\theta_{- (\alpha')^{\vee}} T_{s'}
\right)^{2}\\
&=  (q'_{1})^{-1} \cdot q'_{0} \cdot \iota\left(
- (q'_{1})^{1/2}((q'_{0})^{1/2} - (q'_{0})^{-1/2}) \cdot \theta_{- (\alpha')^{\vee}} T_{s'} + q'_{1}
\right)\\
&= (q'_{0} -1) \left(
- (q'_{1})^{-1/2} \cdot (q'_{0})^{1/2} \cdot \iota\left(
\theta_{- (\alpha')^{\vee}} T_{s'}
\right)
\right) + q'_{0}\\
&= (q'_{0} -1)T_{s', 0} + q'_{0}.
\end{align*}
Let $c'' = (q'_{1})^{1/2} \cdot (q'_{0})^{-1/2} \cdot c'$.
Then, we have
\[
I(T_{s}) = c'' \cdot \theta_{(k-1) (\alpha')^{\vee}/2} \cdot T_{s', 0} \cdot \theta_{-(k-1) (\alpha')^{\vee}/2}.
\]
Since $\theta_{(k-1) (\alpha')^{\vee}/2} \cdot T_{s', 0} \cdot \theta_{-(k-1) (\alpha')^{\vee}/2}$ satisfies the same quadratic relation as $T_{s', 0}$, we have
\begin{align*}
&\quad c''(q_{1} - 1) \left(
\theta_{(k-1) (\alpha')^{\vee}/2} \cdot T_{s', 0} \cdot \theta_{-(k-1) (\alpha')^{\vee}/2}
\right) + q_{1} \\
&= I \left(
(q_{1} - 1)T_{s} + q_{1}
\right)\\
&= I(T_{s}^{2})\\
&= I(T_{s})^{2}\\
&= \left(
c'' \cdot \theta_{(k-1) (\alpha')^{\vee}/2} \cdot T_{s', 0} \cdot \theta_{-(k-1) (\alpha')^{\vee}/2}
\right)^2\\
&= (c'')^2 (q'_{0} -1) \left(
\theta_{(k-1) (\alpha')^{\vee}/2} \cdot T_{s', 0} \cdot \theta_{-(k-1) (\alpha')^{\vee}/2}
\right) + (c'')^2 q'_{0}.
\end{align*}
Hence, we obtain that
\[
\begin{cases}
q_{1} - 1 &= c'' (q'_{0} -1), \\
q_{1} &= (c'')^2 q'_{0}
\end{cases}
\]
Combining them, we obtain
\[
0 = q'_{0} \cdot (c'')^2 - (q'_{0} -1) c'' -1 = (c'' -1)(q'_{0} \cdot c'' +1).
\]
If $q_{0}' \cdot c'' +1 = 0$, we have
\[
q_{1} = (c'')^2 \cdot q'_{0} = \frac{1}{q_{0}'}.
\]
However, since we are assuming that $q_{1}, q'_{0} >1$, it cannot happen.
Thus, we obtain that $c''=1$ and $q_{1} = q'_{0}$.
We also have $c' = (q'_{1})^{-1/2} \cdot (q'_{0})^{1/2}$.
Substituting them to the first equation of \eqref{qcomparison2}, we obtain
\[
c \cdot ((q_{1}')^{1/2} - (q_{1}')^{-1/2}) = (q_{0}^{1/2} - q_{0}^{-1/2}).
\]
Since $c = \pm 1$ and $q_{0}, q_{1}' > 1$, we obtain $c = 1$ and $q_{0} = q'_{1}$.
\end{proof}

We will generalize Proposition~\ref{keypropositionappendixinteger} a bit.
\begin{lemma}
\label{keypropositionappendixhalfinteger}
There is no algebra homomorphism
\[
I \colon \mathcal{H} \rightarrow \mathcal{H}'
\]
such that
\[
I(T_{s}) = c' \cdot \theta_{k (\alpha')^{\vee}} T_{s'} + b'
\]
for some $c' \in \mathbb{C}^{\times}$, $k \in (1/2) \cdot \mathbb{Z}$, and $b' \in \mathbb{C}[Y']$, and
\[
I(\theta_{\alpha^{\vee}}) = c \cdot \theta_{n (\alpha')^{\vee}}
\]
for some $c \in \mathbb{C}^{\times}$ and positive half-integer $n$.
\end{lemma}
\begin{proof}
Since $n$ is a half-integer, we have $(\alpha')^{\vee}/2 \in Y'$.
Hence, we obtain that $\alpha' \not \in 2X'$ that implies $\lambda'(\alpha') = (\lambda^*)'(\alpha')$ and $q'_{1} = q'_{0}$.
Thus, equation~\eqref{comparison} in the proof of Proposition~\ref{keypropositionappendixinteger} becomes
\begin{align*}
cc' \cdot \theta_{k (\alpha')^{\vee}}
(q_{1}'-1) \frac{\theta_{n (\alpha')^{\vee}} - \theta_{-n (\alpha')^{\vee}}}{\theta_{0} - \theta_{-(\alpha')^{\vee}}} 
+ cb' (\theta_{n (\alpha')^{\vee}} - \theta_{-n (\alpha')^{\vee}})
=
c (q_{1} - 1)\theta_{n (\alpha')^{\vee}} + q_{1}^{1/2}(q_{0}^{1/2} - q_{0}^{-1/2})
\end{align*}
in this case.
We also note that equation~\eqref{c=pm1} holds in this case too.
We regard it as an equation in the ring of Laurent polynomials $\mathbb{C}[S, S^{-1}]$ via the isomorphism
\[
\mathbb{C}[\mathbb{Z} (R')^{\vee}/2] \rightarrow \mathbb{C}[S, S^{-1}]
\]
defined as
\[
\theta_{(\alpha')^{\vee}/2} \mapsto S,
\]
and obtain
\begin{comment}
\begin{multline}
\label{halfintegercomparison}
cc' \cdot S^{2k} \left(
(q_{1}'-1) + (q_{1}')^{1/2}((q_{0}')^{1/2} - (q_{0}')^{-1/2}) Y^{-2}
\right)\frac{S^{2n} - S^{-2n}}{1 - Y^{-4}} 
+ cb' (S^{2n} - S^{-2n})\\
=
c (q_{1} - 1) S^{2n} + q_{1}^{1/2}(q_{0}^{1/2} - q_{0}^{-1/2})
\end{multline}
Since $S= \sqrt{-1}$ is not a pole of the right hand side, it is not a pole of the left hand side either, hence we obtain
\[
\left((q_{1}')^{1/2} - (q_{0}')^{1/2}\right)
\left((q_{1}')^{1/2} + (q_{0}')^{-1/2}\right) 
= (q_{1}'-1) - (q_{1}')^{1/2}((q_{0}')^{1/2} - (q_{0}')^{-1/2}) = 0.
\]
Since $q'_{0}, q'_{1} > 1$, we obtain $q'_{1} = q'_{0}$.
Substituting it to equation~\eqref{halfintegercomparison}, we have
\end{comment}
\begin{align}
\label{modifiedhalfintegercomparison}
c c' \cdot S^{2k} (q'_{1} -1)\frac{S^{2n} - S^{-2n}}{1 - S^{-2}} +  cb' (S^{2n} - S^{-2n})
=
c (q_{1} - 1) S^{2n} + q_{1}^{1/2}(q_{0}^{1/2} - q_{0}^{-1/2}).
\end{align}
If $n >1$, we can take $\zeta \in \mathbb{C}^{\times}$ such that $\zeta^{2n} = c = 1/c$ and $\zeta^{2} \neq 1$.
Then, substituting $S = \zeta$ to equation~\eqref{modifiedhalfintegercomparison}, we have
\[
0 = q_{1} -1 + q_{1}^{1/2}(q_{0}^{1/2} - q_{0}^{-1/2}) >0,
\]
a contradiction.
Hence, $n = 1/2$, and we obtain
\[
c c' (q'_{1} - 1)S^{1+2k} + cb'(S - S^{-1}) = c(q_{1} - 1)S + q_{1}^{1/2}(q_{0}^{1/2} - q_{0}^{-1/2}).
\]
Substituting $S = 1$ and $S= -1$ to it, we obtain 
\[
\begin{cases}
c c' (q'_{1} -1) &= c (q_{1} -1), \\
q_{0}^{1/2} - q_{0}^{-1/2} &= 0
\end{cases}
\]
if $k$ is an integer, and
\[
\begin{cases}
c c' (q'_{1} -1) &= q_{1}^{1/2}(q_{0}^{1/2} - q_{0}^{-1/2}), \\
q_{1} - 1 &= 0
\end{cases}
\]
if $k$ is a half-integer.
However, since we are assuming that $q_{1} , q_{0}> 1$, 
\[
q_{0}^{1/2} - q_{0}^{-1/2} > 0
\]
and 
\[
q_{1} -1>0.
\]
Thus, both cannot happen either.
\end{proof}
Similarly, we can prove the following:
\begin{lemma}
\label{keypropositionappendixhalfintegerkver}
There is no algebra homomorphism
\[
I \colon \mathcal{H} \rightarrow \mathcal{H}'
\]
such that
\[
I(T_{s}) = c' \cdot \theta_{k (\alpha')^{\vee}} T_{s'} + b'
\]
for some $c' \in \mathbb{C}^{\times}$, $b' \in \mathbb{C}[Y']$, and half-integer $k$, and
\[
I(\theta_{\alpha^{\vee}}) = c \cdot \theta_{n (\alpha')^{\vee}}
\]
for some $c \in \mathbb{C}^{\times}$ and $n \in (1/2) \cdot \mathbb{Z}_{>0}$.
\end{lemma}
\begin{proof}
Since $k$ is a half-integer, we have $(\alpha')^{\vee}/2 \in Y'$, hence $q'_{1} = q'_{0}$ in this case too.
Then, we obtain equation~\eqref{modifiedhalfintegercomparison} in the proof of Lemma~\ref{keypropositionappendixhalfinteger}.
If $n$ is a half-integer, the claim follows from Proposition~\ref{keypropositionappendixhalfinteger}.
Hence, we may assume that $n$ is an integer.
Then, substituting $S =1$ and $S= -1$ to equation~\eqref{modifiedhalfintegercomparison}, we have
\[
q_{1}'-1 = 0,
\]
a contradiction.
\end{proof}
Now, we obtain a generalization of Proposition~\ref{keypropositionappendixinteger}.
\begin{corollary}
\label{generalizationkeypropositionappendixinteger}
Let 
\[
I \colon \mathcal{H} \rightarrow \mathcal{H}'
\]
be an algebra homomorphism such that
\[
I(T_{s}) = c' \cdot \theta_{k (\alpha')^{\vee}} T_{s'} + b'
\]
for some $c' \in \mathbb{C}^{\times}$, $k \in (1/2) \cdot \mathbb{Z}$, and $b' \in \mathbb{C}[Y']$.
We also assume that 
\[
I(\theta_{\alpha^{\vee}}) = c \cdot \theta_{n (\alpha')^{\vee}}
\]
for some $c \in \mathbb{C}^{\times}$ and $n \in (1/2) \cdot \mathbb{Z}_{>0}$.
Then, we obtain that $k$ is an integer, and
\[
\begin{cases}
I(T_{s}) &= \theta_{k (\alpha')^{\vee}/2} \cdot T_{s'} \cdot \theta_{- k (\alpha')^{\vee}/2}, \\
I(\theta_{\alpha^{\vee}}) &= \theta_{(\alpha')^{\vee}}, \\
q_{1} &= q_{1}', \\
q_{0} &= q_{0}'
\end{cases} 
\]
if $k$ is even, and
\[
\begin{cases}
I(T_{s}) &= \theta_{(k-1) (\alpha')^{\vee}/2} \cdot T_{s', 0} \cdot \theta_{- (k-1) (\alpha')^{\vee}/2}, \\
I(\theta_{\alpha^{\vee}}) &= \theta_{(\alpha')^{\vee}}, \\
q_{1} &= q'_{0}, \\
q_{0} &= q'_{1}
\end{cases} 
\]
if $k$ is odd
\end{corollary}
\begin{proof}
According to Lemma~\ref{keypropositionappendixhalfinteger} and Lemma~\ref{keypropositionappendixhalfintegerkver}, $k$ and $n$ cannot be half-integers.
Then, the claim follows from Proposition~\eqref{keypropositionappendixinteger}.
\end{proof}
\printindex
\providecommand{\bysame}{\leavevmode\hbox to3em{\hrulefill}\thinspace}
\providecommand{\MR}{\relax\ifhmode\unskip\space\fi MR }
\providecommand{\MRhref}[2]{%
  \href{http://www.ams.org/mathscinet-getitem?mr=#1}{#2}
}
\providecommand{\href}[2]{#2}

\end{document}